\pdfoutput=1
\documentclass[a4paper,twoside]{article}

\usepackage{amssymb}
\usepackage{amsmath}
\usepackage{amsthm}

\usepackage{mathrsfs}

\usepackage[T1]{fontenc}
\usepackage{lmodern}

\usepackage{natbib}
\usepackage{url}

\usepackage{microtype}

\usepackage{esint}                

\newtheoremstyle{citing}
  {3pt}
  {3pt}
  {\itshape}
  {}
  {\bfseries}
  {.}
  {.5em}
  {\thmnote{#3}}
\theoremstyle{citing}
\newtheorem*{citing}{}

\theoremstyle{definition}

\swapnumbers
\theoremstyle{plain}

\newtheorem{theorem}{Theorem}[section]
\newtheorem{lemma}[theorem]{Lemma}
\newtheorem{corollary}[theorem]{Corollary}

\theoremstyle{remark}
\newtheorem{remark}[theorem]{Remark}
\newtheorem{example}[theorem]{Example}

\theoremstyle{definition}
\newtheorem*{intro_definition}{Definition}
\newtheorem{definition}[theorem]{Definition}
\newtheorem{miniremark}[theorem]{}

\newcounter{counter1}
\newcounter{counter2}

\newcommand{\Var}{\mathbf{V}}     
\newcommand{\RVar}{\mathbf{RV}}   
\newcommand{\IVar}{\mathbf{IV}}   
\newcommand{\var}{\mathbf{v}}     

\newcommand{\Lp}[1]{\mathbf{L}_{#1}}

\newcommand{\Lploc}[1]{\mathbf{L}_{#1}^{\mathrm{loc}}}

\newcommand{\Sob}[3]{\mathbf{W}_{#1}^{{#2},{#3}}}

\newcommand{\trunc}{\mathbf{T}}

\newcommand{\ccspace}[1]{\mathscr{K}(#1)}

\newcommand{\nat}{\mathscr{P}}

\newcommand{\integers}{\integers}

\newcommand{\rel}{\mathbf{R}}

\newcommand{\complex}{\mathbf{C}}

\newcommand{\grass}[2]{\mathbf{G}(#1,#2)}

\newcommand{\orthproj}[2]{\mathbf{O}^\ast({#1},{#2})}

\newcommand{\oball}[2]{\mathbf{U}(#1,#2)}

\newcommand{\cball}[2]{\mathbf{B}(#1,#2)}

\newcommand{\density}{\boldsymbol{\Theta}}

\newcommand{\unitmeasure}[1]{\boldsymbol{\alpha}(#1)}

\newcommand{\besicovitch}[1]{\boldsymbol{\beta}(#1)}

\newcommand{\isoperimetric}[1]{\boldsymbol{\gamma}(#1)}

\newcommand{\id}[1]{\mathbf{1}_{#1}}

\newcommand{\weakD}{\mathbf{D}}

\newcommand{\derivative}[2]{{#1}\,\weakD{#2}}

\newcommand{\boundary}[2]{{#1}\,\partial{#2}}

\newcommand{\ud}{\ensuremath{\,\mathrm{d}}}

\DeclareMathOperator{\with}{:}
\newcommand{\classification}[3]{{#1} \cap \{ {#2} \with {#3} \}}
\newcommand{\eqclassification}[3]{{(#1)} \cap \{ {#2} \with {#3} \}}

\newcommand{\project}[1]{#1_\natural}

\newcommand{\eqproject}[1]{(#1)_\natural}

\newcommand{\perpproject}[1]{#1_\natural^\perp}

\newcommand{\lIm}{[}
\newcommand{\rIm}{]}

\newcommand{\vdim}{{m}}

\newcommand{\codim}{{n-m}}
\newcommand{\adimI}{{m+1}}
\newcommand{\adim}{{n}}

\newcommand{\norm}[3]{\boldsymbol{|} #1 \boldsymbol{|}_{#2;#3}}

\newcommand{\nuqnorm}[3]{\boldsymbol{\nu}_{#1}(#2,#3)}

\newcommand{\nuqastnorm}[3]{\boldsymbol{\nu}_{#1}^\ast(#2,#3)}

\newcommand{\intertextenum}[1]{\setcounter{counter2}{\value{enumi}}\end{enumerate}#1\begin{enumerate}\setcounter{enumi}{\value{counter2}}}

\newcommand{\printRoman}[1]{\setcounter{counter1}{#1}\Roman{counter1}}

\DeclareMathOperator{\without}{\sim}

\newcommand{\restrict}{\mathop{\llcorner}}

\newcommand{\class}[1]{#1}

\newcommand{\tint}[2]{{\textstyle\int_{#1}^{#2}}}
\newcommand{\tfint}[2]{{\textstyle\fint_{#1}^{#2}}}

\newcommand{\tsum}[2]{{\textstyle\sum_{#1}^{#2}}}

\DeclareMathOperator{\card}{card}

\newcommand{\Clos}[1]{\mathop{\mathrm{Clos}}#1}

\DeclareMathOperator{\Bdry}{Bdry}

\newcommand{\measureball}[2]{{#1}\,{#2}}

\DeclareMathOperator{\Nor}{Nor}     
\DeclareMathOperator{\Tan}{Tan}     
\DeclareMathOperator{\spt}{spt}     
\DeclareMathOperator{\im}{im}       
\DeclareMathOperator{\Int}{Int}     
\DeclareMathOperator{\diam}{diam}   
\DeclareMathOperator{\Lip}{Lip}     
\DeclareMathOperator{\grad}{grad}   
\DeclareMathOperator{\discr}{discr} 
\DeclareMathOperator{\trace}{trace} 
\DeclareMathOperator{\dmn}{dmn}     
\DeclareMathOperator{\dist}{dist}   
\DeclareMathOperator{\Hom}{Hom}     
\DeclareMathOperator{\Div}{div}     
\DeclareMathOperator{\sign}{sign}   
\DeclareMathOperator{\Span}{span}   
\DeclareMathOperator{\ap}{ap}       
\DeclareMathOperator*{\aplim}{\mathrm{ap}\, \lim}   

\newcommand{\Lpnorm}[3]{{#1}_{({#2})}({#3})}
\newcommand{\eqLpnorm}[3]{{(#1)}_{({#2})}({#3})}

\newcommand{\SWnorm}[3]{\mathbf{SW}_{#1} ( {#2}, {#3} )}
\newcommand{\SWloc}[1]{\mathbf{SW}_{#1}^{\mathrm{loc}}}
\newcommand{\SW}[1]{\mathbf{SW}_{#1}}
\newcommand{\VSobnorm}[3]{\mathbf{W}_{#1} ( {#2}, {#3} )}
\newcommand{\VSobloc}[1]{\mathbf{W}_{#1}^{\mathrm{loc}}}
\newcommand{\VSob}[1]{\mathbf{W}_{#1}}

\newcommand{\verify}{\textbf{To verify! }}

\newcommand{\optional}[1]{}      

\theoremstyle{definition}
\swapnumbers
\newtheorem{step}{Step}	
\newenvironment{proofinsteps}{\begin{proof}\setcounter{step}{0}}{\end{proof}}

\begin{document}


\title{A sharp lower bound on the mean curvature integral with critical power
for integral varifolds}
\author{Ulrich Menne}
\date{August 20, 2014}
\maketitle


\part{Weakly differentiable functions} \label{part:weakly}
\begin{abstract}
	The present paper is intended to provide the basis for the study of
	weakly differentiable functions on a rectifiable varifold with locally
	bounded first variation. The concept proposed here is defined by means
	of integration by parts identities for certain compositions with
	smooth functions. In this class the idea of ``zero boundary values''
	is realised using the relative perimeter of level sets.  Results
	include a variety of Sobolev Poincar{\'e} type embeddings, embeddings
	into spaces of continuous and sometimes H{\"o}lder continuous
	functions, pointwise differentiability results both of approximate and
	integral type as well as a coarea formula.
	
	As applications the finiteness of the geodesic distance associated to
	varifolds with suitable integrability of the mean curvature and a
	characterisation of curvature varifolds are obtained.
\end{abstract}
\addcontentsline{toc}{section}{\numberline{}Introduction}
\section*{Introduction} \subsubsection*{Overview}
The main purpose of this paper is to present a concept of weakly
differentiable functions on nonsmooth ``surfaces'' in Euclidean space with
arbitrary dimension and codimension arising in variational problems involving
the area functional. The model used for such surfaces are rectifiable
varifolds whose first variation with respect to area is representable by
integration (that is, in the terminology of Simon \cite[39.2]{MR756417},
rectifiable varifolds with locally bounded first variation). This includes
area minimising rectifiable currents\footnote{See Allard
\cite[4.8\,(4)]{MR0307015}.}, in particular perimeter minimising ``Caccioppoli
sets'', or almost every time slice of Brakke's mean curvature
flow\footnote{See Brakke \cite[\S 4]{MR485012}.} just as well as surfaces
occuring in diffuse interface models\footnote{See for instance Hutchinson and
Tonegawa \cite{MR1803974} or R\"oger and Tonegawa \cite{MR2377408}.} or image
restoration models\footnote{See for example Ambrosio and Masnou
\cite{MR1959769}.}. The envisioned concept should be defined without reference
to an approximation by smooth functions and it should be as broad as possible
so as to still allow for substantial positive results.

In order to integrate well the first variation of the varifold into the
concept of weakly differentiable function, it appeared necessary to provide an
entirely new notion rather than to adapt one the many concepts of weakly
differentiable functions which have been invented for different purposes. For
instance, to study the support of the varifold as metric space with its
geodesic distance, stronger conditions on the first variation are needed, see
Section \ref{sec:geodesic_distance}.

\subsection*{Description of results} \subsubsection*{Setup and basic results}
To describe the results obtained, consider the following set of hypotheses;
the notation is explained in Section \ref{sec:notation}.
\begin{citing} [General hypothesis]
	Suppose $\vdim$ and $\adim$ are positive integers, $\vdim \leq \adim$,
	$U$ is an open subset of $\rel^\adim$, $V$ is an $\vdim$ dimensional
	rectifiable varifold in $U$ whose first variation $\delta V$ is
	representable by integration.
\end{citing}
The study of weakly differentiable functions is closely related to the study
of connectedness properties of the underlying space or varifold. Therefore it
is instructive to begin with the latter.
\begin{intro_definition}
	If $V$ is as in the general hypothesis, it is called
	\emph{indecomposable} if and only if there is no $\| V \| + \| \delta
	V \|$ measurable set $E$ such that $\| V \| ( E ) > 0$, $\| V \| ( U
	\without E ) > 0$ and $\delta ( V \restrict E \times
	\grass{\adim}{\vdim} ) = ( \delta V ) \restrict E$.
\end{intro_definition}
The basic theorem involving this notion is the following.
\begin{citing} [Decomposition theorem, see \ref{thm:decomposition}]
	If $V$ is as in the general hypothesis, then there exists a countable
	disjointed collection $F$ of Borel sets whose union is $U$ such that
	$V \restrict E \times \grass{\adim}{\vdim}$ is nonzero and
	indecomposable and $\delta ( V \restrict E \times
	\grass{\adim}{\vdim}) = ( \delta V ) \restrict E$ whenever $E \in F$.
\end{citing}
Employing the following definition, the Borel partition $F$ of $U$ is required
to consist of members whose distributional $V$ boundary vanishes and which
cannot be split nontrivially into smaller Borel sets with that property.
\begin{intro_definition}
	If $V$ is as in the general hypothesis and $E$ is $\| V \| + \| \delta
	V \|$ measurable, then the \emph{distributional $V$ boundary of $E$}
	is defined by
	\begin{gather*}
		\boundary{V}{E} = ( \delta V ) \restrict E - \delta ( V
		\restrict E \times \grass{\adim}{\vdim} ) \in \mathscr{D}' (
		U, \rel^\adim ).
	\end{gather*}
\end{intro_definition}
However, unlike the decomposition into connected components for topological
spaces, the preceding decomposition is nonunique in an essential way. In fact,
the varifold corresponding to the union of three lines in $\rel^2$ meeting at
the origin at equal angles may also be decomposed into two ``Y-shaped''
varifolds each consisting of three rays meeting at equal angles, see
\ref{remark:nonunique_decomposition}.

The seemingly most natural definition of weakly differentiable function would
be to require an integration by parts identity involving the first variation
of the varifold. However, the resulting class of real valued functions is
neither stable under truncation of its members nor does a coarea formula hold,
see \ref{remark:too_big_sobolev} and
\ref{remark:no_coarea_ineq_for_too_big_sobolev}. Therefore, instead one
requires an integration by parts identity for the composition of the function
in question with smooth functions whose derivative has compact support, see
\ref{def:v_weakly_diff}. The resulting class of functions of $\rel^l$ valued
functions is denoted by
\begin{gather*}
	\trunc (V,\rel^l).
\end{gather*}
Whenever it is defined the weak derivative $\derivative{V}{f} (a)$ is a linear
map of $\rel^\adim$ into $\rel^l$. In this process it seems natural not to
require local integrability of $\derivative{V}{f}$ but only
\begin{gather*}
	\tint{K \cap \{ z \with |f(z)| \leq s \}}{} | \derivative{V}{f} | \ud
	\| V \| < \infty
\end{gather*}
whenever $K$ is a compact subset of $U$ and $0 \leq s < \infty$; this is in
analogy with definition of the space ``$\mathscr{T}_{\mathrm{loc}}^{1,1} ( U
)$'' introduced by B{\'e}nilan, Boccardo, Gallou{\"e}t, Gariepy, Pierre and
V{\'a}zquez in \cite[p.~244]{MR1354907} for the case of Lebesgue measure, see
\ref{remark:comparison_trunc_spaces}. In both cases the letter ``T'' in the
name of the space stands for truncation.

Stability properties under composition (for example truncation) then follow
readily, see \ref{lemma:basic_v_weakly_diff} and \ref{lemma:comp_lip}. Also,
it is evident that members of $\trunc(V,\rel^l)$ may be defined on compenents
separately, see \ref{thm:tv_on_decompositions}. The space $\trunc
(V,\rel^l)$ is stable under addition of a locally Lipschitzian function but it
is not closed with respect to addition in general, see
\ref{thm:addition}\,\eqref{item:addition:add} and \ref{example:star}. A
similar statement holds for multiplication of functions, see
\ref{thm:addition}\,\eqref{item:addition:mult} and \ref{example:star}.
Adopting an axiomatic viewpoint, consider a class of functions which
satisfies the following three conditions:
\begin{enumerate}
	\item It admits members constant on components.
	\item It is closed under truncation.
	\item It is closed under addition.
\end{enumerate}
Then this class necessarily contains characteristic functions with
nonvanishing distributional boundary representable by integration, see
\ref{example:axioms}. Clearly, these characteristic functions should not
belong to a class of ``weakly differentiable'' functions but rather to a class
of functions of ``bounded variation''. The reason for this behaviour is the
afore-mentioned nonuniqueness of decompositions of varifolds.

Much of the further development of the theory rests on the following
observation.
\begin{citing} [Coarea formula, functional analytic form, see
\ref{remark:associated_distribution} and \ref{thm:tv_coarea}]
	Suppose $V$ satisfies the general hypothesis, $f \in \trunc(V)$, and
	$E(t) = \{ z \with f(z) > t \}$ for $t \in \rel$.

	Then there holds
	\begin{gather*}
		\tint{}{} \left < \psi (z,f(z)), \derivative{V}{f}(z) \right >
		\ud \| V \| z = \tint{}{} \boundary{V}{E(t)} ( \psi (\cdot, t
		) ) \ud \mathscr{L}^1 t, \\
		\tint{}{} g (z,f(z)) | \derivative{V}{f}(z) | \ud \| V \| z =
		\tint{}{} \tint{}{} g (z,t) \ud \| \boundary{V}{E(t)} \| z \ud
		\mathscr{L}^1 t.
	\end{gather*}
	whenever $\psi \in \mathscr{D} (U \times \rel, \rel^\adim )$ and
	$g : U \times \rel \to \rel$ is a continuous function with compact
	support.
\end{citing}
The conclusions in fact hold for a natural wider class of functions, see
\ref{lemma:push_on_product}.

An example is proved by the notion of ``zero boundary values'' on a relatively
open part $G$ of the boundary of $U$ for functions on $\trunc(V, \rel^l)$ is
defined, see \ref{def:trunc_g}. The definition is based on properties of the
distributional $V$ boundaries of the sets $E(t) = \{ z \with |f(z)| > t \}$
for $\mathscr{L}^1$ almost all $0 < t < \infty$. Essentially, one requires
that the first variation of the injection of the varifold $V \restrict E(t)
\times \grass{\adim}{\vdim}$ into $\rel^\adim \without B$ only consists of
parts induced from $(\delta V) \restrict E(t)$ and the distributional $V$
boundary of $E(t)$. The definition gives rise to the subspaces $\trunc_G ( V,
\rel^l )$ of $\trunc (V, \rel^l )$ with $\trunc_\varnothing ( V, \rel^l ) =
\trunc ( V, \rel^l )$, see \ref{remark:trunc}. The space satisfies useful
truncation and closure properties, see \ref{lemma:trunc_tg} and
\ref{lemma:closeness_tg}. Moreover, under a natural integrability hypothesis,
the multiplication of a member of $\trunc_G (V,\rel^l)$ by a Lipschitzian
function belongs to $\trunc_G (V,\rel^l)$, see \ref{thm:mult_tg}. Whereas the
usage of level sets in the definition of $\trunc_G (V,\rel)$ is taylored to
work nicely in the proofs of embedding results in Section
\ref{sec:embeddings}, the stability property under multiplication requires a
more delicate proof in turn.

\subsubsection*{Embedding results and structural results} To proceed to deeper
results on functions in $\trunc (V,\rel^l)$, the usage of the isoperimetric
inequality for varifolds seems inevitable. The latter works best under the
following additional hypothesis.
\begin{citing} [Density hypothesis]
	Suppose $V$ is as in the general hypothesis and satisfies
	\begin{gather*}
		\density^\vdim ( \| V \|, z ) \geq 1 \quad \text{for $\| V \|$
		almost all $z$}.
	\end{gather*}
\end{citing}
The key to prove effective versions of Sobolev Poincar{\'e} embedding results
is the following theorem.
\begin{citing} [Relative isoperimetric inquality, see \ref{corollary:rel_iso_ineq}]
	Suppose $V$ satisfies the general hypothesis and the density
	hypothesis, $E$ is $\|V \| + \| \delta V \|$ measurable, $1 \leq Q
	\leq M < \infty$, $\adim \leq M$, $\Gamma =
	\Gamma_{\ref{thm:rel_iso_ineq}} ( M )$, $0 < r < \infty$,
	\begin{gather*}
		\| V \| (E) \leq (Q-M^{-1}) \unitmeasure{\vdim} r^\vdim, \quad
		\| V \| ( \classification{E}{z}{ \density^\vdim ( \| V \|, z )
		< Q } ) \leq \Gamma^{-1} r^\vdim,
	\end{gather*}
	and $A = \{ z \with \oball{z}{r} \subset U \}$.

	Then
	\begin{gather*}
		\| V \| ( E \cap A)^{1-1/\vdim} \leq \Gamma \big ( \|
		\boundary{V}{E} \| (U) + \| \delta V \| ( E ) \big ),
	\end{gather*}
	where $0^0 = 0$.
\end{citing}
In the special case that $\boundary{V}{E} = 0$, $\| \delta V \| (E) = 0$ and
\begin{gather*}
	\density^\vdim ( \| V \|, z ) \geq Q \quad \text{for $\| V \|$ almost
	all $z$}
\end{gather*}
the varifold $V \restrict E \times \grass{\adim}{\vdim}$ is stationary and the
support of its weight measure, $\spt \| V \|$, cannot intersect $A$ by the
monotonicity identity. The value of the theorem lies in quantifying this
behaviour. Much of the usefulness of the result for the purposes of the
present paper stems from the fact that values of $Q$ larger than $1$ are
permitted. This allows to effectively apply the result in neighbourhoods of
points where the density function $\density^\vdim ( \| V \|, \cdot )$ has a
value larger than $1$ and is approximately continuous. The case $Q=1$ is
conceptionally contained in a result of Hutchinson \cite[Theorem 1]{MR1066398}
which treats Lipschitzian functions.

If the set $E$ satisfies a ``zero boundary value'' condition, see
\ref{miniremark:zero_boundary}, on a relatively open subset $G$ of $\Bdry U$,
the conclusion may be sharped by replacing $A = \{ z \with \oball{z}{r}
\subset U \}$ by
\begin{gather*}
	A'= \{ z \with \oball{z}{r} \subset \rel^\adim \without B \}, \quad
	\text{where $B = ( \Bdry U ) \without G$}.
\end{gather*}

In order to state a version of the resulting Sobolev Poincar{\'e}
inequalities, recall a less known notation from Federer's treatise on
geometric measure theory, see \cite[2.4.12]{MR41:1976}.
\begin{intro_definition}
	Suppose $\phi$ measures $X$ and $f$ is a $\phi$ measurable function
	with values in some Banach space $Y$.

	Then one defines $\phi_{(p)} (f)$ for $1 \leq p \leq \infty$ by the
	formulae
	\begin{gather*}
		\phi_{(p)}(f) = ( \tint{}{} |f|^p \ud \phi )^{1/p}
		\quad \text{in case $1 \leq p < \infty$}, \\
		\phi_{(\infty)}(f) = \inf \big \{ t \with \text{$t \geq 0$,
		$\phi ( \{ x \with |f(x)| > t \} ) = 0$} \big \}.
	\end{gather*}
\end{intro_definition}
In comparison to the more common notation $\| f \|_{\mathbf{L}_p ( \phi, Y
)}$ it puts the measure in focus and avoids iterated subscripts.

\begin{citing} [Sobolev Poincar{\'e} inequality, zero median version, see
\ref{thm:sob_poin_summary}\,\eqref{item:sob_poin_summary:interior:p=1}]
	Suppose $1 \leq M < \infty$.

	Then there exists a positive, finite number $\Gamma$ with the
	following property.

	If $\vdim$, $\adim$, $U$, $V$ satisfy the general hypothesis and the
	density hypothesis, $\adim \leq M$, $f \in \trunc ( V, \rel^l )$, $1
	\leq Q \leq M$, $0 < r < \infty$, $E = U \cap \{ z \with f(z) \neq 0
	\}$,
	\begin{gather*}
		\| V \| ( E ) \leq ( Q-M^{-1} ) \unitmeasure{\vdim} r^\vdim,
		\\
		\| V \| ( E \cap \{ z \with \density^\vdim ( \| V \|, z ) < Q
		\} ) \leq \Gamma^{-1} r^\vdim, \\
	 	\text{$\beta = \infty$ if $\vdim = 1$}, \quad
		\text{$\beta = \vdim/(\vdim-1)$ if $\vdim > 1$},
	\end{gather*}
	$A = \{ z \with \oball{z}{r} \subset U \}$, then
	\begin{gather*}
		\eqLpnorm{\| V \| \restrict A }{\beta}{f} \leq \Gamma \big (
		\Lpnorm{\| V \|}{1}{ \derivative{V}{f} } + \Lpnorm{\| \delta V
		\|}{1} {f} \big ).
	\end{gather*}
\end{citing}
Again, the main improvement is the applicability with values $Q > 1$. If $f$
additionally belongs to $\trunc_G(V,\rel^l)$, then $A$ may be by $A'$
as in the relative isoperimetric inequality. Not surprisingly, there also
exists a version of the Sobolev inequality for members of $\trunc_{\Bdry U} (
V, \rel )$. For Lipschitzian functions a Sobolev inequality was obtained by
Allard \cite[7.1]{MR0307015} and Michael and Simon \cite[Theorem
2.1]{MR0344978} for varifolds and by Federer in \cite[\S 2]{MR0388226} for
rectifiable currents which are absolutely minimising with respect to a
positive, parametric integrand.
\begin{citing} [Sobolev inequality, see \ref{thm:sob_poin_summary}\,\eqref{item:sob_poin_summary:global:p=1}]
	Suppose $1 \leq M < \infty$.

	Then there exists a positive, finite number $\Gamma$ with the
	following property.

	If $\vdim$, $\adim$, $U$, $V$, satisfy the general hypothesis and the
	density hypothesis, $\adim \leq M$, $f \in \trunc_{\Bdry U} ( V,
	\rel^l )$,
	\begin{gather*}	
		E = U \cap \{ z \with f(z) \neq 0 \}, \quad \| V \| ( E ) <
		\infty, \\
	 	\text{$\beta = \infty$ if $\vdim = 1$}, \quad
		\text{$\beta = \vdim/(\vdim-1)$ if $\vdim > 1$},
	\end{gather*}
	then
	\begin{gather*}
		\Lpnorm{\| V \|}{\beta}{f} \leq \Gamma \big ( \Lpnorm{\| V
		\|}{1}{ \derivative{V}{f} } + \Lpnorm{\| \delta V \|}{1} {f}
		\big ).
	\end{gather*}
\end{citing}
Coming back the Sobolev Poincar{\'e} inequalities, one may also establish a
version with several ``medians''. The number of medians needed is controlled
by the total weight measure of the varifold a natural way.
\begin{citing} [Sobolev Poincar{\'e} inequality, several medians, see
\ref{thm:sob_poincare_q_medians}\,\eqref{item:sob_poincare_q_medians:p=1}]
	Suppose $1 \leq M < \infty$.

	Then there exists a positive, finite number $\Gamma$ with the
	following property.

	If $l$ is a positive integer, $\vdim$, $\adim$, $U$, $V$ satisfy the
	general hypothesis and the density hypothesis, $\sup \{ l, \adim \}
	\leq M$, $f \in \trunc (V,\rel^l)$, $1 \leq Q \leq M$, $P$ is a
	positive integer, $0 < r < \infty$,
	\begin{gather*}
		\| V \| ( U ) \leq ( Q-M^{-1} ) (P+1) \unitmeasure{\vdim}
		r^\vdim, \\
		\| V \| ( \{ z \with \density^\vdim ( \| V \|, z ) < Q \} )
		\leq \Gamma^{-1} r^\vdim, \\
	 	\text{$\beta = \infty$ if $\vdim = 1$}, \quad
		\text{$\beta = \vdim/(\vdim-1)$ if $\vdim > 1$},
	\end{gather*}	
	and $A = \{ z \with \oball{z}{r} \subset U \}$, then there exists a
	subset $Y$ of $\rel^l$ such that $1 \leq \card Y \leq P$ and
	\begin{gather*}
		\eqLpnorm{\| V \| \restrict A}{\beta}{f_Y} \leq \Gamma
		P^{1/\beta} \big ( \Lpnorm{\| V \|}{1}{\derivative{V}{f}} + \|
		\delta V \|(f_Y) \big),
	\end{gather*}
	where $f_Y (z) = \dist (f(z),Y)$ for $z \in \dmn f$.
\end{citing}
If $l = 1$, the approach of Hutchinson, see \cite[Theorem 3]{MR1066398},
carries over unchanged and, in fact, yields a somewhat stronger statement, see
\ref{thm:sob_poin_several_med}\,\eqref{item:sob_poin_several_med:p=1}. If $l >
1$ and $P > 1$ the selection procedure for $Y$ is significantly more delicate.

In order to precisely state the next result, the concept of approximate
tangent vectors and approximate differentiability (see
\cite[3.2.16]{MR41:1976}) will be recalled.
\begin{intro_definition}
	Suppose $\phi$ measures an open subset $U$ of a normed vectorspace
	$X$, $a \in U$, and $m$ is a positive integer.

	Then $\Tan^m ( \phi, a )$ denotes the closed cone of \emph{$(\phi,m)$
	approximate tangent vectors} at $a$ consisting of all $v \in X$ such
	that
	\begin{gather*}
		\density^{\ast m} ( \phi \restrict \mathbf{E}
		(a,v,\varepsilon), a ) > 0 \quad \text{for every $\varepsilon
		> 0$}, \\
		\text{where $\mathbf{E}(a,v,\varepsilon) = X \cap \{ x
		\with \text{$|r(x-a) -v| < \varepsilon$ for some $r > 0$} \}$}.
	\end{gather*}
	Moreover, if $X$ is an inner product space, then the cone of
	\emph{$(\phi,m)$ approximate normal vectors} at $a$ is defined to be
	\begin{gather*}
		\Nor^m ( \phi, a ) = X \cap \{ u \with \text{$u \bullet v \leq
		0$ for $v \in \Tan^m ( \phi, a )$} \}.
	\end{gather*}
\end{intro_definition}
If $V$ is an $\vdim$ dimensional rectifiable varifold in an open subset $U$ of
$\rel^\adim$, then at $\| V \|$ almost all $a$, $T = \Tan^\vdim ( \| V \|, a
)$ is an $\vdim$ dimensional plane such that
\begin{gather*}
	r^{-\vdim} \tint{}{} f(r^{-1}(z-a)) \ud \| V \| z \to \density^\vdim (
	\| V \|, a ) \tint{T}{} f \ud \mathscr{H}^\vdim \quad \text{as
	$r \to 0+$}
\end{gather*}
whenever $f : U \to \rel$ is a continuous function with compact support.
However, at individual points the requirement that $\Tan^\vdim ( \|V \|, a )$
forms an $\vdim$ dimensional plane differs from the condition that $\|V\|$
admits an $\vdim$ dimensional approximate tangent plane in the sense of
\cite[11.8]{MR756417}.
\begin{intro_definition}
	Suppose $\phi$, $U$, $a$ and $m$ are as in the preceding definition
	and $f$ maps a subset of $X$ into another normed vectorspace $Y$.

	Then $f$ is called \emph{$(\phi,m)$ approximately differentiable} at
	$a$ if and only if there exist $\eta \in Y$ and a continuous linear
	map $\zeta : X \to Y$ such that
	\begin{gather*}
		\density^\vdim ( \phi \restrict X \without \{ x \with
		|f(x)-\eta-\zeta (x-a)| \leq \varepsilon |x-a| \}, a ) = 0
		\quad \text{for every $\varepsilon > 0$}.
	\end{gather*}
	In this case $\zeta | \Tan^\vdim ( \phi, a )$ is unique and it is
	called the \emph{$(\phi,m)$ approximate differential} of $f$ at $a$,
	denoted
	\begin{gather*}
		(\phi,m) \ap Df(a).
	\end{gather*}
\end{intro_definition}
Moreover, the following abbreviation will be convenient.
\begin{intro_definition}
	Whenever $S$ is an $\vdim$ dimensional plane in $\rel^\adim$, the
	orthogonal projection of $\rel^\adim$ onto $S$ will be denoted by
	$\project{S}$.
\end{intro_definition}

The Sobolev Poincar{\'e} inequality with one median, that is with $P=1$, and a
suitably chosen number $Q$ is the key to prove the following structural result
for weakly diffentiable functions.
\begin{citing} [Approximate differentiability, see \ref{thm:approx_diff}]
	Suppose $\vdim$, $\adim$, $U$, and $V$ satisfy the general hypothesis
	and the density hypothesis, and $f \in \trunc (V,\rel^l)$.

	Then $f$ is $( \| V \|, \vdim )$ approximately differentiable with
	\begin{gather*}
		\derivative{V}{f} (a) = ( \| V \|, \vdim ) \ap Df(a) \circ
		\project{\Tan^\vdim ( \| V \|, a )}
	\end{gather*}
	at $\| V \|$ almost all $a$.
\end{citing}
The preceding assertion consists of two parts. Firstly, it yields that
\begin{gather*}
	\derivative{V}{f}(a) | \Nor^\vdim ( \| V \|, a ) = 0 \quad \text{for
	$\| V \|$ almost all $a$};
\end{gather*}
a property that is not required by the definition. Secondly, it asserts that
\begin{gather*}
	\derivative{V}{f}(a) | \Tan^\vdim ( \| V \|, a ) = ( \| V \|, \vdim )
	\ap D f(a) \quad \text{for $\|V \|$ almost all $a$}.
\end{gather*}
This is somewhat similar to the situation for the generalised mean curvature
of an integral $\vdim$ varifold where it was first obtained that its
tangential compenent vanishes by Brakke in \cite[5.8]{MR485012} and secondly
that its normal component is induced by approximate quantities by the author
in \cite[4.8]{snulmenn.c2}.

\begin{citing} [Differentiability in Lebesgue spaces, see
\ref{thm:diff_lebesgue_spaces}\,\eqref{item:diff_lebesgue_spaces:m>1=p}]
	Suppose $\vdim$, $\adim$, $U$, and $V$ satisfy the general hypothesis
	and the density hypothesis, and $f \in \trunc (V,\rel^l)$.

	If $\vdim > 1$, $\beta = \vdim/(\vdim-1)$, and $f \in \Lploc{1} ( \|
	\delta V \|, \rel^l )$, then
	\begin{gather*}
		\lim_{r \to 0+} r^{-\vdim} \tint{\cball ar}{} ( | f(z)-f(a)-
		\derivative Vf (a) (z-a) | / | z-a| )^\beta \ud \| V \| z = 0
	\end{gather*}
	for $\| V \|$ almost all $a$.
\end{citing}
The result is derived from the approximate differentiability result mainly by
means of the zero median version of the Sobolev Poincar{\'e} inequality.
Another consequence of the approximate differentiability result is the
rectifiability of the distributional boundary allmost all superlevelsets of
weakly differentiable functions which supplements the functional analytic form
of the coarea formula.
\begin{citing} [Coarea formula, measure theoretic form, see \ref{corollary:coarea}]
	Suppose $\vdim$, $\adim$, $U$, and $V$ satisfy the general hypothesis
	and the density hypothesis, $f \in \trunc (V,\rel)$, and $E(t) = \{ z
	\with f(z) > t \}$ for $t \in \rel$.

	Then there exists an $\mathscr{L}^1$ measurable function $W$ with
	values in the weakly topologised space of $\vdim-1$ dimensional
	rectifiable varifolds in $U$ such that for $\mathscr{L}^1$ almost all
	$t$ there holds
	\begin{gather*}
		\Tan^{\vdim-1} ( \| W(t) \|, z ) = \Tan^\vdim ( \| V \|, z )
		\cap \ker \derivative{V}{f} (z) \in \grass{\adim}{\vdim-1}, \\
		\density^{\vdim-1} ( \| W(t) \|, z ) = \density^\vdim ( \| V
		\|, z )
	\end{gather*}
	for $\| W(t) \|$ almost all $z$ and
	\begin{gather*}
		\boundary{V}{E(t)}(\theta) = \tint{}{}
		|\derivative{V}{f}(z)|^{-1} \derivative{V}{f} (z)(\theta(z))
		\ud \| W(t) \| z \quad \text{for $\theta \in \mathscr{D}
		(U,\rel^\adim )$}.
	\end{gather*}
\end{citing}
Notice that this structural result does extend to general sets whose
distributional boundary is representable by integration, see
\ref{remark:no_rectifiable_structure}. This is in contrast to the behaviour of
sets of locally finite perimeter in Euclidean space.  \subsubsection*{Critical
mean curvature} Several of the preceding estimates and structural results may
be sharped in case the generalised mean curvature satisfies an appropriate
integrability condition.
\begin{citing} [Mean curvature hypothesis]
	Suppose $\vdim$, $\adim$, $U$ and $V$ are as in the general
	hypothesis and satisfies the following condition.

	If $\vdim > 1$ then for some $h \in \Lploc{\vdim} ( \| V \|,
	\rel^\adim )$ there holds
	\begin{gather*}
		( \delta V ) ( g ) = - \tint{}{} g \bullet h
		\ud \| V \| z \quad \text{for $g \in
		\mathscr{D}(U,\rel^\adim)$}.
	\end{gather*}
	In this case $\psi$ will denote the Radon measure over $U$
	characterised by the condition $\psi (B) = \tint{B}{} |h|^\vdim \ud \|
	V \|$ whenever $B$ is a Borel subset of $U$.
\end{citing}
Clearly, if this condition is satisfied, then the function $h$ is $\| V \|$
almost equal to the generalised mean curvature vector $\mathbf{h}(V,\cdot)$ of
$V$.  The integrability exponent $\vdim$ is ``critical'' with respect to
homothetic rescaling of the varifold. Replacing it by a slightly larger number
would entail upper semicontinuity of $\density^\vdim ( \| V \|, \cdot)$ and
the applicability of Allard's regularity theory, see Allard \cite[\S
8]{MR0307015}. In contrast, replacing the exponent by a slightly smaller
number would allow for examples in which the varifold is locally highly
disconnected, see \cite[1.2]{snulmenn.isoperimetric}.

The importance of this hypothesis for the present considerations lies in the
fact that in the relative isoperimetric inequality, the summand ``$\| \delta V
\| ( U )$'' may be dropped provided $V$ satisfies additionally the mean
curvature hypothesis with $\psi (E)^{1/\vdim} \leq \Gamma^{-1}$. As a first
consequence, one obtains the following version of the Sobolev Poincar{\'e}
inequality.
\begin{citing} [Sobolev Poincar{\'e} inequality, zero median, critial mean
curvature, see
\ref{thm:sob_poin_summary}\,\eqref{item:sob_poin_summary:interior:q<m=p}\,\eqref{item:sob_poin_summary:interior:p=m<q}]
	Suppose $1 \leq M < \infty$.

	Then there exists a positive, finite number $\Gamma$ with the
	following property.

	If $\vdim$, $\adim$, $U$, $V$, and $\psi$ satisfy the general
	hypothesis, the density hypothesis and the mean curvature hypothesis,
	$\adim \leq M$, $f \in \trunc ( V, \rel )$, $1 \leq Q \leq M$, $0 < r
	< \infty$, $E = U \cap \{ z \with f(z) \neq 0 \}$,
	\begin{gather*}
		\| V \| ( E ) \leq ( Q-M^{-1} ) \unitmeasure{\vdim} r^\vdim,
		\quad \psi ( E ) \leq \Gamma^{-1}, \\
		\| V \| ( E \cap \{ z \with \density^\vdim ( \| V \|, z ) < Q
		\} ) \leq \Gamma^{-1} r^\vdim,
	\end{gather*}
	and $A = \{ z \with \oball{z}{r} \subset U \}$, then the following two
	statements hold:
	\begin{enumerate}
		\item If $1 \leq q < \vdim = p$, then
		\begin{gather*}
			\eqLpnorm{\| V \| \restrict A}{\vdim q/(\vdim-q)}{f}
			\leq \Gamma (\vdim-q)^{-1} \Lpnorm{\| V \|}{q}{
			\derivative{V}{f} }.
		\end{gather*}
		\item If $1 < p = \vdim < q \leq \infty$, then
		\begin{gather*}
			\eqLpnorm{\| V \| \restrict A}{\infty}{f} \leq
			\Gamma^{1/(1/\vdim-1/q)} \| V \| ( E)^{1/\vdim-1/q}
			\Lpnorm{\| V \|}{q}{ \derivative{V}{f} }.
		\end{gather*}
	\end{enumerate}
\end{citing}
Even if $Q = q = 1$ and $f$ and $V$ correspond to smooth objects, this
estimate appears to be new; at least, it seems not to be straightforward to
derive from Hutchinson \cite[Theorem 1]{MR1066398}.

A similar result holds for $\vdim = 1$, see
\ref{thm:sob_poin_summary}\,\eqref{item:sob_poin_summary:interior:p=m=1}.
Again, if $f$ additionally belongs to $\trunc_G(V,\rel^l)$, then $A$ may be
replaced by $A'$. Also, appropriate versions of the Sobolev inequality may be
furnished, see
\ref{thm:sob_poin_summary}\,\eqref{item:sob_poin_summary:global:p=m=1}--\eqref{item:sob_poin_summary:global:p=m<q}.
The same is true with respect to the Sobolev Poincar{\'e} inequalities with
several medians, see
\ref{thm:sob_poincare_q_medians}\,\eqref{item:sob_poincare_q_medians:p=m=1}--\eqref{item:sob_poincare_q_medians:p=m<q}
and
\ref{thm:sob_poin_several_med}\,\eqref{item:sob_poin_several_med:p=m=1}--\eqref{item:sob_poin_several_med:p=m<q}.

Before stating the stronger differentiability properties that result from the
mean curvature hypothesis, recall the definition of relative differential from
\cite[3.1.21-22]{MR41:1976}.
\begin{intro_definition}
	Suppose $X$ and $Y$ are normed vectorspaces, $S \subset X$, and $a \in
	\Clos S$, anf $f : S \to Y$.

	Then the \emph{tangent cone} of $S$ at $a$, denoted $\Tan(S,a)$, is
	the set of all $v \in X$ such that for every $\varepsilon > 0$ there
	exist $x \in S$ and $0 < r \in \rel$ with $|x-a| < \varepsilon$ and $|
	r(x-a)-v| < \varepsilon$. Moreover, $f$ is called \emph{differentiable
	relative to $S$ at $a$} if and only if there exist $\eta \in Y$ and a
	continuous linear map $\zeta : X \to Y$ such that
	\begin{gather*}
		|f(x)-\eta-\zeta(x-a) |/|x-a| \to 0 \quad \text{as $S \owns x
		\to a$}.
	\end{gather*}
	In this case $\zeta|\Tan(S,a)$ is unique and denoted $Df(a)$.
\end{intro_definition}
\begin{citing} [Differentiability in Lebesgue spaces, critical mean curvature,
see \ref{thm:diff_lebesgue_spaces}]
	Suppose $\vdim$, $\adim$, $U$, and $V$ satisfy the general hypothesis,
	the density hypothesis and the mean curvature hypothesis, $f \in
	\trunc (V,\rel^l)$, $1 \leq q \leq \infty$, and $\derivative Vf \in
	\Lploc{q} ( \| V \|, \Hom (\rel^\adim, \rel^l) )$.

	Then the following two statements hold.
	\begin{enumerate}
		\item If $q < \vdim = p$ and $\eta = \vdim q/(\vdim-q)$, then
		\begin{gather*}
			\lim_{r \to 0+} r^{-\vdim} \tint{\cball ar}{} ( |
			f(z)-f(a)- \derivative Vf (a)(z-a) | / | z-a| )^\eta
			\ud \| V \| z = 0
		\end{gather*}
		for $\| V \|$ almost all $a$.
		\item If $\vdim =1$ or $\vdim < q$, then there exists a
		subset $S$ of $U$ with $\| V \| ( U \without S ) = 0$ such
		that $f|S$ is differentiable relative to $S$ at $a$ with
		\begin{gather*}
			D (f|S) (a) = \derivative Vf (a) | \Tan^\vdim ( \| V
			\|, a) \quad \text{for $\| V \|$ almost all $a$},
		\end{gather*}
		in particular $\Tan(S,a) = \Tan^\vdim ( \| V \|, a )$ for such
		$a$.
	\end{enumerate}
\end{citing}
Considering $V$ associated to two crossing lines, it is evident that one
cannot expect a function in $\trunc(V,\rel)$ to be $\| V \|$ almost equal to a
continuous function even if $\delta V = 0$ and $\derivative{V}{f} = 0$. Yet,
in case $f$ is continuous, its modulus of continuity may be locally estimated
by its weak derivate. This estimate depends on $V$ but not on $f$ as
formulated in the next theorem.
\begin{citing} [Oscillation estimate, see \ref{thm:mod_continuity}\,\eqref{item:mod_continuity:m>1}]
	Suppose $\vdim$, $\adim$, $U$, and $V$ satisfy the general hypothesis,
	the density hypothesis and the mean curvature hypothesis, $K$ is a
	compact subset of $U$, $0 < \varepsilon < \dist ( K, \rel^\adim
	\without U )$, and $1 < \vdim < q$.

	Then there exists a positive, finite number $\Gamma$ with the
	following property.

	If $f : \spt \| V \| \to \rel^l$ is a continuous function, $f \in
	\trunc (V,\rel^l)$, and $\gamma = \sup \{ \eqLpnorm{\| V \| \restrict
	\cball{a}{\varepsilon}}{q}{ \derivative{V}{f} } \with a \in K \}$,
	then
	\begin{gather*}
		|f(z)-f(w)| \leq l^{1/2} \varepsilon \gamma \quad
		\text{whenever $z,w \in K \cap \spt \| V \|$ and $|z-w| \leq
		\Gamma^{-1}$}.
	\end{gather*}
\end{citing}
A similar result holds for $\vdim = 1$, see
\ref{thm:mod_continuity}\,\eqref{item:mod_continuity:m=1}. This theorem rests
on the fact that connected components of $\spt \| V \|$ are relatively open in
$\spt \| V \|$, see
\ref{corollary:conn_structure}\,\eqref{item:conn_structure:open}. If a
varifold $V$ satisfies the general hypothesis, the density hypothesis and the
mean curvature hypothesis then the decomposition of $\spt \| V \|$ into
connected components yields a locally finite decomposition into relatively
open and closed subsets whose distributional $V$ boundary vanishes, see
\ref{corollary:conn_structure}\,\eqref{item:conn_structure:piece}--\eqref{item:conn_structure:open}.
Moreover, any decomposition of the varifold $V$ will refine the decomposition
of the topological space $\spt \| V \|$ into connected components, see
\ref{corollary:conn_structure}\,\eqref{item:conn_structure:union}.

When the amount of total weight measure (``mass'') available excludes the
possibility of two separate sheets, the oscillation estimate may be sharpened
to yield H{\"o}lder continuity even without assuming a priori the continuity
of the function, see \ref{thm:hoelder_continuity}.

\subsubsection*{Two applications} Despite the rather weak oscillation
estimate in the general case, this estimate is still sufficient to prove that
the geodesic distance between any two points in the same connected component
of the support is finite.
\begin{citing} [Geodesic distance, see \ref{thm:conn_path_finite_length}]
	Suppose $\vdim$, $\adim$, $U$, and $V$ satisfy the general hypothesis,
	the density hypothesis and the mean curvature hypothesis, $C$ is a
	connected component of $\spt \| V \|$, and $w,z \in C$.

	Then there exist $- \infty < a \leq b < \infty$ and a Lipschitzian
	function $g : \{ t \with a \leq t \leq b \} \to \spt \| V \|$ such
	that $g(a) = w$ and $g(b) = z$.
\end{citing}
The proof follows a common pattern in theory of metric spaces, see
\ref{remark:metric_spaces}.

Finally, the presently introduced notion of weak differentiability may also be
used to reformulate the defining condition for curvature varifolds, see
\ref{def:curvature_varifold}, \ref{remark:hutchinson_reformulations}.
\begin{citing} [Characterisation of curvature varifolds, see
\ref{thm:curvature_varifolds}]
	Suppose $\vdim$ and $\adim$ are positive integers, $\vdim \leq \adim$,
	$U$ is an open subset of $\rel^\adim$, $V$ is an $\vdim$ dimensional
	integral varifold in $U$,
	\begin{gather*}
		Y = \Hom (\rel^\adim, \rel^\adim ) \cap \{ \tau \with \tau =
		\tau^\ast \}, \quad Z = U \cap \{ z \with \Tan^\vdim ( \| V
		\|, z ) \in \grass{\adim}{\vdim} \},
	\end{gather*}
	and $R : Z \to Y$ satisfies
	\begin{gather*}
		R(z) = \project{\Tan^\vdim ( \| V \|, z )} \quad
		\text{whenever $z \in Z$}.
	\end{gather*}

	Then $V$ is a curvature varifold if and only if $\| \delta V \|$ is a
	Radon measure absolutely continuous with respect to $\| V \|$ and $R
	\in \trunc(V, Y )$.
\end{citing}
The condition is clearly necessary. The key to prove its sufficiency is to
relate the mean curvature vector of $V$ to the weak differential of the
tangent plane map $R$. This may be accomplished, for instance, by means of the
approximate differentiability result, see \ref{thm:approx_diff}, applied to
$R$, in conjunction with the previously obtained second order
rectifiability result for such varifolds, see \cite[4.8]{snulmenn.c2}.
\subsection*{Possible lines of further study}
\paragraph{Sobolev spaces.} The results obtained by the author on the area
formula for the Gauss map, see \cite[Theorem 3]{snulmenn.mfo1230}, rest on
several estimates for Lipschitzian solutions of certain linear elliptic
equations on varifolds satisfying the mean curvature hypothesis. The
formulation of these estimates necessitated several ad hoc formulations of
concepts such as ``zero boundary values'' which would not seem natural from
the point of partial differential equations. In order to avoid repetition the
author decided to directly build an adequate framework for these results
rather than first publish the Lipschitzian case along with a proof of the
results announced in \cite{snulmenn.mfo1230}.\footnote{Would the author have
fully anticipated the effort needed for such a construction he might have
reconciled himself with some repetition.} The first part of such framework is
provided in the present paper. Continuing this programme, a notion of Sobolev
spaces, complete vectorspaces contained in $\trunc(V,\rel^l)$ in which locally
Lipschitzian functions are suitably dense, has already been developed but is
not included here for length considerations.
\paragraph{Functions of locally bounded variation.} It seems worthwhile to
investigate to which extend results of the present paper for weakly
differentiable functions extend to a class of ``functions of locally bounded
variation''. A possible definition is suggested in \ref{remark:bv}.
\paragraph{Intermediate conditions on the mean curvature.} The mean curvature
hypothesis may be weakened by replacing $\Lploc{\vdim}$ by $\Lploc{p}$ for
some $1 < p < \vdim$. In view of the Sobolev Poincar{\'e} type inequalities
obtained for height functions in \cite[Theorem 4.4]{snulmenn.poincare}, one
might seek for adaquate formulations of the Sobolev Poincar{\'e} inequalities
and the resulting differentiability results for functions belonging to
$\trunc(V,\rel^l)$ in these intermediate cases. This could potentially have
applications to structural results for curvature varifolds, see
\ref{remark:curv_flatness} and \ref{remark:curv_to_do}.  Additionally, the
case $p = \vdim-1$ seems to be related to the study of the geodesic distance
for indecomposable varifolds, see \ref{remark:topping}.
\paragraph{Multiple valued weakly differentiable functions.} For convergence
considerations it appears useful to extend the concept of weakly
differentiable functions to a more general class of ``multiple-valued''
functions in the spirit of Moser \cite[\S 4]{62659}, see \ref{remark:moser}.
\subsection*{Organisation of the paper} In Section \ref{sec:notation} the
notation is introduced. Sections \ref{sec:tvs} and \ref{sec:monotonicity}
provide variations of known results for later reference. In the remaining
sections the theory of weakly differentiable functions as described above is
presented.

\subsection*{Acknowledgements} The author would like to thank Dr.~S{\l}awomir
Ko{\l}asinski and Dr.~Leobardo Rosales for several conversations concerning
parts of this paper.

\section{Notation} \label{sec:notation}
The notation of \cite[\S 1,\,\S 2]{snulmenn.decay} is used. Additionally,
whenever $M$ is a submanifold of $\rel^\adim$ of class $\class{2}$ the mean
curvature of $M$ at $z \in M$ is denoted by $\mathbf{h} (M;z)$, cp. Allard
\cite[2.5\,(2)]{MR0307015}. The reader might want to recall in particular the
following less commonly used symbols: $\nat$ denoting the positive integers,
$\oball{a}{r}$ and $\cball{a}{r}$ denoting respectively the open and closed
ball with centre $a$ and radius $r$, and $\bigodot^i ( V,W )$ and $\bigodot^i
V$ denoting the vector space of all $i$ linear symmetric functions (forms)
mapping $V^i$ into $W$ and $\rel$ respectively, see \cite[2.2.6, 2.8.1,
1.10.1]{MR41:1976}.

Define $\RVar_0 ( U )$ etc.

A function $f$ mapping a subset of $\rel$ into $\overline{\rel}$ is called
\emph{nondecreasing} if and only if $f(s) \leq f(t)$ whenever $s,t \in \dmn f$
and $s \leq t$.

Whenever $f$ maps a topological space $X$ continuously into a topological
vector space $V$ the support of $f$ is given by
\begin{gather*}
	\spt f = \Clos ( \classification{X}{x}{f(x)\neq 0} ).
\end{gather*}
In this part several concepts of weakly differentiable functions on varifolds
are explored. They provide the natural setting in which to prove results on
partial differential equations on varifolds.
\section{Topological vector spaces} \label{sec:tvs}
Some basic results on topological vector spaces and Lusin spaces are gathered
here mainly from \cite{MR0426084,MR910295}. In particular, the topology on
$\mathscr{D}(U,Y)$ employed in \cite[4.1.1]{MR41:1976} is compared to the more
commonly used locally convex topology on that space, see
\ref{remark:top_comp}.
\begin{definition} [see \protect{\cite[Chapter \printRoman{2}, Definition 2,
p.~94]{MR0426084}}] \label{def:lusin}
	Suppose $X$ is a topological space.

	Then $X$ is called a \emph{Lusin space} if and only if $X$ is a
	Hausdorff topological space and there exists a complete, separable
	metric space $W$ and a continuous univalent map $f : W \to X$ whose
	image is $X$.
\end{definition}
\begin{remark} \label{remark:lusin}
	Any subset of a Lusin space is sequentially separable.
\end{remark}
\begin{definition} [see \protect{\cite[\printRoman{2}, p.~23,
def.~1]{MR910295}}]
	A topological vector space is called a \emph{locally convex space} if
	and only if there exists a fundamental system of neighbourhoods of $0$
	that are convex sets; its topology is called \emph{locally convex
	topology}.
\end{definition}
\begin{definition}
	The locally convex spaces form a category; its morphisms are the
	continuous linear maps. An inductive limit in this category is called
	\emph{strict} if and only if the morphisms of the corresponding
	inductive system are homeomorphic embeddings.
\end{definition}
\begin{remark}
	The notion of inductive limit is employed in accordance with
	\cite[p.~67--68]{MR1712872}.
\end{remark}
\begin{definition} [see \protect{\cite[\printRoman{1}, \S 2.4,
prop.~6]{MR1726779}, \cite[\printRoman{2}, p.~27,
prop.~5]{MR910295}}]
	Suppose $E$ is a set [vector space], $E_i$ are topological spaces
	[topological vector spaces] and $f_i : E_i \to E$ are functions
	[linear maps] for $i \in I$.

	Then there exists a unique topology [unique locally convex topology]
	on $E$ such that a function [linear map] $g : E \to F$ into a
	topological space [locally convex space] $F$ is continuous if and only
	if $g \circ f_i$ are continuous for $i \in I$. This topology is called
	\emph{final topology} [\emph{locally convex final topology}] on $E$
	with respect to the family $f_i$.
\end{definition}
\begin{definition} [see \protect{\cite[\printRoman{2}, p.~29, Example
\printRoman{2}]{MR910295}}\footnote{In the terminology of
\cite[\printRoman{2}, p.~29, Example \printRoman{2}]{MR910295} a locally
compact Hausdorff space is called a locally compact space.}] \label{def:kx}

	Suppose $X$ is a locally compact Hausdorff space. Consider the
	inductive system consisting of the locally convex spaces
	$\mathscr{K}(X) \cap \{ f \with \spt f \subset K \}$ with the topology
	of uniform convergence corresponding to all compact subsets $K$ of $U$
	and its inclusion maps.

	Then $\mathscr{K}(X)$ endowed with the locally convex final topology
	with respect to the inclusions of the topological vector spaces
	$\mathscr{K} (X) \cap \{ f \with \spt f \subset K \}$ is the inductive
	limit of the above system in the category of locally convex spaces.
\end{definition}
\begin{remark} [see \protect{\cite[\printRoman{2}, p.~29, Example
\printRoman{2}]{MR910295}}]
	The locally convex topology on $\mathscr{K} (X)$ is Hausdorff and
	induces the given topology on each closed subspace $\mathscr{K} ( X)
	\cap \{ f \with \spt f \subset K \}$. Moreover, the space
	$\mathscr{K}(X)^\ast$ of Daniell integrals on $\mathscr{K}(X)$ agrees
	with the space of continuous linear functionals on $\mathscr{K} (X)$.
\end{remark}
\begin{remark}
	If $K(i)$ is a sequence of compact subsets of $X$ with $K(i) \subset
	\Int K(i+1)$ for $i \in \nat$ and $X = \bigcup_{i=1}^\infty K(i)$,
	then $\mathscr{K}(X)$ is the strict inductive limit of the sequence of
	locally convex spaces $\mathscr{K}(X) \cap \{ f \with \spt f \subset
	K(i) \}$.
\end{remark}
\begin{definition} \label{def:duy}
	Suppose $U$ is an open subset of $\rel^\vdim$ and $Y$ is a Banach
	space. Consider the inductive system consisting of the locally convex
	spaces $\mathscr{D}_K (U,Y)$ with the topology induced from
	$\mathscr{E}(U,Y)$ corresponding to all compact subsets $K$ of $U$ and
	its inclusion maps.

	Then $\mathscr{D}(U,Y)$ endowed with the locally convex final topology
	with respect to the inclusions of the topological vector spaces
	$\mathscr{D}_K (U,Y)$ is the inductive limit of the above inductive
	system in the category of locally convex spaces.
\end{definition}
\begin{remark} \label{remark:basic}
	The locally convex topology on $\mathscr{D} (U,Y)$ is Hausdorff and
	induces the given topology on each closed subspace
	$\mathscr{D}_K(U,Y)$.
\end{remark}
\begin{remark} \label{remark:ind-limit}
	If $K(i)$ is a sequence of compact subsets of $U$ with $K(i) \subset
	\Int K(i+1)$ for $i \in \nat$ and $U = \bigcup_{i=1}^\infty K(i)$,
	then $\mathscr{D} (U,Y)$ is the strict inductive limit of the sequence
	of locally convex spaces $\mathscr{D}_{K(i)} ( U,Y )$. In particular,
	the convergent sequences in $\mathscr{D}(U,Y)$ are precisely the
	convergent sequences in some $\mathscr{D}_K(U,Y)$, see
	\cite[\printRoman{3}, p.~3, prop.~2; \printRoman{2}, p.~32,
	prop.~9\,(i)\,(ii); \printRoman{3}, p.~5, prop.~6]{MR910295}.
\end{remark}
\begin{remark} \label{remark:tvs}
	The locally convex topology on $\mathscr{D} (U,Y)$ is also the
	inductive limit topology in the category of topological vector spaces,
	see \cite[\printRoman{1}, p.~9, Lemma to prop.~7; \printRoman{2},
	p.~75, exerc.~14]{MR910295}.
\end{remark}
\begin{remark} \label{remark:top_comp}
	Consider the inductive system consisting of the topological spaces
	$\mathscr{D}_K (U,Y)$ corresponding to all compact subsets $K$ of $U$
	and its inclusion maps. Then $\mathscr{D}(U,Y)$ endowed with the final
	topology with respect to the inclusions of the topological vector
	spaces $\mathscr{D}_K (U,Y)$ is the inductive limit of this inductive
	system in the category of topological spaces.

	Denoting this topology by $S$ and the topology described in
	\ref{def:duy} by $T$, the following three statements hold.
	\begin{enumerate}
		\item Amongst all locally convex topologies on
		$\mathscr{D}(U,Y)$, the topology $T$ is characterised by the
		following property: A linear map from $\mathscr{D}(U,Y)$ into
		some locally convex space is $T$ continuous if and only if it
		is $S$ continuous.
		\item The $S$ closed sets are precisely the $T$ sequentially
		closed sets, see \ref{remark:ind-limit}.
		\item If $U$ is nonempty and $\dim Y > 0$, then $S$ is
		strictly finer than $T$, compare Shirai \cite[Th{\'e}or{\`e}me
		5]{MR0105613}.\footnote{In \cite{MR0105613} the terms
		\guillemotleft topologie\guillemotright{} [respectively
		\guillemotleft v{\'e}ritable-topologie\guillemotright{}] are
		employed without being fully explained. In verifying the cited
		result, the reader may replace them with operators satisfying
		conditions (a), (b) and (d) [respectively all] of the
		Kuratowski closure axioms, see \cite[p.~43]{MR0370454}.
		Finally, Shirai treats the case $U = \rel^\adim$ and $Y =
		\rel$. However, the extension to the case stated poses no
		difficulty.} In particular, in this case $S$ is not compatible
		with the vector space structure of $\mathscr{D}(U,Y)$ by
		\ref{remark:tvs}.
	\end{enumerate}
\end{remark}
\begin{theorem}
	Suppose $E$ is the strict inductive limit of an increasing sequence of
	locally convex spaces $E_i$ and the spaces $E_i$ are separable,
	complete (with respect to its topological vector space structure) and
	metrisable for $i \in \nat$.\footnote{That is, $E_i$ are separable
	Fr{\'e}chet spaces in the terminology of \cite[\printRoman{2},
	p.~24]{MR910295}.}

	Then $E$ and the dual $E'$ of $E$ endowed with the compact open
	topology are Lusin spaces and the Borel family of $E'$ is generated by
	the sets
	\begin{gather*}
		E' \cap \{ u \with u(x) < t \}
	\end{gather*}
	corresponding to $x \in D$ and $t \in \rel$ whenever $D$ is a dense
	subset of $E$.
\end{theorem}
\begin{proof}
	First, note that $E_i$ is a Lusin space for $i \in \nat$, since it may
	be metrised by a translation invariant metric by \cite[\printRoman{2},
	p.~34, prop.~11]{MR910295}. Second, note that $E$ is Hausdorff by
	\cite[\printRoman{2}, p.~32, prop.~9\,(i)]{MR910295}. Therefore $E$ is
	a Lusin space by \cite[Chapter \printRoman{2}, Corollary 2 to Theorem
	5, p.~102]{MR0426084}. Since every compact subset $K$ of $E$ is a
	compact subset of $E_i$ for some $i$ by \cite[\printRoman{3}, p.~3,
	prop.~2; \printRoman{2}, p.~32, prop.~9\,(ii); \printRoman{3}, p.~5,
	prop.~6]{MR910295}, it follows that $E'$ is a Lusin space from
	\cite[Chapter 2, Theorem 7]{MR0426084}. Defining the continuous,
	univalent map $\iota : E' \to \rel^D$ by $\iota (u) = u|D$ for $u \in
	E'$, the initial topology on $E'$ induced by $\iota$ is a Hausdorff
	topology coarser than the compact open topology, hence their Borel
	families agree by \cite[Chapter \printRoman{2}, Corollary 2 to Theorem
	4, p.~101]{MR0426084}.
\end{proof}
\begin{example} \label{example:kx_lusin}
	Suppose $X$ is a locally compact Hausdorff space which admits a
	countable base of its topology.

	Then $\mathscr{K}(X)$ with its locally convex topology (see
	\ref{def:kx}) and $\mathscr{K}(X)^\ast$ with its weak topology are
	Lusin spaces and the Borel family of $\mathscr{K}(X)^\ast$ is
	generated by the sets $\mathscr{K}(X)^\ast \cap \{ \mu \with \mu(f) <
	t \}$ corresponding to $f \in \mathscr{K}(X)$ and $t \in \rel$.
	Moreover, the topological space
	\begin{gather*}
		\mathscr{K}(X)^\ast \cap \{ \mu \with \text{$\mu$ is monotone}
		\}
	\end{gather*}
	is homeomorphic to a complete, separable metric space;\footnote{Such
	spaces are termed polish spaces in \cite[Chapter \printRoman{2},
	Definition 1]{MR0426084}.} in fact, choosing a countable dense subset
	$D$ of $\mathscr{K}(X)^+$, the image of its homeomorphic embedding
	into $\rel^D$ is closed.
\end{example}
\begin{example} \label{example:distrib_lusin}
	Suppose $U$ is an open subset of $\rel^\vdim$ and $Y$ is a separable
	Banach space.

	Then $\mathscr{D}(U,Y)$ with its locally convex topology (see
	\ref{def:duy}) and $\mathscr{D}'(U,Y)$ with its weak topology are
	Lusin spaces and the Borel family of $\mathscr{D}'(U,Y)$ is generated
	by the sets $\mathscr{D}'(U,Y) \cap \{ T \with T (\phi) < t \}$
	corresponding to $\phi \in \mathscr{D}(U,Y)$ and $t \in \rel$.
	Moreover, recalling \ref{remark:lusin} and \cite[4.1.5]{MR41:1976}, it
	follows that
	\begin{gather*}
		\big ( \mathscr{D}'(U,Y) \times \mathscr{K}(U)^\ast \big )
		\cap \{ ( T, \| T \| ) \with \text{$T$ is representable by
		integration} \}
	\end{gather*}
	is a Borel function whose domain is a Borel set (with respect to the
	weak topology on both spaces).
\end{example}
\section{Distributions on products}
The purpose of the present section is to separate functional analytic
considerations from those employing properties of the varifold or the weakly
differentiable function in deriving the various coarea type formulas in
\ref{lemma:meas_fct}, \ref{thm:tv_coarea}, \ref{thm:ap_coarea}, and
\ref{corollary:coarea}.
\begin{miniremark} \label{miniremark:extension}
	Suppose $U$ is an open subset of $\rel^\adim$, $Y$ is a Banach space
	and $S \in \mathscr{D}' ( U,Y )$.

	Then $\| S \|$ is defined to be the largest Borel regular measure on
	$U$ such that
	\begin{gather*}
		\| S \| (C) = \sup \{ S ( \theta ) \with \text{$\theta \in
		\mathscr{D} (U,Y)$ with $\spt \theta \subset C$ and $| \theta
		(z) | \leq 1$ for $z \in U$} \}
	\end{gather*}
	whenever $C$ is an open subset of $U$; existence may be shown by use
	of the techniques in \cite[2.5.14]{MR41:1976}. Moreover, $S$ is
	representable by integration if and only if $\| S \|$ is a Radon
	measure; in this case the value of the unique $\| S\|_{(1)}$
	continuous extension of $S$ to $\Lp{1} ( \| S \|, Y )$ at $\theta \in
	\Lp{1} ( \| S \|, Y )$ will be denoted $S(\theta)$.

	This concept is modelled on Allard \cite[4.2]{MR0307015}. It agrees
	agrees with the one employed in \cite[4.1.5]{MR41:1976} in case $S$ is
	representable by integration. However, in general they are different
	since with the present definition it may happen that $| S ( \theta ) |
	> \| S \| ( |\theta|)$ for some $\theta \in \mathscr{D} ( U, Y )$.
\end{miniremark}
\begin{remark} \label{remark:extension_additive}
	Suppose $U$ is an open subset of $\rel^\adim$, $Y$ is a Banach space
	and $S, T \in \mathscr{D}' ( U,Y )$ are representable by integration.
	Then
	\begin{gather*}
		(S+T) ( \theta ) = S ( \theta ) + T ( \theta ) \quad \text{for
		$\theta \in \Lp{1} ( \| S \| + \| T \|, Y )$};
	\end{gather*}
	in fact, since $\| S + T \| \leq \| S \| + \| T \|$, considering a
	sequence $\theta_i \in \mathscr{D} ( U,Y )$ with $\eqLpnorm{\| S \|+\|
	T \|}{1}{\theta-\theta_i} \to 0$ as $i \to \infty$ yields the
	assertion.
\end{remark}
\begin{miniremark} \label{miniremark:distrib_on_products}
	Suppose $U$ and $V$ are open subsets of Euclidean spaces and
	$Y$ is a Banach space.

	Then the image of $\mathscr{D}^0 (U) \otimes \mathscr{D}^0 ( V)
	\otimes Y$ in $\mathscr{D} ( U \times V, Y )$ is sequentially dense,
	compare \cite[1.1.3, 4.1.2, 4.1.3]{MR41:1976}.
\end{miniremark}
\begin{definition}
	Suppose $1 \leq p \leq \infty$, $\phi$ is a Radon measure on a locally
	compact Hausdorff space $X$, and $Y$ is a Banach space.

	Then $f \in \Lploc{p} ( \phi, Y )$ if and only if $f \in \Lp{p} ( \phi
	\restrict K, Y )$ whenever $K$ is a compact subset of $X$.
\end{definition}
\begin{miniremark} \label{miniremark:distrib_on_r}
	If $R \in \mathscr{D}_0 ( \rel )$ is representable by integration and
	$\| R\|$ is absolutely continuous with respect to $\mathscr{L}^1$,
	then there exists $k \in \Lploc{1} ( \mathscr{L}^1 )$ such that
	\begin{gather*}
		R ( \gamma ) = \tint{}{} \gamma k \ud \mathscr{L}^1 \quad
		\text{whenever $\gamma \in \Lp{1} ( \| R \| )$}
	\end{gather*}
	by \cite[2.5.8]{MR41:1976}\footnote{In \cite[2.5.8]{MR41:1976} the
	hypothesis $\mu (f) = 0$ whenever $f(x) =0$ for $\phi$ almost all $x$
	should be added. (Consider $\phi = 0$, $X = \{ 0 \}$
	and $\mu = \boldsymbol{\delta}_0$.)} and \ref{miniremark:extension},
	hence
	\begin{gather*}
		k(t) = \lim_{\varepsilon \to 0+} \varepsilon^{-1} R
		(i_{t,\varepsilon}) \quad\text{for $\mathscr{L}^1$ almost all
		$t$},
	\end{gather*}
	where $i_{t,\varepsilon}$ is the characteristic function of the
	interval $\{ u \with t < u \leq t + \varepsilon \}$, by \cite[2.8.17,
	2.9.8]{MR41:1976}\footnote{In \cite[2.8.17]{MR41:1976} the hypothesis
	``$0 \leq (V)\,\limsup_{S \to x} [\delta(S) + \phi ( \widehat{S}) /
	\phi (S) ]$ for $\phi$ almost all $x$'' should be replaced by ``$V
	\cap \{ (x,S) \with \phi (S) = 0 \}$ is not fine at $x$ for $\phi$
	almost all $x$''. (Consider $\phi = \mathscr{L}^1$, $V = \{ (x,\{x\})
	\with x \in \rel \} \cup \{ (x,\cball{x}{r}) \with x \in \rel, 0 < r <
	\infty \}$, $\delta = \diam$, and $\tau = 2$.)}.
\end{miniremark}
\begin{lemma} \label{lemma:push_on_product}
	Suppose $\adim \in \nat$, $\mu$ is a Radon measure over an open subset
	$U$ of $\rel^\adim$, $f$ is a $\mu$ measurable real valued function,
	$F$ is a $\mu$ measurable $\Hom ( \rel^\adim, \rel )$ valued function
	with
	\begin{gather*}
		\tint{K \cap \{ z \with a < f (z) < b \}}{} |F| \ud \mu <
		\infty
	\end{gather*}
	whenever $K$ is a compact subset of $U$ and $- \infty < a \leq
	b < \infty$, and $T \in \mathscr{D}' ( U \times \rel, \rel^\adim )$
	satisfies
	\begin{gather*}
		T ( \psi ) = \tint{}{} \left < \psi (z,f(z)), F(z) \right >
		\ud \mu z \quad \text{for $\psi \in \mathscr{D} ( U \times
		\rel, \rel^\adim )$}
	\end{gather*}

	Then $T$ is representable by integration and
	\begin{gather*}
		T ( \psi ) = \tint{}{} \left < \psi (z,f(z)), F(z) \right >
		\ud \mu z, \quad \tint{}{} g \ud \| T \| = \tint{}{} g(z,f(z))
		| F(z)| \ud \mu z
	\end{gather*}
	whenever $\psi \in \Lp{1} ( \| T \|, \rel^\adim )$ and $g$ is an
	$\overline{\rel}$ valued $\| T \|$ integrable function.
\end{lemma}
\begin{proof}
	Define $p : U \times \rel \to U$ by
	\begin{gather*}
		p (z,t) = z \quad \text{for $(z,t) \in U \times \rel$}.
	\end{gather*}
	Define the measure $\nu$ over $U$ by $\nu (A) = \tint{A}{\ast} |F| \ud
	\mu$ for $A \subset U$.  Let $G : \dmn f \to U \times \rel$ be defined
	by $G(z) = (z,f(z))$ for $z \in \dmn f$. Employing \cite[2.2.2, 2.2.3,
	2.4.10]{MR41:1976} yields that $\nu \restrict \{ z \with a < f(z) < b
	\}$ is a Radon measure for $- \infty < a \leq b < \infty$, hence so is
	$G_\# ( \nu \restrict A )$ whenever $A$ is $\mu$ measurable by
	\cite[2.2.2, 2.2.3, 2.2.17, 2.3.5]{MR41:1976}. Noting $F | \dmn f = F
	\circ p \circ G$ and hence
	\begin{gather*}
		T(\psi) = \tint{}{} \left < \psi \circ G, |F|^{-1} F \right
		> \ud \nu = \tint{}{} \left < \psi, |F \circ p |^{-1} F \circ
		p \right > \ud G_\# \nu
	\end{gather*}
	for $\psi \in \mathscr{D} (U \times \rel, \rel^\adim )$ by
	\cite[2.4.10, 2.4.18\,(1)]{MR41:1976}, one infers $\| T \| = G_\# \nu$
	and hence the conclusion by means of \cite[2.2.2, 2.2.3, 2.4.10,
	2.4.18\,(1)\,(2)]{MR41:1976}.
\end{proof}
\begin{theorem} \label{thm:distribution_on_product}
	Suppose $U$ is an open subset of $\rel^\adim$, $Y$ is a separable
	Banach space, $T \in \mathscr{D}' ( U \times \rel, Y)$ is
	representable by integration, $R_\theta \in \mathscr{D}_0 ( \rel )$
	satisfy
	\begin{gather*}
		R_\theta ( \gamma ) = T_{(z,t)} ( \gamma (t) \theta (z) )
		\quad \text{whenever $\gamma \in \mathscr{D}^0 ( \rel )$ and
		$\theta \in \mathscr{D} (U,Y )$}.
	\end{gather*}
	and $S(t) : \mathscr{D} ( U, Y ) \to \rel$ satisfy, see
	\ref{miniremark:distrib_on_r},
	\begin{gather*}
		S(t)( \theta ) = \lim_{\varepsilon \to 0+} \varepsilon^{-1}
		R_\theta ( i_{t,\varepsilon} ) \in \rel \quad \text{for
		$\theta \in \mathscr{D} (U, Y )$}
	\end{gather*}
	whenever $t \in \rel$, that is $t \in \dmn S$ if and only if the limit
	exists and belongs to $\rel$ for $\theta \in \mathscr{D} ( U,
	Y )$.

	Then the following two statements hold.
	\begin{enumerate}
		\item \label{item:distribution_on_product:inequality} If $g$
		is an $\{ u \with 0 \leq u \leq \infty \}$ valued $\| T \|$
		measurable function, then
		\begin{gather*}
			\tint{}{} \tint{}{} g(z,t) \ud \| S(t) \| z \ud
			\mathscr{L}^1 t \leq \tint{}{} g \ud \| T \|.
		\end{gather*}
		\item \label{item:distribution_on_product:absolute} If $\|
		R_\theta \|$ is absolutely continuous with respect to
		$\mathscr{L}^1$ for $\theta \in \mathscr{D}(U,Y)$,
		then
		\begin{gather*}
			T ( \psi ) = \tint{}{} S(t) ( \psi (\cdot, t ) ) \ud
			\mathscr{L}^1 t, \quad \tint{}{} g \ud \| T \| =
			\tint{}{} \tint{}{} g (z,t) \ud \| S(t) \| z \ud
			\mathscr{L}^1 t.
		\end{gather*}
		whenever $\psi \in \Lp{1} ( \| T \|, Y )$ and $g$ is
		an $\overline{\rel}$ valued $\| T \|$ integrable function.
	\end{enumerate}
\end{theorem}
\begin{proof}
	Define $p : U \times \rel \to \rel$ and $q : U \times \rel \to U$ by
	\begin{gather*}
		p (z,t) = z \quad \text{and} \quad q (z,t) = t \quad \text{for
		$(z,t) \in U \times \rel$}.
	\end{gather*}
	
	In order to prove \eqref{item:distribution_on_product:inequality}, one
	may assume $g \in \mathscr{K} (U \times \rel)^+$. Observe that one may
	use \cite[2.9.19]{MR41:1976} in conjunction with \ref{remark:lusin},
	\ref{remark:ind-limit}, and \ref{example:distrib_lusin} to deduce that
	$S(t)$ belongs to $\mathscr{D}' (U,Y)$ and is representable
	by integration for $\mathscr{L}^1$ almost all $t$ and that $\| S(t) \|
	\in \mathscr{K} (U)^\ast$ depends $\mathscr{L}^1$ measurably on $t$.
	Suppose $f \in \mathscr{K} (U)^+$, $\gamma \in \mathscr{K} ( \rel)^+$
	and $g (z,t) = f (z) \gamma (t)$ for $(z,t) \in U \times \rel$.
	\begin{align*}
		& \| S (t) \| ( f ) \leq \mathbf{D} ( q_\# ( \| T \| \restrict
		f \circ p ), \mathscr{L}^1, V, t ) \quad \text{for
		$\mathscr{L}^1$ almost all $t$}, \\
		& \qquad \text{where $V = \{ (s,I) \with \text{$s \in I
		\subset \rel$ and $I$ is a compact interval} \}$}.
	\end{align*}
	Employing \cite[2.4.10, 2.8.17, 2.9.7]{MR41:1976}, one infers that
	\begin{gather*}
		\begin{aligned}
			\tint{}{} \| S (t) \| ( f ) \gamma(t) \ud
			\mathscr{L}^1 t & \leq \tint{}{} \mathbf{D} ( q_\# (
			\| T \| \restrict f \circ p ), \mathscr{L}^1, V, t )
			\gamma (t) \ud \mathscr{L}^1 t \\
			& \leq q_\# ( \| T \| \restrict ( f \circ p ) ) (
			\gamma ) = \| T \| ( ( f \circ p ) ( \gamma \circ q )
			) = \| T \| (g).
		\end{aligned}
	\end{gather*}
	An arbitrary $g \in \mathscr{K} (U \times \rel)^+$ may be approximated
	by a sequence of functions which are nonnegative linear combinations
	of functions of the previously considered type, compare \cite[4.1.2,
	4.1.3]{MR41:1976}.

	To prove the first equation in
	\eqref{item:distribution_on_product:absolute}, it is sufficient to
	exhibit a sequentially dense subset $D$ of $\mathscr{D} (U \times
	\rel, Y )$ such that
	\begin{gather*}
		T ( \psi ) = \tint{}{} S(t) (\psi(\cdot,t)) \ud \mathscr{L}^1
		t \quad \text{whenever $\psi \in D$}
	\end{gather*}
	by \eqref{item:distribution_on_product:inequality} and
	\ref{remark:ind-limit}. If $\theta \in \mathscr{D} (U,Y)$, then
	\ref{miniremark:distrib_on_r} implies
	\begin{gather*}
		R_\theta ( \gamma ) = \tint{}{} \gamma (t) S(t) ( \theta ) \ud
		\mathscr{L}^1 t, \\
		T_{(z,t)} ( \gamma (t) \theta (z) ) = R_\theta ( \gamma ) =
		\tint{}{} S(t)_z ( \gamma(t) \theta (z) ) \ud \mathscr{L}^1 t
	\end{gather*}
	for $\gamma \in \mathscr{D}^0 ( \rel )$. One may now take $D$ to be
	the image of $\mathscr{D}^0 ( \rel ) \otimes \mathscr{D} (U,Y)$ in
	$\mathscr{D} ( U \times \rel, Y )$, see
	\ref{miniremark:distrib_on_products}.

	The second equation in \eqref{item:distribution_on_product:absolute}
	follows from \eqref{item:distribution_on_product:inequality} and the
	first equation in \eqref{item:distribution_on_product:absolute}.
\end{proof}
\section{Monotonicity identity} \label{sec:monotonicity}
In preparation to the oscillation estimate in \ref{thm:mod_continuity}, some
simple consequences of the monotonicity identity are proven here in
\ref{corollary:monotonicity} and \ref{corollary:density_1d}.
\begin{lemma} \label{lemma:calculus}
	Suppose $\vdim > 1$, $0 \leq s < r < \infty$, $0 \leq \gamma <
	\infty$, $I = \{ t \with s < t \leq r \}$, and $f : I \to \{ t : 0 < t
	< \infty \}$ is a function satisfying for $t \in I$
	\begin{gather*}
		\limsup_{u \to t-} f(u) \leq f (t), \quad
		f (t) \leq f(r) + \gamma \tint{t}{r} u^{-1} f (u)^{1-1/\vdim}
		\ud \mathscr{L}^1 u.
	\end{gather*}

	Then
	\begin{gather*}
		f (t) \leq \big ( 1 + ( \gamma/ \vdim ) f (r)^{-1/\vdim} \log
		( r/t ) \big )^\vdim f (r) \quad \text{for $t \in I$}.
	\end{gather*}
\end{lemma}
\begin{proof}
	Whenever $s < t < r$ and $f (t) < y$ there holds
	$\density^{\ast 1} ( \mathscr{L}^1 \restrict \{ u \with f(u) \geq y
	\}, t ) \leq 1/2$, hence $f$ is $\mathscr{L}^1 \restrict I$ measurable
	by \cite[2.9.11]{MR41:1976}.

	Suppose $f (r) < \alpha < \infty$ and $\beta = ( \gamma/ \vdim
	) \alpha^{-1/\vdim}$ and consider the set $J$ of all $t \in I$ such
	that
	\begin{gather*}
		f (u) \leq ( 1 + \beta \log ( r/u ) )^\vdim \alpha \quad
		\text{whenever $t \leq u \leq r$}.
	\end{gather*}
	Clearly, $J$ is an interval and $r$ belongs to the interior of $J$
	relative to $I$. The same holds for $t$ with $s < t \in \Clos J$
	since
	\begin{gather*}
		\begin{aligned}
			f ( t ) & \leq f ( r ) + \gamma \alpha^{1-1/\vdim}
			\tint{t}{r} u^{-1} ( 1 + \beta \log ( r/u))^{\vdim-1}
			\ud \mathscr{L}^1 u \\
			& = f ( r ) + \alpha \big ( ( 1 + \beta \log ( r/t)
			)^\vdim - 1 \big ) < \alpha ( 1 + \beta \log ( r/t )
			)^\vdim.
		\end{aligned}
	\end{gather*}
	Therefore $I$ equals $J$.
\end{proof}
\begin{theorem} \label{thm:monotonicity}
	Suppose $\vdim, \adim \in \nat$, $\vdim \leq \adim$, $a \in
	\rel^\adim$, $0 < r < \infty$, $V \in \Var_\vdim ( \oball{a}{r}
	)$, and $\zeta \in \mathscr{D}^0 ( \{ s \with 0 < s < r \} )$.

	Then there holds
	\begin{gather*}
		\begin{aligned}
			& - \tint{0}{r} \zeta'(s) s^{-\vdim}
			\measureball{\| V \|}{ \cball{a}{s} } \ud
			\mathscr{L}^1 s \\
			& \qquad = \tint{( \oball{a}{r} \without \{ a \} )
			\times \grass{\adim}{\vdim} }{} \zeta ( |z-a| )
			|z-a|^{-\vdim-2} | \perpproject{S} (z-a) |^2 \ud V
			(z,S) \\
			& \qquad \quad - ( \delta V )_z \left ( \big (
			\tint{|z-a|}{r} s^{-\vdim-1} \zeta (s) \ud
			\mathscr{L}^1 s \big ) (z-a) \right ).
		\end{aligned}
	\end{gather*}
\end{theorem}
\begin{proof}
	Assume $a=0$ and let $I = \{ s \with - \infty < s < r \}$
	and $J = \{ s \with 0 < s < r \}$.
	
	If $\eta \in \mathscr{E}^0 (I)$, $\sup \spt \eta < r$, $0 \notin
	\spt \eta'$ and $g : \oball{a}{r} \to \rel^\adim$ is associated to
	$\eta$ by $g(z) = \eta ( |z| ) z$ for $z \in \oball{a}{r}$, then
	$Dg(0) \bullet \project{S} = \vdim \eta(0)$ and
	\begin{align*}
		Dg(z) \bullet \project{S} & = | \project{S} (z) |^2 |z|^{-1}
		\eta' (|z|) + \vdim \eta (|z|) \\
		& = - | \perpproject{S} (z) |^2 |z|^{-1} \eta'(|z|) + |z|
		\eta' (|z|) + \vdim \eta ( |z| )
	\end{align*}
	whenever $z \in \rel^\adim$, $0 < |z| < r$, and $S \in
	\grass{\adim}{\vdim}$. Define $\eta, \phi \in \mathscr{E}^0 (I)$ by
	\begin{gather*}
		\eta (s) = - \tint{\sup \{ s, 0 \}}{r} u^{-\vdim-1} \zeta(u)
		\ud \mathscr{L}^1 u \quad \text{and} \quad \phi (s) = s
		\eta'(s) + \vdim \eta(s)
	\end{gather*}
	for $s \in I$, hence $\sup \spt \phi \leq \sup \spt \eta <
	r$, $0 \notin \spt \eta'$, and
	\begin{gather*}
		\eta'(s) = s^{-\vdim-1} \zeta(s), \quad \eta''(s) = -
		(\vdim+1) s^{-\vdim-2} \zeta(s) + s^{-\vdim-1} \zeta'(s)
	\end{gather*}
	for $s \in J$. Using Fubini's theorem, one computes with $g$ as before
	that
	\begin{gather*}
		\begin{aligned}
			& \delta V (g) + \tint{( \rel^\adim \times
			\grass{\adim}{\vdim}) \cap \{ (z,S) \with 0 < |z| < r
			\}}{} \zeta (|z|) |z|^{-\vdim-2} | \perpproject{S} (z)
			|^2 \ud V (z,S) \\
			& \quad = \tint{}{} \phi (|z|) \ud \| V \| z = -
			\tint{}{} \tint{|z|}{r} \phi '(s) \ud \mathscr{L}^1 s
			\ud \| V \| z = - \tint{0}{r} \phi'(s) \measureball{\|
			V \|}{\cball{a}{s}} \ud \mathscr{L}^1 s.
		\end{aligned}
	\end{gather*}
	Finally, notice that $\phi'(s) = s \eta''(s) + ( \vdim+1)
	\eta' (s) = s^{-\vdim} \zeta'(s)$ for $s \in J$.
\end{proof}
\begin{remark}
	This is a slight generalisation of Simon's version of the monotonicity
	identity, see \cite[17.3]{MR756417}, included here for the convenience
	of the reader.
\end{remark}
\begin{miniremark} \label{miniremark:situation_general_varifold}
	Suppose $\vdim, \adim \in \nat$, $\vdim \leq \adim$, $U$ is an
	open subset of $\rel^\adim$, $V \in \Var_\vdim ( U )$, $\| \delta V
	\|$ is a Radon measure, and $\eta ( V, \cdot )$ is a $\| \delta V \|$
	measurable $\mathbf{S}^{\adim-1}$ valued function satisfying
	\begin{gather*}
		( \delta V ) (g) = \tint{}{} g(z) \bullet \eta ( V, z ) \ud \|
		\delta V \| z \quad \text{for $g \in \mathscr{D} ( U,
		\rel^\adim )$},
	\end{gather*}
	see Allard \cite[4.3]{MR0307015}.
\end{miniremark}
\begin{corollary} \label{corollary:monotonicity}
	Suppose $\vdim$, $\adim$, $U$, $V$, and $\eta$ are as in
	\ref{miniremark:situation_general_varifold}.

	Then there holds
	\begin{gather*}
		\begin{aligned}
			&s^{-\vdim} \measureball{\| V \|}{ \cball as } +
			\tint{( \cball ar \without \cball as )\times
			\grass{\adim}{\vdim}}{} | z-a |^{-\vdim-2} |
			\perpproject{S} (z-a) |^2 \ud V(z,S) \\
			& \qquad = r^{-\vdim} \measureball{\| V \|}{ \cball ar
			} \\
			& \qquad \phantom{=}\ + \vdim^{-1}
			\tint{\cball{a}{r}}{} ( \sup \{ |z-a|,s \}^{-\vdim}-
			r^{-\vdim}) (z-a) \bullet \eta (V,z) \ud \| \delta V
			\| z
		\end{aligned}
	\end{gather*}
	whenever $a \in \rel^\adim$, $0 < s < r < \infty$, and $\cball ar
	\subset U$.
\end{corollary}
\begin{proof}
	Letting $\zeta$ approximate the charateristic function of $\{ t \with
	s < t \leq r \}$, the assertion is a consequence of
	\ref{thm:monotonicity}.
\end{proof}
\begin{remark} \label{remark:monotonicity}
	Using Fubini's theorem, the last summand can be expressed as
	\begin{gather*}
		\tint{s}{r} t^{-\vdim-1} \tint{\cball at}{} (z-a) \bullet \eta
		(V,z) \ud \| \delta V \| z \ud \mathscr{L}^1 t.
	\end{gather*}
\end{remark}
\begin{remark}
	\ref{corollary:monotonicity} and \ref{remark:monotonicity} are
	a minor variations of Simon \cite[17.3, 17.4]{MR756417}.
\end{remark}
\begin{corollary} \label{corollary:density_1d}
	Suppose $\vdim$, $\adim$, $U$, and $V$ are as in
	\ref{miniremark:situation_general_varifold}, $\vdim = 1$, and $T = U
	\cap \{ z \with \| \delta V \| ( \{ z \} ) > 0 \}$.

	Then the following four statements hold.
	\begin{enumerate}
		\item \label{item:density_1d:upper_bound} If $a \in
		\rel^\adim$, $0 < s < r < \infty$, and $\cball{a}{r} \subset
		U$, then
		\begin{gather*}
			s^{-1} \measureball{\| V \|}{ \cball as } +
			\tint{(\cball{a}{r} \without \{ a \} ) \times
			\grass{\adim}{\vdim}}{} |z-a|^{-3} | \perpproject{S}
			(z-a) |^2 \ud V (z,S) \\
			\leq r^{-1} \measureball{\| V \|}{\cball{a}{r}} + \|
			\delta V \| ( \cball{a}{r} \without \{ a \} )
		\end{gather*}
		\item \label{item:density_1d:real} $\density^1 ( \| V
		\|, \cdot )$ is a real valued function on $U$.
		\item \label{item:density_1d:upper_semi} $\density^1 ( \| V
		\|, \cdot )$ is upper semicontinuous at $a$ whenever $a \in U
		\without T$.
		\item \label{item:density_1d:lower_bound} If $\density^1 ( \|
		V \|, z ) \geq 1$ for $\| V \|$ almost all $z$, then $a
		\in \spt \| V \|$ implies
		\begin{gather*}
			\density^1 ( \| V \|, a ) \geq 1 \quad \text{if $a
			\not \in T$} \qquad \text{and} \qquad \density^1 ( \| V
			\|, a) \geq 1/2 \quad \text{if $a \in T$}.
		\end{gather*}
	\end{enumerate}
\end{corollary}
\begin{proof}
	If $a \in U$, $0 < s < r < \infty$ and $\cball ar \subset U$, then
	\begin{gather*}
		\big | \sup \{ |z-a|, s \}^{-1} - r^{-1} \big | |z-a| \leq 1
		\quad \text{whenever $z \in \cball ar$}.
	\end{gather*}
	Therefore \eqref{item:density_1d:upper_bound} follows from
	\ref{corollary:monotonicity}. \eqref{item:density_1d:upper_bound}
	readily implies \eqref{item:density_1d:real} and
	\eqref{item:density_1d:upper_semi} and the first half of
	\eqref{item:density_1d:lower_bound}. To prove the second half, choose
	$\eta$ as in \ref{miniremark:situation_general_varifold} and consider
	$a \in T$. One may assume $a = 0$ and in view of Allard
	\cite[4.10\,(2)]{MR0307015} also $U = \rel^\adim$. Abbreviating $v =
	\eta (V,0) \in \mathbf{S}^{\adim-1}$ and defining the reflection $f :
	\rel^\adim \to \rel^\adim$ by $f(z) = z- 2 (z \bullet v) v$ for $z \in
	\rel^\adim$, one infers the second half of
	\eqref{item:density_1d:lower_bound} by applying the first half of
	\eqref{item:density_1d:lower_bound} to the varifold $V + f_\# V$.
\end{proof}
\begin{corollary} \label{corollary:density_ratio_estimate}
	Suppose $\vdim$, $\adim$, $U$, and $V$ are as in
	\ref{miniremark:situation_general_varifold}, $a \in \rel^\adim$, $0 <
	s < r < \infty$, $\cball{a}{r} \subset U$, $0 \leq \beta < \infty$,
	\begin{gather*}
		\measureball{\| \delta V \|}{\cball{a}{t}} \leq \beta r^{-1}
		\| V \| ( \cball{a}{r} )^{1/\vdim} \| V \| ( \cball {a}{t}
		)^{1-1/\vdim} \quad \text{for $s < t < r$}.
	\end{gather*}

	Then there holds
	\begin{gather*}
		s^{-\vdim} \measureball{\| V \|}{\cball{a}{s}} \leq \big ( 1 +
		\vdim^{-1} \beta \log ( r/s) \big )^\vdim r^{-\vdim}
		\measureball{\| V \|}{\cball{a}{r}}.
	\end{gather*}
\end{corollary}
\begin{proof}
	Assuming $\mu (r) > 0$ and taking $t = \inf \{ u \with \mu (u) > 0 \}$
	and $f (u) = u^{-\vdim} \mu (u)$ for $\sup \{s,t\} < u \leq r$, this
	is a consequence of \ref{corollary:monotonicity},
	\ref{remark:monotonicity} and \ref{lemma:calculus}.
\end{proof}
\section{Decompositions of varifolds} \label{sec:decomposition}
In this section the existence of a decomposition of rectifiable varifolds
whose first variation is representable by integration is established in
\ref{thm:decomposition}. If the first variation is sufficiently regular, this
decomposition may be linked to the decomposition of the support of the weight
measure into connected components, see \ref{corollary:conn_structure}.
\begin{miniremark} \label{miniremark:situation_general}
	Suppose $\vdim, \adim \in \nat$, $\vdim \leq \adim$, $1 \leq p \leq
	\vdim$, $U$ is an open
	subset of $\rel^\adim$, $V \in \Var_\vdim ( U )$, $\| \delta V \|$ is
	a Radon measure, $\density^\vdim ( \| V \|, z ) \geq 1$ for $\| V \|$
	almost all $z$. If $p > 1$, then suppose additionally that
	$\mathbf{h} ( V; \cdot ) \in \Lploc{p} ( \| V \|, \rel^\adim)$ and
	\begin{gather*}
		\delta V (g) = - \tint{}{} g (z) \bullet \mathbf{h} (V;z) \ud
		\| V \| z \quad \text{for $g \in \mathscr{D} ( U, \rel^\adim
		)$}.
	\end{gather*}
	Therefore $V \in \RVar_\vdim ( U)$ by Allard
	\cite[5.5\,(1)]{MR0307015}. If $p < \infty$, then define
	\begin{gather*}
		\psi = \| \delta V \| \quad \text{if $p=1$}, \qquad \psi = \|
		V \| \restrict | \mathbf{h} ( V ; \cdot ) |^p \quad
		\text{else.}
	\end{gather*}
\end{miniremark}
\begin{definition} \label{def:v_boundary}
	Suppose $\vdim, \adim \in \nat$, $\vdim \leq \adim$, $U$ is an open
	subset of $\rel^\adim$, $V \in \Var_\vdim ( U )$, $\| \delta V \|$ is
	a Radon measure, and $E$ is $\| V \| + \| \delta V \|$ measurable.

	Then the \emph{distributional $V$ boundary of $E$} is given by
	\begin{gather*}
		\boundary{V}{E} = ( \delta V ) \restrict E - \delta ( V
		\restrict E \times \grass{\adim}{\vdim} ) \in \mathscr{D}' (
		U, \rel^\adim ).
	\end{gather*}
\end{definition}
\begin{remark} \label{remark:v_boundary}
	If $E_1$ and $E_2$ are $\| V \| + \| \delta V \|$ measurable sets and
	$E_1 \subset E_2$, then
	\begin{gather*}
		\boundary{V}{(E_2 \without E_1)} = ( \boundary{V}{E_2} ) - (
		\boundary{V}{E_1} ).
	\end{gather*}
\end{remark}
\begin{remark} \label{remark:partition}
	If $G$ is a countable, disjointed collection of $\| V \| + \| \delta V
	\|$ measurable sets with $\boundary{V}{E} = 0$ for $E \in G$ and $\| V
	\| ( U \without \bigcup G ) = 0$, then $\| \delta V \| ( U \without
	\bigcup G ) = 0$; in fact $\delta V = \delta \big ( \sum_{E \in G} V
	\restrict E \times \grass{\adim}{\vdim} \big ) = \sum_{E \in G} \delta
	( V \restrict E \times \grass{\adim}{\vdim} ) = \sum_{E \in G} (
	\delta V) \restrict E$.
\end{remark}
\begin{example}
	If $\adim = \vdim$, $U = \rel^\vdim$, and $\| V \| (A) =
	\mathscr{L}^\vdim (A)$ for $A \subset \rel^\vdim$, then
	$\boundary{V}{E}$ is representable by integration if and only if
	$\partial ( \mathbf{E}^\vdim \restrict E )$ is representable by
	integration; in this case
	\begin{gather*}
		\boundary{V}{E} ( \theta ) = - \tint{E}{} \Div \theta
		\ud\mathscr{L}^\vdim = - \tint{}{} \theta (z) \bullet
		\mathbf{n} ( E, z ) \ud \mathscr{H}^{\vdim-1} z
	\end{gather*}
	for $\theta \in \mathscr{D} ( \rel^\vdim, \rel^\vdim )$ by
	\cite[4.5.16, 4.5.6]{MR41:1976}. In the terminology of \cite[\S
	3.3]{MR2003a:49002} such $E$ is called ``of locally finite
	perimeter in $\rel^\vdim$''.
\end{example}
\begin{lemma} \label{lemma:piece}
	Suppose $\vdim$, $\adim$, $p$, $U$ and $V$ are as in
	\ref{miniremark:situation_general}, $p = \vdim$, $A$ is a $\| V \| +
	\| \delta V \|$ measurable set with $\boundary{V}{A} = 0$.
	
	Then $A$ is $\| V \|$ almost equal to $\spt ( \| V \| \restrict A )$.
\end{lemma}
\begin{proof}
	Define $W = V \restrict ( A \times \grass{\adim}{\vdim} )$, hence
	\begin{gather*}
		\| W \| = \| V \| \restrict A, \quad \| V \| ( A \without \spt
		\| W \| ) = 0.
	\end{gather*}
	Moreover, one obtains
	\begin{gather*}
		\density_\ast^\vdim ( \| W \|, z ) \geq 1 \quad \text{whenever
		$z \in \spt \| W \|$ and $\| \delta W \| ( \{ z \} ) = 0$}, \\
		\density^\vdim ( \| W \|, z ) = 0 \quad \text{for
		$\mathscr{H}^\vdim$ almost all $z \in U \without A$}
	\end{gather*}
	by \cite[2.5]{snulmenn.isoperimetric} and
	\cite[2.10.19\,(4)]{MR41:1976}, hence $\| V \| ( ( \spt \| W \| )
	\without A ) = 0$ by Allard \cite[3.5\,(1b)]{MR0307015}.
\end{proof}
\begin{remark}
	If $A = \rel^2 \cap \{ (x_1,x_2) \with x_2 = 0 < |x_1| < 1
	\}$, $B = \rel^2 \cap \{ (x_1,x_2) \with x_1 = 0 \leq x_2 \leq 1 \}$
	and $V \in \IVar_1 ( \rel^2 )$ is characterised by $\| V \| =
	\mathscr{H}^1 \restrict ( A \cup B )$, then $\boundary{V}{A} = 0$ but
	$A$ is not $\| V \| + \| \delta V \|$ almost equal to $\spt ( \| V \|
	\restrict A )$.
\end{remark}
\begin{definition}
	Suppose $\vdim, \adim \in \nat$, $\vdim \leq \adim$, $U$ is an open
	subset of $\rel^\adim$, $V \in \Var_\vdim ( U )$ and $\| \delta V \|$
	is a Radon measure.

	Then $V$ is called \emph{indecomposable} if there exists no $\| V \| +
	\| \delta V \|$ measurable set $B$ such that
	\begin{gather*}
		\| V \| (B) > 0, \quad \| V \| ( U \without B ) > 0, \quad
		\boundary{V}{B} = 0.
	\end{gather*}
\end{definition}
\begin{remark}
	The same definition results if $B$ is required to be a Borel set.
\end{remark}
\begin{remark}
	If $V$ is indecomposable then so is $\lambda V$ for $0 < \lambda <
	\infty$. This is in contrast with the notion bearing the same name
	employed by Mondino in \cite[2.15]{MR3148123}.
\end{remark}
\begin{lemma} \label{lemma:connected}
	Suppose $\vdim, \adim \in \nat$, $\vdim \leq \adim$, $U$ is an open
	subset of $\rel^\adim$, $V \in \Var_\vdim ( U )$, $\| \delta V \|$
	is a Radon measure, and $V$ is indecomposable.

	Then $\spt \| V \|$ is connected.
\end{lemma}
\begin{proof}
	If $\spt \| V \|$ were not connected, there would exist nonempty
	relatively closed subsets $A_1$ and $A_2$ of $\spt \| V \|$ such that
	$A_1 \cap A_2 = \varnothing$ and $A_1 \cup A_2 = \spt \| V \|$, hence
	$U \without A_1$ and $U \without A_2$ would be open and
	\begin{gather*}
		\| V \| ( A_1 ) = \| V \| ( U \without A_2 ) > 0, \quad
		\| V \| ( A_2 ) = \| V \| ( U \without A_1 ) > 0.
	\end{gather*}
	Therefore one could construct $f \in \mathscr{E}^0 ( U )$ such
	that
	\begin{gather*}
		A_i \subset \Int \{ z \with f(z) = i \} \quad
		\text{for $i = \{ 1, 2 \}$};
	\end{gather*}
	in fact applying \cite[3.1.13]{MR41:1976} with $\Phi = \{ U \without
	A_1, U \without A_2 \}$ one would obtain $h$, $S$ and $v_s$ and
	could take
	\begin{gather*}
		f (z) = 1 + \sum_{s \in T} v_s (z) \quad \text{for $z \in
		U$},
	\end{gather*}
	where $T = S \cap \big \{ s \with 20 h(s) = \inf \{ 1, \dist
	(s,\rel^\adim \without ( U \without A_1 ) ) \} \big \}$. This would
	yield $\delta ( V \restrict A_i \times \grass{\adim}{\vdim} ) = (
	\delta V ) \restrict A_i$ for $i = \{ 1, 2 \}$ by Allard
	\cite[4.10\,(1)]{MR0307015}.
\end{proof}
\begin{definition} \label{def:component}
	Suppose $\vdim, \adim \in \nat$, $\vdim \leq \adim$, $U$ is an open
	subset of $\rel^\adim$, $V \in \Var_\vdim ( U )$, and $\| \delta V \|$
	is a Radon measure.

	Then $W$ is called a \emph{component of $V$} if and only if $0 \neq W
	\in \Var_\vdim ( U )$ is indecomposable and there exists a $\| V \| +
	\| \delta V \|$ measurable set $A$ such that
	\begin{gather*}
		W = V \restrict ( A \times \grass{\adim}{\vdim} ), \quad
		\boundary{V}{A} = 0.
	\end{gather*}
\end{definition}
\begin{remark} \label{remark:unique_component}
	Suppose $B$ is a $\| V \| + \| \delta V \|$ measurable set. Then $A$
	is $\| V \| + \| \delta V \|$ almost equal to $B$ if and only if 
	\begin{gather*}
		W = V \restrict ( B \times \grass{\adim}{\vdim} ), \quad
		\boundary{V}{B} = 0.
	\end{gather*}
\end{remark}
\begin{remark} \label{remark:component}
	If $C$ is a connected component of $\spt \| V \|$ and $W$ is a
	component of $V$ with $C \cap \spt \| W \| \neq \varnothing$, then
	$\spt \| W \| \subset C$ by \ref{lemma:connected}.
\end{remark}
\begin{definition} \label{def:decomposition}
	Suppose $\vdim, \adim \in \nat$, $\vdim \leq \adim$, $U$ is an open
	subset of $\rel^\adim$, $V \in \Var_\vdim ( U )$,  $\| \delta V \|$ is
	a Radon measure, and $G \subset \Var_\vdim ( U )$.

	Then $G$ is called a \emph{decomposition of $V$} if and only if the
	following three conditions are satisfied:
	\begin{enumerate}
		\item Each member of $G$ is a component of $V$.
		\item Whenever $W$ and $X$ are distinct members of $G$ there
		exist disjoint $\| V \| + \| \delta V \|$ measurable sets $A$
		and $B$ with $\boundary{V}{A} = 0 = \boundary{V}{B}$ and
		\begin{gather*}
			W = V \restrict ( A \times \grass{\adim}{\vdim} ),
			\quad X = V \restrict ( B \times \grass{\adim}{\vdim}
			).
		\end{gather*}
		\item $V ( f ) = \tsum{W \in G}{} W (f)$ whenever $f \in
		\mathscr{K} ( U \times \grass{\adim}{\vdim} )$.
	\end{enumerate}
\end{definition}
\begin{remark} \label{remark:decomp_rep}
	Clearly, $G$ is countable.

	Moreover, using \ref{remark:unique_component} one constructs a
	function $\zeta$ mapping $G$ into the class of all Borel subsets of
	$U$ such that distinct members of $G$ are mapped onto disjoint sets
	and
	\begin{gather*}
		W = V \restrict ( \zeta(W) \times \grass{\adim}{\vdim} ),
		\quad \boundary{V}{(\zeta(W))} = 0
	\end{gather*}
	whenever $W \in G$. Consequently,
	\begin{gather*}
		{\textstyle ( \| V \| + \| \delta V \| ) \big ( U \without
		\bigcup \im \zeta \big ) = 0}
	\end{gather*}
	and $\boundary{V}{\left ( \bigcup \zeta \lIm H \rIm \right )} = 0$
	whenever $H \subset G$.
\end{remark}
\begin{remark} \label{remark:decomposition}
	Suppose $\vdim$, $\adim$, $p$, $U$ and $V$ are as in
	\ref{miniremark:situation_general}, $p = \vdim$, and $G$ is a
	decomposition of $W$. Observe that \ref{remark:decomp_rep} and
	\cite[2.5]{snulmenn.isoperimetric} imply
	\begin{gather*}
		\card (G \cap \{ W \with K \cap \spt \| W \| \neq \varnothing
		\} ) < \infty
	\end{gather*}
	whenever $K$ is a compact subset of $U$, hence
	\begin{gather*}
		\spt \| V \| = {\textstyle \bigcup \{ \spt \| W \| \with W \in
		G \}}.
	\end{gather*}
	Notice that \cite[1.2]{snulmenn.isoperimetric} readily shows that both
	assertions need not to hold in case $p < \vdim$.
\end{remark}
\begin{theorem} \label{thm:decomposition}
	Suppose $\vdim, \adim \in \nat$, $\vdim \leq \adim$, $U$ is an open
	subset of $\rel^\adim$, $V \in \RVar_\vdim ( U )$, and $\| \delta V
	\|$ is a Radon measure.

	Then there exists a decomposition of $V$.
\end{theorem}
\begin{proof}
	Assume $V \neq 0$.

	Denote by $R$ the family of Borel subsets $A$ of $U$ such that
	$\boundary{V}{A} = 0$. Notice that
	\begin{gather*}
		\text{$A \in R$ if and only if $A \without B \in R$} \quad
		\text{whenever $B \subset A$ and $B \in R$}, \\
		\bigcap_{i=1}^\infty A_i \in R \quad \text{whenever $A_i$ is a
		sequence in $R$ with $A_{i+1} \subset A_i$ for $i \in \nat$}.
	\end{gather*}
	Let $P = R \cap \{ A \with \| V \| ( A ) > 0 \}$. Next, define
	\begin{gather*}
		\delta_i = \unitmeasure{\vdim} 2^{-\vdim-1} i^{-1-2\vdim},
		\quad \varepsilon_i = 2^{-1} i^{-2}
	\end{gather*}
	for $i \in \nat$ and let $E_i$ denote the Borel set of $a \in
	\rel^\adim$ satisfying
	\begin{gather*}
		|a| \leq i, \quad \oball{a}{2\varepsilon_i} \subset U, \quad
		\density^\vdim ( \| V \|, a ) \geq 1/i, \\
		\measureball{\| \delta V \|}{ \cball{a}{r} } \leq
		\unitmeasure{\vdim} i r^\vdim \quad \text{for $0 < r <
		\varepsilon_i$}
	\end{gather*}
	whenever $i \in \nat$. Clearly, $E_i \subset E_{i+1}$ for $i \in \nat$
	and $\| V \| ( U \without \bigcup_{i=1}^\infty E_i ) = 0$ by
	Allard \cite[3.5\,(1a)]{MR0307015} and \cite[2.8.18,
	2.9.5]{MR41:1976}. Moreover, define
	\begin{gather*}
		P_i = R \cap \{ A \with \| V \| ( A \cap E_i ) > 0 \}
	\end{gather*}
	and notice that $P_i \subset P_{i+1}$ for $i \in \nat$ and $P =
	\bigcup_{i=1}^\infty P_i$. One observes the lower bound given by
	\begin{gather*}
		\| V \| ( A \cap \cball{a}{\varepsilon_i} ) \geq \delta_i
	\end{gather*}
	whenever $A \in R$, $i \in \nat$, $a \in E_i$ and $\density^{\ast
	\vdim} ( \| V \| \restrict A, a ) \geq 1/i$; in fact, noting
	\begin{gather*}
		\tint{s}{t} u^{-\vdim} \measureball{\| \delta ( V \restrict A
		\times \grass{\adim}{\vdim} ) \|}{ \cball{a}{u} } \ud
		\mathscr{L}^1 u \leq \unitmeasure{\vdim} i (t-s)
	\end{gather*}
	for $0 < s < t < \varepsilon_i$, the inequality follows from
	\ref{lemma:monotonicity}. Let $Q_i$ denote the set of $A \in P$ such
	that there is no $B$ satisfying
	\begin{gather*}
		B \subset A, \quad B \in P_i, \quad A \without B \in P_i.
	\end{gather*}

	Denote by $\Omega$ the class of Borel partitions $H$ of $U$ with $H
	\subset P$ and let $G_0 = \{ U \} \in \Omega$. The previously observed
	lower bound implies
	\begin{gather*}
		\delta_i \card ( H \cap P_i ) \leq \| V \| ( U \cap \{ z \with
		\dist (z,E_i) \leq \varepsilon_i \} ) < \infty
	\end{gather*}
	whenever $H$ is a disjointed subfamily of $P$, since for each $A \in
	H \cap P_i$ there exists $a \in E_i$ with $\density^\vdim ( \| V \|
	\restrict A, a ) = \density^\vdim ( \| V \|, a ) \geq 1/i$ by
	\cite[2.8.18, 2.9.11]{MR41:1976}, hence
	\begin{gather*}
		\| V \| ( A \cap \{ z \with \dist (z,E_i) \leq \varepsilon_i
		\} ) \geq \| V \| ( A \cap \cball{a}{\varepsilon_i} ) \geq
		\delta_i.
	\end{gather*}
	In particular, such $H$ is countable.

	Next, one inductively (for $i \in \nat$) defines $\Omega_i$ to be the
	class of all $H \in \Omega$ such that every $A \in G_{i-1}$ is the
	union of some subfamily of $H$ and chooses $G_i \in \Omega_i$ such
	that
	\begin{gather*}
		\card ( G_i \cap P_i ) \geq \card ( H \cap P_i ) \quad
		\text{whenever $H \in \Omega_i$}.
	\end{gather*}
	The maximality of $G_i$ implies $G_i \subset Q_i$; in fact, if there
	would exist $A \in G_i \without Q_i$ there would exist $B$ satisfying
	\begin{gather*}
		B \subset A, \quad B \in P_i, \quad A \without B \in P_i
	\end{gather*}
	and $H = ( G_i \without \{ A \} ) \cap \{ B, A \without B \}$ would
	belong to $\Omega_i$ with
	\begin{gather*}
		\card ( H \cap P_i ) > \card ( G_i \cap P_i ).
	\end{gather*}
	Moreover, it is evident that to each $z \in U$ there corresponds a
	sequence $A_i$ uniquely characterised by the requirements $z \in
	\bigcap_{i=1}^\infty A_i$ and $A_{i+1} \subset A_i \in G_i$ for $i \in
	\nat$.

	Define $G = \bigcup_{i=1}^\infty G_i$ and notice that $G$ is
	countable. Define $C$ to be the collection of sets
	$\bigcap_{i=1}^\infty A_i$ with positive $\| V \|$ measure
	corresponding to all sequences $A_i$ with $A_{i+1} \subset A_i \in
	G_i$ for $i \in \nat$. Clearly, $C$ is a disjointed subfamily of $P$,
	hence $C$ is countable. Next, it will be shown that
	\begin{gather*}
		\| V \| ( U \without {\textstyle \bigcup C } ) = 0.
	\end{gather*}
	In view of \cite[2.8.18, 2.9.11]{MR41:1976} it is sufficent to prove
	\begin{gather*}
		E_i \without {\textstyle \bigcup C} \subset {\textstyle
		\bigcup } \big \{ A \cap \{ z \with \density^{\ast \vdim} ( \|
		V \| \restrict A, z ) < \density^{\ast \vdim} ( \| V \|, z )
		\} \with A \in G \big \}
	\end{gather*}
	for $i \in \nat$. For this purpose consider $a \in E_i \without
	\bigcup C$ with corresponding sequence $A_j$. It follows that $\| V \|
	( \bigcap_{j=1}^\infty A_j ) = 0$, hence there exists $j$ with $\| V
	\| ( A_j \cap \cball{a}{\varepsilon_i} ) < \delta_i$ and the lower
	bound implies
	\begin{gather*}
		\density^{\ast \vdim} ( \| V \| \restrict A_j, a ) < 1/i \leq
		\density^\vdim ( \| V \|, a ).
	\end{gather*}
	One infers that also
	\begin{gather*}
		\| \delta V \| ( U \without {\textstyle \bigcup C} ) = 0,
	\end{gather*}
	since $\delta V = \sum_{A \in C} \delta ( V \restrict A \times
	\grass{\adim}{\vdim} ) = \sum_{A \in C} ( \delta V ) \restrict A = (
	\delta V ) \restrict \bigcup C$.

	It remains to prove that each varifold $V \restrict A \times
	\grass{\adim}{\vdim}$ corresponding to $A \in C$ is indecomposable. If
	this were not the case, then there would exist $A = \bigcap_{i=1} A_i
	\in C$ with $A_{i+1} \subset A_i \in G_i$ for $i \in \nat$ and a Borel
	set $B$ such that
	\begin{gather*}
		\| V \| ( A \cap B ) > 0, \quad \| V \| ( A \without B ) > 0,
		\\
		\delta ( V \restrict (A \cap B) \times \grass{\adim}{\vdim} )
		= \big ( \delta ( V \restrict A \times \grass{\adim}{\vdim} )
		\big ) \restrict B = ( \delta V ) \restrict ( A \cap B ).
	\end{gather*}
	This would imply $A \cap B \in P$ and $A \without B \in P$, hence for
	some $i$ also
	\begin{gather*}
		A \cap B \subset A_i, \quad A \cap B \in P_i, \quad A \without
		B \in P_i
	\end{gather*}
	which would yield
	\begin{gather*}
		A_i \without ( A \cap B ) = ( A_i \without A ) \cup ( A
		\without B ) \in P_i,
	\end{gather*}
	in contradiction to $A_i \in Q_i$.
\end{proof}
\begin{remark} \label{remark:nonunique_decomposition}
	The decomposition of $V$ may be nonunique. In fact,
	considering the six rays
	\begin{gather*}
		R_j = \{ t \exp ( \pi \mathbf{i} j/3) \with 0 < t < \infty
		\} \subset \complex = \rel^2, \quad \text{where $\pi =
		\boldsymbol{\Gamma} (1/2)^2$},
	\end{gather*}
	corresponding to $j \in \{ 0, 1, 2, 3, 4, 5 \}$ and their associated
	varifolds $V_j \in \IVar_1 ( \rel^2 )$ with $\| V_j \| = \mathscr{H}^1
	\restrict R_j$, one notices that $V = \sum_{j=0}^5 V_j \in \IVar_1 (
	\rel^2 )$ is a stationary varifold such that
	\begin{gather*}
		\{ V_0 + V_2 + V_4, V_1 + V_3 + V_5 \} \quad \text{and} \quad
		\{ V_0 + V_3, V_1 + V_4, V_2 + V_5 \}
	\end{gather*}
	are distinct decompositions of $V$.
\end{remark}
\begin{corollary} \label{corollary:conn_structure}
	Suppose $\vdim$, $\adim$, $p$, $U$ and $V$ are as in
	\ref{miniremark:situation_general}, $p = \vdim$, and $H$ is the family
	of all connected components of $\spt \| V \|$.

	Then the following four statements hold.
	\begin{enumerate}
		\item \label{item:conn_structure:union} If $C \in H$, then
		\begin{gather*}
			C = {\textstyle \bigcup \{ \spt \| W \| \with W \in G,
			C \cap \spt \| W \| \neq \varnothing \}}
		\end{gather*}
		whenever $G$ is decomposition of $V$.
		\item \label{item:conn_structure:piece} If $C \in H$, then
		$\boundary{V}{C} = 0$.
		\item \label{item:conn_structure:finite} $\card ( H \cap \{ C
		\with C \cap K \neq \varnothing \}) < \infty$ whenever $K$ is
		a compact subset of $U$.
		\item \label{item:conn_structure:open} If $C \in H$, then $C$
		is open relative to $\spt \| V \|$.
	\end{enumerate}
\end{corollary}
\begin{proof}
	\eqref{item:conn_structure:union} is a consequence of
	\ref{remark:component} and \ref{remark:decomposition}.

	To prove \eqref{item:conn_structure:piece}, let $G' = G \cap \{ W
	\with C \cap \spt \| W \| \neq \varnothing \}$ and $X = \sum_{W \in
	G'} W$. One readily verifies $C = \spt \| X \|$ by means of
	\ref{remark:decomposition} and \eqref{item:conn_structure:union}.
	Hence \eqref{item:conn_structure:piece} follows from
	\ref{lemma:piece} and \ref{remark:decomp_rep}.
	
	In view of \eqref{item:conn_structure:union} and
	\ref{thm:decomposition}, \eqref{item:conn_structure:finite} is a
	consequence of \ref{remark:decomposition}. Finally,
	\eqref{item:conn_structure:open} follows from
	\eqref{item:conn_structure:finite}.
\end{proof}
\begin{remark}
	If $V$ is stationary, then \eqref{item:conn_structure:piece} and
	\eqref{item:conn_structure:open} imply that $V \restrict C \times
	\grass{\adim}{\vdim}$ is stationary and $\spt \| V \restrict C \times
	\grass{\adim}{\vdim} \| = C$ whenever $C \in H$. This observation
	might prove useful in considerations involving a strong maximum
	principle such as Wickramasekera \cite[Theorem
	1.1]{arxiv.1310.4031v1}.
\end{remark}
\section{Basic properties}
In this section generalised weakly differentiable functions are defined in
\ref{def:v_weakly_diff}. Properties studied include behaviour under
composition, see \ref{lemma:basic_v_weakly_diff}, \ref{example:composite} and
\ref{lemma:comp_lip}, addition and multiplication, see
\ref{thm:addition}\,\eqref{item:addition:add}\,\eqref{item:addition:mult}, and
decomposition of the varifold, see \ref{thm:tv_on_decompositions} and
\ref{thm:zero_derivative}. Moreover, coarea formulae in terms of the
distributional boundary of superlevel sets are established, see
\ref{lemma:meas_fct} and \ref{thm:tv_coarea}. A measure theoric description of
the superlevel sets will appear in \ref{corollary:coarea}. The theory is
illustrated by examples in \ref{example:star} and \ref{example:axioms}.
\begin{lemma} \label{lemma:meas_fct}
	Suppose $\vdim, \adim \in \nat$, $U$ is an open subset of
	$\rel^\adim$, $V \in \Var_\vdim (U)$, $\| \delta V \|$ is a Radon
	measure, $f$ is a real valued $\| V \| + \| \delta V \|$ measurable
	function with $\dmn f \subset U$, and $E(t) = \{ z \with f(z) > t \}$
	for $t \in \rel$.

	Then there exists a unique $T \in \mathscr{D}' (U \times \rel,
	\rel^\adim )$ such that, see \ref{miniremark:extension},
	\begin{gather*}
		( \delta V ) ( (\eta \circ f ) \theta ) = \tint{}{} \eta (
		f(z)) D \theta (z) \bullet \project{S} \ud V(z,S) + T_{(z,t)}
		( \eta'(t) \theta (z) )
	\end{gather*}
	whenever $\theta \in \mathscr{D} (U,\rel^\adim)$, $\eta \in
	\mathscr{E}^0 ( \rel )$, $\spt \eta'$ is compact and $\inf \spt \eta >
	- \infty$. Moreover, there holds
	\begin{gather*}
		T ( \psi ) = \tint{}{} \boundary{V}{E(t)} ( \psi ( \cdot, t )
		) \ud \mathscr{L}^1t \quad \text{for $\psi \in \mathscr{D} (U
		\times \rel, \rel^\adim )$}.
	\end{gather*}
\end{lemma}
\begin{proof}
	Define $A = \{ (z,t) \with z \in E(t) \}$ and $g : (U \times
	\grass{\adim}{\vdim} ) \times \rel \to U \times \rel$ by $g((z,S),t) =
	(z,t)$ for $z \in U$, $S \in \grass{\adim}{\vdim}$ and $t \in \rel$.
	Using \cite[2.2.2, 2.2.3, 2.2.17, 2.6.2]{MR41:1976}, one obtains that
	$A$ is $\| V \| \times \mathscr{L}^1$ and $\| \delta V \| \times
	\mathscr{L}^1$ measurable, hence that $g^{-1} \lIm A \rIm$ is $V
	\times \mathscr{L}^1$ measurable since $\| V \| \times \mathscr{L}^1 =
	g_\# ( V \times \mathscr{L}^1 )$

	Define $p : \rel^\adim \times \rel \to \rel^\adim$ by $p(z,t) = z$ for
	$(z,t) \in \rel^\adim \times \rel$ and let $T \in \mathscr{D} ' ( U
	\times \rel, \rel^\adim )$ be defined by
	\begin{align*}
		T ( \psi ) & = \tint{A}{} \psi (z,t) \bullet \eta (V,z) \ud (
		\| \delta V \| \times \mathscr{L}^1 ) (z,t) \\
		& \phantom{=} \ - \tint{g^{-1} \lIm A \rIm}{} ( D \psi (z,t)
		\circ p^\ast ) \bullet \project{S} \ud (V \times
		\mathscr{L}^1) ((z,S),t)
	\end{align*}
	whenever $\psi \in \mathscr{D} ( U \times \rel, \rel^\adim )$, see
	\ref{miniremark:situation_general_varifold}. Fubini's theorem then
	yields the two equations. The uniqueness of $T$ follows from
	\ref{miniremark:distrib_on_products}.
\end{proof}
\begin{remark}
	Notice that characterising equation for $T$ also holds if the
	requirement $\inf \spt \eta > - \infty$ is dropped.
\end{remark}
\begin{definition} \label{def:v_weakly_diff}
	Suppose $l, \vdim, \adim \in \nat$, $\vdim \leq \adim$, $U$ is an open
	subset of $\rel^\adim$, $V \in \RVar_\vdim ( U )$, and $\| \delta V
	\|$ is a Radon measure.

	Then an $\rel^l$ valued $\| V \| + \| \delta V \|$ measurable function
	$f$ with $\dmn f \subset U$ is called \emph{generalised $V$ weakly
	differentiable} if and only if for some $\| V \|$ measurable $\Hom
	(\rel^\adim, \rel^l )$ valued function $F$ the following two
	conditions hold:
	\begin{enumerate}
		\item \label{item:v_weakly_diff:int} If $\eta \in \mathscr{K}
		( \rel^l )$ then $( \eta \circ f ) \cdot F \in \Lploc{1} ( \|
		V \|, \Hom ( \rel^\adim, \rel^l) )$.
		\item \label{item:v_weakly_diff:partial} If $\theta \in
		\mathscr{D} ( U, \rel^\adim )$, $\eta \in \mathscr{E}^0 ( \rel^l
		)$ and $\spt D \eta$ is compact then, see
		\ref{miniremark:extension},
		\begin{gather*}
			\begin{aligned}
				& ( \delta V ) ( ( \eta \circ f ) \theta ) \\
				& \qquad = \tint{}{} \eta(f(z)) \project{S}
				\bullet D \theta (z) \ud V (z,S) + \tint{}{}
				\left < \theta(z), D\eta (f(z)) \circ F (z)
				\right > \ud \| V \| z.
			\end{aligned}
		\end{gather*}
	\end{enumerate}
	The function $F$ is $ \| V \|$ almost unique. Therefore, one may
	define the \emph{generalised $V$ weak derivative of $f$} to be the
	function $\derivative{V}{f}$ characterised by $a \in \dmn
	\derivative{V}{f}$ if and only if
	\begin{gather*}
		(\| V \|, C ) \aplim_{z\to a} F (z) = \tau \quad \text{for
		some $\tau \in \Hom ( \rel^\adim, \rel^l)$}
	\end{gather*}
	and in this case $\derivative{V}{f}(a) = \tau$, where $C = \{
	(a,\cball{a}{r}) \with \cball{a}{r} \subset U \}$. Moreover, the set
	of all $\rel^l$ valued generalised $V$ weakly differentiable functions
	will be denoted by $\trunc ( V, \rel^l )$ and $\trunc (V) = \trunc
	(V,\rel)$.
\end{definition}
\begin{remark} \label{remark:eq_bounded_condition}
	The condition \eqref{item:v_weakly_diff:int} is equivalent to
	$\tint{\classification{K}{z}{|f(z)|\leq s}}{} |F| \ud \| V \| <
	\infty$ whenever $K$ is compact subset of $U$ and $0 \leq s < \infty$.
\end{remark}
\begin{remark} \label{remark:eq_condition_weak_diff}
	The condition \eqref{item:v_weakly_diff:partial} is equivalent to the
	following one: If $\gamma \in \mathscr{D}^0 ( U )$, $v \in
	\rel^\adim$, $\eta \in \mathscr{E}^0 ( \rel^l )$ and $\spt D \eta$ is
	compact then
	\begin{gather*}
		\begin{aligned}
			& ( \delta V ) ( ( \eta \circ f ) \gamma \cdot v ) \\
			& \qquad = \tint{}{} \eta(f(z)) \left <
			\project{S}(v), D \gamma (z) \right > \ud V (z,S) +
			\tint{}{} \gamma (z) \left < v , D\eta (f(z)) \circ F
			(z) \right > \ud \| V \| z.
		\end{aligned}
	\end{gather*}
\end{remark}
\begin{remark} \label{remark:associated_distribution}
	If $f \in \trunc(V)$ then the distribution $T$ associated to $f$ in
	\ref{lemma:meas_fct} is representable by integration and, see
	\ref{lemma:push_on_product},
	\begin{gather*}
		T ( \psi ) = \tint{}{} \left < \psi (z,f(z)),
		\derivative{V}{f} (z) \right > \ud \| V \| z, \\
		\tint{}{} g \ud \| T \| = \tint{}{} g(z,f(z)) |
		\derivative{V}{f} (z) | \ud \| V \| z
	\end{gather*}
	whenever $\psi \in \Lp{1} ( \| T \|, \rel^\adim )$ and $g$ is an
	$\overline{\rel}$ valued $\| T \|$ integrable function.
\end{remark}
\begin{example} \label{example:lipschitzian}
	If $f : U \to \rel^l$ is a locally Lipschitzian function then $f$ is
	generalised $V$ weakly differentiable with
	\begin{gather*}
		\derivative{V}{f}(z) = ( \| V \|, \vdim ) \ap Df (z) \circ
		\project{\Tan^\vdim ( \| V \|, z )} \quad \text{for $\| V\|$
		almost all $z$},
	\end{gather*}
	as may be verified by means of \cite[4.5\,(4)]{snulmenn.decay}.
	Moreover, if $\density^\vdim ( \| V \|, z ) \geq 1$ for $\| V \|$
	almost all $z$, then the equality holds for $f \in \trunc (V,\rel^l)$
	as will be shown in \ref{thm:approx_diff}.
\end{example}
\begin{remark} \label{remark:integration_by_parts}
	\emph{If $\alpha \in \Hom ( \rel^l, \rel )$, $f \in \trunc (V,\rel^l)
	\cap \Lploc{1} ( \| V \| + \| \delta V \|, \rel^l )$, and
	$\derivative{V}{f} \in \Lploc{1} ( \| V \|, \Hom ( \rel^\adim,
	\rel^l))$ then $\alpha \circ f \in \trunc (V)$ and
	\begin{gather*}
		\derivative{V}{(\alpha \circ f)} (z) = \alpha \circ
		\derivative{V}{f} (z) \quad \text{for $\| V \|$ almost all
		$z$}, \\
		( \delta V ) ( ( \alpha \circ f ) \theta ) = \alpha \big (
		\tint{}{} ( \project{S} \bullet D \theta (z) ) f(z) + \left <
		\theta(z), \derivative{V}{f} (z) \right > \ud V (z,S) \big )
	\end{gather*}
	whenever $\theta \in \mathscr{D} ( U, \rel^\adim )$}. To prove this
	suppose $\eta \in \mathscr{E}^0 ( \rel )$ with $\Lip \eta < \infty$,
	choose $\phi \in \mathscr{D}^0 ( \rel^l )$ and $\phi(0)=1$, define
	$\eta_r = ( \phi \circ \boldsymbol{\mu}_r ) ( \eta \circ \alpha )$ for
	$0 < r < \infty$, observe
	\begin{gather*}
		\eta_r (y) \to \eta ( \alpha (y)) \quad \text{and} \quad
		D\eta_r (y) \to D \eta ( \alpha (y) ) \circ \alpha \qquad
		\text{as $r \to 0+$ for $y \in \rel^l$}, \\
		\sup \{ | \eta_r (0) | \with 0 < r \leq 1 \} + \sup \{ | D
		\eta_r (y) | \with y \in \rel^l, 0 < r \leq 1 \} < \infty
	\end{gather*}
	and consider the limit $r \to 0+$ in
	\ref{def:v_weakly_diff}\,\eqref{item:v_weakly_diff:partial} with
	$\eta$ replaced by $\eta_r$.
\end{remark}
\begin{remark}
	The prefix ``generalised'' has been chosen in analogy with the notion
	of ``generalised function of bounded variation'' treated in
	\cite[\S 4.5]{MR2003a:49002} originating from De~Giorgi and Ambrosio
	\cite{MR1152641}.
\end{remark}
\begin{remark}
	The usefulness of partial integration identities involving the first
	variation in defining a concept of weakly differentiable functions on
	varifolds has already been ``expected'' by Anzellotti, Delladio and
	Scianna who developed two notions of functions of bounded variation on
	integral currents, see \cite[p.~261]{MR1441622},
\end{remark}
\begin{remark} \label{remark:bv}
	In order to define a concept of ``generalised (real valued) functions
	of bounded variation'' with respect to a varifold, it could be of
	interest to study the class of those functions $f$ satisfying the
	hypotheses of \ref{lemma:meas_fct} such that the associated function
	$T$ is representable by integration.
\end{remark}
\begin{remark} \label{remark:moser}
	A concept related to the present one has been proposed by Moser in
	\cite[Definition 4.1]{62659} in the context of curvature varifolds
	(see \ref{def:curvature_varifold} and
	\ref{remark:hutchinson_reformulations}); in fact, it allows for
	certain ``multiple-valued'' functions. In studying convergence of
	pairs of varifolds and weakly differentiable functions, it would seem
	natural to investigate the extension of the present concept to such
	functions. (Notice that the usage of the term ``multiple-valued'' here
	is different but related to the one of Almgren in \cite[\S
	1]{MR1777737}).
\end{remark}
\begin{lemma} \label{lemma:basic_v_weakly_diff}
	Suppose $l$, $\vdim$, $\adim$, $U$, and $V$ are as in
	\ref{def:v_weakly_diff}, $f \in \trunc (V,\rel^l)$, $k \in \nat$, $0
	\leq c < \infty$, $A$ is a closed subset of $\rel^l$, $\phi : \rel^l
	\to \rel^k$, $\zeta : \rel^l \to \Hom ( \rel^l, \rel^k )$, $\phi_i \in
	\mathscr{E} ( \rel^l, \rel^k)$, $\spt D \phi_i \subset A$, $\Lip
	\phi_i \leq c$, $\phi | A$ is proper,
	\begin{gather*}
		\phi(y) = \lim_{i \to \infty} \phi_i (y) \quad \text{uniformly
		in $y \in \rel^l$}, \\
		\zeta (y) = \lim_{i \to \infty} D \phi_i ( y ) \quad \text{for
		$y \in \rel^l$}.
	\end{gather*}
	
	Then $\phi \circ f \in \trunc (V,\rel^k)$ and
	\begin{gather*}
		\derivative{V}{( \phi \circ f )} (z) = \zeta ( f (z) ) \circ
		\derivative{V}{f} (z) \quad \text{for $\| V \|$ almost all
		$z$}.
	\end{gather*}
\end{lemma}
\begin{proof}
	Note $\Lip \phi \leq c$, $\| \zeta(y) \| \leq c$ for $y \in \rel^l$,
	and $\zeta(y) = 0$ for $y \in \rel^l \without A$, hence
	\begin{gather*}
		\{y \with \eta ( \phi (y ) ) \zeta (y) \neq 0 \} \subset (
		\phi | A )^{-1} \lIm \spt \eta \rIm \quad \text{for $\eta \in
		\mathscr{K} ( \rel^k )$}.
	\end{gather*}
	One infers that the function mapping $z \in \dmn \derivative{V}{f}$
	onto $\eta ( \phi ( f (z) ) ) \zeta ( f(z)) \circ \derivative{V}{f}
	(z)$ belongs to $\Lploc{1} ( \| V \|, \Hom ( \rel^\adim, \rel^k ) )$.

	Suppose $\eta \in \mathscr{E}^0 ( \rel^k )$ and $K$ is a compact subset
	of $\rel^k$ with $\spt D \eta \subset \Int K$ and $C = ( \phi |
	A)^{-1} \lIm K \rIm$. Then $\im \eta$ is bounded, $C$ is compact and
	\begin{gather*}
		\{ y \with D \eta ( \phi_i (y ) ) \circ D \phi_i (y) \neq 0 \}
		\subset ( \phi_i | A )^{-1} \lIm \spt D \eta \rIm \subset C
		\quad \text{for large $i$},
	\end{gather*}
	in particular $\spt D ( \eta \circ \phi_i ) \subset C$ for such $i$.
	Therefore
	\begin{gather*}
		\begin{aligned}
			& ( \delta V ) ( ( \eta \circ \phi_i \circ f ) \theta
			) = \tint{}{} \eta(\phi_i(f(z))) \project{S}
			\bullet D \theta (z) \ud V (z,S) \\
			& \qquad + \tint{}{} \left < \theta(z), D\eta
			(\phi_i(f(z))) \circ D \phi_i (f(z)) \circ
			\derivative{V}{f} (z) \right > \ud \| V \| z
		\end{aligned}
	\end{gather*}
	and considering the limit $i \to \infty$ yields the conclusion.
\end{proof}
\begin{example} \label{example:composite}
	Amongst the functions $\phi$ and $\zeta$ admitting an approximation as
	in \ref{lemma:basic_v_weakly_diff} are the following:
	\begin{enumerate}
		\item \label{item:composite:scalar_mult} If $r \in \rel$
		then $\phi = \boldsymbol{\mu}_r$ with $\zeta = D
		\boldsymbol{\mu}_r$ is admissible.
		\item \label{item:composite:add_constant} If $y \in \rel^l$
		then $\phi = \boldsymbol{\tau}_y$ with $\zeta =
		D\boldsymbol{\tau}_y$ is admissible.
		\item \label{item:composite:mod} If $c \in \rel^l$
		then one may take $\phi$ and $\zeta$ such that $\phi (y) =
		|y-c|$ for $y \in \rel^l$,
		\begin{gather*}
			\text{$\zeta (y)(v) = |y-c|^{-1} (y-c) \bullet v$ if
			$y \neq c$}, \quad \text{$\zeta (y) = 0$ if $y=c$}
		\end{gather*}
		whenever $v, y \in \rel^l$.
		\item \label{item:composite:1d} If $l=1$ and $t \in
		\rel$ then one may take $\phi$ and $\zeta$ such that
		\begin{gather*}
			\phi (y) = \sup \{ y, t \}, \quad \text{$\zeta (y)(v)
			= v$ if $y > t$}, \quad \text{$\zeta (y) = 0$ if $y
			\leq t$}
		\end{gather*}
		whenever $v,y \in \rel$.
	\end{enumerate}

	To prove \eqref{item:composite:mod}, choose $\Phi \in \mathscr{D}^0 (
	\rel^l )^+$ with $\int \Phi \ud \mathscr{L}^l = 1$ and $\Phi (y) = \Phi
	(-y)$ for $y \in \rel^l$ and take $k=1$, $c=1$, $A = \rel^l$, and
	$\phi_i = \Phi_{1/i} \ast \phi$ in \ref{lemma:basic_v_weakly_diff}
	noting $(\Phi_{1/i} \ast \phi) (c-y) = (\Phi_{1/i} \ast \phi)(c+y)$
	for $y \in \rel^l$, hence $D ( \Phi_{1/i} \ast \phi ) (c)=0$.

	To prove \eqref{item:composite:1d}, choose $\Phi \in \mathscr{D}^0 (
	\rel )^+$ with $\int \Phi \ud \mathscr{L}^1 = 1$, $\spt \Phi \subset
	\cball{0}{1}$ and $\varepsilon = \inf \spt \Phi > 0$, and take $k=1$,
	$c=1$, $A = \classification{\rel}{y}{y \geq t}$, and $\phi_i =
	\Phi_{1/i} \ast \phi$ in \ref{lemma:basic_v_weakly_diff} noting
	$\phi_i (y) = t$ if $-\infty < y \leq t + \varepsilon/i$, hence $D
	\phi_i (y) = 0$ for $- \infty < y \leq t$.
\end{example}
\begin{remark} \label{remark:comparison_trunc_spaces}
	If $l=1$, $\vdim = \adim$, and $\| V \| ( A ) = \mathscr{L}^\vdim (A)$
	for $A \subset U$, then $f \in \trunc (V)$ if and only if $f$ belongs
	to the class $\mathscr{T}^{1,1}_{\mathrm{loc}} ( U )$ introduced by
	B{\'e}nilan, Boccardo, Gallou{\"e}t, Gariepy, Pierre and V{\'a}zquez
	in \cite[p.~244]{MR1354907} and in this case $\derivative{V}{f}$
	corresponds to ``the derivative $Df$ of $f \in
	\mathscr{T}^{1,1}_{\mathrm{loc}} ( U)$'' of \cite[p.~246]{MR1354907}
	as may be verified by use of \ref{lemma:basic_v_weakly_diff},
	\ref{example:composite}\,\eqref{item:composite:scalar_mult}\,\eqref{item:composite:1d}
	and \cite[Lemmas 2.1, 2.3]{MR1354907}.
\end{remark}
\begin{definition} [see \protect{\cite[p.~419]{MR90g:47001a}}]
	Whenever $X$ is Banach space the weak topology is the topology on $X$
	induced by all continuous linear functionals mapping $X$ into $\rel$.
	Topological notions referring to the weak topology will be designated
	by the prefix ``weakly'' whereas topological notions without
	qualification will refer to the metric topology on $X$.
\end{definition}
\begin{lemma} \label{lemma:functional_analysis}
	Suppose $U$ is an open subset of $\rel^\adim$, $\mu$ is a Radon
	measure on $U$, $Y$ is a finite dimensional normed space, $g \in
	\Lp{1} ( \mu )$, $K$ denotes the set of all $f \in \Lp{1} ( \mu, Y )$
	such that
	\begin{gather*}
		|f(x)| \leq g ( x) \quad \text{for $\mu$ almost all $x$},
	\end{gather*}
	$L_1 ( \mu, Y ) = \Lp{1} ( \mu , Y ) / \{ f \with \Lpnorm{\mu}{1}{f} =
	0 \}$ is the (usual) quotient Banach space, and $\pi : \Lp{1} ( \mu, Y
	) \to L_1 ( \mu, Y )$ denotes the canonical projection.

	Then $\pi \lIm K \rIm$ with the topology induced by the weak topology
	on $L_1 ( \mu, Y )$ is compact and metrisable.
\end{lemma}
\begin{proof}
	First, notice that $\pi \lIm K \rIm$ is convex and closed, hence
	weakly closed by \cite[\printRoman{5}.3.13]{MR90g:47001a}.  Therefore
	one may assume that $Y = \rel$ employing the isomorphism $\Lp{1} ( \mu
	)^{\dim Y} \simeq \Lp{1} ( \mu, Y )$ induced by a basis of $Y$ and
	\cite[\printRoman{5}.3.15]{MR90g:47001a}. Since $L_1 ( \mu, \rel )$ is
	separable, the conclusion now follows combining
	\cite[\printRoman{4}.8.9, \printRoman{5}.6.1,
	\printRoman{5}.6.3]{MR90g:47001a}.
\end{proof}
\begin{lemma} \label{lemma:comp_lip}
	Suppose $l$, $\vdim$, $\adim$, $U$, and $V$ are as in
	\ref{def:v_weakly_diff}, $f \in \trunc (V, \rel^l)$, $k \in \nat$, $A$
	is a closed subset of $\rel^l$, $a$ is the characteristic function of
	$f^{-1} \lIm A \rIm$, and $\phi : \rel^l \to  \rel^k$ is a
	Lipschitzian function such that $\phi | A$ is proper and $\phi |
	\rel^l \without A$ is locally constant.

	Then $\phi \circ f \in \trunc (V, \rel^k)$ and
	\begin{gather*}
		\| \derivative{V}{( \phi \circ f )} (z) \| \leq \Lip ( \phi )
		a(z) \| \derivative{V}{f} (z) \| \quad \text{for $\| V \|$
		almost all $z$}.
	\end{gather*}
\end{lemma}
\begin{proof}
	Abbreviate $c = \Lip \phi$. Define $B = \rel^l \cap \{ y \with \dist
	(y,A) \leq 1 \}$ and let $b$ denote the characteristic function of
	$f^{-1} \lIm B \rIm$. Since $\phi | B$ is proper, one may employ
	convolution to construct $\phi_i \in \mathscr{E} ( \rel^l, \rel^k )$
	satisfying $\Lip \phi_i \leq c$, $\spt D\phi_i \subset B$ and
	\begin{gather*}
		\delta_i = \sup \{ | (\phi - \phi_i) (y) | \with y \in
		\rel^l \} \to 0 \quad \text{as $i \to \infty$}.
	\end{gather*}
	Therefore, if $\delta_i < \infty$ then $\phi_i|B$ is proper and
	$\phi_i \circ f \in \trunc (V,\rel^k)$ with
	\begin{gather*}
		\| \derivative{V}{( \phi_i \circ f )}(z) \| \leq c \, b(z) \|
		\derivative{V}{f} (z) \| \quad \text{for $\| V \|$ almost all
		$z$}
	\end{gather*}
	by \ref{lemma:basic_v_weakly_diff} with $A$, $\phi$, and $\zeta$
	replaced by $B$, $\phi_i$, and $D\phi_i$. Choose a sequence of compact
	sets $K_j$ such that $K_j \subset \Int K_{j+1}$ for $j \in \nat$ and
	$U = \bigcup_{j=1}^\infty K_j$ and define $E(j) = K_j \cap \{ z \with
	|f(z)| < j \}$ for $j \in \nat$. In view of
	\ref{lemma:functional_analysis}, possibly passing to a subsequence by
	means of a diagonal process, there exist functions $F_j \in \Lp{1} (
	\| V \| \restrict E(j), \Hom ( \rel^\adim, \rel^k ) )$ such that
	\begin{gather*}
		\| F_j(z) \| \leq c \, b(z) \| \derivative{V}{f} (z) \| \quad
		\text{for $\| V \|$ almost all $z \in E_j$}, \\
		\tint{E(j)}{} G \bullet
		\derivative{V}{(\phi_i \circ f )} \ud \| V \| \to
		\tint{E(j)}{} G \bullet F_j \ud \| V \| \quad \text{as $i \to
		\infty$}
	\end{gather*}
	whenever $G \in \Lp{\infty} ( \| V \|, \Hom (\rel^\adim, \rel^k) )$
	and $j \in \nat$. Noting $E_j \subset E_{j+1}$ and $F_j (z) = F_{j+1}
	(z)$ for $\| V \|$ almost all $E(j)$ for $j \in \nat$, one may define
	a $\| V \|$ measurable function $F$ by $F(z) = \lim_{j \to \infty}
	F_j(z)$ whenever $z \in U$.
	
	In order to verify 
	$\phi \circ f \in \trunc (V,\rel^k)$ with
	$\derivative{V} ( \phi \circ f )(z) = F(z)$ for 
	$\| V \|$ almost all $z$, suppose $\theta \in \mathscr{D} ( U,
	\rel^\adim )$, $\eta \in \mathscr{E}^0 ( \rel^k )$ with $\spt D\eta$
	compact. Then there exists $j \in \nat$ with $\spt \theta \subset K_j$
	and $(\phi|B)^{-1} \lIm \spt D \eta \rIm \subset \oball{0}{j}$.
	Define $G_i, G \in \Lp{\infty} ( \| V \|, \Hom ( \rel^\adim, \rel^k)
	)$ by the requirements
	\begin{gather*}
		G_i(z) \bullet \tau = \left < \theta (z), D\eta ( \phi_i(f(z))
		) \circ \tau \right>, \quad G(z) \bullet \tau = \left < \theta
		(z), D\eta ( \phi(f(z)) ) \circ \tau \right>
	\end{gather*}
	whenever $z \in \dmn f$ and $\tau \in \Hom ( \rel^\adim, \rel^k )$,
	hence
	\begin{gather*}
		\Lpnorm{\| V \|}{\infty}{G_i-G} \to 0 \quad \text{as $i \to
		\infty$}.
	\end{gather*}
	Observing $( \phi_i|B )^{-1} \lIm \spt D \eta \rIm \subset
	\oball{0}{j}$ for large $i$, one infers
	\begin{gather*}
		\begin{aligned}
			& \tint{}{} \left < \theta (z), D \eta ( \phi (f(z)))
			\circ F (z) \right > \ud \| V \| z = \tint{E(j)}{} G
			\bullet F \ud \| V \| \\
			& \qquad = \lim_{i \to \infty} \tint{E(j)}{} G_i (z)
			\bullet \derivative{V}{(\phi_i \circ f)} (z) \ud \| V
			\| z \\
			& \qquad = \lim_{i \to \infty} \tint{}{} \left <
			\theta (z), D\eta ( \phi_i(f(z)) \circ
			\derivative{V}{(\phi_i \circ f )}(z) \right > \ud \| V
			\| z
		\end{aligned}
	\end{gather*}
	as $(G \bullet F) (z) = 0 = ( G_i \bullet \derivative{V}{( \phi_i
	\circ f)} ) (z) $ for $\| V \|$ almost all $z \in U \without E(j)$.
\end{proof}
\begin{miniremark} \label{miniremark:trunc}
	The following approximation of $\id{\rel^l}$ will be useful.

	\emph{There exists a family of functions $\phi_s \in \mathscr{D} (
	\rel^l, \rel^l )$ with $0 < s < \infty$ satisfying
	\begin{gather*}
		\{ y \with \alpha ( \phi_s (y) ) > 0 \} \subset \{ y \with
		\alpha (y) > 0 \} \quad \text{whenever $\alpha \in \Hom (
		\rel^l, \rel )$, $0 < s < \infty$}, \\
		\phi_s (y) = y \quad \text{whenever $y \in \rel^l \cap
		\cball{0}{s}$, $0 < s < \infty$}, \\
		\sup \{ \Lip \phi_s : 0 < s < \infty \} < \infty;
	\end{gather*}}
	in fact one may select $\gamma \in \mathscr{D}^0 ( \rel )$ such that
	\begin{gather*}
		0 \leq \gamma (t) \leq t \quad \text{for $0 \leq t < \infty$},
		\qquad \gamma (t) = t \quad \text{for $-1 \leq t \leq 1$},
	\end{gather*}
	define $\gamma_s = s \gamma \circ \boldsymbol{\mu}_{1/s}$ and $\phi_s
	\in \mathscr{D} ( \rel^l, \rel^l )$ by
	\begin{gather*}
		\phi_s (y) = 0 \quad \text{if $y=0$}, \qquad \phi_s (y) =
		\gamma_s (|y|) |y|^{-1} y \quad \text{if $y \neq 0$},
	\end{gather*}
	whenever $y \in \rel^l$ and $0 < s < \infty$, and conclude
	\begin{gather*}
		\gamma_s (t) = t \quad \text{for $-s \leq t \leq s$}, \qquad
		\Lip \gamma_s = \Lip \gamma, \\
		D \phi_s (y) (v) = \gamma_s' ( |y| ) ( |y|^{-1} y ) \bullet v
		( |y|^{-1} y ) + \gamma_s ( |y| ) |y|^{-1} \big ( v - (
		|y|^{-1} y ) \bullet v ( |y|^{-1} y ) \big )
	\end{gather*}
	whenever $y \in \rel^l \without \{ 0 \}$, $v \in \rel^l$, and $0 < s <
	\infty$, hence $\Lip \phi_s \leq 2 \Lip \gamma < \infty$.
\end{miniremark}
\begin{theorem} \label{thm:addition}
	Suppose $l, \vdim, \adim \in \nat$, $\vdim \leq \adim$, $U$ is an open
	subset of $\rel^\adim$, $V \in \RVar_\vdim (U)$, $\| \delta V \|$ is a
	Radon measure, and $f \in \trunc (V,\rel^l)$.

	Then the following four statements hold:
	\begin{enumerate}
		\item \label{item:addition:zero} If $A =
		\classification{U}{z}{f(z) = 0}$ then
		\begin{gather*}
			\derivative{V}{f} (z) = 0 \quad \text{for $\| V \|$
			almost all $z \in A$}.
		\end{gather*}
		\item \label{item:addition:join} If $k \in \nat$, $g : U \to
		\rel^k$ is locally Lipschitzian, and $h(z) = (f(z),g(z))$ for
		$z \in \dmn f$ then $h \in \trunc (V, \rel^l \times \rel^k )$
		and
		\begin{gather*}
			\derivative{V}{h} (z)(v) = ( \derivative{V}{f}(z)(v),
			\derivative{V}{g} (z)(v)) \quad \text{whenever $v \in
			\rel^\adim$}
		\end{gather*}
		for $\| V \|$ almost all $z$.
		\item \label{item:addition:add} If $g : U \to \rel^l$ is
		locally Lipschitzian then $f+g \in \trunc (V,\rel^l)$ and
		\begin{gather*}
			\derivative{V}{(f+g)} (z) = \derivative{V}{f} (z) +
			\derivative{V}{g} (z) \quad \text{for $\| V \|$ almost
			all $z$}.
		\end{gather*}
		\item \label{item:addition:mult} If $f \in \Lploc{1} ( \| V
		\|, \rel^l )$, $\derivative{V}{f} \in \Lploc{1} ( \| V \|,
		\Hom ( \rel^\adim, \rel^l ) )$, and $g : U \to \rel$ is
		locally Lipschitzian, then $gf \in \trunc ( V, \rel^l )$ and
		\begin{gather*}
			\derivative{V}{(gf)} (z) = \derivative{V}{g} (z)\,f(z)
			+ g(z) \derivative{V}{f} (z) \quad \text{for $\| V \|$
			almost all $z$}.
		\end{gather*}
	\end{enumerate}
\end{theorem}
\begin{proof} [Proof of \eqref{item:addition:zero}]
	By \ref{miniremark:trunc} in conjunction with
	\ref{lemma:basic_v_weakly_diff} one may assume $f$ to be bounded and
	$\derivative{V}{f} \in \Lploc{1} ( \| V \|, \Hom ( \rel^\adim, \rel^l)
	)$, hence by \ref{remark:integration_by_parts} also $l=1$. In this
	case it follows from \ref{lemma:basic_v_weakly_diff} and
	\ref{example:composite}\,\eqref{item:composite:1d} that $f^+$ and
	$f^-$ satisfy the same hypotheses as $f$, hence
	\begin{gather*}
		( \delta V ) ( g \theta ) = \tint{}{} ( \project{S} \bullet
		D\theta (z) ) g (z) + \left < \theta (z), \derivative{V}{g}
		(z) \right > \ud V (z,S)
	\end{gather*}
	for $\theta \in \mathscr{D}^0 ( U )$ and $g \in \{ f, f^+, f^- \}$ by
	\ref{remark:integration_by_parts}. Since $f=f^+-f^-$, this implies
	\begin{gather*}
		\derivative{V}{f} (z) = \derivative{V}{f^+} (z) -
		\derivative{V}{f^-} (z) \quad \text{for $\| V \|$ almost all
		$z$}
	\end{gather*}
	and the formulae derived in
	\ref{example:composite}\,\eqref{item:composite:1d} yield the
	conclusion.
\end{proof}
\begin{proof} [Proof of \eqref{item:addition:join}]
	Define a $\| V \|$ measurable function $H$ with values in $\Hom (
	\rel^\adim, \rel^l \times \rel^k )$ by
	\begin{gather*}
		H(z)(v) = ( \derivative{V}{f} (z)(v), \derivative{V}{g} (z)(v)
		) \quad \text{for $v \in \rel^\adim$}
	\end{gather*}
	whenever $z \in \dmn \derivative{V}{f} \cap \dmn \derivative{V}{g}$.
	Suppose $\theta \in \mathscr{D} ( U, \rel^\adim )$, $C$ is a compact
	subset of $\rel^l$, and $A = g \lIm \spt \theta \rIm$. Choose a
	compact convex subset $D$ of $\rel^k$ with
	\begin{gather*}
		\cball{\xi}{1} \subset D \quad \text{whenever $\xi \in A$}.
	\end{gather*}
	
	Next, it will be shown that \emph{every $\eta \in \mathscr{E}^0 (
	\rel^l \times \rel^k )$ such that $\spt D \eta$ is a compact subset of
	$C \times \rel^k$ belongs to the class $T$ of such $\eta$ satisfying
	\begin{gather*}
		( \delta V ) ( ( \eta \circ h ) \theta ) = \tint{}{} \eta
		(h(z)) \project{S} \bullet D \theta (z) + \left < \theta (z),
		D \eta ( h (z) ) \circ H(z) \right > \ud V (z,S).
	\end{gather*}}
	For this purpose, first consider the case that for some $\eta_1 \in
	\mathscr{E}^0 ( \rel^k )$ and $\eta_2 \in \mathscr{E}^0 ( \rel^l )$
	with
	\begin{gather*}
		\eta ( y_1, y_2 ) = \eta_1 ( y_1 ) \eta_2 ( y_2 ) \quad
		\text{for $(y_1,y_2) \in \rel^l \times \rel^k$}, \\
		\spt D \eta_1 \subset C \quad \text{and} \quad \text{$\spt D
		\eta_2$ is compact}.
	\end{gather*}
	In this case one computes, approximating $( \eta_2 \circ g ) \theta$
	by means of convolution and \cite[4.5\,(3)]{snulmenn.decay} and
	using \ref{example:lipschitzian},
	\begin{gather*}
		\begin{aligned}
			& ( \delta V ) ( ( \eta \circ h ) \theta ) = ( \delta
			V ) ( ( \eta_1 \circ f ) ( \eta_2 \circ g ) \theta )
			\\
			& \qquad = \tint{}{} \eta_1 ( f(z)) \big ( \left <
			\theta (z), D \eta_2 ( g(z) ) \circ \derivative{V}{g}
			(z) \right > + \eta_2 ( g(z)) \project{S} \bullet D
			\theta (z) \big ) \ud V (z,S) \\
			& \qquad \phantom{=} \ + \tint{}{} \eta_2 ( g (z)
			) \left < \theta (z), D \eta_1 ( f (z) ) \circ
			\derivative{V}{f} (z) \right > \ud \| V \| z \\
			& \qquad = \tint{}{} \eta ( h (z) ) \project{S}
			\bullet D \theta (z) + \left <
			\theta (z), D \eta ( h (z) ) \circ H (z) \right > \ud
			V(z,S),
		\end{aligned}
	\end{gather*}
	hence $\eta \in T$. Second, $T$ is a vectorspace. Third, observe if
	$\eta \in \mathscr{E}^0 ( \rel^l \times \rel^k )$, $\spt D \eta$ is a
	compact subset of $C \times \rel^k$, and $\eta_i$ is a sequence in $T$
	with
	\begin{gather*}
		\text{$\eta_i (y) \to \eta(y)$ and $D \eta_i (y) \to D \eta
		(y)$ as $i \to \infty$ uniformly in $y \in \rel^l \times A$}
	\end{gather*}
	then $\eta \in T$. Finally, consider $\eta \in \mathscr{E}^0 ( \rel^l
	\times \rel^k )$ such that $\spt D \eta$ is a compact subset of $C
	\times \rel^k$. Choose $\Phi \in \mathscr{D}^0 ( \rel^k )^+$ with
	$\tint{}{} \Phi \ud \mathscr{L}^k = 1$ and $\spt \Phi \subset
	\oball{0}{1}$ and define $\Phi_\varepsilon ( \xi ) = \varepsilon^{-k}
	\Phi ( \varepsilon^{-1} \xi )$ for $0 < \varepsilon \leq 1$ and $\xi
	\in \rel^k$. One approximates $\eta$ by functions $\phi$ defined by
	\begin{gather*}
		\phi ( y_1, y_2 ) = \tint{}{} \Phi_\varepsilon ( y_2 - \xi )
		\eta ( y_1, \xi ) \ud \mathscr{L}^k \xi \quad \text{for
		$(y_1,y_2) \in \rel^l \times \rel^k$}
	\end{gather*}
	corresponding to small $\varepsilon$ and $\phi$ by functions $\zeta$
	defined by
	\begin{gather*}
		\zeta (y_1,y_2) = \tsum{i=1}{j} \mathscr{L}^k ( B_i )
		\Phi_\varepsilon ( y_2 - b_i ) \eta ( y_1, b_i ) \quad
		\text{for $(y_1,y_2) \in \rel^l \times \rel^k$}
	\end{gather*}
	corresponding to Borel partitions $B_1, \ldots, B_j$ of $D$ with small
	diameter where $b_i \in B_i$. By the first two considerations the
	functions $\zeta$ belong to $T$. Noting
	\begin{gather*}
		\begin{aligned}
			D \phi (y_1,y_2) (v_1,v_2) & = \tint{}{}
			\Phi_\varepsilon ( y_2 - \xi ) \left <(v_1,0), D \eta
			( y_1, \xi ) \right > + \\
			& \phantom{= \tint{}{}} \ \left < v_2, D
			\Phi_\varepsilon ( y_2 - \xi ) \right > \eta ( y_1,
			\xi) \ud \mathscr{L}^k \xi, \\
			D \zeta (y_1,y_2) (v_1,v_2) & = \tsum{i=1}{j}
			\mathscr{L}^k ( B_i ) \big ( \Phi_\varepsilon ( y_2 -
			b_i ) \left < ( v_1,0 ), D \eta ( y_1, b_i ) \right >
			\\
			& \phantom{= \tsum{i=1}{j} \mathscr{L}^k ( B_i ) \big (}
			\ + \left < v_2, D \Phi_\varepsilon ( y_2 - b_i )
			\right > \eta ( y_1, b_i ) \big )
		\end{aligned}
	\end{gather*}
	whenever $(v_1,v_2),(y_1,y_2) \in \rel^l \times \rel^k$ and
	\begin{gather*}
		\Lip ( \eta | \rel^l \times D ) + \Lip ( D\eta | \rel^l \times
		D ) < \infty,
	\end{gather*}
	one infers that $\zeta$ and $D \zeta$ approach $\phi$ and $D \phi$
	uniformly on $\rel^l \times A$ as the diameter of the partition
	approaches $0$, hence the functions $\phi$ belong to $T$. Noting
	\begin{gather*}
		D \phi (y_1,y_2) = \tint{}{} \Phi_\varepsilon ( \xi ) D
		\eta ( y_1, y_2-\xi ) \ud \mathscr{L}^k \xi \quad \text{for
		$(y_1,y_2) \in \rel^l \times \rel^k$},
	\end{gather*}
	one obtains that $\phi$ and $D \phi$ approach $\eta$ and $D \eta$
	uniformly on $\rel^l \times A$ as $\varepsilon$ approaches $0$, hence
	also $\eta \in T$.

	The assertion of the preceding paragraph readily implies the
	conclusion.
\end{proof}
\begin{proof} [Proof of \eqref{item:addition:add}]
	Assume for some $0 < r < \infty$ that $\im g \subset \cball{0}{r}$.
	Suppose $\eta \in \mathscr{E}^0 ( \rel^l )$ and $\spt D \eta$ is
	compact. Let $\alpha : \rel^l \times \rel^l \to \rel^l$ denote
	addition and define $A = \rel^l \times \cball{0}{2r}$. Choosing $\zeta
	\in \mathscr{D}^0 ( \rel^l )$ with
	\begin{gather*}
		\cball{0}{r} \subset \Int \{ y \with \zeta(y) = 1 \}, \quad
		\spt \zeta \subset \cball{0}{2r},
	\end{gather*}
	one defines $\phi : \rel^l \times \rel^l \to \rel^l$ by
	\begin{gather*}
		\phi (y_1,y_2) = \zeta (y_2) ( \eta \circ \alpha ) ( y_1,y_2 )
		\quad \text{for $(y_1,y_2) \in \rel^l \times \rel^l$}.
	\end{gather*}
	Defining $h$ as in \eqref{item:addition:join} and noting
	\begin{gather*}
		\text{$\alpha | A$ is proper}, \quad \spt D \phi \subset A
		\cap \alpha^{-1} \lIm \spt \eta \rIm, \quad \text{$\spt D
		\phi$ is compact}, \\
		\phi (y_1,y_2) = ( \eta \circ \alpha ) ( y_1,y_2 ) \quad
		\text{and} \quad D \phi ( y_1,y_2 ) = D ( \eta \circ \alpha )
		( y_1,y_2 )
	\end{gather*}
	for $y_1 \in \rel^l$, $y_2 \in \rel^l \cap \cball{0}{r}$, the
	conclusion now follows from \eqref{item:addition:join} by use of
	\ref{def:v_weakly_diff}\,\eqref{item:v_weakly_diff:partial} with $f$
	and $\eta$ replaced by $h$ and $\phi$.
\end{proof}
\begin{proof} [Proof of \eqref{item:addition:mult}]
	Assume $g$ to be bounded and define $h$ as in
	\eqref{item:addition:join}.

	First, consider the special case that $f$ is bounded. Then there
	exists $0 < r < \infty$ such that $\im h \subset A(r)$, where $A(s) =
	( \rel^l \cap \cball{0}{s} ) \times ( \rel \cap \cball{0}{s} )$ for $0
	< s < \infty$. Define $\alpha : \rel^l \times \rel \to \rel^l$ by
	$\alpha (y_1,y_2) = y_2 y_1$ for $(y_1,y_2) \in \rel^l \times \rel$
	and choose $\zeta \in \mathscr{D}^0 ( \rel^l \times \rel )$ such that
	\begin{gather*}
		A(r) \subset \Int \{ y \with \zeta (y) = 1 \}, \quad \spt \zeta
		\subset A(2r).
	\end{gather*}
	For $\eta \in \mathscr{E}^0 ( \rel^l )$ define $\phi = \zeta \cdot (
	\eta \circ \alpha )$, note
	\begin{gather*}
		\phi (y) = ( \eta \circ \alpha ) (y) \quad \text{and} \quad D
		\phi (y) = D ( \eta \circ \alpha ) ( y )
	\end{gather*}
	whenever $y \in A(r)$, and the conclusion in the special case follows
	from \eqref{item:addition:join} by use of
	\ref{def:v_weakly_diff}\,\eqref{item:v_weakly_diff:partial} with $f$
	and $\eta$ replaced by $h$ and $\phi$.

	In the general case choose $\phi_s$ as in \ref{miniremark:trunc},
	hence the functions $\phi_s \circ f$ satisfy the hypotheses of the
	special case by \ref{lemma:basic_v_weakly_diff} and passing to the
	limit $s \to \infty$ yields the conclusion.
\end{proof}
\begin{remark}
	The approximation procedure in the proof of \eqref{item:addition:join}
	uses ideas from \cite[4.1.2,\,3]{MR41:1976}.
\end{remark}
\begin{remark}
	The need for some strong restriction on $g$ in
	\eqref{item:addition:join} and \eqref{item:addition:add} will be
	illustrated in \ref{example:star}.
\end{remark}
\begin{miniremark} \label{miniremark:measurablity}
	If $\phi$ is a measure, $A$ is $\phi$ measurable and $B$ is a
	$\phi \restrict A$ measurable subset of $A$, then $B$ is $\phi$
	measurable.
\end{miniremark}
\begin{theorem} \label{thm:tv_on_decompositions}
	Suppose $l, \vdim, \adim \in \nat$, $\vdim \leq \adim$, $U$ is an open
	subset of $\rel^\adim$, $V \in \Var_\vdim ( U )$, $\| \delta V \|$ is
	a Radon measure, $G$ is a decomposition of $V$, $\zeta$ is associated
	to $G$ as in \ref{remark:decomp_rep}, $f_W \in \trunc (W, \rel^l )$
	for $W \in G$, and
	\begin{gather*}
		{\textstyle f = \bigcup \{ f_W | \zeta (W) \with W \in G \},
		\quad F = \bigcup \{ \derivative{W}{f_W} | \zeta (W) \with W
		\in G \}}.
	\end{gather*}
	
	Then the following three statements hold:
	\begin{enumerate}
		\item \label{item:tv_on_decompositions:f} $f$ is $\| V \| + \|
		\delta V \|$ measurable.
		\item \label{item:tv_on_decompositions:F} $F$ is $\| V \|$
		measurable.
		\item \label{item:tv_on_decompositions:tv} If $( \eta \circ
		f ) \cdot F \in \Lploc{1} ( \| V \|, \Hom ( \rel^\adim, \rel^l
		) )$ for $\eta \in \ccspace{\rel^l}$, then $f \in \trunc
		(V,\rel^l)$ and
		\begin{gather*}
			\derivative{V}{f} (z) = F (z) \quad \text{for $\| V
			\|$ almost all $z$}.
		\end{gather*}
	\end{enumerate}
\end{theorem}
\begin{proof}
	Clearly, $f | \zeta (W) = f_W | \zeta (W)$ and $F | \zeta (W) =
	\derivative{W}{f_W} | \zeta (W)$ for $W \in G$.
	\eqref{item:tv_on_decompositions:f} and
	\eqref{item:tv_on_decompositions:F} readily follow by means of
	\ref{miniremark:measurablity}, since $\| W \| = \| V \| \restrict
	\zeta ( W )$ and $\| \delta W \| = \| \delta V \| \restrict \zeta (W)$
	for $W \in G$. If $\theta \in \mathscr{D} ( U, \rel^l )$, $\eta \in
	\mathscr{E}^0 ( \rel^l )$ and $\spt D\eta$ is compact, then the
	hypothesis of \eqref{item:tv_on_decompositions:tv} implies
	\begin{gather*}
		\begin{aligned}
			& ( \delta V ) ( ( \eta \circ f ) \theta ) = \tsum{W
			\in G}{} ( \delta W ) ( ( \eta \circ f_W ) \theta ) \\
			& \ = \tsum{W \in G}{} \tint{}{} \eta ( f_W (z))
			D \theta (z) \bullet \project{S} + \left < \theta (z),
			D \eta ( f_W ( z ) ) \circ \derivative{W}{f_W} (z)
			\right > \ud W (z,S) \\
			& \ = \tsum{W \in G}{} \tint{\zeta(W) \times
			\grass{\adim}{\vdim}}{} \eta ( f(z) ) D \theta (z)
			\bullet \project{S} + \left < \theta (z), D\eta (f(z))
			\circ F(z) \right > \ud V (z,S) \\
			& \ = \tint{}{} \eta (f(z)) D \theta (z) \bullet
			\project{S} + \left < \theta (z), D \eta(f(z)) \circ F
			(z) \right > \ud V (z,S).
		\end{aligned} \qedhere
	\end{gather*}
\end{proof}
\begin{example} \label{example:star}
	Consider $R_j = \mathbf{C} \cap \{ r \exp ( \pi \mathbf{i} j/3) \with
	0 < r < \infty \}$ and $V_j = \IVar_1 ( \mathbf{C} )$ such that $\|
	V_j \| = \mathscr{H}^1 \restrict R_j$ for $j \in \{ 1, 2, 3, 4, 5, 6
	\}$ and note $V_1 + V_3 + V_5$, $V_2 + V_4 + V_6$, $V_1 + V_4$, and
	$V_2 + V_3 + V_5 + V_6$ are stationary, define $f \in \trunc ( V_1 +
	V_3 + V_5 )$ and $g \in \trunc ( V_1 + V_4 )$ by
	\begin{gather*}
		f(z) = 1 \quad \text{if $z \in R_1 \cup R_3 \cup R_5$}, \quad
		f(z) = 0 \quad \text{else}, \\
		g(z) = 1 \quad \text{if $z \in R_1 \cup R_4$}, \qquad g(z) = 0
		\quad \text{else}
	\end{gather*}
	whenever $z \in \mathbf{C}$, define $h : \mathbf{C} \to \rel \times
	\rel$ by $h(z)=(f(z),g(z))$ for $z \in \mathbf{C}$, and let $V =
	\sum_{j=1}^6 V_j$.

	Then \ref{thm:tv_on_decompositions} implies $f, g \in \trunc (V)$
	with
	\begin{gather*}
		\derivative{V}{f} (z) = 0 = \derivative{V}{g}(z) \quad
		\text{for $\| V \|$ almost all $z$},
	\end{gather*}
	but neither $h \notin \trunc (V, \rel \times \rel)$ nor $f+g \in
	\trunc (V)$ nor $gf \in \trunc (V)$; in fact the characteristic
	function of $R_1$ which equals $gf$ does not belong to $\trunc (V)$,
	hence $f+g \notin \trunc (V)$ by \ref{lemma:basic_v_weakly_diff} and
	$h \in \trunc (V,\rel \times \rel)$ would imply $f+g \in \trunc (V)$
	by the method of \ref{thm:addition}\,\eqref{item:addition:add}.
\end{example}
\begin{remark} \label{remark:too_big_sobolev}
	Here some properties of the class $\mathbf{W} ( V, \rel^l )$
	consisting of all $f \in \Lploc{1} ( \| V \| + \| \delta V \|,
	\rel^l)$ such that for some $F \in \Lploc{1} ( \| V \|, \Hom (
	\rel^\adim, \rel^l ) )$
	\begin{gather*}
		( \delta V ) ( ( \alpha \circ f ) \theta ) = \alpha \big (
		\tint{}{} ( \project{S} \bullet D \theta (z)) f (z) + \left <
		\theta (z), F(z) \right > \ud V (z,S) \big )
	\end{gather*}
	whenever $\theta \in \mathscr{D} (U, \rel^\adim )$ and $\alpha \in
	\Hom ( \rel^l, \rel )$, associated with $V$ and $\rel^l$ whenever $l,
	\vdim, \adim \in \nat$, $\vdim \leq \adim$, $U$ is an open subset of
	$\rel^\adim$, $V \in \RVar_\vdim (U)$ and $\| \delta V \|$ is a Radon
	measure, will be discussed briefly.

	Clearly, $F$ is $\| V \|$ almost unique and $\mathbf{W} (V, \rel^l )$
	is a vectorspace. However, it may happen that $f \in \mathbf{W} (V,
	\rel^l )$ but $\phi \circ f \notin \mathbf{W} (V,\rel^l)$ for some
	$\phi \in \mathscr{D}^0 ( \rel^l )$; in fact the function $f+g$
	constructed in \ref{example:star} provides an example.  Note also that
	\begin{gather*}
		\classification{\trunc (V, \rel^l ) \cap \Lploc{1} ( \| V \| +
		\| \delta V \|, \rel^l )}{f} { \derivative{V}{f} \in \Lploc{1}
		( \| V \|, \Hom ( \rel^\adim, \rel^l ) }
	\end{gather*}
	is contained in $\mathbf{W} (V,\rel^l)$ by
	\ref{remark:integration_by_parts}.
\end{remark}
\begin{example} \label{example:axioms}
	The considerations of \ref{example:star} may be axiomatised as
	follows.

	Suppose $S$ denotes the family of stationary one dimensional integral
	varifolds in $\rel^2$ and, whenever $V \in S$, $C(V)$ is a class of
	real valued functions on $\rel^2$ satisfying the following conditions.
	\begin{enumerate}
		\item \label{item:axioms:decomp} If $G$ is a decomposition of
		$V$, $\zeta$ is associated to $G$ as in
		\ref{remark:decomp_rep}, $\xi : G \to \rel$ and $f : \rel^2
		\to \rel$ satisfies
		\begin{gather*}
			f(z) = \xi (W) \quad \text{if $z \in \zeta(W)$, $W \in
			G$}, \qquad f(z) = 0 \quad \text{if $z \in \rel^2
			\without {\textstyle\bigcup} \im \zeta$},
		\end{gather*}
		then $f \in C(V)$.
		\item \label{item:axioms:addition} If $f,g \in C(V)$, then
		$f+g \in C(V)$.
		\item \label{item:axioms:truncation} If $f \in C(V)$, $\phi
		\in \mathscr{E}^0 ( \rel )$ and $\spt \phi'$ is compact, then
		$\phi \circ f \in C(V)$.  \intertextenum{
		
		Then, using the terminology of
		\ref{remark:nonunique_decomposition}, the characteristic
		function of any ray $R_j$ belongs to $C(V)$. Moreover, the
		same holds, if the conditions \eqref{item:axioms:addition} and
		\eqref{item:axioms:truncation} are replaced by the following
		condition.}
		\item \label{item:axioms:multiplication} If $f,g \in C(V)$,
		then $fg \in C(V)$.
	\end{enumerate}
\end{example}
\begin{lemma} \label{lemma:level_sets}
	Suppose $\vdim, \adim \in \nat$, $\vdim \leq \adim$, $U$ is an open
	subset of $\rel^\adim$, $V \in \RVar_\vdim ( U )$, $\| \delta V \|$ is
	a Radon measure, $f \in \trunc (V)$, $t \in \rel$, and $E = \{ z \with
	f(z)>t \}$.

	Then there holds
	\begin{gather*}
		\boundary{V}{E}(\theta) = \lim_{\varepsilon \to 0+}
		\varepsilon^{-1} \tint{\{z \with t < f(z) \leq
		t+\varepsilon\}}{} \left < \theta, \derivative{V}{f} \right >
		\ud \| V \| \quad \text{for $\theta \in \mathscr{D} (U,
		\rel^\adim )$}.
	\end{gather*}
\end{lemma}
\begin{proof}
	Define $\phi_\varepsilon (u) = \varepsilon^{-1} \inf \{ \varepsilon ,
	\sup \{ u-t, 0 \} \}$ for $u \in \rel$, $0 < \varepsilon \leq 1$ and
	notice that
	\begin{gather*}
		\phi_\varepsilon (u) = 0 \quad \text{if $u \leq t$}, \qquad
		\phi_\varepsilon (u) = 1 \quad \text{if $u \geq
		t+\varepsilon$}, \\
		\phi_\varepsilon (u) \uparrow 1 \quad \text{as $\varepsilon
		\to 0+$ if $u > t$}
	\end{gather*}
	whenever $u \in \rel$. Consequently, one infers $\phi_\varepsilon
	\circ f \in \trunc (V)$ with
	\begin{gather*}
		\text{$\derivative{V}{( \phi_\varepsilon \circ f )} (z) =
		\varepsilon^{-1} \derivative{V}{f} (z)$ if $t < f(z) \leq
		t+\varepsilon$}, \quad \text{$\derivative{V}{(
		\phi_\varepsilon \circ f )} (z) = 0$ else}
	\end{gather*}
	for $\| V \|$ almost all $z$ from \ref{lemma:basic_v_weakly_diff},
	\ref{example:composite}\,\eqref{item:composite:scalar_mult}\,\eqref{item:composite:add_constant}\,\eqref{item:composite:1d}
	and \ref{thm:addition}\,\eqref{item:addition:zero}. It follows
	\begin{gather*}
		\begin{aligned}
			\boundary{V}{E} ( \theta ) & = \lim_{\varepsilon
			\to 0+} ( \delta V ) ( ( \phi_\varepsilon \circ f )
			\theta ) - \lim_{\varepsilon \to 0+} \tint{}{} (
			\phi_\varepsilon \circ f ) (z) \project{S} \bullet D
			\theta (z) \ud V (z,S) \\
			& = \lim_{\varepsilon \to 0+} \varepsilon^{-1}
			\tint{\{ z \with t < f (z) \leq t + \varepsilon \}}{}
			\left < \theta(z), \derivative{V}{f} (z) \right > \ud
			\| V \| z,
		\end{aligned}
	\end{gather*}
	where \ref{remark:integration_by_parts} with $\alpha$ and $f$ replaced
	by $\id{\rel}$ and $\phi_\varepsilon \circ f$ was employed.
\end{proof}
\begin{theorem} \label{thm:tv_coarea}
	Suppose $\vdim, \adim \in \nat$, $U$ is an open subset of
	$\rel^\adim$, $V \in \RVar_\vdim (U)$, $\| \delta V \|$ is a Radon
	measure, and $f \in \trunc(V)$, $T \in \mathscr{D} ( U \times \rel,
	\rel^\adim )$ satisfies
	\begin{gather*}
		T ( \psi ) = \tint{}{} \left < \psi (z,f(z)),
		\derivative{V}{f} (z) \right > \ud \| V \|z \quad \text{for
		$\psi \in \mathscr{D} (U \times \rel, \rel^\adim )$},
	\end{gather*}
	and $E(t) = \{ z \with f(z) > t \}$ for $t \in \rel$.

	Then $T$ is representable by integration and
	\begin{gather*}
		T ( \psi ) = \tint{}{} \boundary{V}{E(t)} ( \psi (\cdot, t ) )
		\ud \mathscr{L}^1 t, \quad \tint{}{} g \ud \| T \| = \tint{}{}
		\tint{}{} g (z,t) \ud \| \boundary{V}{E(t)} \| z \ud
		\mathscr{L}^1 t.
	\end{gather*}
	whenever $\psi \in \Lp{1} ( \| T \|, \rel^\adim )$ and $g$ is an
	$\overline{\rel}$ valued $\| T \|$ integrable function.
\end{theorem}
\begin{proof}
	Taking \ref{remark:associated_distribution} into account,
	\ref{lemma:meas_fct} yields
	\begin{gather*}
		\tint{}{} \gamma \circ f \left < \theta, \derivative{V}{f}
		\right > \ud \| V \| = T_{(z,t)} ( \gamma(t) \theta (z) ) =
		\tint{}{} \gamma(t) \boundary{V}{E(t)} ( \theta ) \ud
		\mathscr{L}^1 t
	\end{gather*}
	for $\gamma \in \mathscr{D}^0 ( \rel )$ and $\theta \in \mathscr{D}
	(U,\rel^\adim)$. In view of \ref{lemma:level_sets}, the conclusion is
	implied by
	\ref{thm:distribution_on_product}\,\eqref{item:distribution_on_product:absolute}
	with $Y = \rel^\adim$.
\end{proof}
\begin{remark}
	The formulation of \ref{thm:tv_coarea} is modelled on
	\cite[4.5.9\,(13)]{MR41:1976}.
\end{remark}
\begin{remark} \label{remark:no_coarea_ineq_for_too_big_sobolev}
	The equalities in \ref{thm:tv_coarea} are not valid for functions in
	$\mathbf{W} (V)$ with $\derivative{V}{f}$ in the definition of $T$
	replaced by the function $F$ occuring in the definition of $\mathbf{W}
	(V)$, see \ref{remark:too_big_sobolev}; in fact, the function $f+g$
	constructed in \ref{example:star} provides an example.
\end{remark}
\begin{corollary} \label{corollary:boundary_controls_interior}
	Suppose $l, \vdim, \adim \in \nat$, $\vdim \leq \adim$, $U$ is an open
	subset of $\rel^\adim$, $V \in \RVar_\vdim ( U )$, $\| \delta V
	\|$ is a Radon measure, and $f \in \trunc ( V, \rel^l )$.

	Then there holds
	\begin{gather*}
		\eqLpnorm{\| \delta V \| \restrict Z}{\infty}{f} \leq
		\eqLpnorm{\| V \| \restrict Z}{\infty}{f}
	\end{gather*}
	whenever $Z$ is an open subset of $U$.
\end{corollary}
\begin{proof}
	Assume $Z = U$ and abbreviate $s = \Lpnorm{\| V \|}{\infty}{f}$.
	Recalling \ref{lemma:basic_v_weakly_diff} and
	\ref{example:composite}\,\eqref{item:composite:mod}, one applies
	\ref{thm:tv_coarea} with $f$ replaced by $|f|$ to infer
	\begin{gather*}
		\boundary{V}{E(t)} = 0 \quad \text{for $\mathscr{L}^1$ almost
		all $s < t < \infty$},
	\end{gather*}
	where $E(t) = \{ z \with |f(z)|>t \}$, hence $\| \delta V \| (E(t)) =
	0$ for those $t$.
\end{proof}
\begin{theorem} \label{thm:zero_derivative}
	Suppose $l, \vdim, \adim \in \nat$, $\vdim \leq \adim$, $U$ is an open
	subset of $\rel^\adim$, $V \in \RVar_\vdim ( U )$, $\| \delta V \|$ is
	a Radon measure, and $f \in \trunc ( V, \rel^l )$ with
	\begin{gather*}
		\derivative{V}{f} (z) = 0 \quad \text{for $\| V \|$ almost all
		$z$}.
	\end{gather*}

	Then there exists a decomposition $G$ of $V$ and $\xi : G \to \rel^l$
	such that
	\begin{gather*}
		f(z)=\xi(W) \quad \text{for $\| W \| + \| \delta W \|$ almost
		all $z$}
	\end{gather*}
	whenever $W \in G$.
\end{theorem}
\begin{proof}
	Define $B(y,r) = \{ z \with | f(z) - y | \leq r \}$ for $y \in \rel^l$
	and $0 \leq r < \infty$. First, one observes that \ref{thm:tv_coarea}
	in conjunction with \ref{lemma:basic_v_weakly_diff},
	\ref{example:composite}\,\eqref{item:composite:mod} implies
	\begin{gather*}
		\boundary{V}{B(y,r)} = 0 \quad \text{for $y \in \rel^l$ and $0
		\leq r < \infty$}.
	\end{gather*}

	Next, a countable subset $Y$ of $\rel^l$ such that
	\begin{gather*}
		f(z) \in Y \quad \text{for $\| V \|$ almost all $z$}
	\end{gather*}
	will be constructed. For this purpose define $\delta_i =
	\unitmeasure{\vdim} 2^{-\vdim-1} i^{-1-2\vdim}$, $\varepsilon_i =
	2^{-1} i^{-2}$ and Borel sets $E_i$ consisting of all $a \in
	\rel^\adim$ satisfying
	\begin{gather*}
		|a| \leq i, \quad \oball{a}{2\varepsilon_i} \subset U, \quad
		\density^\vdim ( \| V \|, a ) \geq 1/i, \\
		\measureball{\| \delta V \|}{\cball{a}{r}} \leq
		\unitmeasure{\vdim} is^\vdim \quad \text{for $0 < s <
		\varepsilon_i$}
	\end{gather*}
	whenever $i \in \nat$. Clearly, $E_i \subset E_{i+1}$ for $i \in \nat$
	and $\| V \| ( U \without \bigcup_{i=1}^\infty E_i ) = 0$ by Allard
	\cite[3.5\,(1a)]{MR0307015} and \cite[2.8.18, 2.9.5]{MR41:1976}.
	Abbreviating $\mu_i = f_\# ( \| V \| \restrict E_i )$, one obtains
	\begin{gather*}
		\| V \| ( B(y,r) \cap \{ z \with \dist ( z, E_i ) \leq
		\varepsilon_i \} ) \geq \delta_i
	\end{gather*}
	whenever $y \in \spt \mu_i$, $0 \leq r < \infty$ and $i \in \nat$; in
	fact, assuming $r > 0$ there exists $a \in E_i$ with $\density^\vdim (
	\| V \| \restrict B(y,r), a ) = \density^\vdim ( \| V \|, a ) \geq
	1/i$ by \cite[2.8.18, 2.9.11]{MR41:1976}, hence
	\begin{gather*}
		\tint{s}{t} u^{-\vdim} \measureball{\| \delta ( V \restrict
		B(y,r) \times \grass{\adim}{\vdim} ) \|}{\cball{a}{u}} \ud
		\mathscr{L}^1 u \leq \unitmeasure{\vdim} i (t-s)
	\end{gather*}
	for $0 < s < t < \varepsilon_i$ implying $\| V \| ( B (y,r) \cap
	\cball{a}{\varepsilon_i} ) \geq \delta_i$ by \ref{lemma:monotonicity}.
	Consequently,
	\begin{gather*}
		\delta_i \card \spt \mu_i \leq \| V \| ( U \cap \{ z \with
		\dist (z,E_i) \leq \varepsilon_i \} ) < \infty
	\end{gather*}
	and one may take $Y = \bigcup_{i=1}^\infty \spt \mu_i$.

	In view of \ref{remark:partition} it follows $( \| V \| + \| \delta V
	\| ) ( U \without \bigcup \{ B (y,0) \with y \in Y \} ) = 0$. Applying
	\ref{thm:decomposition} to $V \restrict B(y,0) \times
	\grass{\adim}{\vdim}$ for $y \in Y$ now readily yields the conclusion.
\end{proof}
\section{Zero boundary values} \label{sec:zero}
In this section a notion of zero boundary values for weakly differentiable
functions based on the behaviour of superlevel sets of their modulus functions
is introduced. Stability of this class under composition, see
\ref{lemma:trunc_tg}, convergence, see \ref{lemma:closeness_tg} and
\ref{remark:closeness_tg}, and multiplication by a Lipschitzian function, see
\ref{thm:mult_tg}, are investigated. The deeper parts of this study rest on a
characterisation of such functions in terms of an associated distribution
built from the distributional boundaries of superlevel sets, see
\ref{thm:char_tg}.
\begin{definition} \label{def:trunc_g}
	Suppose $l, \vdim, \adim \in \nat$, $\vdim \leq \adim$, $U$ is an open
	subset of $\rel^\adim$, $V \in \RVar_\vdim ( U )$, $\| \delta V \|$ is
	a Radon measure, and $G$ is a relatively open subset of $\Bdry U$.

	Then $\trunc_G (V, \rel^l)$ is defined to be the set of all $f \in
	\trunc ( V, \rel^l )$ such that with
	\begin{gather*}
		B = ( \Bdry U ) \without G \qquad \text{and} \qquad E(t) = \{
		z \with |f(z)| > t \} \quad \text{for $0 < t < \infty$}
	\end{gather*}
	the following conditions hold for $\mathscr{L}^1$ almost all $0 < t <
	\infty$, see \ref{miniremark:extension},
	\begin{gather*}
		\| V \| ( E(t) \cap K ) + \| \delta V \| ( E(t) \cap K ) + \|
		\boundary{V}{E(t)} \| ( U \cap K ) < \infty, \\
		\tint{E(t) \times \grass{\adim}{\vdim}}{} D \theta (z) \bullet
		\project{S} \ud V(z,S) = ( (\delta V) \restrict E(t) ) (
		\theta | U ) - ( \boundary{V}{E(t)} ) ( \theta|U )
	\end{gather*}
	whenever $K$ is a compact subset of $\rel^\adim \without B$ and
	$\theta \in \mathscr{D} ( \rel^\adim \without B, \rel^\adim )$.
	Moreover, let $\trunc_G ( V ) = \trunc_G (V, \rel)$.
\end{definition}
\begin{remark} \label{remark:trunc}
	Note $\boundary{V}{E(t)}$ is representable by integration for
	$\mathscr{L}^1$ almost all $0 < t < \infty$ by
	\ref{lemma:basic_v_weakly_diff},
	\ref{example:composite}\,\eqref{item:composite:mod} and
	\ref{thm:tv_coarea}, hence $\trunc_\varnothing (V,\rel^l) = \trunc
	(V,\rel^l)$.
\end{remark}
\begin{lemma} \label{lemma:boundary}
	Suppose $l$, $\vdim$, $\adim$, $U$, $V$, $G$, $B$, and $E(t)$ are as
	in \ref{def:trunc_g}, $f \in \trunc_G ( V, \rel^l )$, and
	\begin{gather*}
		\Clos \spt \big ( ( \| V \| + \| \delta V \| ) \restrict E(t)
		\big ) \subset \rel^\adim \without B \quad \text{for
		$\mathscr{L}^1$ almost all $0 < t < \infty$}.
	\end{gather*}

	Then $f \in \trunc_{\Bdry U} ( V, \rel^l )$.
\end{lemma}
\begin{proof}
	Define $A(t) = \Clos \spt \big ( ( \| V \| + \| \delta V \| )
	\restrict E(t) \big )$ for $0 < t < \infty$. Consider $0 < t < \infty$
	with $A(t) \subset \rel^\adim \without B$ and
	\begin{gather*}
		\| V \| ( E(t) \cap K ) + \| \delta V \| ( E(t) \cap K ) + \|
		\boundary{V}{E(t)} \| ( U \cap K ) < \infty, \\
		\tint{E(t) \times \grass{\adim}{\vdim}}{} D \theta (z) \bullet
		\project{S} \ud V(z,S) = ( (\delta V) \restrict E(t) ) (
		\theta | U ) - ( \boundary{V}{E(t)} ) ( \theta|U )
	\end{gather*}
	whenever $K$ is a compact subset of $\rel^\adim \without B$ and
	$\theta \in \mathscr{D} ( \rel^\adim \without B, \rel^\adim )$.

	Suppose $K$ is a compact subset of $\rel^\adim$ and $\theta \in
	\mathscr{D} ( \rel^\adim, \rel^\adim)$. Then $K \cap A(t)$ is a
	compact subset of $\rel^\adim \without B$, hence there exists $\phi
	\in \mathscr{D}^0 ( \rel^\adim )$ such that
	\begin{gather*}
		A(t) \cap \spt \theta \subset \Int \{ z \with \phi (z) = 1
		\}, \quad \spt \phi \subset \rel^\adim \without B.
	\end{gather*}
	Since $\theta$ and $\phi \theta$ agree on a neighbourhood of $A(t)$,
	one infers from the first paragraph with $K$ and $\theta$ replaced by
	$K \cap A(t)$ and $\phi \theta | \rel^\adim \without B$ that
	\begin{gather*}
		\| V \| ( E(t) \cap K ) + \| \delta V \| ( E(t) \cap K ) + \|
		\boundary{V}{E(t)} \| ( U \cap K ) < \infty, \\
		\tint{E(t) \times \grass{\adim}{\vdim}}{} D \theta (z) \bullet
		\project{S} \ud V(z,S) = ( (\delta V) \restrict E(t) ) (
		\theta | U ) - ( \boundary{V}{E(t)} ) ( \theta|U ),
	\end{gather*}
	as all relevant distributions are supported in $A(t)$.
\end{proof}
\begin{lemma} \label{lemma:restriction_tg}
	Suppose $l$, $\vdim$, $\adim$, $U$, $V$, $G$, and $B$ are as in
	\ref{def:trunc_g}, $f \in \trunc_G ( V, \rel^l )$, $Z$ is an open
	subset of $\rel^\adim \without B$, $H = Z \cap \Bdry U$, and $X = V |
	\mathbf{2}^{( U \cap Z ) \times \grass{\adim}{\vdim}}$.

	Then $f | Z \in \trunc_H ( X, \rel^l )$.
\end{lemma}
\begin{proof}
	Clearly, $X \in \RVar_\vdim ( U \cap Z )$ and
	\begin{gather*}
		\delta X ( \theta |Z ) = \delta V ( \theta ), \quad
		\boundary{X}{(Z \cap E)} ( \theta | Z ) = \boundary{V}{E} (
		\theta )
	\end{gather*}
	whenever $\theta \in \mathscr{D} ( U, \rel^\adim )$ with $\spt \theta
	\subset U \cap Z$ and $E$ is $\| V \| + \| \delta V \|$ measurable.

	Noting $H = Z \cap \Bdry ( U \cap Z )$, one infers that $H$ is a
	relatively open subset of $\Bdry ( U \cap Z )$.

	Let $C = ( \Bdry ( U \cap Z ) ) \without H$. Suppose $0 < t < \infty$
	satisfies the conditions of \ref{def:trunc_g}, define $E = \{ z \with |
	f(z) | > t \}$, and note
	\begin{gather*}
		\text{$K \cap \Clos ( U \cap Z )$ is a compact subset of $Z$},
		\\
		Z \cap E \cap K \subset E \cap K \cap \Clos ( U \cap K ),
		\quad U \cap Z \cap K \subset U \cap K \cap \Clos ( U \cap Z
		), \\
		\| X \| ( Z \cap E \cap K ) + \| \delta X \| ( Z \cap E \cap
		K) + \| \boundary{X}{(Z \cap E)} \| ( U \cap Z \cap K ) <
		\infty
	\end{gather*}
	whenever $K$ is a compact subset of $\rel^\adim \without C$. Suppose
	$\theta \in \mathscr{D} ( \rel^\adim, \rel^\adim )$ with $\spt \theta
	\subset \rel^\adim \without C$ and choose $\eta \in \mathscr{D}^0 (
	\rel^\adim )$ with
	\begin{gather*}
		\spt \eta \subset Z, \quad ( \spt \theta ) \cap \Clos ( U \cap
		Z ) \subset \Int \{ z \with \eta (z) = 1 \}.
	\end{gather*}
	One obtains
	\begin{gather*}
		\theta | U \cap Z = \eta \theta | U \cap Z, \quad \eta \theta
		| U \without Z = 0, \quad D ( \eta \theta ) | U \without Z =
		0, \\
		( ( \delta X) \restrict Z \cap E ) ( \theta | U \cap Z ) = ((
		\delta V ) \restrict E ) ( \eta \theta | U ), \\
		( \boundary{X}{(Z \cap E )} ) ( \theta | U \cap Z ) = (
		\boundary{V}{E} ) ( \eta \theta | U ), \\
		\tint{(Z \cap E ) \times \grass{\adim}{\vdim}}{} \project{S}
		\bullet D \theta (z) \ud X (z,S) = \tint{E \times
		\grass{\adim}{\vdim}}{} \project{S} \bullet D ( \eta \theta )
		(z) \ud V (z,S)
	\end{gather*}
	and the conclusion follows.
\end{proof}
\begin{remark} \label{remark:restriction_tg}
	Recalling $H = Z \cap \Bdry ( U \cap Z )$ and defining $C = ( \Bdry (
	U \cap Z ) ) \without H$, one notes
	\begin{gather*}
		\classification{U}{z}{ \oball{z}{r} \subset Z } \subset
		\classification{U \cap Z}{z}{\oball{z}{r} \subset \rel^\adim
		\without C } \quad \text{for $0 < r < \infty$};
	\end{gather*}
	this fact will be useful for localisation procedures.
\end{remark}
\begin{example}
	Suppose $\vdim = \adim = 1$, $U = \rel \without \{ 0 \}$, $V \in
	\RVar_\vdim ( U )$ with $\| V \| ( A ) = \mathscr{L}^1 ( A )$ for $A
	\subset U$, $f : U \to \rel$ with $f(z) = 1$ for $z \in U$, $Z =
	\classification{U}{z}{z<0}$, and $X = V | \mathbf{2}^{Z \times
	\grass{\adim}{\vdim}}$.

	Then $\delta V = 0$ and $f \in \trunc_{\{ 0 \}} ( V )$ but $f | Z
	\notin \trunc_{\{ 0 \}} (X)$.
\end{example}
\begin{lemma} \label{lemma:trunc_tg}
	Suppose $\vdim$, $\adim$, $U$, $V$, $G$, and $B$ are as in
	\ref{def:trunc_g}, $f \in \trunc_G ( V )^+$, $0 \leq c < \infty$, $A$
	is a closed subset of $\rel$, $\phi : \rel \to \rel \cap \{ y \with y
	\geq 0 \}$, $\phi(0)=0$, $\zeta : \rel \to \Hom ( \rel, \rel )$,
	$\phi_i \in \mathscr{E}^0 ( \rel )$, $\spt D\phi_i \subset A$, $\Lip
	\phi_i \leq c$, $\phi|A$ is proper, and
	\begin{gather*}
		\phi (y) = \lim_{i \to \infty} \phi_i (y) \quad
		\text{uniformly in $y \in \rel$}, \\
		\zeta (y) = \lim_{i \to \infty} D\phi_i (y) \quad \text{for $y
		\in \rel$}.
	\end{gather*}

	Then $\phi \circ f \in \trunc_G ( V )^+$ and
	\begin{gather*}
		\derivative{V}{( \phi \circ f )} (z) = \zeta (f(z)) \circ
		\derivative{V}{f} (z) \quad \text{for $\| V \|$ almost all
		$z$}.
	\end{gather*}
\end{lemma}
\begin{proof}
	Clearly, $\Lip \phi \leq c$.

	Let $I = \{ t \with 0 < t < \infty \}$. Define $F$ to be the class of
	all Borel subsets $Y$ of $I$ such that $E = f^{-1} \lIm Y \rIm$
	satisfies
	\begin{gather*}
		( \| V \| + \| \delta V \| ) ( E \cap K ) + \| \boundary{V}{E}
		\| ( U \cap K ) < \infty, \\
		\tint{E \times \grass{\adim}{\vdim}}{} D \theta (z) \bullet
		\project{S} \ud V(z,S) = ( ( \delta V ) \restrict E ) (
		\theta|U ) - ( \boundary{V}{E} ) ( \theta|U )
	\end{gather*}
	whenever $K$ is a compact subset of $\rel^\adim \without B$ and
	$\theta \in \mathscr{D} ( \rel^\adim \without B, \rel^\adim )$. In
	view of \ref{lemma:basic_v_weakly_diff} it is sufficient to prove
	\begin{gather*}
		(\phi|\Clos I)^{-1} \lIm \{ y \with t < y < \infty \} \rIm \in
		F \quad \text{for $\mathscr{L}^1$ almost all $t \in I$}.
	\end{gather*}

	Let $C$ denote the set of $t \in I$ such that either $\{ y \with t < y
	< \infty \}$ or $\{ y \with t \leq y < \infty \}$ does not belong to
	$F$. Since $( \| V \| + \| \delta V \| ) ( f^{-1} \lIm \{ t \} \rIm )
	= 0$ for all but countably many $t$, one observes that $\mathscr{L}^1 (
	C ) = 0$. Moreover, if $Y_1 \in F$ and $Y_1 \subset Y_2 \in F$
	then $Y_2 \without Y_1 \in F$, as one readily verifies using
	\ref{remark:extension_additive} and \ref{remark:v_boundary}.

	Next, the following assertion will be shown: \emph{If $Y$ is an open
	subset of $I$ with $\inf Y > 0$ and if $\Bdry Y$ consists of finitely
	many points not meeting $C$, then $Y \in F$}. To prove this, assume $Y
	\neq \varnothing$ and let $J = \{ y \with \inf Y < y < \infty \} \in
	F$. Let $G$ denote the family of connected components of $J \without
	Y$. Observe that $G$ is a finite disjointed family of possibly
	degenerated closed intervals. Notice that $G \subset F$ since
	\begin{gather*}
		\Bdry I \subset \Bdry Y, \quad I = \{ y \with \inf I \leq y <
		\infty \} \without \{ y \with \sup I < y < \infty \}
	\end{gather*}
	whenever $I \in G$. It follows
	\begin{gather*}
		{\textstyle J \without \bigcup H \in F} \quad \text{whenever
		$H \subset G$}
	\end{gather*}
	by means of induction on the cardinality of $H$. In particular, $Y = J
	\without \bigcup G \in F$.

	Suppose now $t \in I \without \phi \lIm C \rIm$ and
	\begin{gather*}
		N ( \phi | \{ y \with |y| \leq i \}, t ) < \infty \quad
		\text{whenever $i \in \nat$}.
	\end{gather*}
	Notice that these conditions are satisfied by $\mathscr{L}^1$ almost
	all $t \in I$ by \cite[3.2.3\,(1)]{MR41:1976}. Let $Y = (\phi|\Clos
	I)^{-1} \lIm \{ y \with t < y < \infty \} \rIm$. From
	\begin{gather*}
		\Bdry Y \subset \{ y \with \phi (y) = t \}
	\end{gather*}
	one infers $C \cap \Bdry Y = \varnothing$ and $N(\phi,t) < \infty$
	since $A \cap \{ y \with \phi(y)=t \}$ is compact and $\phi$ is
	locally constant on $\rel \without A$. Therefore $Y \in F$ by the
	assertion of the preceding paragraph.
\end{proof}
\begin{example}
	Suppose $\vdim = \adim = 1$, $U = \rel \without \{ 0 \}$, $V \in
	\RVar_\vdim ( U )$ with $\| V \| ( A ) = \mathscr{L}^1 (A)$ for $A
	\subset U$ and $f = \sign | U$.

	Then $\delta V = 0$, $f \in \trunc_{\{0\}} ( V )$ but $f^+ \notin
	\trunc_{\{0\}} (V)$.
\end{example}
\begin{theorem} \label{thm:char_tg}
	Suppose $l$, $\vdim$, $\adim$, $U$, $V$, $G$, and $B$ are as in
	\ref{def:trunc_g}, $f \in \trunc (V,\rel^l)$, $J = \{ t \with 0 < t <
	\infty \}$, $E(t) = \classification{U}{z}{|f(z)| > t }$ for $t \in J$,
	\begin{gather*}
		( \| V \| + \| \delta V \| ) ( K \cap E(t) ) < \infty
	\end{gather*}
	whenever $K$ is a compact subset of $\rel^\adim \without B$ and $t \in
	J$, and the distributions $R(t) \in \mathscr{D}' ( \rel^\adim \without
	B, \rel^\adim )$ and $T \in \mathscr{D}' ( ( \rel^\adim \without B )
	\times J, \rel^\adim )$ satisfy, see \ref{miniremark:extension},
	\begin{gather*}
		R(t)(\theta) = ( ( \delta V ) \restrict E(t) ) ( \theta | U )
		- \tint{E(t) \times \grass{\adim}{\vdim}}{} \project{S}
		\bullet D \theta (z) \ud V (z,S) \quad \text{for $t \in J$},
		\\
		T ( \eta ) = \tint{0}{\infty} R(t) ( \eta (\cdot,t) ) \ud
		\mathscr{L}^1 t
	\end{gather*}
	whenever $\theta \in \mathscr{D} ( \rel^\adim \without B, \rel^\adim
	)$ and $\eta \in \mathscr{D} ( ( \rel^\adim \without B ) \times J,
	\rel^\adim )$.

	Then the following three statements are equivalent:
	\begin{enumerate}
		\item \label{item:char_tg:def} $\tint{I}{} \|
		\boundary{V}{E(t)} \| ( U \cap K ) \ud \mathscr{L}^1 t <
		\infty$ whenever $K$ is a compact subset of $\rel^\adim
		\without B$ and $I$ is a compact subset of $J$, and $f \in
		\trunc_G ( V, \rel^l )$.
		\item \label{item:char_tg:slices} $\tint{I}{} \| R(t) \| ( K )
		\ud \mathscr{L}^1 t < \infty$ whenever $K$ is a compact subset
		of $\rel^\adim \without B$ and $I$ is a compact subset of $J$,
		and $\| R(t) \| ( ( \rel^\adim \without B ) \without U ) = 0$
		for $\mathscr{L}^1$ almost all $t \in J$.
		\item \label{item:char_tg:graph} $T$ is representable by
		integration and $\| T \| ( ( ( \rel^\adim \without B ) \without
		U ) \times J ) = 0$.
	\end{enumerate}
\end{theorem}
\begin{proof}
	Clearly, \eqref{item:char_tg:def} implies \eqref{item:char_tg:slices}.

	If \eqref{item:char_tg:slices} holds and $H$ is an open set such that
	$\Clos H$ is a compact subset of $J$, then
	\begin{gather*}
		\| T \| ( Z \times H ) \leq \tint{H}{} \| R (t) \| ( Z ) \ud
		\mathscr{L}^1 t < \infty
	\end{gather*}
	whenever $Z$ is an open set such that $\Clos Z$ is a compact subset of
	$\rel^\adim \without B$, hence by approximation the same holds with
	$Z$ replaced by $K \without U$ whenever $K$ is a compact subset of
	$\rel^\adim \without B$ and \eqref{item:char_tg:graph} follows.

	Suppose now that \eqref{item:char_tg:graph} holds and define $p : (
	\rel^\adim \without B ) \times J \to \rel^\adim \without B$ and $q : (
	\rel^\adim \without B ) \times J \to J$ by
	\begin{gather*}
		p (z,t) = z \quad \text{and} \quad q(z,t)=t \quad \text{for
		$z \in \rel^\adim \without B$ and $t \in J$}.
	\end{gather*}
	Since $\sup \{ | R(s)(\theta) | \with \delta < s < \infty \} < \infty$
	for $\theta \in \mathscr{D} ( \rel^\adim \without B, \rel^\adim )$ and
	$0 < \delta < \infty$, denoting by $b_{t,\varepsilon}$ the
	characteristic function of $\cball{t}{\varepsilon}$, one obtains by
	approximating $b_{t,\varepsilon}$ that
	\begin{gather*}
		T ( ( b_{t,\varepsilon} \circ q ) \cdot ( \theta \circ p ) ) =
		\tint{\cball{t}{\varepsilon}}{} R(s)(\theta) \ud \mathscr{L}^1
		s
	\end{gather*}
	whenever $\theta \in \mathscr{D} ( \rel^\adim \without B, \rel^\adim)$
	and $0 < \varepsilon < t < \infty$. Employing a countable dense subset
	of
	\begin{gather*}
		\classification{\mathscr{D} ( \rel^\adim \without B,
		\rel^\adim ) }{\theta}{ \text{$\spt \theta \subset Z$ and $|
		\theta (z) | \leq 1$ for $z \in \rel^\adim \without B$} },
	\end{gather*}
	one concludes, by use of \cite[2.2.17, 2.8.18,
	2.9.2,\,5,\,7,\,9]{MR41:1976},
	\begin{gather*}
		\| R(t) \| (Z) \leq \density^{\ast 1} ( q_\# ( \| T \|
		\restrict ( Z \times H) ), t ) < \infty \quad \text{for
		$\mathscr{L}^1$ almost all $t \in H$}, \\
		\tint{H}{} \| R (t) \| (Z) \ud \mathscr{L}^1 t \leq \| T
		\| ( Z \times H ) = (p_\# (  \| T \| \restrict q^{-1} \lIm H
		\rIm ) ) (Z)
	\end{gather*}
	whenever $Z$ is an open set such that $\Clos Z$ is a compact subset of
	$\rel^\adim \without B$ and $H$ is an open set such that $\Clos H$ is
	a compact subset of $J$. Since $p_\# ( \| T \| \restrict q^{-1} \lIm H
	\rIm )$ is a Radon measure, one may use approximation to obtain the
	last inequality for every Borel subset $Z$ of $\rel^\adim \without B$,
	in particular, taking $Z = ( \rel^\adim \without B ) \without U$,
	\eqref{item:char_tg:slices} follows.
\end{proof}
\begin{lemma} \label{lemma:closeness_tg}
	Suppose $l$, $\vdim$, $\adim$, $U$, $V$, $G$, and $B$ are as in
	\ref{def:trunc_g}, $J = \{ t \with 0 < t < \infty$, $f \in \trunc (V,
	\rel^l )$, $f_i \in \trunc_G ( V, \rel^l )$, and
	\begin{gather*}
		( \| V \| + \| \delta V \| ) ( \classification{K}{z}{ |f(z)|>t
		} ) < \infty, \\
		f_i \to f \quad \text{as $i \to \infty$ in $( \| V \| + \|
		\delta V \| ) \restrict ( U \cap K )$ measure}, \\
		\varrho ( K, I, t, \delta ) < \infty \quad \text{for $0 <
		\delta < \infty$}, \qquad \varrho ( K, I, t, \delta ) \to 0
		\quad \text{as $\delta \to 0+$}
	\end{gather*}
	whenever $K$ is a compact subset of $\rel^\adim \without B$, $I$ is a
	compact subset of $J$, and $\inf I > t \in J$, where $\varrho (K, I, t,
	\delta )$ denotes the supremum of all numbers
	\begin{gather*}
		\limsup_{i \to \infty} \tint{\classification{K \cap A}{z}{|f_i
		(z) | \in I}}{} | \derivative{V}{f_i} | \ud \| V \|
	\end{gather*}
	corresponding to $\| V \|$ measurable sets $A$ with $\| V \| (
	\classification{A \cap K}{z}{|f(z)| > t} ) \leq \delta$.

	Then $f \in \trunc_G ( V, \rel^l )$.
\end{lemma}
\begin{proof}
	Define
	\begin{gather*}
		E_i (t) = \classification{U}{z}{ |f_i(z)| > t } \quad
		\text{and} \quad E(t) = \classification{U}{z}{ |f(z)| > t }
		\quad \text{for $0 < t < \infty$}, \\
		G_i = \eqclassification{U \times J}{(z,t)}{z \in E_i (t)},
		\quad G = \eqclassification{U \times J}{(z,t)}{ z \in E(t) }.
	\end{gather*}
	Denoting by $g_i$ and $g$ the characteristic functions of $G_i$ and
	$G$ respectively and noting
	\begin{gather*}
		E_i (t) \subset E(s) \cup \{ z \with | (f-f_i)(z) | > t-s \}
		\quad \text{for $0 < s < t < \infty$},
	\end{gather*}
	one obtains
	\begin{gather*}
		\tint{( U \cap K ) \times I}{} |g| \ud ( \| V \| + \| \delta V
		\| ) \times \mathscr{L}^1 < \infty, \\
		\lim_{i \to \infty} \tint{( U \cap K ) \times I}{} |g-g_i|
		\ud ( \| V \| + \| \delta V \| ) \times \mathscr{L}^1 = 0
	\end{gather*}
	whenever $K$ is a compact subset of $\rel^\adim \without B$ and $I$ is
	a compact subset of $J$. Therefore one may define $T \in \mathscr{D}'
	( ( \rel^\adim \without B ) \times J , \rel^\adim )$ by
	\begin{gather*}
			T ( \eta ) = \tint{J}{} ( ( \delta V ) \restrict
			E(t) ) ( \eta ( \cdot, t ) | U )
			- \tint{E(t) \times
			\grass{\adim}{\vdim}}{} \project{S} \bullet D \eta (
			\cdot, t ) (z) \ud \| V \| z \ud \mathscr{L}^1 t
	\end{gather*}
	for $\eta \in \mathscr{D} ( ( \rel^\adim \without B ) \times J,
	\rel^\adim )$ and infer by use of Fubini's theorem
	\begin{gather*}
		\begin{aligned}
			T ( \eta ) & = \lim_{i \to \infty} \tint{J}{} ( (
			\delta V ) \restrict E_i (t) ) ( \eta ( \cdot, t ) | U
			) \\
			& \phantom{=\lim_{i \to \infty}} \ \quad - \tint{E_i (t)
			\times \grass{\adim}{\vdim}}{} \project{S} \bullet D
			\eta ( \cdot, t ) (z) \ud \| V \| z \ud \mathscr{L}^1
			t
		\end{aligned}
	\end{gather*}
	for $\eta \in \mathscr{D} ( ( \rel^\adim \without B ) \times J ,
	\rel^\adim )$. From \ref{thm:tv_coarea} in conjunction with
	\ref{lemma:basic_v_weakly_diff},
	\ref{example:composite}\,\eqref{item:composite:mod} one obtains
	\begin{gather*}
		\tint{H}{} \| \boundary{V}{E_i(t)} \| (A) \ud \mathscr{L}^1 t
		= \tint{\classification{A}{z}{|f_i(z)| \in H}}{} |
		\derivative{V}{f_i} | \ud \| V \|
	\end{gather*}
	whenever $A$ is $\| V \|$ measurable, $H$ is an open subset of $J$,
	and $i \in \nat$. Consequently, taking $A = U \cap ( \Int K ) \without
	C$ and $H = \Int I$, one infers
	\begin{gather*}
		\| T \| \big ( ( ( \Int K ) \without C ) \times \Int I \big )
		\leq \varrho ( K, I, t, \delta )
	\end{gather*}
	whenever $K$ is a compact subset of $\rel^\adim \without B$, $I$ is a
	compact subset of $J$, $\inf I > t \in J$, $0 < \delta < \infty$, $C$
	is a compact subset of $U$, and
	\begin{gather*}
		\| V \| ( \eqclassification{K \without C}{z}{|f(z)| > t} )
		\leq \delta.
	\end{gather*}
	In particular, taking $C = \varnothing$ and $\delta$ sufficiently large,
	one concludes that $T$ is representable by integration and taking $C$
	such that $\| V \| ( \eqclassification{K \without C}{z}{|f(z)| > t }
	)$ is small yields
	\begin{gather*}
		\| T \| ( ( ( \rel^\adim \without B ) \without U ) \times J )
		= 0.
	\end{gather*}
	The conclusion now follows from
	\ref{thm:char_tg}\,\eqref{item:char_tg:def}\,\eqref{item:char_tg:graph}.
\end{proof}
\begin{remark} \label{remark:closeness_tg}
	The conditions on $\varrho$ are satisfied for instance if for any
	compact subset $K$ of $\rel^\adim \without B$ there holds
	\begin{gather*}
		\begin{aligned}
			\text{either} & \quad \tint{U \cap K}{} |
			\derivative{V}{f} | \ud \| V \| < \infty, \quad
			\lim_{i \to \infty} \eqLpnorm{\| V \| \restrict U \cap
			K}{1}{ \derivative{V}{f} - \derivative{V}{f_i} } = 0,
			\\
			\text{or} & \quad \limsup_{i \to \infty} \eqLpnorm{\|
			V \| \restrict U \cap K}{q}{\derivative{V}{f_i}} <
			\infty \quad \text{for some $1 < q \leq \infty$};
		\end{aligned}
	\end{gather*}
	in fact if $I$ is a compact subset of $J$ and $\inf I > t \in J$ then,
	in the first case,
	\begin{gather*}
		\limsup_{i \to \infty} \tint{\classification{K \cap A}{z}{|f_i
		(z) | \in I }}{} | \derivative{V}{f_i} | \ud \| V \| \leq
		\tint{\classification{K \cap A}{z}{|f(z)| > t}}{} |
		\derivative{V}{f} | \ud \| V \|
	\end{gather*}
	whenever $A$ is $\| V \|$ measurable and, in the second case,
	\begin{gather*}
		\varrho ( K, I, t, \delta ) \leq \delta^{1-1/q} \limsup_{i \to
		\infty} \eqLpnorm{\| V \| \restrict U \cap
		K}{q}{\derivative{V}{f_i}} \quad \text{for $0 < \delta <
		\infty$}.
	\end{gather*}
\end{remark}
\begin{miniremark} \label{miniremark:stepfunction}
	If $f$ is a nonnegative $\mathscr{L}^1$ measurable $\overline{\rel}$
	valued function, $N \subset \rel$, $\mathscr{L}^1 ( N ) = 0$,
	$\varepsilon > 0$, and $j \in \nat$, then there exist $r_1, \ldots,
	r_j$ such that 
	\begin{gather*}
		\varepsilon (i-1) < r_i < \varepsilon i \quad \text{and} \quad
		r_i \notin N \qquad \text{for $i=1, \ldots, j$}, \\
		\tsum{i=1}{j} ( r_i-r_{i-1} ) f(r_i) \leq 2 \tint{}{} f \ud
		\mathscr{L}^1,
	\end{gather*}
	where $r_0=0$; in fact it is sufficient to choose $r_i$ such that
	\begin{gather*}
		\varepsilon (i-1) < r_i < \varepsilon i, \quad r_i \notin N,
		\quad \varepsilon f(r_i) \leq \tint{\varepsilon
		(i-1)}{\varepsilon i} f \ud \mathscr{L}^1
	\end{gather*}
	for $i=1, \ldots, j$, and note $r_i - r_{i-1} \leq 2 \varepsilon$.
\end{miniremark}
\begin{theorem} \label{thm:mult_tg}
	Suppose $\vdim$, $\adim$, $U$, $V$, $G$, and $B$ are as in
	\ref{def:trunc_g}, $f \in \trunc_G ( V, \rel^l )$, $g : U \to \rel$,
	and
	\begin{gather*}
		\tint{U \cap K}{} | f | + | \derivative{V}{f} | \ud \| V \| <
		\infty, \quad \Lip ( g | K ) < \infty
	\end{gather*}
	whenever $K$ is a compact subset of $\rel^\adim \without B$.

	Then $g f \in \trunc_G ( V, \rel^l )$.
\end{theorem}
\begin{proof}
	Define $h= gf$ and note $h \in \trunc ( V, \rel^l )$ by
	\ref{thm:addition}\,\eqref{item:addition:mult}. Define a function
	$e$ by $e = ( \Clos g ) | \rel^\adim \without B$ and note
	\begin{gather*}
		\dmn e = U \cup G, \quad e | U = g, \quad \Lip ( e | K ) <
		\infty
	\end{gather*}
	whenever $K$ is a compact subset of $\rel^\adim \without B$. Moreover,
	let
	\begin{gather*}
		J = \{ t \with 0 < t < \infty \}, \quad 
		A = \eqclassification{U \times J}{(z,t)}{|h(z)| > t }.
	\end{gather*}
	and define $p : (
	\rel^\adim \without B ) \times J \to \rel^\adim \without B$ by
	\begin{gather*}
		p (z,t) = z \quad \text{for $z \in \rel^\adim \without B$ and
		$t \in J$}.
	\end{gather*}
	Noting $( \| V \| + \| \delta V \| ) ( \classification{K}{z}{ |h(z)| >
	t } ) < \infty$ whenever $K$ is a compact subset of $\rel^\adim
	\without B$ and $t \in J$, the proof may be carried out by showing
	that the distribution $T \in \mathscr{D}' ( ( \rel^\adim \without B)
	\times J, \rel^\adim )$ defined by
	\begin{gather*}
		\begin{aligned}
			T ( \eta ) & = \tint{J}{} ( ( \delta V ) \restrict \{
			z \with | h(z) | > t \} ) ( \eta ( \cdot, t ) | U ) \\
			& \phantom{=} \ \qquad- \tint{\{ z \with | h(z) | > t
			\}}{} \project{\Tan^\vdim ( \| V \|, z )} \bullet D
			\eta ( \cdot, t ) (z) \ud \| V \| z \ud \mathscr{L}^1
			t
		\end{aligned}
	\end{gather*}
	for $\eta \in \mathscr{D} ( ( \rel^\adim \without B ) \times J,
	\rel^\adim )$ satisfies the conditions of
	\ref{thm:char_tg}\,\eqref{item:char_tg:graph} with $f$ replaced by
	$h$.

	For this purpose, define subsets $C(r)$, $D(s)$, and $E(r,s)$ of $U$,
	varifolds $W_r \in \Var_\vdim ( \rel^\adim \without B )$, and
	distributions $R(r), S(r,s) \in \mathscr{D}' ( \rel^\adim \without B,
	\rel^\adim )$ by
	\begin{gather*}
		C (r) = \{ z \with |f(z)| > r \}, \quad D (s) = \{ z \with |
		e(z) | > s \}, \\
		E (r,s) = C(r) \cap D(s), \quad W_r ( \phi ) = \tint{C(r)
		\times \grass{\adim}{\vdim}}{} \phi \ud V, \\
		R(r) ( \theta ) = ( ( \delta V ) \restrict C(r) ) ( \theta | U
		) - \tint{C(r)}{} \project{\Tan^\vdim ( \| V \|, z )} \bullet
		D \theta (z) \ud \| V \| z, \\
		S(r,s) ( \theta ) = ( ( \delta V ) \restrict E(r,s) ) ( \theta
		| U ) - \tint{E(r,s)}{} \project{\Tan^\vdim ( \| V \|, z )}
		\bullet D \theta (z) \ud \| V \| z
	\end{gather*}
	whenever $0 < r < \infty$, $0 < s < \infty$, $\phi \in \mathscr{K} ( (
	\rel^\adim \without B ) \times \grass{\adim}{\vdim} )$, and $\theta
	\in \mathscr{D} ( \rel^\adim \without B, \rel^\adim )$. Let $Q$
	consist of all $r \in J$ such that
	\begin{gather*}
		\text{$R(r)$ is representable by integration and $\| R (r) \|
		( ( \rel^\adim \without B ) \without U ) = 0$}.
	\end{gather*}
	Note $\mathscr{L}^1 ( J \without Q ) = 0$ as $f \in \trunc_G
	(V,\rel^l)$ and
	\begin{gather*}
		S(r,s) = R(r) \restrict D(s) + \boundary{W_r}{D(s)} \quad
		\text{whenever $r \in Q$ and $0 < s < \infty$}.
	\end{gather*}
	One readily verifies by means of \ref{example:lipschitzian} in
	conjunction with Kirszbraun's extension theorem, see
	e.g.~\cite[2.10.43]{MR41:1976}, and \cite[2.10.19\,(4)]{MR41:1976}
	that $e \in \trunc ( W )$ and
	\begin{gather*}
		\derivative{W}{e} (z) = \derivative{V}{g} (z)
		\quad \text{for $\| V \|$ almost all $z \in A$}
	\end{gather*}
	whenever $A$ is $\| V \|$ measurable, $W \in \Var_\vdim ( \rel^\adim
	\without B )$, $W ( \phi ) = \tint{A \times \grass{\adim}{\vdim}}{}
	\phi \ud V$ for $\phi \in \mathscr{K} ( ( \rel^\adim \without B )
	\times \grass{\adim}{\vdim} )$ and $\| \delta W \|$ is a Radon
	measure. In particular, if $r \in Q$ then for $\mathscr{L}^1$ almost
	all $s \in J$
	\begin{gather*}
		\text{$\boundary{W_r}{D(s)}$ is representable by integration
		and $\| \boundary{W_r}{D(s)} \| ( ( \rel^\adim \without B )
		\without U ) = 0$}
	\end{gather*}
	by \ref{thm:tv_coarea}.

	Whenever $I$ is a nonempty finite subset of $Q$ define functions $f_I
	: U \to \rel$ and $h_I : U \to \rel$ by
	\begin{gather*}
		f_I (z) = \sup ( \{ 0 \} \cup ( \classification{I}{r}{ z \in
		C(r) } )), \quad h_I (z) = f_I (z) g(z)
	\end{gather*}
	whenever $z \in U$ and distributions $R_I \in \mathscr{D}' ( (
	\rel^\adim \without B ) \times J, \rel^\adim )$ and $T_I \in
	\mathscr{D}' ( ( \rel^\adim \without B ) \times J, \rel^\adim )$ by
	\begin{gather*}
		\begin{aligned}
			R_I ( \eta ) & = \tint{J}{} ( ( \delta V ) \restrict
			\{ z \with | f_I (z) | > r \} ) ( \eta ( \cdot, r ) |
			U ) \\
			& \phantom{=} \ \qquad - \tint{\{ z \with |f_I (z)| >
			r \}}{} \project{\Tan^\vdim ( \| V \|, z )} \bullet D
			\eta ( \cdot , r ) ( z ) \ud \| V \| z \ud
			\mathscr{L}^1 r, \\
			T_I ( \eta ) & = \tint{J}{} ( ( \delta V ) \restrict
			\{ z \with | h_I (z) | > t \} ) ( \eta ( \cdot, t ) |
			U ) \\
			& \phantom{=} \ \qquad - \tint{\{ z \with |h_I (z)| >
			t \}}{} \project{\Tan^\vdim ( \| V \|, z )} \bullet D
			\eta ( \cdot , t ) ( z ) \ud \| V \| z \ud
			\mathscr{L}^1 t
		\end{aligned}
	\end{gather*}
	whenever $\eta \in \mathscr{D} ( ( \rel^\adim \without B ) \times J,
	\rel^\adim )$.

	Next, \emph{it will be shown
	\begin{gather*}
		\| T_I \| ( Z \times J ) \leq ( ( p_\# \| R_I \| ) \restrict Z
		) ( |e| ) + \tint{U \cap Z}{} |f| | \derivative{V}{g} | \ud \|
		V \|
	\end{gather*}
	whenever $Z$ is an open subset of $\rel^\adim \without B$}. For this
	purpose suppose for some $j$ and some $0 < r_1 < \ldots < r_j <
	\infty$ that $I = \{ r_i \with i = 1, \ldots, j \}$, let $r_0 = 0$,
	note
	\begin{gather*}
		f_I (z) = r_i \quad \text{if $z \in C(r_i) \without C(r_{i+1})$
		for some $i=1, \ldots, j-1$}, \\
		f_I (z) = 0 \quad \text{if $z \in U \without C(r_1)$}, \qquad
		f_I (z) = r_j \quad \text{if $z \in C(r_j)$}
	\end{gather*}
	and define distributions $X, Y \in \mathscr{D}' ( ( \rel^\adim
	\without B ) \times J, \rel^\adim )$ by
	\begin{gather*}
		X ( \eta ) = \tint{J}{} \big ( R (r_1) \restrict D(t/r_1) +
		\tsum{i=2}{j} ( R(r_i) \restrict D (t/r_i) \without D
		(t/r_{i-1}) ) \big ) ( \eta ( \cdot, t ) ) \ud \mathscr{L}^1
		t, \\
		Y ( \eta ) = \tint{J}{} \big ( \tsum{i=1}{j-1} \boundary{(
		W_{r_i} - W_{r_{i+1}} )}{D(t/r_i)} +
		\boundary{W_{r_j}}{D(t/r_j) } \big ) ( \eta (\cdot, t ) ) \ud
		\mathscr{L}^1 t
	\end{gather*}
	for $\eta \in \mathscr{D} ( ( \rel^\adim \without B ) \times J,
	\rel^\adim )$. One computes
	\begin{gather*}
		R_I ( \eta ) = \tsum{i=1}{j} \tint{r_{i-1}}{r_i} R(r_i) ( \eta
		( \cdot, r ) ) \ud \mathscr{L}^1 r \quad \text{for $\eta \in
		\mathscr{D} ( ( \rel^\adim \without B ) \times J, \rel^\adim
		)$}
	\end{gather*}
	and deduces
	\begin{gather*}
		\| R_I \| ( \zeta ) = \tsum{i=1}{j} \tint{r_{i-1}}{r_i} \| R
		(r_i) \| ( \zeta ( \cdot, r ) ) \ud \mathscr{L}^1 r \quad
		\text{for $\zeta \in \mathscr{D}^0 ( ( \rel^\adim \without B )
		\times J )$}
	\end{gather*}
	with the help of \cite[4.1.3]{MR41:1976}. Noting
	\begin{gather*}
		\{ z \with | h_I (z) | > t \} \cap ( C (r_i) \without C
		(r_{i+1}) ) = E (r_i,t/r_i) \without E(r_{i+1},t/r_i), \\
		\{ z \with | h_I (z) | > t \} \cap ( U \without C(r_1) ) =
		\varnothing, \quad \{ z \with | h_I (z) | > t \} \cap C (r_j) =
		E (r_j,t/r_j)
	\end{gather*}
	for $i=1, \ldots, j$ and $t \in J$, one obtains
	\begin{gather*}
		T_I ( \eta ) = \tint{J}{} \big ( \tsum{i=1}{j-1} ( S
		(r_i,t/r_i) - S (r_{i+1}, t/r_i ) ) + S ( r_j,t/r_j ) \big
		)( \eta ( \cdot, t )) \ud \mathscr{L}^1 t
	\end{gather*}
	whenever $\eta \in \mathscr{D} ( ( \rel^\adim \without B ) \times J,
	\rel^\adim )$. Computing
	\begin{gather*}
		\begin{aligned}
			& \tsum{i=1}{j-1} ( S(r_i,t/r_i) - S (r_{i+1}, t/r_i)
			) + S (r_j, t/r_j) \\
			& \qquad = \tsum{i=1}{j} R (r_i) \restrict D (t/r_i) +
			\tsum{i=1}{j} \boundary{W_{r_i}}{D (t/r_i)} \\
			& \qquad \phantom{=} \ - \tsum{i=1}{j-1} R (r_{i+1})
			\restrict D ( t/r_i ) - \tsum{i=1}{j-1}
			\boundary{W_{r_{i+1}}}{D(t/r_i)} \\
			& \qquad = R(r_1) \restrict D (t/r_1) + \tsum{i=2}{j}
			R (r_i) \restrict ( D ( t/r_i) \without D (t/r_{i-1})
			) \\
			& \qquad \phantom{=} \ + \tsum{i=1}{j-1}
			\boundary{(W_{r_i}-W_{r_{i+1}})}{D ( t/r_i )} +
			\boundary{W_{r_j}}{D(t/r_j)}
		\end{aligned}
	\end{gather*}
	whenever $t \in J$ yields
	\begin{gather*}
		T_I = X + Y.
	\end{gather*}
	Then $\| X \| ( Z \times J )$ does not exceed
	\begin{gather*}
		\begin{aligned}
			& \tint{J}{} \big ( \| R(r_1) \| \restrict D (t/r_1) +
			\tsum{i=2}{j} \| R (r_i ) \| \restrict ( D (t/r_i)
			\without D(t/r_{i-1}) ) \big ) (Z) \ud \mathscr{L}^1 t
			\\
			& \qquad = \tsum{i=1}{j} \tint{J}{} ( \| R (r_i) \|
			\restrict Z ) ( D (t/r_i) ) \ud \mathscr{L}^1 t \\
			& \qquad \phantom{=} \ - \tsum{i=2}{j} \tint{J}{} ( \|
			R (r_i) \| \restrict Z ) ( D (t/r_{i-1}) ) \ud
			\mathscr{L}^1 t \\
			& \qquad = \tsum{i=1}{j} (r_i-r_{i-1}) ( \| R (r_i) \|
			\restrict Z ) ( |e| ) = ( ( p_\# \| R_I \| ) \restrict
			Z ) ( |e| ).
		\end{aligned}
	\end{gather*}
	Moreover, $\| Y \| ( Z \times J )$ does not exceed, using
	\ref{thm:tv_coarea},
	\begin{gather*}
		\begin{aligned}
			& \tint{J}{} \big ( \tsum{i=1}{j-1} \|
			\boundary{(W_{r_i}-W_{r_{i+1}})}{D (t/r_i)} \| + \|
			\boundary{W_{r_j}}{D(t/r_j)} \| \big ) (Z) \ud
			\mathscr{L}^1 t \\
			& \qquad = \tsum{i=1}{j-1} r_i \tint{Z}{} |
			\derivative{(W_{r_i}-W_{r_{i+1}})}{e} | \ud \|
			W_{r_i}-W_{r_{i+1}} \| + r_j \tint{Z}{} |
			\derivative{W_{r_j}}{e} | \ud \| W_{r_j} \| \\
			& \qquad = \tsum{i=1}{j-1} r_i \tint{Z \cap ( C(r_i)
			\without C(r_{i+1}))}{} | \derivative{V}{g} | \ud \| V
			\| + r_j \tint{Z \cap C(r_j)}{} | \derivative{V}{g} |
			\ud \| V \| \\
			& \qquad = \tint{U \cap Z}{} |f| | \derivative{V}{g} |
			\ud \| V \|
		\end{aligned}
	\end{gather*}
	The asserted estimate of $\| T_I \| ( Z \times J )$ follows.

	Next, \emph{it will be proven
	\begin{gather*}
		\| T \| ( Z \times J ) \leq \tint{U \cap Z}{} 2 |g| |
		\derivative{V}{f} | + |f| | \derivative{V}{g} | \ud \| V \|
	\end{gather*}
	whenever $Z$ is an open subset of $\rel^\adim \without B$}. Recalling
	the formula for $\| R_I \|$, one may use \ref{miniremark:stepfunction}
	with $f(r)$ replaced by $( \| R (r) \| \restrict Z ) ( | e | )$ to
	construct a sequence $I(k)$ of finite subsets of $Q$ such that
	\begin{gather*}
		( ( p_\# \| R_{I(k)} \| ) \restrict Z ) ( |e| ) \leq 2
		\tint{J}{} ( \| R (r) \| \restrict Z ) ( | e | ) \ud
		\mathscr{L}^1 r, \\
		\dist ( r, I (k) ) \to 0 \quad \text{as $k \to \infty$ for $r
		\in J$}.
	\end{gather*}
	Define
	\begin{gather*}
		A(k) = \eqclassification{U \times J}{(z,t)}{ |h_{I(k)} (z) | >
		t }
	\end{gather*}
	for $k \in \nat$. Noting
	\begin{gather*}
		|h_{I(k)} (z) | \to | h(z) | \quad \text{as $k \to \infty$ for
		$z \in U$}
	\end{gather*}
	and recalling $( \| V \| + \| \delta V \| ) ( \classification{K}{z}{
	|h(z)| > t } ) < \infty$ for $t \in J$ whenever $K$ is a compact
	subset of $\rel^\adim \without B$, one infers
	\begin{gather*}
		\big ( ( \| V \| + \| \delta V \| ) \times \mathscr{L}^1 \big)
		( L \cap A \without A (k) ) \to 0
	\end{gather*}
	as $k \to \infty$ whenever $L$ is a compact subset of $( \rel^\adim
	\without B ) \times J$. Since $A(k) \subset A$ for $k \in \nat$, it
	follows
	\begin{gather*}
		T_{I(k)} \to T \quad \text{as $k \to \infty$}
	\end{gather*}
	and, in conjunction with the assertion of the preceding paragraph,
	\begin{gather*}
		\| T \| ( Z \times J ) \leq 2 \tint{J}{} ( \| R (r) \|
		\restrict Z ) ( | e | ) \ud \mathscr{L}^1 r + \tint{U \cap
		Z}{} |f| | \derivative{V}{g} | \ud \| V \|.
	\end{gather*}
	Therefore the assertion of the present paragraph is a consequence of
	\ref{thm:char_tg}\,\eqref{item:char_tg:def}\,\eqref{item:char_tg:slices}
	and \ref{remark:associated_distribution}, \ref{thm:tv_coarea}.
	
	Finally, the assertion of the preceding paragraph extends to all
	Borels subset $Z$ of $\rel^\adim \without B$ by approximation and the
	conclusion follows.
\end{proof}
\section{Relative isoperimetric inequality} \label{sec:rel_iso}
In this section a general isoperimetric inequality for sets on varifolds
satisfying a lower density bound is established, see \ref{thm:rel_iso_ineq}.
As corollary one obtains two relative isoperimetric inequalities, see
\ref{corollary:rel_iso_ineq} and \ref{corollary:true_rel_iso_ineq}, under the
relevant conditions on the first variation of the varifold, that is $p = 1$
and $p = \vdim$ in \ref{miniremark:situation_general}.
\begin{miniremark} \label{miniremark:zero_boundary}
	Suppose $\vdim$, $\adim$, $U$, $V$ are as in
	\ref{miniremark:situation_general}, $E$ is $\| V \| + \| \delta V \|$
	measurable, $B$ is a closed subset of $\Bdry U$, and, see
	\ref{miniremark:extension},
	\begin{gather*}
		\| V \| ( E \cap K ) + \| \delta V \| ( E \cap K ) + \|
		\boundary{V}{E} \| ( U \cap K ) < \infty, \\
		\tint{E \times \grass{\adim}{\vdim}}{} \project{S} \bullet D
		\theta (z) \ud V (z,S) =  ( ( \delta V ) \restrict E ) (
		\theta |U ) - ( \boundary{V}{E} ) ( \theta | U )
	\end{gather*}
	whenever $K$ is compact subset of $\rel^\adim \without B$ and $\theta
	\in \mathscr{D} ( \rel^\adim \without B, \rel^\adim )$. Defining $W
	\in \Var_\vdim ( \rel^\adim \without B )$ by
	\begin{gather*}
		W (A) = V ( A \cap ( E \times \grass{\adim}{\vdim} ) ) \quad
		\text{for $A \subset ( \rel^\adim \without B ) \times
		\grass{\adim}{\vdim}$},
	\end{gather*}
	this implies
	\begin{gather*}
		\| \delta W \| (A) \leq \| \boundary{V}{E} \| ( U \cap A ) +
		\| \delta V \| ( E \cap A ) \quad \text{for $A \subset
		\rel^\adim \without B$}.
	\end{gather*}
\end{miniremark}
\begin{lemma} \label{lemma:lower_mass_bound_Q}
	Suppose $1 \leq M < \infty$.

	Then there exists a positive, finite number $\Gamma$ with the
	following property.

	If $\vdim, \adim \in \nat$, $\vdim \leq \adim \leq M$, $1 \leq Q \leq
	M$, $a \in \rel^\adim$, $0 < r < \infty$, $V \in \Var_\vdim (
	\oball{a}{r} )$, $\| \delta V \|$ is a Radon measure, $\density^\vdim
	( \| V \|, z ) \geq 1$ for $\| V \|$ almost all $z$, $a \in \spt \| V
	\|$,
	\begin{gather*}
		\measureball{\| \delta V \|}{\cball{a}{s}} \leq \Gamma^{-1} \|
		V \| ( \cball{a}{s} )^{1-1/\vdim} \quad \text{for $0 < s <
		r$}, \\
		\| V \| ( \{ z \with \density^\vdim ( \| V \|, z ) < Q \} )
		\leq \Gamma^{-1} \measureball{\| V \|}{\oball{a}{r}},
	\end{gather*}
	then
	\begin{gather*}
		\measureball{\| V \|}{\oball{a}{r}} \geq ( Q-M^{-1} )
		\unitmeasure{\vdim} r^\vdim.
	\end{gather*}
\end{lemma}
\begin{proof}
	If the lemma were false for some $M$, there would exist a sequence
	$\Gamma_i$ with $\Gamma_i \uparrow 0$ as $i \to \infty$, and sequences
	$\vdim_i$, $\adim_i$, $Q_i$, $a_i$, $r_i$, and $V_i$ showing that
	$\Gamma = \Gamma_i$ does not have the asserted property.

	One could assume for some $\vdim, \adim \in \nat$, $1 \leq Q \leq M$
	that $\vdim \leq \adim \leq M$,
	\begin{gather*}
		\vdim = \vdim_i, \quad \adim = \adim_i, \quad  a_i = 0, \quad
		r_i = 1
	\end{gather*}
	for $i \in \nat$ and $Q_i \to Q$ as $i \to \infty$ and, by
	\cite[2.5]{snulmenn.isoperimetric},
	\begin{gather*}
		\measureball{\| V_i \|}{\cball{0}{s}} \geq ( 2 \vdim
		\isoperimetric{\vdim} )^{-\vdim} s^\vdim \quad \text{whenever
		$0 < s < 1$ and $i \in \nat$}.
	\end{gather*}
	Defining $V \in \Var_\vdim ( \rel^\adim \cap \oball{0}{1} )$ to be the
	limit of some subsequence of $V_i$, one would obtain
	\begin{gather*}
		\measureball{\| V \|}{\oball{0}{1}} \leq ( Q - M^{-1} )
		\unitmeasure{\vdim}, \quad 0 \in \spt \| V \|, \quad \delta V
		= 0.
	\end{gather*}
	Finally, using Allard \cite[5.4, 8.6, 5.1\,(2)]{MR0307015}, one would
	then conclude that
	\begin{gather*}
		\density^\vdim ( \| V \|, z ) \geq Q \quad \text{for $\| V \|$
		almost all $z$}, \\
		\density^\vdim ( \| V \|, 0 ) \geq Q, \quad \measureball{\| V
		\|}{\oball{0}{1}} \geq Q \unitmeasure{\vdim},
	\end{gather*}
	a contradiction.
\end{proof}
\begin{remark}
	Considering stationary varifolds whose support is contained in two
	affine planes with $\density^\vdim ( \| V \| , a )$ a small positive
	number, shows that the hypotheses ``$\density^\vdim ( \| V \|, z )
	\geq 1$ for $\| V \|$ almost all $z$'' cannot be omitted.
\end{remark}
\begin{remark} \label{remark:kuwert_schaetzle}
	Taking $Q=1$ in \ref{lemma:lower_mass_bound_Q} (or applying
	\cite[2.6]{snulmenn.isoperimetric}) yields the following proposition:
	\emph{If $\vdim$, $\adim$, $p$, $U$, $V$, and $\psi$ are as in
	\ref{miniremark:situation_general}, $p = \vdim$, $a \in \spt \| V \|$,
	and $\psi ( \{ a \} ) = 0$, then $\density_\ast^\vdim ( \| V \|, a )
	\geq 1$.} If $\| \delta V \|$ is absolutely continuous with respect to
	$\| V \|$ then the condition $\psi ( \{ a \} ) = 0$ is redundant.

	If $\vdim = \adim$ and $f : \rel^\vdim \to \{ t \with 1 \leq t <
	\infty \}$ is a weakly differentiable function with $\weakD f \in
	\Lploc{\vdim} ( \mathscr{L}^\vdim, \Hom ( \rel^\vdim, \rel ) )$, then
	the varifold $V \in \RVar_\vdim ( \rel^\vdim )$ defined by the
	requirement $\| V \| ( B ) = \tint{B}{} f \ud \mathscr{L}^\vdim$
	whenever $B$ is a Borel subset of $\rel^\vdim$ satisfies the
	conditions of \ref{miniremark:situation_general} with $p = \vdim$
	since
	\begin{gather*}
		\delta V ( g) = - \tint{}{} \left < g (z), \weakD (\log \circ
		f ) (z) \right > \ud \| V \| z \quad \text{for $g \in
		\mathscr{D} ( \rel^\adim, \rel^\adim )$}.
	\end{gather*}
	If $\vdim > 1$ and $T$ is a countable subset of $\rel^\vdim$, one may
	use well known properties of Sobolev functions, in particular example
	\cite[4.43]{MR2424078}, to construct $f$ such that $T \subset \{ z
	\with \density^\vdim ( \| V \|, z ) = \infty \}$.  It is therefore
	evident that the conditions of \ref{miniremark:situation_general} with
	$p = \vdim>1$ are insufficient to guarantee finiteness or upper
	semicontinuity of $\density^\vdim ( \| V \|, \cdot )$ at each point of
	$U$. However, employing Brakke \cite[5.8]{MR485012}, the following
	proposition was obtained by Kuwert and Sch\"atzle in \cite[Appendix
	A]{MR2119722}: \emph{If $\vdim$, $\adim$, $p$, $U$, and $V$ are as in
	\ref{miniremark:situation_general}, $p = \vdim = 2$, and $V \in
	\IVar_2 ( U )$, then $\density^2 ( \| V \|, \cdot )$ is a real valued,
	upper semicontinuous function on $U$.}
\end{remark}
\begin{remark}
	The preceding remark is a corrected and extended version of the
	author's remark in \cite[2.7]{snulmenn.isoperimetric} where the last
	two sentences should have referred to integral varifolds.
\end{remark}
\begin{theorem} \label{thm:rel_iso_ineq}
	Suppose $1 \leq M < \infty$.

	Then there exists a positive, finite number $\Gamma$ with the
	following properity.

	If $\vdim, \adim \in \nat$, $\vdim \leq \adim \leq M$, $1 \leq Q \leq
	M$, $U$ is an open subset of $\rel^\adim$, $W \in \Var_\vdim ( U )$,
	$S, T \in \mathscr{D}' ( U, \rel^\adim )$ are representable by
	integration, $\delta W = S + T$, $\density^\vdim ( \| W \|, z ) \geq
	1$ for $\| W \|$ almost all $z$, $0 < r < \infty$, and
	\begin{gather*}
		\| S \| ( U ) \leq \Gamma^{-1} \qquad \text{if $\vdim = 1$},
		\\
		S(f) \leq \Gamma^{-1} \Lpnorm{\| W\|}{\vdim/(\vdim-1)}{f}
		\quad \text{for $f \in \mathscr{D} ( U, \rel^\adim )$} \qquad
		\text{if $\vdim > 1$}, \\
		\| W \| ( U ) \leq ( Q-M^{-1} ) \unitmeasure{\vdim} r^\vdim,
		\quad
		\| W \| ( \{ z \with \density^\vdim ( \| W \|, z )
		< Q \} ) \leq \Gamma^{-1} r^\vdim,
	\end{gather*}
	then
	\begin{gather*}
		\| W \| ( \{ z \with \oball{z}{r} \subset U \})^{1-1/\vdim}
		\leq \Gamma \, \| T \| ( U ),
	\end{gather*}
	where $0^0=0$.
\end{theorem}
\begin{proof}
	Define
	\begin{gather*}
		\Delta_1 = \Gamma_{\ref{lemma:lower_mass_bound_Q}} ( 2 M
		)^{-1}, \quad \Delta_2 = \inf \{ ( 2
		\isoperimetric{\vdim} )^{-1} \with M \geq \vdim \in \nat
		\}, \\
		\Delta_3 = \Delta_1 \inf \{ ( 2 \vdim
		\isoperimetric{\vdim})^{-\vdim} \with M \geq \vdim \in \nat
		\}, \quad \Delta_4 = \sup \{ \besicovitch{\adim} \with M \geq
		\adim \in \nat \}, \\
		\Delta_5 = (1/2) \inf \{ \Delta_2, \Delta_3 \}, \quad \Gamma =
		\Delta_4 \Delta_5^{-1}.
	\end{gather*}
	Notice that $\isoperimetric{1} \geq 1/2$, hence $2 \Delta_5 \leq
	\Gamma_{\ref{lemma:lower_mass_bound_Q}} ( 2M )^{-1}$.

	Suppose $\vdim$, $\adim$, $Q$, $U$, $W$, $S$, $T$, and $r$ satisfy the
	hypotheses in the body of the theorem with $\Gamma$.

	Abbreviate $A = \classification{U}{z}{ \oball{z}{r} \subset U }$.
	Clearly, if $\vdim=1$ then $\| S \| ( U ) \leq \Delta_5$.  Observe, if
	$\vdim>1$ then
	\begin{gather*}
		\| S \| (C) \leq \Delta_5 \| W \| ( C )^{1-1/\vdim} \quad
		\text{whenever $C \subset U$}.
	\end{gather*}

	Next, the following assertion will be shown: \emph{If $z \in A \cap
	\spt \| W \|$ then there exists $0 < s < r$ such that
	\begin{gather*}
		\Delta_5 \| W \| ( \cball{z}{s} )^{1-1/\vdim} <
		\measureball{\| T \|}{ \cball{z}{s} }.
	\end{gather*}}
	Since $\measureball{\| \delta W \|}{\cball{z}{s}} \leq \Delta_5 \|
	W \| ( \cball{z}{s} )^{1-1/\vdim} + \measureball{\| T
	\|}{\cball{z}{s}}$, it is sufficient to exhibit $0 < s < r$ with
	\begin{gather*}
		2 \Delta_5 \| W \| ( \cball{z}{s} )^{1-1/\vdim} <
		\measureball{\| \delta W \|}{ \cball{z}{s} }.
	\end{gather*}
	As $\measureball{\| W \|}{\oball{z}{r}} \leq \| W \| ( U ) < \big (
	Q-(2M)^{-1} \big ) \unitmeasure{\vdim} r^\vdim$ the nonexistence of
	such $s$ would imply by use of \cite[2.5]{snulmenn.isoperimetric} and
	\ref{lemma:lower_mass_bound_Q}
	\begin{gather*}
		\begin{aligned}
			\Delta_1 ( 2 \vdim \isoperimetric{\vdim} )^{-\vdim}
			r^\vdim & \leq \Delta_1 \measureball{\| W
			\|}{\oball{z}{r}} \\
			& < \| W \| ( \classification{\oball{z}{r}}{\xi}{
			\density^\vdim ( \| W \|, \xi ) < Q } ) \leq \Delta_5
			r^\vdim,
		\end{aligned}
	\end{gather*}
	a contradiction.

	By the assertion of the preceding paragraph there exist countable
	disjointed families of closed balls $F_1, \ldots,
	F_{\besicovitch{\adim}}$ such that
	\begin{gather*}
		A \cap \spt \| W \| \subset {\textstyle\bigcup\bigcup} \{ F_i
		\with i=1, \ldots, \besicovitch{\adim} \} \subset U, \\
		\| W \| ( D )^{1-1/\vdim} \leq \Delta_5^{-1} \| T \| ( D )
		\quad \text{for $D \in {\textstyle\bigcup} \{ F_i \with i=1,
		\ldots, \besicovitch{\adim} \}$}.
	\end{gather*}
	If $\vdim > 1$ then, defining $\beta = \vdim / (\vdim-1)$, one
	estimates
	\begin{gather*}
		\begin{aligned}
			\| W \| (A) & \leq \tsum{i=1}{\besicovitch{\adim}}
			\tsum{D \in F_i}{} \| W \| ( D ) \leq
			\Delta_5^{-\beta} \tsum{i=1}{\besicovitch{\adim}}
			\tsum{D \in F_i}{} \| T \| ( D )^\beta \\
			& \leq \Delta_5^{-\beta}
			\tsum{i=1}{\besicovitch{\adim}} \big ( \tsum{D \in
			F_i}{} \| T \| (D) \big )^\beta \leq
			\Delta_5^{-\beta} \besicovitch{\adim} \| T \| ( U
			)^\beta.
		\end{aligned}
	\end{gather*}
	If $\vdim = 1$ and $\| W \| (A) > 0$ then $\Delta_5 \leq \| T \| (D)
	\leq \| T \| ( U )$ for some $D \in \bigcup \{ F_i \with i = 1,
	\ldots, \besicovitch{\adim} \}$.
\end{proof}
\begin{corollary} \label{corollary:rel_iso_ineq}
	Suppose $\vdim$, $\adim$, $U$, $V$, $E$, and $B$ are as in
	\ref{miniremark:zero_boundary}, $1 \leq Q \leq M < \infty$, $\adim
	\leq M$, $\Gamma = \Gamma_{\ref{thm:rel_iso_ineq}} ( M )$, $0 < r <
	\infty$, and
	\begin{gather*}
		\| V \| (E) \leq (Q-M^{-1}) \unitmeasure{\vdim} r^\vdim, \quad
		\| V \| ( \classification{E}{z}{ \density^\vdim ( \| V \|, z )
		< Q } ) \leq \Gamma^{-1} r^\vdim.
	\end{gather*}
	
	Then
	\begin{gather*}
		\| V \| ( \classification{E}{z}{ \oball{z}{r} \subset
		\rel^\adim \without B } )^{1-1/\vdim} \leq \Gamma \big ( \|
		\boundary{V}{E} \| (U) + \| \delta V \| ( E ) \big ),
	\end{gather*}
	where $0^0 = 0$.
\end{corollary}
\begin{proof}
	Define $W$ as in \ref{miniremark:zero_boundary} and note
	$\density^\vdim ( \| W \|, z ) = \density^\vdim ( \| V \|, z )$ for
	$\| W \|$ almost all $z$ by \cite[2.9.11]{MR41:1976}. Therefore
	applying \ref{thm:rel_iso_ineq} with $U$, $S$, and $T$ replaced by
	$\rel^\adim \without B$, $0$, and $\delta W$ yields the conclusion.
\end{proof}
\begin{corollary} \label{corollary:true_rel_iso_ineq}
	Suppose $\vdim$, $\adim$, $p$, $U$, $V$, and $\psi$ are as in
	\ref{miniremark:situation_general}, $p = \vdim$, $\adim \leq M$, $E$
	and $B$ are related to $\vdim$, $\adim$, $U$, and $V$, as in
	\ref{miniremark:zero_boundary}, $1 \leq Q \leq M < \infty$, $\Gamma =
	\Gamma_{\ref{thm:rel_iso_ineq}} ( M )$, $0 < r < \infty$, and
	\begin{gather*}
		\| V \| ( E ) \leq (Q-M^{-1}) \unitmeasure{\vdim} r^\vdim,
		\quad \psi (E)^{1/\vdim} \leq \Gamma^{-1} , \\
		\| V \| ( \classification{E}{z}{ \density^\vdim ( \| V \|, z )
		< Q } ) \leq \Gamma^{-1} r^\vdim.
	\end{gather*}

	Then
	\begin{gather*}
		\| V \| ( \classification{E}{z}{ \oball{z}{r} \subset
		\rel^\adim \without B } )^{1-1/\vdim} \leq \Gamma \, \|
		\boundary{V}{E} \| ( U ),
	\end{gather*}
	where $0^0=0$.
\end{corollary}
\begin{proof}
	Define $W$ as in \ref{miniremark:zero_boundary} and note
	$\density^\vdim ( \| W \|, z ) = \density^\vdim ( \| V \|, z )$ for
	$\| W \|$ almost all $z$ by \cite[2.9.11]{MR41:1976}. If $\vdim > 1$
	then approximation shows
	\begin{gather*}
		( ( \delta V ) \restrict E ) ( \theta | U ) = - \tint{E}{}
		\theta (z) \bullet \mathbf{h} (V;z) \ud \| V \| z \quad
		\text{for $\theta \in \mathscr{D} ( \rel^\adim \without B,
		\rel^\adim )$};
	\end{gather*}
	in fact since $\| ( \delta V ) \restrict E \| = \| \delta V \|
	\restrict E$ and $\theta |U \in \Lp{1} ( \| \delta V \| \restrict E,
	\rel^\adim )$ the problem reduces to the case $\spt \theta \subset U$
	which is readily treated. Therefore applying \ref{thm:rel_iso_ineq}
	with $U$, $S(\theta)$ and $T( \theta )$ replaced by $\rel^\adim
	\without B$, $(( \delta V ) \restrict E ) ( \theta | U )$ and $- (
	\boundary{V}{E} ) ( \theta | U )$ yields the conclusion.
\end{proof}
\begin{remark}
	The corollary is related to the ``relative isoperimetric inequality'',
	see \cite[4.5.2, 4.5.3]{MR41:1976} or \cite[(3.43)]{MR2003a:49002} for
	the case $\vdim = \adim$. In case $\vdim = \adim-1$ an inequality of
	this kind for area minimising currents was obtained by De~Giorgi, see
	Bombieri and Giusti \cite[Theorem 2]{MR0308945}.
\end{remark}
\begin{remark}
	Evidently, considering $E=U$, $B = \Bdry U$, $Q=1$, and stationary
	varifolds $V$ such that $\| V \| ( U ) $ is a small positive number
	shows that $\{ z \with \oball{z}{r} \cap B = \varnothing \}$ cannot be
	replaced by $U$ in the conclusion of the corollary.
\end{remark}
\begin{example}
	The set $\classification{U}{z}{ \oball{z}{r} \subset \rel^\adim
	\without B }$ may be rather small even if $V \in \IVar_\vdim (U)$ is
	stationary and indecomposable.

	Consider $\vdim = 2$, $\adim=3$,
	\begin{gather*}
		U = \classification{\rel^3}{(z_1,z_2,z_3)}{ z_1^2 + z_2^2 <1
		}, \quad B = \Bdry U, \quad Q = 1, \\
		N = \classification{\rel^3}{(z_1,z_2,z_3)}{ \cosh z_3 = (
		z_1^2 + z_2^2)^{1/2} },
	\end{gather*}
	$V_i = \var ( U \cap \boldsymbol{\mu}_{1/i} \lIm N \rIm ) \in \Var_2 (
	U )$ for $i \in \nat$, $E = \classification{U}{(z_1,z_2,z_3)}{z_3 <
	0}$ and choose $r_i$ as in \ref{corollary:true_rel_iso_ineq}. Since
	$N$ is the catenoid, $\delta V_i = 0$. Moreover, one readily verifies
	\begin{gather*}
		\lim_{i \to \infty} \| \boundary{V_i}{E} \| ( U ) = 0, \qquad
		\lim_{i \to \infty} \| V_i \| ( E \cap \cball{0}{s} ) =
		\unitmeasure{2} s^2 \quad \text{for $0 < s \leq 1$}.
	\end{gather*}
	Therefore \ref{corollary:true_rel_iso_ineq} implies $\liminf_{i \to
	\infty} r_i \geq 1$. The indecomposability of the $V_i$ is a
	consequence of Allard \cite[4.6\,(3)]{MR0307015}.
\end{example}
\section{Embeddings into Lebesgue spaces} \label{sec:embeddings}
In this section a variety of Sobolev-Poincar{\'e} type inequalities for weakly
differentiable functions are established by means of the relative
isoperimetric inequalities \ref{corollary:rel_iso_ineq} and
\ref{corollary:true_rel_iso_ineq}. The key are local estimates under a
smallness condition on set of points where the function is nonzero, see
\ref{thm:sob_poin_summary}\,\eqref{item:sob_poin_summary:interior}. These
estimate are formulated in such a way as to improve in case the function is
zero on an open part of the boundary. Consequently, Sobolev inequalities are
essentially a special case, see
\ref{thm:sob_poin_summary}\,\eqref{item:sob_poin_summary:global}. Local
integrability results also follow, see \ref{corollary:integrability}. Finally,
versions without the previously hypothesed smallness condition are derived in
\ref{thm:sob_poincare_q_medians} and \ref{thm:sob_poin_several_med}.

The differentiability results which will be derived in \ref{thm:approx_diff}
and \ref{thm:diff_lebesgue_spaces} is based on
\ref{thm:sob_poin_summary}\,\eqref{item:sob_poin_summary:interior} whereas the
oscillation estimates which will be proven in \ref{thm:mod_continuity} and
\ref{thm:hoelder_continuity} employ \ref{thm:sob_poin_several_med}.

\begin{theorem} \label{thm:sob_poin_summary}
	Suppose $1 \leq M < \infty$.

	Then there exists a positive, finite number $\Gamma$ with the
	following property.

	If $l \in \nat$, $\vdim$, $\adim$, $p$, $U$, $V$, and $\psi$ are as in
	\ref{miniremark:situation_general}, $\adim \leq M$, $G$ is a
	relatively open subset of $\Bdry U$, $B = ( \Bdry U ) \without G$, and
	$f \in \trunc_G ( V, \rel^l )$, then the following two statements
	hold:
	\begin{enumerate}
		\item \label{item:sob_poin_summary:interior} Suppose $1 \leq Q
		\leq M$, $0 < r < \infty$, $E = U \cap \{ z \with f(z) \neq 0
		\}$,
		\begin{gather*}
			\| V \| ( E ) \leq ( Q-M^{-1} ) \unitmeasure{\vdim}
			r^\vdim, \\
			\| V \| ( E \cap \{ z \with \density^\vdim ( \| V \|,
			z ) < Q \} ) \leq \Gamma^{-1} r^\vdim,
		\end{gather*}
		and $A = U \cap \{ z \with \oball{z}{r} \subset \rel^\adim
		\without B \}$. Then the following four implications hold:
		\begin{enumerate}
			\item \label{item:sob_poin_summary:interior:p=1} If $p
			= 1$, $\beta = \infty$ if $\vdim = 1$ and $\beta =
			\vdim/(\vdim-1)$ if $\vdim > 1$, then
			\begin{gather*}
				\eqLpnorm{\| V \| \restrict A}{\beta}{f} \leq
				\Gamma \big ( \Lpnorm{\| V \|}{1}{
				\derivative{V}{f} } + \Lpnorm{\| \delta V
				\|}{1} {f} \big ).
			\end{gather*}
			\item \label{item:sob_poin_summary:interior:p=m=1} If
			$p = \vdim = 1$ and $\psi ( E ) \leq \Gamma^{-1}$,
			then
			\begin{gather*}
				\eqLpnorm{\| V \| \restrict A}{\infty}{f} \leq
				\Gamma \, \Lpnorm{\| V \|}{1}{
				\derivative{V}{f} }.
			\end{gather*}
			\item \label{item:sob_poin_summary:interior:q<m=p} If
			$1 \leq q < \vdim = p$ and $\psi ( E ) \leq
			\Gamma^{-1}$, then
			\begin{gather*}
				\eqLpnorm{\| V \| \restrict A}{\vdim
				q/(\vdim-q)}{f} \leq \Gamma (\vdim-q)^{-1}
				\Lpnorm{\| V \|}{q}{ \derivative{V}{f} }.
			\end{gather*}
			\item \label{item:sob_poin_summary:interior:p=m<q} If
			$1 < p = \vdim < q \leq \infty$ and $\psi ( E ) \leq
			\Gamma^{-1}$, then
			\begin{gather*}
				\eqLpnorm{\| V \| \restrict A}{\infty}{f} \leq
				\Gamma^{1/(1/\vdim-1/q)} \| V \| ( E
				)^{1/\vdim-1/q} \Lpnorm{\| V \|}{q}{
				\derivative{V}{f} }.
			\end{gather*}
		\end{enumerate}
		\item \label{item:sob_poin_summary:global} Suppose $G = \Bdry
		U$, $E = U \cap \{ z \with f(z) \neq 0 \}$, and $\| V \| ( E )
		< \infty$. Then the following four implications hold:
		\begin{enumerate}
			\item \label{item:sob_poin_summary:global:p=1} If $p
			= 1$, $\beta = \infty$ if $\vdim = 1$ and $\beta =
			\vdim/(\vdim-1)$ if $\vdim > 1$, then
			\begin{gather*}
				\Lpnorm{\| V \|}{\beta}{f} \leq \Gamma \big
				( \Lpnorm{\| V \|}{1}{ \derivative{V}{f} } +
				\Lpnorm{\| \delta V \|}{1} {f} \big ).
			\end{gather*}
			\item \label{item:sob_poin_summary:global:p=m=1} If
			$p = \vdim = 1$ and $\psi ( E ) \leq \Gamma^{-1}$,
			then
			\begin{gather*}
				\Lpnorm{\| V \|}{\infty}{f} \leq \Gamma \,
				\Lpnorm{\| V \|}{1}{ \derivative{V}{f} }.
			\end{gather*}
			\item \label{item:sob_poin_summary:global:q<m=p} If
			$1 \leq q < \vdim = p$ and $\psi ( E ) \leq
			\Gamma^{-1}$, then
			\begin{gather*}
				\Lpnorm{\| V \|}{\vdim q/(\vdim-q)}{f} \leq
				\Gamma (\vdim-q)^{-1} \Lpnorm{\| V \|}{q}{
				\derivative{V}{f} }.
			\end{gather*}
			\item \label{item:sob_poin_summary:global:p=m<q} If
			$1 < p = \vdim < q \leq \infty$ and $\psi ( E ) \leq
			\Gamma^{-1}$, then
			\begin{gather*}
				\Lpnorm{\| V \|}{\infty}{f} \leq
				\Gamma^{1/(1/\vdim-1/q)} \| V \| (
				E)^{1/\vdim-1/q} \Lpnorm{\| V \|}{q}{
				\derivative{V}{f} }.
			\end{gather*}
		\end{enumerate}
	\end{enumerate}
\end{theorem}
\begin{proofinsteps}
	Denote by \eqref{item:sob_poin_summary:interior:q<m=p}$'$
	[respectively \eqref{item:sob_poin_summary:global:q<m=p}$'$] the
	implication resulting from
	\eqref{item:sob_poin_summary:interior:q<m=p} [respectively
	\eqref{item:sob_poin_summary:global:q<m=p}] through omission of the
	factor $(\vdim-q)^{-1}$ and addition of the requirement $q=1$. It is
	sufficient to construct functions
	$\Gamma_{\eqref{item:sob_poin_summary:interior:p=1}}$,
	$\Gamma_{\eqref{item:sob_poin_summary:interior:p=m=1}}$,
	$\Gamma_{\eqref{item:sob_poin_summary:interior:q<m=p}}$,
	$\Gamma_{\eqref{item:sob_poin_summary:interior:q<m=p}'}$,
	$\Gamma_{\eqref{item:sob_poin_summary:interior:p=m<q}}$,
	$\Gamma_{\eqref{item:sob_poin_summary:global:p=1}}$,
	$\Gamma_{\eqref{item:sob_poin_summary:global:p=m=1}}$,
	$\Gamma_{\eqref{item:sob_poin_summary:global:q<m=p}}$,
	$\Gamma_{\eqref{item:sob_poin_summary:global:q<m=p}'}$, and
	$\Gamma_{\eqref{item:sob_poin_summary:global:p=m<q}}$ corresponding to
	the implications \eqref{item:sob_poin_summary:interior:p=1},
	\eqref{item:sob_poin_summary:interior:p=m=1},
	\eqref{item:sob_poin_summary:interior:q<m=p},
	\eqref{item:sob_poin_summary:interior:q<m=p}$'$,
	\eqref{item:sob_poin_summary:interior:p=m<q},
	\eqref{item:sob_poin_summary:global:p=1},
	\eqref{item:sob_poin_summary:global:p=m=1},
	\eqref{item:sob_poin_summary:global:q<m=p},
	\eqref{item:sob_poin_summary:global:q<m=p}$'$, and
	\eqref{item:sob_poin_summary:global:p=m<q} whose value at $M$ for $1
	\leq M < \infty$ is a positive, finite number such that the respective
	implication is true for $M$ with $\Gamma$ replaced by this value.
	Define
	\begin{gather*}
		\Gamma_{\eqref{item:sob_poin_summary:interior:p=1}} (M) =
		\Gamma_{\eqref{item:sob_poin_summary:interior:p=m=1}} (M) =
		\Gamma_{\eqref{item:sob_poin_summary:interior:q<m=p}'} (M) =
		\Gamma_{\ref{thm:rel_iso_ineq}} ( M ), \quad
		\Gamma_{\eqref{item:sob_poin_summary:global:p=1}} ( M )=
		\Gamma_{\eqref{item:sob_poin_summary:interior:p=1}} ( \sup \{
		2, M \} ), \\
		\Gamma_{\eqref{item:sob_poin_summary:global:p=m=1}} ( M ) =
		\Gamma_{\eqref{item:sob_poin_summary:interior:p=m=1}} ( \sup
		\{ 2, M \} ), \quad
		\Gamma_{\eqref{item:sob_poin_summary:global:q<m=p}'} ( M ) =
		\Gamma_{\eqref{item:sob_poin_summary:interior:q<m=p}'} ( \sup
		\{ 2, M \} ), \\
		\Gamma_{\eqref{item:sob_poin_summary:global:q<m=p}} (M) = M^2
		\Gamma_{\eqref{item:sob_poin_summary:global:q<m=p}'} ( M ), \\
		\Delta_1 (M) = ( \sup \{ 1/2, 1- 1/( 2M^2-1) \})^{1/M},
		\quad \Delta_2 (M) = 2 M / (1-\Delta_1(M)), \\
		\Delta_3 (M) = M \sup \{ \unitmeasure{\vdim}
		\with M \geq \vdim \in \nat \}, \\
		\Delta_4 (M) = \inf \{ \unitmeasure{\vdim} \with M \geq \vdim
		\in \nat \} / 2, \\
		\Delta_5 (M) = \sup \{ 1,
		\Gamma_{\eqref{item:sob_poin_summary:global:q<m=p}} (M)(
		\Delta_2(M) \Delta_3(M) + 1) \}, \\
		\Delta_6 (M) = 4^{M+1} \Delta_3(M)^M \Delta_5(M)^{M+1}, \\
		\Gamma_{\eqref{item:sob_poin_summary:interior:q<m=p}} (M) =
		\sup \{ 2
		\Gamma_{\eqref{item:sob_poin_summary:interior:q<m=p}'} ( 2M),
		\Delta_6 (M) ( 1 + \Delta_3(M)
		\Gamma_{\eqref{item:sob_poin_summary:interior:q<m=p}'} (2M) )
		\}, \\
		\Delta_7 (M) = \sup \{ 1,
		\Gamma_{\eqref{item:sob_poin_summary:interior:q<m=p}'} (2M)
		\}, \\
		\Delta_8 (M) = \inf \big \{ \inf \{ 1, \Delta_3(M)^{-1}
		\Delta_4(M) \Delta_1(M)^M \} ( 1 - \Delta_1(M) )^M, 1/2 \big
		\}, \\
		\Delta_9 ( M ) = 2 \Delta_7(M) \Delta_8(M)^{-1}, \\
		\Gamma_{\eqref{item:sob_poin_summary:interior:p=m<q}} (M) =
		\sup \big \{ \Delta_9 (M), \Delta_1 (M)^{-M}
		\Gamma_{\eqref{item:sob_poin_summary:interior:q<m=p}'} ( 2M )
		\big \}, \\
		\Gamma_{\eqref{item:sob_poin_summary:global:p=m<q}} (M) =
		\Gamma_{\eqref{item:sob_poin_summary:interior:p=m<q}} ( \sup
		\{ 2, M \} ).
	\end{gather*}
	whenever $1 \leq M < \infty$.

	In order to verify the asserted properties suppose $1 \leq M < \infty$
	and abbreviate $\delta_i = \Delta_i(M)$ whenever $i = 1, \ldots, 9$.
	In the verification of assertion ($X$) [respectively ($X$)$'$], where
	$X$ is one of \ref{item:sob_poin_summary:interior:p=1},
	\ref{item:sob_poin_summary:interior:p=m=1},
	\ref{item:sob_poin_summary:interior:q<m=p},
	\ref{item:sob_poin_summary:interior:p=m<q},
	\ref{item:sob_poin_summary:global:p=1},
	\ref{item:sob_poin_summary:global:p=m=1},
	\ref{item:sob_poin_summary:global:q<m=p}, and
	\ref{item:sob_poin_summary:global:p=m<q}, [respectively
	\ref{item:sob_poin_summary:interior:q<m=p} and
	\ref{item:sob_poin_summary:global:q<m=p}] it will be assumed that the
	quantities occuring in its hypotheses are defined and satisfy these
	hypotheses with $\Gamma$ replaced by $\Gamma_{(X)} (M)$ [respectively
	$\Gamma_{(X)'} (M)$].

	\begin{step} \label{step:sob_poin_summary:1}
		Verification of the property of
		$\Gamma_{\eqref{item:sob_poin_summary:interior:p=1}}$.
	\end{step}

	Replacing $l$ and $f$ by $1$ and $|f|$, one may assume $l=1$ and $f
	\geq 0$ by \ref{lemma:basic_v_weakly_diff},
	\ref{example:composite}\,\eqref{item:composite:mod} and define $E(t) =
	\{ z \with f(z) > t \}$ for $0 \leq t < \infty$. If $\vdim = 1$ then
	$\Gamma_{\ref{thm:rel_iso_ineq}}(M)^{-1} \leq \| \boundary{V}{E(t)} \|
	(U) + \| \delta V \| ( E(t))$ for $\mathscr{L}^1$ almost all $t$ with
	$0 < t < \eqLpnorm{\| V \| \restrict A}{\beta}{f}$ by
	\ref{corollary:rel_iso_ineq} and the conclusion follows from
	\ref{thm:tv_coarea}. If $\vdim > 1$, define
	\begin{gather*}
		f_t = \inf \{ f, t \}, \quad \phi (t) = \eqLpnorm{\| V \|
		\restrict A}{\beta}{f_t} \leq t \, \| V \| ( E )^{1/\beta} <
		\infty
	\end{gather*}
	for $0 \leq t < \infty$ and use Minkowski's inequality to conclude
	\begin{gather*}
		0 \leq \phi (t+h) - \phi (t) \leq \eqLpnorm{\| V \| \restrict
		A}{\beta}{f_{t+h}-f_t} \leq h \, \| V \| ( A \cap E(t)
		)^{1/\beta}
	\end{gather*}
	for $0 \leq t < \infty$ and $0 < h < \infty$. Therefore $\phi$ is
	Lipschitzian and one infers from \cite[2.9.19]{MR41:1976} and
	\ref{corollary:rel_iso_ineq} that
	\begin{gather*}
		0 \leq \phi'(t) \leq \| V \| ( A \cap E(t) )^{1-1/\vdim} \leq
		\Gamma_{\ref{thm:rel_iso_ineq}}(M) \big ( \|
		\boundary{V}{E(t)} \| ( U ) + \| \delta V \| (E(t)) \big)
	\end{gather*}
	for $\mathscr{L}^1$ almost all $0 < t < \infty$, hence $\eqLpnorm{\| V
	\| \restrict A}{\beta}{f} = \lim_{t \to \infty} \phi (t) =
	\tint{0}{\infty} \phi' \ud \mathscr{L}^1$ by \cite[2.9.20]{MR41:1976}
	and \ref{thm:tv_coarea} implies the conclusion.

	\begin{step} \label{step:sob_poin_summary:2}
		Verification of the properties of
		$\Gamma_{\eqref{item:sob_poin_summary:interior:p=m=1}}$ and
		$\Gamma_{\eqref{item:sob_poin_summary:interior:q<m=p}'}$.
	\end{step}

	Omitting the terms involving $\delta V$ from the proof of Step
	\ref{step:sob_poin_summary:1} and using
	\ref{corollary:true_rel_iso_ineq} instead of
	\ref{corollary:rel_iso_ineq}, a proof of Step
	\ref{step:sob_poin_summary:2} results.
	
	\begin{step}
		Verification of the properties of
		$\Gamma_{\eqref{item:sob_poin_summary:global:p=1}}$,
		$\Gamma_{\eqref{item:sob_poin_summary:global:p=m=1}}$, and
		$\Gamma_{\eqref{item:sob_poin_summary:global:q<m=p}'}$.
	\end{step}

	Taking $Q = 1$, one may apply
	\eqref{item:sob_poin_summary:interior:p=1},
	\eqref{item:sob_poin_summary:interior:p=m=1}, and
	\eqref{item:sob_poin_summary:interior:q<m=p}$'$ with $M$ replaced by
	$\sup \{ 2, M \}$ and a sufficiently large number $r$.

	\begin{step}
		Verification of the property of
		$\Gamma_{\eqref{item:sob_poin_summary:global:q<m=p}}$.
	\end{step}

	Replacing $f$ and $l$ by $|f|$ and $1$, one may assume $l=1$ and $f
	\in \trunc_{\Bdry U} ( V )^+$ by \ref{lemma:basic_v_weakly_diff},
	\ref{example:composite}\,\eqref{item:composite:mod}. Moreover, one may
	assume $f$ to be bounded by \ref{lemma:basic_v_weakly_diff},
	\ref{example:composite}\,\eqref{item:composite:1d}.  Noting
	$f^{q(\vdim-1)/(\vdim-q)} \in \trunc_{\Bdry U} ( V )^+$ by
	\ref{lemma:trunc_tg}, one now applies
	\eqref{item:sob_poin_summary:global:q<m=p}$'$ with $f$ replaced by
	$f^{q(\vdim-1)/(\vdim-q)}$ to deduce the assertion by the method of
	\cite[4.5.15]{MR41:1976}.

	\begin{step}
		Verification of the property of
		$\Gamma_{\eqref{item:sob_poin_summary:interior:q<m=p}}$.
	\end{step}

	One may assume $\Lpnorm{\| V \|}{q}{ \derivative{V}{f} } < \infty$
	and, possibly rescaling, also $r=1$. Moreover, observe that one may
	assume $f$ to be bounded and $l=1$ by replacing $f(z)$ by $\inf \{
	|f(z)|, t \}$ for $0 < t < \infty$ and considering $t \to \infty$ by
	\ref{lemma:basic_v_weakly_diff},
	\ref{example:composite}\,\eqref{item:composite:mod}\,\eqref{item:composite:1d}.
	Consequently,
	\begin{gather*}
		\tint{}{} |f| + | \derivative{V}{f} | \ud \| V \| < \infty
	\end{gather*}
	by \ref{thm:addition}\,\eqref{item:addition:zero} and H\"older's
	inequality. Define
	\begin{gather*}
		r_i = \delta_1 + (i-1)(1-\delta_1)/M, \quad
		A_i = \classification{U}{z}{ \oball{z}{r_i} \subset \rel^\adim
		\without B } \quad
	\end{gather*}
	whenever $i=1, \ldots, \vdim$. Observe that
	\begin{gather*}
		Q-M^{-1} \leq r_1^\vdim \big (Q-(2M)^{-1}\big ), \quad
		r_1^\vdim \geq 1/2, \quad r_\vdim \leq 1, \\
		r_{i+1} - r_i = 2 / \delta_2 \quad \text{for $i=1, \ldots,
		\vdim-1$}.
	\end{gather*}
	Choose $g_i \in \mathscr{E}^0 ( U )$ with
	\begin{gather*}
		\spt g_i \subset A_i, \qquad g_i(z) = 1 \quad \text{for $z \in
		A_{i+1}$}, \\
		0 \leq g_i (z) \leq 1 \qquad \text{for $z \in U$}, \quad \Lip
		g_i \leq \delta_2
	\end{gather*}
	whenever $i=1, \ldots, \vdim-1$. Notice that $g_i f \in \trunc (V)$ by
	\ref{thm:addition}\,\eqref{item:addition:mult} with
	\begin{gather*}
		\derivative{V}{(g_if)} (z) = \derivative{V}{g_i} (z) f(z) +
		g_i (z) \derivative{V}{f} (z) \quad \text{for $\| V \|$ almost
		all $z$}.
	\end{gather*}
	Moreover, one infers $g_i f \in \trunc_{\Bdry U} ( V )$ by
	\ref{thm:mult_tg} and \ref{lemma:boundary}.

	If $i \in \{ 1, \ldots, \vdim-1 \}$ and $1-(i-1)/\vdim \geq 1/q \geq
	1-i/\vdim$ then by \eqref{item:sob_poin_summary:global:q<m=p} and
	H\"older's inequality
	\begin{gather*}
		\begin{aligned}
			& \eqLpnorm{\| V \| \restrict A_{i+1}}{\vdim
			q/(\vdim-q)}{f} \leq \Lpnorm{\| V \|}{\vdim
			q/(\vdim-q)} { g_i f } \\
			& \qquad \leq
			\Gamma_{\eqref{item:sob_poin_summary:global:q<m=p}}
			(M) (\vdim-q)^{-1} \big ( \delta_2 \eqLpnorm{\| V \|
			\restrict A_i}{q}{f} + \Lpnorm{\| V \|}{q}{
			\derivative{V}{f} } \big ) \\
			& \qquad \leq \delta_5 ( \vdim-q )^{-1} \big (
			\eqLpnorm{\| V \| \restrict A_i }{\vdim/(\vdim-i)} {f}
			+ \Lpnorm{\| V \|}{q}{ \derivative{V}{f} } \big ).
		\end{aligned}
	\end{gather*}
	In particular, replacing $q$ by $\vdim/(\vdim-i)$ and using H\"older's
	inequality, one infers
	\begin{gather*}
		\eqLpnorm{ \| V \| \restrict A_{i+1} }{\vdim/(\vdim-i-1)} {f}
		\leq 2 \delta_3 \delta_5 \big ( \eqLpnorm{\| V \| \restrict
		A_i }{\vdim/(\vdim-i)} {f} + \Lpnorm{\| V \|}{q}
		{\derivative{V}{f}} \big )
	\end{gather*}
	whenever $i \in \{ 1, \ldots, \vdim-2 \}$ and $1-i/\vdim \geq 1/q$,
	choosing $j \in \{ 1, \ldots, \vdim-1 \}$ with $1-(j-1)/\vdim \geq 1/q
	> 1-j/\vdim$ and iterating this $j-1$ times, also
	\begin{gather*}
		\begin{aligned}
			& \eqLpnorm{ \| V \| \restrict A_j }{\vdim/(\vdim-j)}
			{f} \\
			& \qquad \leq ( 4 \delta_3 \delta_5)^{j-1} \big (
			\eqLpnorm{\| V \| \restrict A_1 }{\vdim/(\vdim-1)} {f}
			+ \Lpnorm{\| V \|}{q} {\derivative{V}{f}} \big ).
		\end{aligned}
	\end{gather*}
	Together this yields
	\begin{gather*}
		\begin{aligned}
			& \eqLpnorm{\| V \| \restrict A_{j+1} }{\vdim
			q/(\vdim-q)} {f} \\
			& \qquad \leq \delta_6 (\vdim-q)^{-1} \big (
			\eqLpnorm{\| V \| \restrict A_1}{\vdim/(\vdim-1)} {f}
			+ \Lpnorm{\| V \|}{q}{ \derivative{V}{f} } \big ).
		\end{aligned}
	\end{gather*}
	and the conclusion then follows from
	\eqref{item:sob_poin_summary:interior:q<m=p}$'$ with $M$ and $r$
	replaced by $2M$ and $r_1$.

	\begin{step}
		Verification of the property of
		$\Gamma_{\eqref{item:sob_poin_summary:interior:p=m<q}}$.
	\end{step}

	Assume
	\begin{gather*}
		r=1, \quad \| V \| ( E ) > 0, \quad \Lpnorm{\| V
		\|}{q}{\derivative{V}{f}} < \infty,
	\end{gather*}
	abbreviate $\beta = \vdim / ( \vdim-1 )$ and $\eta = 1/(1/\vdim-1/q)$,
	and suppose
	\begin{gather*}
		\mu = \| V \| ( E ), \quad \Lpnorm{\| V \|}{q}{
		\derivative{V}{f} } < \gamma < \infty.
	\end{gather*}
	Notice that $0 < \delta_1 < 1$, $0 < \delta_8 < 1$, $\eta \geq 1$ and
	$\mu > 0$, and define
	\begin{gather*}
		r_i = 1 - (1-\delta_1)^i, \quad t_i = \delta_9^\eta
		\mu^{1/\vdim-1/q} \gamma ( 1 - 2^{-i} ), \\
		Z_i = \classification{\rel^\adim}{z}{\dist (z,B) > r_i },
		\quad U_i = U \cap Z_i, \\
		H_i = Z_i \cap \Bdry U, \quad C_i = ( \Bdry U_i ) \without
		H_i, \quad X_i = V | \mathbf{2}^{U_i \times
		\grass{\adim}{\vdim} }, \\
		a_i = \| V \| ( \classification{U_i}{z}{|f(z)|>t_i} ), \\
		f_i (z) = \sup \{ | f(z) | - t_i, 0 \} \quad \text{whenever $z
		\in \dmn f$}
	\end{gather*}
	whenever $i$ is a nonnegative integer. Notice that
	\begin{gather*}
		r_{i+1} - r_i = \delta_1 (1-\delta_1)^i, \quad U_{i+1} \subset
		U_i, \quad t_{i+1} - t_i = \delta_9^\eta \mu^{1/\vdim-1/q}
		\gamma 2^{-i-1} > 0
	\end{gather*}
	whenever $i$ is a nonnegative integer. The conclusion is readily
	deduced from the assertion,
	\begin{gather*}
		a_i \leq \mu \delta_8^{\eta i} \quad \text{\emph{whenever $i$
		is a nonnegative integer}},
	\end{gather*}
	which will be proven by induction.

	The case $i=0$ is trivial. To prove the case $i=1$, one applies
	\eqref{item:sob_poin_summary:interior:q<m=p}$'$ with $M$ and $r$
	replaced by $2M$ and $\delta_1$ and H\"older's inequality in
	conjunction with \ref{thm:addition}\,\eqref{item:addition:zero} to
	obtain
	\begin{gather*}
		\eqLpnorm{\| V \| \restrict U_1}{\beta}{f} \leq
		\Gamma_{\eqref{item:sob_poin_summary:interior:q<m=p}'} (2M)
		\Lpnorm{\| V \|}{1}{ \derivative{V}{f} } \leq \delta_7
		\mu^{1-1/q} \gamma,
	\end{gather*}
	hence
	\begin{gather*}
		a_1^{1-1/\vdim} \leq (t_1-t_0)^{-1} \eqLpnorm{\| V \|
		\restrict U_1}{\beta}{f} \leq \delta_9^{-\eta} \mu^{1-1/\vdim} 2
		\delta_7 \leq \mu^{1-1/\vdim} \delta_8^{\eta ( 1-1/\vdim )}.
	\end{gather*}
	Assuming the assertion to be true for some $i \in \nat$, notice that
	\begin{gather*}
		\mu \leq \delta_3, \quad \eta i \geq 1, \quad a_i \leq \mu
		\delta_8^{\eta i} \leq \delta_4 \delta_1^\vdim
		(1-\delta_1)^{i \vdim} \leq (1/2) \unitmeasure{\vdim}
		(r_{i+1}-r_i)^\vdim, \\
		f_i \in \trunc_G ( V, \rel^l ), \quad f_i | Z_i \in
		\trunc_{H_i} ( X_i, \rel^l ), \\
		H_i = Z_i \cap \Bdry U_i, \quad U_{i+1} \subset
		\classification{U_i}{z}{ \oball{z}{r_{i+1}-r_i} \subset
		\rel^\adim \without C_i },
	\end{gather*}
	by \ref{lemma:basic_v_weakly_diff},
	\ref{example:composite}\,\eqref{item:composite:mod}\,\eqref{item:composite:1d}
	and \ref{lemma:restriction_tg}, \ref{remark:restriction_tg} with $f$,
	$Z$, $H$, $X$, $C$, and $r$ replaced by $f_i$, $Z_i$, $H_i$, $X_i$,
	$C_i$, and $r_{i+1}-r_i$. Therefore
	\eqref{item:sob_poin_summary:interior:q<m=p}$'$ applied with $U$, $V$,
	$Q$, $M$, $G$, $B$, $f$, and $r$ replaced by $U_i$, $X_i$, $1$, $2M$,
	$H_i$, $C_i$, $f_i | Z$, and $r_{i+1}-r_i$ yields
	\begin{gather*}
		\eqLpnorm{\| V \| \restrict U_{i+1} }{\beta}{f_i} \leq
		\Gamma_{\eqref{item:sob_poin_summary:interior:q<m=p}'} (2M)
		\eqLpnorm{\| V \| \restrict U_i }{1}{ \derivative{V}{f_i} },
	\end{gather*}
	hence, using H\"older's inequality in conjunction with
	\ref{lemma:basic_v_weakly_diff},
	\ref{example:composite}\,\eqref{item:composite:mod}\,\eqref{item:composite:1d}
	and
	\ref{thm:addition}\,\eqref{item:addition:zero},
	\begin{gather*}
		\eqLpnorm{\| V \| \restrict U_{i+1}}{\beta}{f_i} \leq \delta_7
		a_i^{1-1/q} \gamma.
	\end{gather*}
	Noting $\delta_9^{-\eta} 2 \delta_7 \leq \delta_8^{\eta (1-1/\vdim)}$
	and $2 \delta_8^{\eta (1-1/q)} \leq \delta_8^{\eta (1-1/\vdim)}$, it
	follows
	\begin{gather*}
		\begin{aligned}
			a_{i+1}^{1-1/\vdim} & \leq ( t_{i+1} - t_i )^{-1}
			\eqLpnorm{\| V \| \restrict U_{i+1}}{\beta}{f_i} \\
			& \leq \mu^{1-1/\vdim} \delta_9^{-\eta} 2 \delta_7
			\big ( 2 \delta_8^{\eta(1-1/q)} \big )^i \leq
			\mu^{1-1/\vdim} \delta_8^{\eta (1-1/\vdim) (i+1)}
		\end{aligned}
	\end{gather*}
	and the assertion is proven.

	\begin{step}
		Verification of the property of
		$\Gamma_{\eqref{item:sob_poin_summary:global:p=m<q}}$.
	\end{step}

	Taking $Q = 1$, one may apply
	\eqref{item:sob_poin_summary:interior:p=m<q} with $M$ replaced by
	$\sup \{ 2, M \}$ and a sufficiently large number $r$.
\end{proofinsteps}
\begin{remark}
	The role of \ref{corollary:rel_iso_ineq} and
	\ref{corollary:true_rel_iso_ineq} respectively in the Steps
	\ref{step:sob_poin_summary:1} and \ref{step:sob_poin_summary:2} of the
	preceding proof is identical to the one of Allard
	\cite[7.1]{MR0307015} in Allard \cite[7.3]{MR0307015}.

	The method of deduction of \eqref{item:sob_poin_summary:global:q<m=p}
	from \eqref{item:sob_poin_summary:global:q<m=p}$'$ is classical, see
	\cite[4.5.15]{MR41:1976}.

	The method of deduction of
	\eqref{item:sob_poin_summary:interior:q<m=p} from
	\eqref{item:sob_poin_summary:interior:q<m=p}$'$ and
	\eqref{item:sob_poin_summary:global:q<m=p} is outlined by Hutchinson
	in \cite[pp.~60--61]{MR1066398}.

	The iteration procedure employed in the proof
	\eqref{item:sob_poin_summary:interior:p=m<q} bears formal resemblance
	with \cite[Lemma 4.1\,(i)]{MR0251373}. However, here the estimate of
	$a_{i+1}$ in terms of $a_i$ requires a smallness hypothesis on $a_i$.
\end{remark}
\begin{corollary} \label{corollary:integrability}
	Suppose $l \in \nat$, $\vdim$, $\adim$, $U$, $V$, $p$, and $\psi$ are
	as in \ref{miniremark:situation_general}, $1 \leq q \leq \infty$, $f
	\in \trunc ( V, \rel^l )$, and $\derivative{V}{f} \in \Lploc{q} ( \| V
	\|, \Hom ( \rel^\adim, \rel^l ) )$.

	Then the following four statements hold:
	\begin{enumerate}
		\item \label{item:integrability:m=1} If $\vdim = 1$, then $f
		\in \Lploc{\infty} ( \| V \| + \| \delta V \|, \rel^l )$.
		\item \label{item:integrability:m>1} If $\vdim > 1$ and $f \in
		\Lploc{1} ( \| \delta V \|, \rel^l )$, then $f \in
		\Lploc{\vdim/(\vdim-1)} ( \| V \|, \rel^l )$.
		\item \label{item:integrability:q<m=p} If $1 \leq q < \vdim =
		p$, then $f \in \Lploc{\vdim q/(\vdim-q)} ( \| V \|, \rel^l
		)$.
		\item \label{item:integrability:m=p<q} If $1 < \vdim = p < q
		\leq \infty$, then $f \in \Lploc{\infty} ( \| V \|, \rel^l )$.
	\end{enumerate}
\end{corollary}
\begin{proof}
	Assume $( \| V \| + \psi ) ( U ) + \Lpnorm{\| V \|}{q}{
	\derivative{V}{f} } < \infty$ and $\Lpnorm{\| \delta V \|}{1}{f} <
	\infty$ in case of \eqref{item:integrability:m>1}. Suppose $K$ is a
	compact subset of $U$. Choose $0 < r < \infty$ with $\oball{z}{r}
	\subset U$ for $z \in K$ and $0 < s < \infty$ with
	\begin{gather*}
		\| V \| ( E ) \leq (1/2) \unitmeasure{\vdim} r^\vdim, \quad
		\psi ( E ) \leq \Gamma_{\ref{thm:sob_poin_summary}} ( 2M
		)^{-1},
	\end{gather*}
	where $E = \{ z \with |f(z)| > s \}$. Select $\gamma \in \mathscr{E}^0
	( \rel )$ with $\gamma (t) = 0$ if $t \leq s$ and $\gamma (t) = t$ if
	$t \geq s+1$, define $\phi : \rel^l \to \rel^l$ by $\phi (0) = 0$ and
	$\phi (y) = \gamma ( |y| ) |y|^{-1} y$ for $y \in \rel^l \without \{ 0
	\}$ and notice that $\phi \circ f \in \trunc ( V, \rel^l )$ with
	\begin{gather*}
		\| \derivative{V}{(\phi \circ f)}(z) \| \leq ( \Lip \phi ) \|
		\derivative{V}{f} \| ( z ) \quad \text{for $\| V \|$ almost
		all $z$}
	\end{gather*}
	by \ref{lemma:basic_v_weakly_diff}. Applying
	\ref{thm:sob_poin_summary}\,\eqref{item:sob_poin_summary:interior}
	with $M$, $G$, $f$, and $Q$ replaced by $2M$, $\varnothing$, $\phi
	\circ f$, and $1$ in conjunction with H{\"o}lder's inequality, and
	noting \ref{corollary:boundary_controls_interior} in case of
	\eqref{item:integrability:m=1}, the conclusion follows.
\end{proof}
\begin{definition}
	Whenever $G$ is a set of $\overline{\rel}$ valued functions $\sum_G$
	shall denote the function whose domain is the set of $z$ such that
	$\sum_{g \in G} g(z) \in \overline{\rel}$ and whose value at $z$
	equals $\sum_{g \in G} g(z)$. This function will also be denoted by
	$\sum_{g \in G} g$.
\end{definition}
\begin{miniremark} \label{miniremark:lpnorms}
	\emph{If $1 \leq q \leq p \leq \infty$, $\phi$ measures $X$, $G$ is a
	countable collection of $\phi$ measurable $\{ t \with 0 \leq t \leq
	\infty \}$ valued functions, $1 \leq \gamma < \infty$,
	\begin{gather*}
		\card ( G \cap \{ g \with g(z) \neq 0 \} ) \leq \gamma \quad
		\text{for $\phi$ almost all $z$},
	\end{gather*}
	and $f (z) = \sum_{g \in G} g(z)$ for $\phi$ almost all $z$, then 
	\begin{align*}
		\Lpnorm{\phi}{p}{f} & \leq \gamma \big ( \tsum{g \in G}{}
		\Lpnorm{\phi}{p}{g}^q \big )^{1/q} &&
		\text{if $q < \infty$}, \\
		\Lpnorm{\phi}{p}{f} & \leq \gamma \sup ( \{ 0 \} \cup \{
		\Lpnorm{\phi }{p}{g} \with g \in G \} ) && \text{if $q =
		\infty$};
	\end{align*}}
	in fact $\Lpnorm{\phi}{p}{f}^q \leq \gamma^q \big ( \sum_{g \in G} \int
	g^p \ud \phi \big )^{q/p} \leq \gamma^q \sum_{g \in G}
	\Lpnorm{\phi}{p}{g}^q$ if $p < \infty$ and $\Lpnorm{\phi}{\infty}{f^q}
	\leq \gamma^q \Lpnorm{\phi}{\infty}{ \sum_{g \in G} g^q} \leq \gamma^q
	\sum_{g\in G} \Lpnorm{\phi}{\infty}{g}^q$ if $1 \leq q < p = \infty$.
\end{miniremark}
\begin{lemma} \label{lemma:partition_of_unity}
	Suppose $l \in \nat$.

	Then there exists a positive, finite number $\Gamma$ with the
	following property.

	If $D$ is a closed subset of $\rel^l$ and $0 < t < \infty$, then there
	exists a countable family $S$ contained in $\mathscr{D}^0 ( \rel^l )$
	satisfying the following four conditions.
	\begin{enumerate}
		\item \label{item:partition_of_unity:contained} If $v \in S$
		then $\spt v \subset \oball{d}{t}$ for some $d \in D$.
		\item \label{item:partition_of_unity:card} If $y \in \rel^l$
		then $\card ( S \cap \{ v \with \cball{y}{t/4} \cap \spt v
		\neq \varnothing \} ) \leq \Gamma$.
		\item \label{item:partition_of_unity:sum} There holds $D
		\subset \Int \left \{ \sum_{v \in S} v(y) = 1 \right \}$ and
		$\sum_{v \in S} v(y) \leq 1$ for $y \in \rel^l$.
		\item \label{item:partition_of_unity:estimate} If $v \in S$
		and $y \in \rel^l$ then $0 \leq v(y) \leq 1$ and $\| Dv(y) \|
		\leq \Gamma/t$.
	\end{enumerate}
\end{lemma}
\begin{proof}
	Let $\Gamma = \sup \{ (129)^l, 40 V_1 \}$, where $V_1$ is obtained
	from \cite[3.1.13]{MR41:1976} with $m$ replaced by $l$.

	In order to verify that $\Gamma$ has the asserted property, suppose
	$D$ and $t$ are as in the body of the lemma. One may assume $t=1$.

	Define $\Phi = \{ \rel^l \without D \} \cup \{ \oball{d}{1} \with d
	\in D \}$ and $h : \rel^l \to \rel$ by
	\begin{gather*}
		h(y) = {\textstyle\frac{1}{20}} \sup \big \{ \inf \{ 1, \dist
		(y, \rel^l \without U ) \with U \in \Phi \} \big \} \quad
		\text{for $y \in \rel^l$}.
	\end{gather*}
	Observe that $h(y) \geq \frac{1}{40}$ for $y \in \rel^l$. Employing
	\cite[3.1.13]{MR41:1976}, one obtains a countable family $T$ contained
	in $\mathscr{D}^0 ( \rel^l )$ satisfying the following four
	conditions.
	\begin{enumerate}
		\item [\eqref{item:partition_of_unity:contained}$'$] If $v \in
		T$ then $\spt v \subset \cball{y}{10h(y)}$ for some $y \in
		\rel^l$.
		\item [\eqref{item:partition_of_unity:card}$'$] If $y \in
		\rel^l$ then $\card ( T \cap \{ v \with \cball{y}{10h(y)} \cap
		\spt v \neq \varnothing \} ) \leq (129)^l$.
		\item [\eqref{item:partition_of_unity:sum}$'$] If $y \in
		\rel^l$ then $\sum_{v \in T} v(y)= 1$.
		\item [\eqref{item:partition_of_unity:estimate}$'$] If $v \in
		T$ and $y \in \rel^l$ then $0 \leq v(y) \leq 1$ and $\| Dv(y)
		\| \leq V_1 \cdot h(y)^{-1}$.
	\end{enumerate}
	Now, one readily verifies that $S = T \cap \{ v \with D \cap \spt v
	\neq \varnothing \} )$ has the asserted property.
\end{proof}
\begin{theorem} \label{thm:sob_poincare_q_medians}
	Suppose $1 \leq M < \infty$.

	Then there exists a positive, finite number $\Gamma$ with the
	following property.

	If $l \in \nat$, $\vdim$, $\adim$, $p$, $U$, $V$, and $\psi$ are as in
	\ref{miniremark:situation_general}, $\sup \{ l, \adim \} \leq M$, $f
	\in \trunc (V,\rel^l)$, $1 \leq Q \leq M$, $P \in \nat$, $0 < r <
	\infty$,
	\begin{gather*}
		\| V \| ( U ) \leq ( Q-M^{-1} ) (P+1) \unitmeasure{\vdim}
		r^\vdim, \\
		\| V \| ( \{ z \with \density^\vdim ( \| V \|, z ) < Q \} )
		\leq \Gamma^{-1} r^\vdim,
	\end{gather*}	
	$A = \{ z \with \oball{z}{r} \subset U \}$, and $f_Y : \dmn f \to
	\rel$ is defined by
	\begin{gather*}
		f_Y (z) = \dist (f(z),Y) \quad \text{whenever $z \in \dmn f$,
		$\varnothing \neq Y \subset \rel^l$},
	\end{gather*}
	then the following four statements hold.
	\begin{enumerate}
		\item \label{item:sob_poincare_q_medians:p=1} If $p = 1$,
		$\beta = \infty$ if $\vdim = 1$ and $\beta = \vdim/(\vdim-1)$
		if $\vdim > 1$, then there exists a subset $Y$ of $\rel^l$
		such that $1 \leq \card Y \leq P$ and
		\begin{gather*}
			\eqLpnorm{\| V \| \restrict A}{\beta}{f_Y} \leq \Gamma
			P^{1/\beta} \big ( \Lpnorm{\| V
			\|}{1}{\derivative{V}{f}} + \| \delta V \|(f_Y) \big
			).
		\end{gather*}
		\item \label{item:sob_poincare_q_medians:p=m=1} If $p = \vdim
		= 1$ and $\psi (U) \leq \Gamma^{-1}$, then there exists a
		subset $Y$ of $\rel^l$ such that $1 \leq \card Y \leq P$ and
		\begin{gather*}
			\eqLpnorm{\| V \| \restrict A}{\infty}{f_Y} \leq
			\Gamma \, \Lpnorm{\| V \|}{1}{\derivative{V}{f}}.
		\end{gather*}
		\item \label{item:sob_poincare_q_medians:q<m=p} If $1 \leq q <
		\vdim = p$ and $\psi (U) \leq \Gamma^{-1}$, then there exists
		a subset $Y$ of $\rel^l$ such that $1 \leq \card Y \leq P$ and
		\begin{gather*}
			\eqLpnorm{\| V \| \restrict A}{\vdim q/(\vdim-q)}{f_Y}
			\leq \Gamma P^{1/q-1/\vdim} (\vdim-q)^{-1} \Lpnorm{\|
			V \|}{q}{\derivative{V}{f}}.
		\end{gather*}
		\item \label{item:sob_poincare_q_medians:p=m<q} If $1 < p =
		\vdim < q \leq \infty$ and $\psi ( U ) \leq \Gamma^{-1}$, then
		there exists a subset $Y$ of $\rel^l$ such that $1 \leq \card
		Y \leq P$ and
		\begin{gather*}
			\eqLpnorm{\| V \| \restrict A}{\infty}{f_Y} \leq
			\Gamma^{1/(1/\vdim-1/q)} r^{1-\vdim/q} \Lpnorm{\| V
			\|}{q}{\derivative{V}{f}}.
		\end{gather*}
	\end{enumerate}
\end{theorem}
\begin{proof}
	Define
	\begin{gather*}
		\Delta_1 = \sup \{ 1, \Gamma_{\ref{thm:rel_iso_ineq}} ( M )
		\}^M, \quad \Delta_2 = ( \sup \{ 1/2, 1-1/(2M^2-1) \} )^{1/M},
		\\
		\Delta_3 = \sup ( \{ 1 \} \cup \{
		\Gamma_{\ref{lemma:partition_of_unity}} (l) \with M \geq l \in
		\nat \} ), \quad \Delta_4 = \Delta_2^{-M}
		\Gamma_{\ref{thm:sob_poin_summary}} ( 2M), \\
		\Delta_5 = (1/2) \inf \{ \unitmeasure{\vdim} \with M \geq
		\vdim \in \nat \}, \\
		\Delta_6 = \sup \{ 1, \Gamma_{\ref{thm:sob_poin_summary}} (
		2M) \} ( \Delta_3^3 + 1 ) \Delta_5^{-1} (1-\Delta_2)^{-M}, \\
		\Delta_7 = \sup \{ \unitmeasure{\vdim} \with M \geq \vdim \in
		\nat \}, \quad \Delta_8 = 4 M \Delta_6 \Delta_7 +
		\Gamma_{\ref{thm:sob_poin_summary}} ( 2M ), \\
		\Delta_9 = M \Delta_7 \Gamma_{\ref{thm:sob_poin_summary}} ( 2M
		), \quad \Gamma = \sup \{ \Delta_1, \Delta_4, \Delta_8, 2
		\Delta_6 ( 1 + \Delta_9 ) \}
	\end{gather*}
	and notice that $\Delta_5 \leq 1 \leq \inf \{ \Delta_6, \Delta_7 \}$.

	In order to verify that $\Gamma$ has the asserted property, suppose
	$l$, $\vdim$, $\adim$, $p$, $U$, $V$, $\psi$, $f$, $Q$, $P$, $r$, $A$,
	and $f_Y$ are related to $\Gamma$ as in the body of the theorem.
	Abbreviate $\eta = \vdim q/(\vdim-q)$ if
	\eqref{item:sob_poincare_q_medians:q<m=p} holds and $\zeta =
	1/\vdim-1/q$ if \eqref{item:sob_poincare_q_medians:p=m<q} holds.
	Moreover, define
	\begin{align*}
		\gamma & = \Lpnorm{\| V \|}{1}{\derivative{V}{f}} && \quad
		\text{if \eqref{item:sob_poincare_q_medians:p=1} or
		\eqref{item:sob_poincare_q_medians:p=m=1} hold}, \\
		\gamma & = (\vdim-q)^{-1} \Lpnorm{\| V
		\|}{q}{\derivative{V}{f}} && \quad \text{if
		\eqref{item:sob_poincare_q_medians:q<m=p} holds}, \\
		\gamma & = \Delta_9^{1/\zeta} \Lpnorm{\| V
		\|}{q}{\derivative{V}{f}} && \quad \text{if
		\eqref{item:sob_poincare_q_medians:p=m<q} holds}
	\end{align*}
	and assume $r = 1$.
	
	First, the \emph{case $\gamma = 0$} will be considered.
	
	For this purpose choose $G$ and $\xi$ as in \ref{thm:zero_derivative}
	and let
	\begin{gather*}
		H = G \cap \{ W \with \| W \| ( U ) \leq ( Q-M^{-1} )
		\unitmeasure{\vdim} \}.
	\end{gather*}
	Applying \ref{corollary:rel_iso_ineq} and
	\ref{corollary:true_rel_iso_ineq} with $V$, $E$, and $B$ replaced by
	$W$, $U$, and $\Bdry U$, one infers
	\begin{align*}
		& \| W \| ( A )^{1/\beta} \leq \Delta_1 \, \| \delta W \| ( U
		) && \quad \text{if \eqref{item:sob_poincare_q_medians:p=1}
		holds, where $0^0=0$,} \\
		& \| W \| ( A ) = 0 && \quad \text{if
		\eqref{item:sob_poincare_q_medians:p=m=1} or
		\eqref{item:sob_poincare_q_medians:q<m=p} or
		\eqref{item:sob_poincare_q_medians:p=m<q} hold}
	\end{align*}
	whenever $W \in H$. Since $\card ( G \without H ) \leq P$, there
	exists $Y$ such that
	\begin{gather*}
		\xi \lIm G \without H \rIm \subset Y \subset \rel^l \quad
		\text{and} \quad 1 \leq \card Y \leq P.
	\end{gather*}
	If \eqref{item:sob_poincare_q_medians:p=1} holds, it follows that,
	using \ref{remark:decomp_rep},
	\begin{align*}
		& \eqLpnorm{\| V \| \restrict A}{\beta}{f_Y} \leq \tsum{W \in
		H}{} \dist (\xi(W),Y) \| W \| (A)^{1/\beta} \\
		& \qquad \leq \Delta_1 \tsum{W \in H}{} \dist (\xi(W),Y) \|
		\delta W \| ( U ) = \Delta_1 \| \delta V \|(f_Y).
	\end{align*}
	If \eqref{item:sob_poincare_q_medians:p=m=1} or
	\eqref{item:sob_poincare_q_medians:q<m=p} or
	\eqref{item:sob_poincare_q_medians:p=m<q} hold, the corresponding
	estimate is trivial.

	Second, the \emph{case $\gamma > 0$} will be considered.

	For this purpose assume $\gamma < \infty$ and define
	\begin{gather*}
		Z = \{ z \with \cball{z}{\Delta_2} \subset U \}, \quad t =
		\Delta_6 \gamma, \\
		T = \rel^l \cap \big \{ y \with \| V \| ( f^{-1} \lIm \oball
		yt \rIm ) > (Q-M^{-1}) \unitmeasure{\vdim} \big \}.
	\end{gather*}
	Choose $Y \subset \rel^l$ satisfying
	\begin{gather*}
		1 \leq \card Y \leq P \quad \text{and} \quad T \subset \{ y
		\with \dist (y,Y) < 2t \};
	\end{gather*}
	in fact, if $T \neq \varnothing$ then one may take $Y$ to be a maximal
	subset of $T$ with respect to inclusion such that $\{ \oball yt \with
	y \in Y \}$ is disjointed. Abbreviate $g = \dist ( \cdot, Y )$ and, if
	\eqref{item:sob_poincare_q_medians:p=1} holds, define
	\begin{gather*}
		\delta = \gamma + \| \delta V \| ( g \circ f )
	\end{gather*}
	and assume $\delta < \infty$. If
	\eqref{item:sob_poincare_q_medians:p=m=1} or
	\eqref{item:sob_poincare_q_medians:q<m=p} or
	\eqref{item:sob_poincare_q_medians:p=m<q} hold, define $\delta =
	\gamma$. Let
	\begin{gather*}
		D = \rel^l \cap \{ y \with g(y) \geq 2 \Delta_6 \delta \},
		\quad C = f^{-1} \lIm D \rIm.
	\end{gather*}

	Next, \emph{it will be shown that
	\begin{align*}
		& \| V \| ( Z \cap C ) \leq (1/2) \unitmeasure{\vdim} (
		1-\Delta_2)^\vdim && \quad \text{if
		\eqref{item:sob_poincare_q_medians:p=1} or
		\eqref{item:sob_poincare_q_medians:q<m=p} hold}, \\
		& \| V \| ( Z \cap C ) = 0 && \quad \text{if
		\eqref{item:sob_poincare_q_medians:p=m=1} or
		\eqref{item:sob_poincare_q_medians:p=m<q} hold}.
	\end{align*}}
	To prove this assertion, choose $S$ as in
	\ref{lemma:partition_of_unity} and let $C_v = U \cap \{ z \with
	v(f(z)) \neq 0 \}$ for $v \in S$. Since $D \cap T = \varnothing$ as
	$\Delta_6 \delta \geq t$, it follows that
	\begin{gather*}
		\| V \|( C_v ) \leq ( Q-M^{-1} ) \unitmeasure{\vdim} \leq \big
		(Q-(2M)^{-1} \big ) \unitmeasure{\vdim} \Delta_2^\vdim, \\
		\| V \| ( C_v \cap \{ z \with \density^\vdim ( \| V \| , z ) <
		Q \} ) \leq \Gamma^{-1} \leq \Delta_4^{-1} \leq
		\Gamma_{\ref{thm:sob_poin_summary}} ( 2M )^{-1}
		\Delta_2^\vdim, \\
		\psi ( C_v ) \leq \Gamma^{-1} \leq
		\Gamma_{\ref{thm:sob_poin_summary}} ( 2M )^{-1} \quad \text{if
		\eqref{item:sob_poincare_q_medians:p=m=1} or
		\eqref{item:sob_poincare_q_medians:q<m=p} or
		\eqref{item:sob_poincare_q_medians:p=m<q} hold}
	\end{gather*}
	whenever $v \in S$. Denoting by $c_v$ the characteristic function of
	$C_v$, applying \ref{lemma:comp_lip} with $A$ replaced by $\spt v$ and
	noting \ref{thm:addition}\,\eqref{item:addition:zero} yields $v
	\circ f \in \trunc(V)$ with
	\begin{gather*}
		| \derivative{V}{(v \circ f)} (z) | \leq \Delta_3 t^{-1}
		c_v(z) | \derivative{V}{f} (z) | \quad \text{for $\| V \|$
		almost all $z$}
	\end{gather*}
	whenever $v \in S$. Moreover, notice that
	\begin{gather*}
		\card ( S \cap \{ v \with z \in C_v \} ) \leq \Delta_3 \quad
		\text{for $z \in \dmn f$}, \\
		\tsum{v \in S}{} v(f(z)) = 1 \quad \text{for $z \in C$},
		\qquad \spt v \subset \{ y \with g(y) \geq \Delta_6 \delta \}
		\quad \text{for $v \in S$}.
	\end{gather*}
	If \eqref{item:sob_poincare_q_medians:p=1} holds, one estimates, using
	\ref{thm:sob_poin_summary}\,\eqref{item:sob_poin_summary:interior:p=1}
	with $M$, $G$, $f$, and $r$ replaced by $2M$, $\varnothing$, $v \circ
	f$, and $\Delta_2$ and \cite[2.4.18\,(1)]{MR41:1976},
	\begin{align*}
		& \| V \| ( Z \cap C )^{1/\beta} \leq \eqLpnorm{\| V \|
		\restrict Z}{\beta}{ \tsum{v \in S}{} v \circ f } \leq \tsum{v
		\in S}{} \eqLpnorm{\| V \| \restrict Z}{\beta}{ v \circ f} \\
		& \qquad \leq \Gamma_{\ref{thm:sob_poin_summary}} ( 2M )
		\tsum{v \in S}{} \big ( \Delta_3 t^{-1} \eqLpnorm{\| V \|
		\restrict C_v}{1}{ \derivative{V}{f} } + \| \delta V \| ( v
		\circ f ) \big ), \\
		& \qquad \leq \Gamma_{\ref{thm:sob_poin_summary}} ( 2M ) \big
		( \Delta_6^{-1} \Delta_3^2 + ( f_\# \| \delta V \| ) ( \{ y
		\with g (y) \geq \Delta_6 \delta \} ) \big ) \\
		& \qquad \leq \Gamma_{\ref{thm:sob_poin_summary}} ( 2M ) (
		\Delta_3^2 + 1 ) \Delta_6^{-1} \leq \Delta_5 ( 1- \Delta_2)^M
		\leq \big ((1/2) \unitmeasure{\vdim} (1-\Delta_2)^\vdim \big
		)^{1/\beta},
	\end{align*}
	where $0^0=0$. If \eqref{item:sob_poincare_q_medians:p=m=1} holds,
	using
	\ref{thm:sob_poin_summary}\,\eqref{item:sob_poin_summary:interior:p=m=1}
	with $M$, $G$, $f$, and $r$ replaced by $2M$, $\varnothing$, $v \circ
	f$, and $\Delta_2$, one estimates
	\begin{align*}
		& \Lpnorm{\| V \|}{\infty}{ \tsum{v \in S}{} v \circ f } \leq
		\tsum{v \in S}{} \eqLpnorm{\| V \| \restrict Z}{\infty}{v
		\circ f} \\
		& \qquad \leq \Gamma_{\ref{thm:sob_poin_summary}} ( 2M )
		\tsum{v \in S}{} \Delta_3 t^{-1} \eqLpnorm{ \| V \| \restrict
		C_v }{1}{ \derivative{V}{f} } \\
		& \qquad \leq \Gamma_{\ref{thm:sob_poin_summary}} ( 2M )
		\Delta_3^2 \Delta_6^{-1} \leq \Delta_5 ( 1-\Delta_2 )^M < 1,
	\end{align*}
	hence $\| V \| ( Z \cap C ) = 0$. If
	\eqref{item:sob_poincare_q_medians:q<m=p} holds, using
	\ref{thm:sob_poin_summary}\,\eqref{item:sob_poin_summary:interior:q<m=p}
	with $M$, $G$, $f$, and $r$ replaced by $2M$, $\varnothing$, $v \circ
	f$, and $\Delta_2$ and \ref{miniremark:lpnorms}, one estimates
	\begin{align*}
		& \| V \| ( Z \cap C )^{1/\eta} \leq \eqLpnorm{\| V \|
		\restrict Z}{\eta}{ \tsum{v \in S}{} v \circ f } \\
		& \qquad \leq \Delta_3 \big ( \tsum{v \in S}{} \eqLpnorm{\| V
		\| \restrict Z}{\eta}{ v \circ f }^q \big )^{1/q} \\
		& \qquad \leq \Gamma_{\ref{thm:sob_poin_summary}} ( 2M )
		(\vdim-q)^{-1} \Delta_3^2 t^{-1} \big ( \tsum{v \in S}{}
		\eqLpnorm{\| V \| \restrict C_v}{q}{ \derivative{V}{f} }^q
		\big )^{1/q} \\
		& \qquad \leq \Gamma_{\ref{thm:sob_poin_summary}} ( 2M )
		\Delta_3^3 \Delta_6^{-1} \leq \Delta_5 (1-\Delta_2)^M \leq
		\big ( (1/2) \unitmeasure{\vdim} ( 1-\Delta_2)^\vdim \big
		)^{1/\eta}.
	\end{align*}
	If \eqref{item:sob_poincare_q_medians:p=m<q} holds, using
	\ref{thm:sob_poin_summary}\,\eqref{item:sob_poin_summary:interior:p=m<q}
	with $M$, $G$, $f$, and $r$ replaced by $2M$, $\varnothing$, $v \circ
	f$, and $\Delta_2$ and noting $\Gamma_{\ref{thm:sob_poin_summary}} (
	2M)^{1/\zeta} (M\Delta_7)^\zeta \leq \Delta_9^{1/\zeta}$, one obtains
	\begin{gather*}
		\eqLpnorm{\| V \| \restrict Z}{\infty}{v \circ f} \leq
		\Delta_9^{1/\zeta} \Delta_3 t^{-1} \eqLpnorm{\| V \| \restrict
		C_v}{q}{\derivative{V}{f}}
	\end{gather*}
	and therefore $\| V \| (Z \cap C) = 0$, since if $q < \infty$ then,
	using \ref{miniremark:lpnorms},
	\begin{align*}
		& \eqLpnorm{\| V \| \restrict Z}{\infty}{ \tsum{v \in S}{} v
		\circ f } \leq \Delta_3 \big ( \tsum{v \in S}{} \eqLpnorm{\| V
		\| \restrict Z}{\infty}{v \circ f}^q \big )^{1/q} \\
		& \qquad \leq \Delta_9^{1/\zeta} \Delta_3^2 t^{-1} \big (
		\tsum{v \in S}{} \eqLpnorm{\| V \| \restrict
		C_v}{q}{\derivative{V}{f}}^q \big)^{1/q} \\
		& \qquad \leq \Delta_3^3 \Delta_6^{-1} \leq \Delta_5
		(1-\Delta_2)^M < 1,
	\end{align*}
	and if $q = \infty$ then $\eqLpnorm{ \| V \| \restrict Z}{\infty}{
	\sum_{v \in S} v \circ f} < 1$ follows similarly.

	If \eqref{item:sob_poincare_q_medians:p=m=1} or
	\eqref{item:sob_poincare_q_medians:p=m<q} hold, noting $2 \Delta_6
	\Delta_9^{1/\zeta} \leq \Gamma^{1/\zeta}$ if
	\eqref{item:sob_poincare_q_medians:p=m<q} holds, the conclusion is
	evident from the assertion of the preceding paragraph. Suppose now
	that \eqref{item:sob_poincare_q_medians:p=1} or
	\eqref{item:sob_poincare_q_medians:q<m=p} hold and define $h : \rel^l
	\to \rel$ by
	\begin{gather*}
		h(y) = \sup \{ 0, g(y)- 2 \Delta_6 \delta \} \quad \text{for
		$y \in \rel^l$}.
	\end{gather*}
	Notice that $h \circ f \in \trunc (V)$ with
	\begin{gather*}
		| \derivative{V}{(h \circ f)} (z) | \leq | \derivative{V}{f}
		(z) | \quad \text{for $\| V \|$ almost all $z$}
	\end{gather*}
	by \ref{lemma:comp_lip} with $A$ replaced by $\rel^l$. If
	\eqref{item:sob_poincare_q_medians:p=1} holds, applying
	\ref{thm:sob_poin_summary}\,\eqref{item:sob_poin_summary:interior:p=1}
	with $M$, $U$, $V$, $G$, $f$, $Q$, and $r$ replaced by $2M$, $Z$, $V |
	\mathbf{2}^{Z \times \grass{\adim}{\vdim}}$, $\varnothing$, $h \circ
	f$, $1$, and $1-\Delta_2$ implies
	\begin{gather*}
		\eqLpnorm{\| V \| \restrict A}{\beta}{ h \circ f} \leq
		\Gamma_{\ref{thm:sob_poin_summary}} ( 2M ) \big ( \Lpnorm{\| V
		\|}{1}{ \derivative{V}{f} } + \| \delta V \| ( h \circ f )
		\big ), \\
		\eqLpnorm{\| V \| \restrict A}{\beta}{g \circ f} \leq 4 M
		\Delta_6 \Delta_7 P^{1/\beta} \delta +
		\Gamma_{\ref{thm:sob_poin_summary}} ( 2M) \delta \leq \Gamma
		P^{1/\beta} \delta,
	\end{gather*}
	since $\| V \| ( A )^{1/\beta} \leq 2 M \Delta_7 P^{1/\beta}$. If
	\eqref{item:sob_poincare_q_medians:q<m=p} holds, applying
	\ref{thm:sob_poin_summary}\,\eqref{item:sob_poin_summary:interior:q<m=p}
	with $M$, $U$, $V$, $G$, $f$, $Q$, and $r$ replaced by $2M$, $Z$, $V |
	\mathbf{2}^{Z \times \grass{\adim}{\vdim}}$, $\varnothing$, $h \circ
	f$, $1$, and $1-\Delta_2$ implies
	\begin{gather*}
		\eqLpnorm{\| V \| \restrict A}{\eta}{ h \circ f} \leq
		\Gamma_{\ref{thm:sob_poin_summary}} ( 2M ) \gamma, \\
		\eqLpnorm{\| V \| \restrict A}{\eta}{ g \circ f} \leq 4 M
		\Delta_6 \Delta_7 P^{1/\eta} \gamma +
		\Gamma_{\ref{thm:sob_poin_summary}} ( 2M ) \gamma \leq \Gamma
		P^{1/\eta} \gamma
	\end{gather*}
	since $\| V \| (A)^{1/\eta} \leq 2M \Delta_7 P^{1/\eta}$.
\end{proof}
\begin{theorem} \label{thm:sob_poin_several_med}
	Suppose $\vdim$, $\adim$, $p$, $U$, $V$, and $\psi$ are as in
	\ref{miniremark:situation_general}, $\adim \leq M < \infty$, $P \in
	\nat$, $1 \leq Q \leq M$, $f \in \trunc (V)$, $0 < r < \infty$, $Z$ is
	a finite subset of $U$,
	\begin{gather*}
		\| V \| ( U ) \leq ( Q-M^{-1} ) (P+1) \unitmeasure{\vdim}
		r^\vdim, \\
		\| V \| ( \{ z \with \density^\vdim ( \| V \|, z ) < Q \} )
		\leq \Gamma^{-1} r^\vdim,
	\end{gather*}
	and $A = \{ z \with \oball{z}{r} \subset U \}$, then there exists a
	subset $Y$ of $\rel$ with $1 \leq \card Y \leq P + \card Z$ such that
	the following four statements hold with $g = \dist ( \cdot, Y ) \circ
	f$:
	\begin{enumerate}
		\item \label{item:sob_poin_several_med:p=1} If $p = 1$, $\beta
		= \infty$ if $\vdim = 1$ and $\beta = \vdim/(\vdim-1)$ if
		$\vdim > 1$, then
		\begin{gather*}
			\eqLpnorm{\| V \| \restrict A}{\beta}{g} \leq
			\Gamma_{\ref{thm:sob_poin_summary}} (M) \big (
			\Lpnorm{\| V \|}{1}{ \derivative{V}{f} } + \| \delta V
			\|(g) \big ).
		\end{gather*}
		\item \label{item:sob_poin_several_med:p=m=1} If $p = \vdim =
		1$ and $\psi ( U \without Z ) \leq
		\Gamma_{\ref{thm:sob_poin_summary}} (M)^{-1}$, then
		\begin{gather*}
			\eqLpnorm{\| V \| \restrict A}{\infty}{g} \leq
			\Gamma_{\ref{thm:sob_poin_summary}} (M) \, \Lpnorm{\|
			V \|}{1}{ \derivative{V}{f} }.
		\end{gather*}
		\item \label{item:sob_poin_several_med:q<m=p} If $1 \leq q <
		\vdim = p$ and $\psi ( U ) \leq
		\Gamma_{\ref{thm:sob_poin_summary}} (M)^{-1}$, then
		\begin{gather*}
			\eqLpnorm{\| V \| \restrict A}{\vdim q/(\vdim-q)}{g}
			\leq \Gamma_{\ref{thm:sob_poin_summary}} (M)
			(\vdim-q)^{-1} \Lpnorm{\| V \|}{q}{ \derivative{V}{f}
			}.
		\end{gather*}
		\item \label{item:sob_poin_several_med:p=m<q} If $1 < p =
		\vdim < q \leq \infty$ and $\psi ( U ) \leq
		\Gamma_{\ref{thm:sob_poin_summary}} (M)^{-1}$, then
		\begin{gather*}
			\eqLpnorm{\| V \| \restrict A}{\infty}{g} \leq
			\Gamma^{1/(1/\vdim-1/q)} Q^{1/\vdim-1/q} r^{1-\vdim/q}
			\Lpnorm{\| V \|}{q}{ \derivative{V}{f} },
		\end{gather*}
		where $\Gamma = \Gamma_{\ref{thm:sob_poin_summary}} (M) \sup
		\{ \unitmeasure{k} \with M \geq k \in \nat \}$.
	\end{enumerate}
\end{theorem}
\begin{proof}
	Choose a nonempty subset $Y$ of $\rel$ satisfying
	\begin{gather*}
		\card Y \leq P + \card Z, \quad f \lIm Z \rIm \subset Y, \\
		\| V \| ( U \cap \{ z \with f(z) \in I \} ) \leq (Q-M^{-1})
		\unitmeasure{\vdim} r^\vdim \quad \text{for $I \in F$},
	\end{gather*}
	where $F$ is the family of connected components of $\rel \without Y$.
	Notice that
	\begin{gather*}
		\dist (b,Y) = \dist (b, \rel\without I) \quad \text{and} \quad
		\dist (b,\rel \without J) = 0
	\end{gather*}
	whenever $b \in I \in F$ and $I \neq J \in F$, and
	\begin{gather*}
		\dist (b,Y) = \dist (b,\rel \without I) = 0 \quad
		\text{whenever $b \in Y$ and $I \in F$}.
	\end{gather*}
	Defining $f_I = \dist ( \rel \without I ) \circ f$, one infers $f_I
	\in \trunc (V)$ and
	\begin{gather*}
		| \derivative{V}{f} (z) | = | \derivative{V}{f_I}(z) | \quad
		\text{for $\| V \|$ almost all $z \in f^{-1} \lIm I \rIm$}, \\
		\derivative{V}{f} (z) = 0 \quad \text{for $\| V \|$ almost all
		$z \in f^{-1} \lIm Y \rIm$}, \\
		\derivative{V}{f_I} (z) = 0 \quad \text{for $\| V \|$ almost
		all $z \in f^{-1} \lIm \rel \without I \rIm$}
	\end{gather*}
	whenever $I \in F$ by \ref{lemma:basic_v_weakly_diff},
	\ref{example:composite} and
	\ref{thm:addition}\,\eqref{item:addition:zero}.

	The conclusion now will be obtained by applying
	\ref{thm:sob_poin_summary}\,\eqref{item:sob_poin_summary:interior}
	with $l$, $G$, and $f$ replaced by $1$, $\varnothing$, and $f_I$ for
	$I \in F$ and employing \ref{miniremark:lpnorms}. For instance, in
	case of \eqref{item:sob_poin_several_med:q<m=p} one estimates
	\begin{gather*}
		\begin{aligned}
			\eqLpnorm{\| V \| \restrict A}{\vdim q/(\vdim-q)}{g} &
			\leq \big ( \tsum{I \in F}{} \eqLpnorm{\| V \|
			\restrict A}{q}{ f_I }^q \big )^{1/q} \\
			& \leq \gamma \big ( \tsum{I \in F}{} \Lpnorm{\| V
			\|}{q}{ \derivative{V}{f_I} }^q \big )^{1/q} = \gamma
			\, \Lpnorm{\| V \|}{q}{ \derivative{V}{f} },
		\end{aligned}
	\end{gather*}
	where $\gamma = \Gamma_{\ref{thm:sob_poin_summary}} (M)
	(\vdim-q)^{-1}$. The cases \eqref{item:sob_poin_several_med:p=1} and
	\eqref{item:sob_poin_several_med:p=m=1} follow similarly. Defining
	$\Delta = \sup \{ \unitmeasure{k} \with M \geq k \in \nat \}$ and
	noting $\Delta \geq 1$, hence $\unitmeasure{\vdim}^{(1/\vdim-1/q)^2}
	\leq \Delta^{(1/\vdim-1/q)^2} \leq \Delta$, the same holds for
	\eqref{item:sob_poin_several_med:p=m<q}.
\end{proof}
\begin{remark}
	The method of deduction of \ref{thm:sob_poin_several_med} from
	\ref{thm:sob_poin_summary} is that of Hutchinson \cite[Theorem
	3]{MR1066398} which is derived from Hutchinson \cite[Theorem
	1]{MR1066398}.
\end{remark}
\begin{remark}
	A nonempty choice of $Z$ will occur in \ref{thm:mod_continuity}.
\end{remark}
\section{Differentiability properties}
In this section, approximate differentiability, see \ref{thm:approx_diff}, and
differentiability in Lebesgue spaces, see \ref{thm:diff_lebesgue_spaces} are
established for weakly differentiable functions. The primary ingredient is the
Sobolev Poincar{\'e} type inequality
\ref{thm:sob_poin_summary}\,\eqref{item:sob_poin_summary:interior}.
\begin{miniremark} \label{miniremark:convention}
	Following \cite{MR41:1976}, the convention that $a = b$ is true if
	and only if either both expressions ``$a$'' and ``$b$'' are defined
	and equal or both expressions are undefined will be employed.
\end{miniremark}
\begin{lemma} \label{lemma:approx_diff}
	Suppose $l, \vdim, \adim \in \nat$, $\vdim \leq \adim$, $U$ is an open
	subset of $U$, $V \in \RVar_\vdim (U)$, $f$ is a $\| V \|$ measurable
	$\rel^l$ valued function, and $A$ is the set of points at which $f$ is
	$( \| V \|, \vdim )$ approximately differentiable.

	Then the following three statements hold.
	\begin{enumerate}
		\item \label{item:approx_diff:measurable} The set $A$ is $\| V
		\|$ measurable and $( \| V \|, \vdim ) \ap Df(z) \circ
		\project{\Tan^\vdim ( \| V \|, z )}$ depends $\| V \|
		\restrict A$ measurably on $z$.
		\item \label{item:approx_diff:cover} There exists a sequence
		of functions $f_i : U \to \rel^l$ of class $1$ such that
		\begin{gather*}
			\| V \| ( A \without \{ z \with \text{$f(z) = f_i(z)$
			for some $i$} \} ) = 0.
		\end{gather*}
		\item \label{item:approx_diff:unrectifiable} If $g : U \to
		\rel^l$ is locally Lipschitzian, then
		\begin{gather*}
			\| V \| ( U \cap \{ z \with f(z) = g(z) \} \without A
			) = 0.
		\end{gather*}
		\item \label{item:approx_diff:comp} If $g$ is a $\| V \|$
		measurable $\rel^l$ valued function and $B = U \cap \{ z \with
		f(z) = g(z) \}$, then (see \ref{miniremark:convention})
		\begin{gather*}
			( \| V \|, \vdim ) \ap Df(z) = ( \| V \|, \vdim ) \ap
			Dg(z) \quad \text{for $\| V \|$ almost all $z \in
			B$}.
		\end{gather*}
	\end{enumerate}
\end{lemma}
\begin{proof}
	\eqref{item:approx_diff:measurable} is a consequence of
	\cite[4.5\,(1)]{snulmenn.decay}.

	In order to prove \eqref{item:approx_diff:cover}, one may reduce the
	problem. Firstly, to the case that $\| V \| ( U \without M ) = 0$ for
	some $\vdim$ dimensional submanifold of $U$ of class $1$ by
	\cite[2.10.19\,(4), 3.2.29]{MR41:1976}. Secondly, to the case that for
	some $1 < \tau < \infty$ the varifold satisfies additionally that
	$\tau^{-1} \leq \density^\vdim ( \| V \|, z ) \leq \tau$ for $\| V \|$
	almost all $z$ by \cite[2.10.19\,(4)]{MR41:1976}, hence thirdly to the
	case that $\density^\vdim ( \| V \|, z ) = 1$ for $\| V \|$ almost all
	$z$ by \cite[2.10.19\,(1)\,(3)]{MR41:1976}. Finally, to the case $\| V
	\| = \mathscr{L}^\vdim$ by \cite[3.1.19\,(4), 3.2.3, 2.8.18,
	2.9.11]{MR41:1976} which may be treated by means of
	\cite[3.1.16]{MR41:1976}.

	\eqref{item:approx_diff:comp} is a consequence of
	\cite[2.10.19\,(4)]{MR41:1976} and implies
	\eqref{item:approx_diff:unrectifiable} by
	\cite[4.5\,(2)]{snulmenn.decay}.
\end{proof}
\begin{theorem} \label{thm:approx_diff}
	Suppose $l \in \nat$, $\vdim$, $\adim$, $U$, and $V$ are as in
	\ref{miniremark:situation_general}, and $f \in \trunc (V,\rel^l)$.

	Then $f$ is $( \| V \|, \vdim )$ approximately differentiable with
	\begin{gather*}
		\derivative{V}{f} (a) = ( \| V \|, \vdim ) \ap Df(a) \circ
		\project{\Tan^\vdim ( \| V \|, a )}
	\end{gather*}
	at $\| V \|$ almost all $a$.
\end{theorem}
\pagebreak
\begin{proof}
	In view of \ref{lemma:approx_diff}\,\eqref{item:approx_diff:comp}, one
	employs \ref{miniremark:trunc} and \ref{lemma:basic_v_weakly_diff} to
	reduce the problem to the case that $f$ is bounded and
	$\derivative{V}{f} \in \Lploc{1} ( \| V \|, \Hom ( \rel^\adim,
	\rel^l))$. In particular, one may assume $l = 1$ by
	\ref{remark:integration_by_parts}.  Define $\beta = \infty$ if $\vdim
	= 1$ and $\beta = \vdim/(\vdim-1)$ if $\vdim > 1$.

	First, it will be shown that
	\begin{gather*}
		\limsup_{s \to 0+} s^{-1} \eqLpnorm{s^{-\vdim} \| V \|
		\restrict \cball{a}{s}}{\beta}{f(\cdot)-f(a)} < \infty \quad
		\text{for $\| V \|$ almost all $a$}.
	\end{gather*}
	For this purpose define $C = \{ (a, \cball{a}{r} ) \with \cball{a}{r}
	\subset U \}$ and consider a point $a$ satisfying for some $M$
	the conditions
	\begin{gather*}
		\sup \{ 4, \adim \} \leq M < \infty, \quad 1 \leq
		\density^\vdim ( \| V \|, a ) \leq M, \quad \eqLpnorm{\| V \|
		+ \| \delta V \|}{\infty}{f} \leq M, \\
		\limsup_{s \to 0+} s^{-\vdim} \big ( \tint{\cball{a}{s}}{} |
		\derivative{V}{f} | \ud \| V \| + \measureball{\| \delta V
		\|}{ \cball{a}{s} } \big ) < M, \\
		\text{$\density^\vdim ( \| V \|, \cdot )$ and $f$ are $(\| V
		\|, C )$ approximately continuous at $a$}
	\end{gather*}
	which are met by $\| V \|$ almost all $a$ by \cite[2.10.19\,(1)\,(3),
	2.8.18, 2.9.13]{MR41:1976}. Define $\Delta =
	\Gamma_{\ref{thm:sob_poin_summary}} (M)$. Choose $1 \leq Q \leq M$ and
	$1 < \tau \leq 2$ subject to the requirements $\density^\vdim ( \| V
	\|, a ) < 2 \tau^{-\vdim} ( Q- 1/4)$ and
	\begin{gather*}
		\text{either $Q = \density^\vdim ( \| V \|, a ) = 1$} \quad
		\text{or $Q < \density^\vdim ( \| V \|, a )$}.
	\end{gather*}
	Then pick $0 < r < \infty$ such that $\cball{a}{\tau r} \subset U$
	and
	\begin{gather*}
		\measureball{\| V \|}{\cball{a}{s}} \geq (1/2)
		\unitmeasure{\vdim} s^\vdim, \quad \measureball{\| V
		\|}{\oball{a}{\tau s}} \leq 2 (Q-M^{-1}) \unitmeasure{\vdim}
		s^\vdim, \\
		\| V \| ( \classification{\oball{a}{\tau s}}{z}{
		\density^\vdim ( \| V \|, z ) < Q } ) \leq \Delta^{-1}
		s^\vdim, \\
		\tint{\cball{a}{\tau s}}{} | \derivative{V}{f} | \ud \| V \| +
		\measureball{\| \delta V \|}{ \cball{a}{\tau s} } \leq M
		\tau^\vdim s^\vdim
	\end{gather*}
	for $0 < s \leq r$. Choose $y (s) \in \rel$ such that
	\begin{gather*}
		\| V \| ( \classification{\oball{a}{\tau s}}{z}{ f ( z ) < y
		(s) } ) \leq (1/2) \measureball{\| V \|}{ \oball{a}{\tau s} },
		\\
		\| V \| ( \classification{\oball{a}{\tau s}}{z}{ f ( z ) > y
		(s) } ) \leq (1/2) \measureball{\| V \|}{ \oball{a}{\tau s} }
	\end{gather*}
	for $0 < s \leq r$, in particular
	\begin{gather*}
		| y (s) | \leq M \quad \text{and} \quad f(a) = \lim_{s \to 0+}
		y(s).
	\end{gather*}
	Define $f_s ( z ) = f(z)-y(s)$ whenever $0 < s \leq r$ and $z \in
	\dmn f$. Recalling \ref{lemma:basic_v_weakly_diff} and
	\ref{example:composite}\,\eqref{item:composite:1d}, one applies
	\ref{thm:sob_poin_summary}\,\eqref{item:sob_poin_summary:interior:p=1}
	with $U$, $G$, $f$, and $r$ replaced by $\oball{a}{\tau s}$,
	$\varnothing$, $f_s^+$ respectively $f_s^-$, and $s$ to infer
	\begin{gather*}
		\begin{aligned}
			\eqLpnorm{\| V \| \restrict \cball{a}{(\tau-1)
			s}}{\beta}{f_s} & \leq \Delta \tint{\cball{a}{\tau
			s}}{} | \derivative V{f_s} | \ud \| V \| +
			\tint{\cball{a}{\tau s}}{} |f_s| \ud \| \delta V \| \\
			& \leq 2 \Delta M \tau^\vdim s^\vdim
		\end{aligned}
	\end{gather*}
	for $0 < s \leq r$. Abbreviating $\gamma = 2 M \Delta \tau^\vdim$ and
	noting that
	\begin{gather*}
		\begin{aligned}
			& | y(s)-y(s/2) | \cdot \| V \| ( \cball{a}{(\tau-1)
			s/2})^{1/\beta} \\
			& \leq \eqLpnorm{\| V \| \restrict
			\cball{a}{(\tau-1) s/2}}{\beta}{f_{s/2}} +
			\eqLpnorm{\| V \| \restrict \cball{a}{(\tau-1)
			s}}{\beta}{f_s} \leq 2 \gamma s^\vdim,
		\end{aligned} \\
		| y(s)-y(s/2) | \leq 2^{\vdim+1} \gamma
		\unitmeasure{\vdim}^{-1/\beta} ( \tau-1 )^{1-\vdim} s,
	\end{gather*}
	one obtains for $0 < s \leq r$ that
	\begin{gather*}
		|y(s)-f(a)| \leq 2^{\vdim+2} \gamma
		\unitmeasure{\vdim}^{-1/\beta} (\tau-1)^{1-\vdim} s, \\
		\eqLpnorm{\| V \| \restrict \cball{a}{(\tau-1)s}
		}{\beta}{f(\cdot)-f(a)} \leq \gamma \big ( 1 + 2^{\vdim+3} M (
		\tau-1 )^{1-\vdim} \big ) s^\vdim.
	\end{gather*}

	Combining the assertion of the preceding paragraph with
	\cite[3.7\,(i)]{snulmenn.isoperimetric} applied with $\alpha$, $q$,
	and $r$ replaced by $1$, $1$, and $\infty$, one obtains a sequence of
	locally Lipschitzian functions $f_i : U \to \rel$ such that
	\begin{gather*}
		{\textstyle \| V \| \big ( U \without \bigcup \{ B_i \with i
		\in \nat \} \big ) = 0, \quad \text{where $B_i = U \cap \{ z
		\with f(z)=f_i(z) \}$}}.
	\end{gather*}
	Since $f-f_i \in \trunc (V)$ with
	\begin{gather*}
		\derivative Vf (a) - \derivative{V}{f_i} (a) =
		\derivative{V}{(f-f_i)} (a) = 0 \quad \text{for $\| V \|$
		almost all $a \in B_i$}
	\end{gather*}
	by
	\ref{thm:addition}\,\eqref{item:addition:zero}\,\eqref{item:addition:add},
	the conclusion follows from of \ref{example:lipschitzian} and
	\ref{lemma:approx_diff}\,\eqref{item:approx_diff:comp}.
\end{proof}
\begin{miniremark} \label{miniremark:tangent_spaces}
	\emph{If $\vdim$, $\adim$, $U$, $V$, $\psi$, and $p$ are as in
	\ref{miniremark:situation_general}, $p = \vdim$, $a \in U$, and $\psi
	( \{ a \} ) < \isoperimetric{\vdim}^{-1}$, then
	\begin{gather*}
		\Tan^\vdim ( \| V \|, a ) = \Tan ( \spt \| V \|, a );
	\end{gather*}}
	in fact, this is a consequence of \cite[2.5]{snulmenn.isoperimetric},
	compare \cite[Lemma 17.11]{MR756417}.
\end{miniremark}
\begin{theorem} \label{thm:diff_lebesgue_spaces}
	Suppose $l \in \nat$, $\vdim$, $\adim$, $p$, $U$, and $V$ are as in
	\ref{miniremark:situation_general}, $1 \leq q \leq \infty$, $f \in
	\trunc (V,\rel^l)$, $\derivative Vf \in \Lploc{q} ( \| V \|, \Hom
	(\rel^\adim, \rel^l) )$,
	\begin{gather*}
		C = \{ ( a, \cball{a}{r} ) \with \cball{a}{r} \subset U \},
	\end{gather*}
	and $S$ is the set of points in $\spt \| V \|$ at which $f$ is $( \| V
	\|,C )$ approximately continuous.

	Then the following four statements hold.
	\begin{enumerate}
		\item \label{item:diff_lebesgue_spaces:m>1=p} If $\vdim > 1$,
		$\beta = \vdim/(\vdim-1)$, and $f \in \Lploc{1} ( \| \delta V
		\|, \rel^l )$, then
		\begin{gather*}
			\lim_{r \to 0+} r^{-\vdim} \tint{\cball ar}{} ( |
			f(z)-f(a)- \left <z-a, \derivative Vf (a) \right > | /
			| z-a| )^\beta \ud \| V \| z = 0
		\end{gather*}
		for $\| V \|$ almost all $a$.
		\item \label{item:diff_lebesgue_spaces:m=p=1} If $\vdim = 1$,
		then $f|S$ is differentiable relative to $S$ at $a$ with
		\begin{gather*}
			D (f|S) (a) = \derivative Vf (a) | \Tan^\vdim ( \| V
			\|, a) \quad \text{for $\| V \|$ almost all $a$}.
		\end{gather*}
		\item \label{item:diff_lebesgue_spaces:q<m=p} If $q < \vdim =
		p$ and $\eta = \vdim q/(\vdim-q)$, then
		\begin{gather*}
			\lim_{r \to 0+} r^{-\vdim} \tint{\cball ar}{} ( |
			f(z)-f(a)- \left <z-a, \derivative Vf (a) \right > | /
			| z-a| )^\eta \ud \| V \| z = 0
		\end{gather*}
		for $\| V \|$ almost all $a$.
		\item \label{item:diff_lebesgue_spaces:p=m<q} If $p = \vdim <
		q$, then $f|S$ is differentiable relative to $S$ at $a$ with
		\begin{gather*}
			D (f|S) (a) = \derivative Vf (a) | \Tan^\vdim ( \| V
			\|, a) \quad \text{for $\| V \|$ almost all $a$}.
		\end{gather*}
	\end{enumerate}
\end{theorem}
\begin{proof}
	Assume $p = q = 1$ in case of \eqref{item:diff_lebesgue_spaces:m>1=p}
	or \eqref{item:diff_lebesgue_spaces:m=p=1}. In case of
	\eqref{item:diff_lebesgue_spaces:p=m<q} also assume $p > 1$ and $q <
	\infty$. Define $\zeta = \beta$ in case of
	\eqref{item:diff_lebesgue_spaces:m>1=p}, $\zeta = \eta$ in case of
	\eqref{item:diff_lebesgue_spaces:q<m=p}, and $\zeta = \infty$ in case
	of \eqref{item:diff_lebesgue_spaces:m=p=1} or
	\eqref{item:diff_lebesgue_spaces:p=m<q}. Moreover, let $h_a : \dmn f
	\to \rel^l$ be defined by
	\begin{gather*}
		h_a(z) = f(z)-f(a) - \left <z-a, \derivative Vf (a) \right > 
	\end{gather*}
	whenever $a \in \dmn f \cap \dmn \derivative Vf$ and $z \in \dmn f$.
	
	The following assertion will be shown. \emph{There holds
	\begin{gather*}
		\lim_{s \to 0+} s^{-1} \eqLpnorm{s^{-\vdim} \| V \| \restrict
		\cball{a}{s}}{\zeta}{h_a} = 0 \quad \text{for $\| V \|$ almost
		all $a$}.
	\end{gather*}}
	 In the special case that $f$ is of class
	$1$, in view of \ref{example:lipschitzian}, it is sufficient to
	prove that
	\begin{gather*}
		\lim_{s \to 0+} s^{-1} \eqLpnorm{s^{-\vdim} \| V \| \restrict
		\cball{a}{s}}{\zeta}{\project{\Nor^\vdim ( \| V \|, a )} (
		\cdot - a )} = 0
	\end{gather*}
	for $\| V \|$ almost all $a$; in fact, if $\varepsilon > 0$,
	$\Tan^\vdim ( \| V \|, a ) \in \grass{\adim}{\vdim}$ and $\psi ( \{ a
	\} ) < \isoperimetric{\vdim}^{-1}$, then
	\begin{gather*}
		\density^\vdim ( \| V \| \restrict U \cap \{ z \with |
		\project{\Nor^\vdim ( \| V \|, a )} (z-a) | > \varepsilon
		|z-a| \}, a ) = 0
	\end{gather*}
	in case of \eqref{item:diff_lebesgue_spaces:m>1=p} by
	\cite[3.2.16]{MR41:1976} and
	\begin{gather*}
		\cball{a}{s} \cap \spt \| V \| \subset \{ z \with |
		\project{\Nor^\vdim ( \| V \|, a )} (z-a) | \leq \varepsilon
		|z-a| \} \quad \text{for some $s>0$}
	\end{gather*}
	in case of \eqref{item:diff_lebesgue_spaces:m=p=1} or
	\eqref{item:diff_lebesgue_spaces:q<m=p} or
	\eqref{item:diff_lebesgue_spaces:p=m<q} by
	\ref{miniremark:tangent_spaces} and \cite[3.1.21]{MR41:1976}. To treat
	the general case, one obtains a sequence of functions $f_i : U \to
	\rel^l$ of class $1$ such that
	\begin{gather*}
		{\textstyle \| V \| \big ( U \without \bigcup \{ B_i \with i
		\in \nat \} \big ) = 0, \quad \text{where $B_i = U \cap \{ z
		\with f(z)=f_i(z) \}$}}
	\end{gather*}
	from \ref{thm:approx_diff} and
	\ref{lemma:approx_diff}\,\eqref{item:approx_diff:cover}. Define $g_i =
	f - f_i$ and notice that $g_i \in \trunc ( V, \rel^l )$ and
	\begin{gather*}
		\derivative{V}{g_i} (a) = \derivative{V}{f} (a) -
		\derivative{V}{f_i} (a) \quad \text{for $\| V \|$ almost all
		$a$}
	\end{gather*}
	by \ref{thm:addition}\,\eqref{item:addition:add}. Define Radon
	measures $\mu_i$ over $U$ by
	\begin{gather*}
		\mu_i ( A ) = \tint{A}{\ast} | \derivative{V}{g_i} | \ud \| V
		\| + \tint{A}{\ast} |g_i| \ud \| \delta V \| \quad \text{in
		case of \eqref{item:diff_lebesgue_spaces:m>1=p}}, \\
		\mu_i ( A ) = \tint{A}{\ast} | \derivative{V}{g_i} |^q \ud \|
		V \| \quad \text{in case of
		\eqref{item:diff_lebesgue_spaces:m=p=1} or
		\eqref{item:diff_lebesgue_spaces:q<m=p} or
		\eqref{item:diff_lebesgue_spaces:p=m<q}}
	\end{gather*}
	whenever $A \subset U$ and $i \in \nat$ and notice that $\mu_i ( B_i )
	= 0$ by \ref{thm:addition}\,\eqref{item:addition:zero} (or
	alternately by \ref{thm:approx_diff} and
	\ref{lemma:approx_diff}\,\eqref{item:approx_diff:comp}). The assertion
	will be shown to hold at a point $a$ satisfying for some $i$ that
	\begin{gather*}
		f(a) = f_i (a), \quad \derivative Vf (a) = \derivative V{f_i}
		(a), \\
		\lim_{s \to 0+} s^{-1} \eqLpnorm{s^{-\vdim} \| V \| \restrict
		\cball as }{\zeta}{f_i(\cdot)-f_i(a) - \left < \cdot - a ,
		\derivative V{f_i} (a) \right >} = 0, \\
		\density^\vdim ( \| V \| \restrict
		U \without B_i, a ) = 0, \quad \density^\vdim ( \mu_i, a ) = 0
	\end{gather*}
	and in case of \eqref{item:diff_lebesgue_spaces:m=p=1} also $\psi ( \{
	a \} ) = 0$. These conditions are met by $\| V \|$ almost all $a$ in
	view of the special case, \cite[2.10.19\,(4)]{MR41:1976} and Allard
	\cite[3.5\,(1b)]{MR0307015}. Choosing $0 < r < \infty$ with
	$\cball{a}{2r} \subset U$ and
	\begin{gather*}
		\| V \| ( \oball{a}{2s} \without B_i ) \leq (1/2)
		\unitmeasure{\vdim} s^\vdim, \\
		\measureball{\psi}{\oball{a}{2r}} \leq
		\Gamma_{\ref{thm:sob_poin_summary}} ( 2\adim )^{-1} \quad
		\text{in case of \eqref{item:diff_lebesgue_spaces:m=p=1} or
		\eqref{item:diff_lebesgue_spaces:q<m=p} or
		\eqref{item:diff_lebesgue_spaces:p=m<q}}
	\end{gather*}
	for $0 < s \leq r$, one infers from
	\ref{thm:sob_poin_summary}\,\eqref{item:sob_poin_summary:interior}
	with $U$, $M$, $G$, $f$, $Q$, and $r$ replaced by $\oball a{2s}$,
	$2\adim$, $\varnothing$, $g_i$, $1$, and $s$ that
	\begin{gather*}
		\eqLpnorm{\| V \| \restrict \cball{a}{s} }{\zeta}{g_i} \leq
		\Delta \mu_i ( \cball{a}{2s} )^{1/q} \quad \text{in case of
		\eqref{item:diff_lebesgue_spaces:m>1=p} or
		\eqref{item:diff_lebesgue_spaces:m=p=1} or
		\eqref{item:diff_lebesgue_spaces:q<m=p}}, \\
		\eqLpnorm{\| V \| \restrict \cball{a}{s} }{\zeta}{g_i} \leq
		\Delta s^{1-\vdim/q} \mu_i ( \cball{a}{2s} )^{1/q} \quad
		\text{in case of \eqref{item:diff_lebesgue_spaces:p=m<q}}
	\end{gather*}
	for $0 < s \leq r$, where $\Delta =
	\Gamma_{\ref{thm:sob_poin_summary}} ( 2 \adim )$ in case of
	\eqref{item:diff_lebesgue_spaces:m>1=p} or
	\eqref{item:diff_lebesgue_spaces:m=p=1}, $\Delta = (\vdim-q)^{-1}
	\Gamma_{\ref{thm:sob_poin_summary}} ( 2 \adim )$ in case of
	\eqref{item:diff_lebesgue_spaces:q<m=p}, and $\Delta =
	\Gamma_{\ref{thm:sob_poin_summary}} (2 \adim )^{1/(1/\vdim-1/q)}
	\unitmeasure{\vdim}^{1/\vdim-1/q}$ in case of
	\eqref{item:diff_lebesgue_spaces:p=m<q}. Consequently,
	\begin{gather*}
		\lim_{s \to 0+} s^{-1} \eqLpnorm { s^{-\vdim} \| V \|
		\restrict \cball{a}{s} }{\zeta} {g_i} = 0
	\end{gather*}
	and the assertion follows.

	From the assertion of the preceding paragraph one obtains
	\begin{gather*}
		\lim_{s \to 0+} \eqLpnorm{s^{-\vdim} \| V \| \restrict
		\cball{a}{s} }{\zeta}{ | \cdot - a |^{-1} h_a ( \cdot ) } = 0
		\quad \text{for $\| V \|$ almost all $a$};
	\end{gather*}
	in fact, if $\cball{a}{r} \subset U$ and $\gamma = \sup \{ s^{-1}
	\eqLpnorm{s^{-\vdim} \| V \| \restrict \cball{a}{s} }{\zeta}{h_a}
	\with 0 < s \leq r \}$, then
	\begin{gather*}
		\begin{aligned}
			& \eqLpnorm{s^{-\vdim} \| V \| \restrict \cball{a}{s}
			}{\zeta}{ |\cdot-a|^{-1} h_a} \\
			& \qquad \leq s^{-\vdim/\zeta} \tsum{i=1}{\infty}
			2^{i} s^{-1} \eqLpnorm{\| V \| \restrict (
			\cball{a}{2^{1-i} s} \without \cball{a}{2^{-i}s}
			)}{\zeta}{h_a} \\
			& \qquad \leq 2 \gamma \tsum{i=1}{\infty}
			2^{(1-i)\vdim/\zeta} = 2 \gamma / (1-2^{-\vdim/\zeta})
		\end{aligned}
	\end{gather*}
	in case of \eqref{item:diff_lebesgue_spaces:m>1=p} or
	\eqref{item:diff_lebesgue_spaces:q<m=p} and clearly
	$\eqLpnorm{s^{-\vdim} \| V \| \restrict \cball{a}{s} }{\zeta}{
	|\cdot-a|^{-1} h_a} \leq \gamma$ in case of
	\eqref{item:diff_lebesgue_spaces:m=p=1} or
	\eqref{item:diff_lebesgue_spaces:p=m<q}. This yields the conclusion in
	case of \eqref{item:diff_lebesgue_spaces:m>1=p} or
	\eqref{item:diff_lebesgue_spaces:q<m=p} and that $f|S$ is
	differentiable relative to $S$ at $a$ with
	\begin{gather*}
		D(f|S) (a) = \derivative Vf (a) | \Tan (S,a) \quad \text{for
		$\| V \|$ almost all $a$}
	\end{gather*}
	in case of \eqref{item:diff_lebesgue_spaces:m=p=1} or
	\eqref{item:diff_lebesgue_spaces:p=m<q} by \cite[3.1.22]{MR41:1976}.
	To complete the proof, note that $S$ is dense in $\spt \| V \|$ by
	\cite[2.8.18, 2.9.13]{MR41:1976} and hence if $p = \vdim$, then
	\begin{gather*}
		\Tan (S,a) = \Tan ( \spt \| V \|, a ) = \Tan^\vdim ( \| V \|,
		a )
	\end{gather*}
	whenever $a \in U$ with $\psi ( \{ a \} ) <
	\isoperimetric{\vdim}^{-1}$ by \cite[3.1.21]{MR41:1976} and
	\ref{miniremark:tangent_spaces}.
\end{proof}
\begin{remark}
	The usage of $h_a$ in the last paragraph of the proof is adapted from
	the proof of \cite[4.5.9\,(26)\,(\printRoman{2})]{MR41:1976}.
\end{remark}
\section{Coarea formula}
In this section rectifiability properties of the distributional boundary of
allmost all superlevel sets of real valued weakly differentiable functions are
established, see \ref{corollary:coarea}. The result rests on the approximate
differentiability of such functions, see \ref{thm:approx_diff}. To underline
this fact, it is derived as corollary to a general result for approximately
differentiable functions, see \ref{thm:ap_coarea}.
\begin{theorem} \label{thm:ap_coarea}
	Suppose $\vdim, \adim \in \nat$, $\vdim \leq \adim$, $U$ is an open
	subset of $\rel^\adim$, $V \in \RVar_\vdim ( U )$, $f$ is a $\| V \|$
	measurable real valued function which is $( \| V \|, \vdim )$
	approximately differentiable at $\| V \|$ almost all points, $F$ is a
	$\| V \|$ measurable $\Hom ( \rel^\adim, \rel )$ valued function with
	\begin{gather*}
		F(z) = ( \| V \|, \vdim ) \ap Df(z) \circ \project{\Tan^\vdim
		( \| V \|, z )} \quad \text{for $\| V \|$ almost all $z$},
	\end{gather*}
	and $\tint{K \cap \{ z \with a < f (z) < b \}}{} |F| \ud \| V \| <
	\infty$ whenever $K$ is a compact subset of $U$ and $- \infty < a \leq
	b < \infty$, and $T \in \mathscr{D}' ( U \times \rel, \rel^\adim )$
	and $S(t) : \mathscr{D} ( U, \rel^\adim ) \to \rel$ satisfy
	\begin{gather*}
		T ( \psi ) = \tint{}{} \left < \psi (z,f(z)), F(z) \right >
		\ud \| V \| z \quad \text{for $\psi \in \mathscr{D} ( U \times
		\rel, \rel^\adim )$}, \\
		S(t)( \theta ) = \lim_{\varepsilon \to 0+} \varepsilon^{-1}
		\tint{\{ z \with t < f(z) \leq t+\varepsilon \}}{} \left <
		\theta, F \right > \ud \| V \| \in \rel \quad \text{for
		$\theta \in \mathscr{D} (U, \rel^\adim )$}
	\end{gather*}
	whenever $t \in \rel$, that is $t \in \dmn S$ if and only if the limit
	exists and belongs to $\rel$ for $\theta \in \mathscr{D} ( U,
	\rel^\adim )$.

	Then the following two statements hold.
	\begin{enumerate}
		\item \label{item:ap_coarea:equation} If $\psi \in \Lp{1} ( \|
		T \|, \rel^\adim )$ and $g$ is an $\overline{\rel}$ valued $\|
		T \|$ integrable function, then
		\begin{gather*}
			T ( \psi )= \tint{}{} S(t) ( \psi (\cdot, t ) ) \ud
			\mathscr{L}^1 t, \quad
			\tint{}{} g \ud \| T \| = \tint{}{} \tint{}{} g(z,t)
			\ud \| S(t) \| z \ud \mathscr{L}^1 t.
		\end{gather*}
		\item \label{item:ap_coarea:structure} There exists an
		$\mathscr{L}^1$ measurable function $W$ with values in
		$\RVar_{\vdim-1} (U)$ (with the weak topology) such that for
		$\mathscr{L}^1$ almost all $t$ there holds
		\begin{gather*}
			\Tan^{\vdim-1} ( \| W(t) \|, z ) = \Tan^\vdim ( \| V
			\|, z ) \cap \ker F(z) \in \grass{\adim}{\vdim-1}, \\
			\density^{\vdim-1} ( \| W(t) \|, z ) = \density^\vdim
			( \| V \|, z )
		\end{gather*}
		for $\| W(t) \|$ almost all $z$ and
		\begin{gather*}
			S(t)(\theta) = \tint{}{} \left < \theta, |F|^{-1} F
			\right > \ud \| W(t) \| \quad \text{for $\theta \in
			\mathscr{D} (U,\rel^\adim )$}.
		\end{gather*}
	\end{enumerate}
\end{theorem}
\begin{proof}
	First, notice that \ref{lemma:push_on_product} implies
	\begin{gather*}
		T ( \psi ) = \tint{}{} \left < \psi(z,f(z)), F(z)
		\right > \ud \| V \| z, \quad \| T \| (g) = \tint{}{}
		g(z,f(z)) | F(z) | \ud \| V \| z
	\end{gather*}
	whenever $\psi \in \Lp{1} ( \| T \|, \rel^\adim )$ and $g$ is an
	$\overline{\rel}$ valued $\| T \|$ integrable function.

	Next, the following assertion will be shown. \emph{There exists an
	$\mathscr{L}^1$ measurable function $W$ with values in
	$\RVar_{\vdim-1} (U)$ such that for $\mathscr{L}^1$ almost all $t$
	there holds
	\begin{gather*}
		\Tan^{\vdim-1} ( \| W(t) \|, z ) = \Tan^\vdim ( \| V \|, z )
		\cap \ker F(z) \in \grass{\adim}{\vdim-1}, \\
		\density^{\vdim-1} ( \| W(t) \|, z ) = \density^\vdim ( \| V
		\|, z )
	\end{gather*}
	for $\| W(t) \|$ almost all $z$ and
	\begin{gather*}
		T( \psi ) = \tint{}{} \tint{}{} \left < \psi(z,f(z)),
		|F(z)|^{-1} F(z) \right > \ud \| W(t) \| z \ud \mathscr{L}^1
		t, \\
		\tint{}{} g \ud \| T \| = \tint{}{} \tint{}{} g(z,t) \ud \| W
		(t) \| z \ud \mathscr{L}^1 t
	\end{gather*}
	whenever $\psi \in \Lp{1} ( \| T \|, \rel^\adim )$ and $g$ is an
	$\overline{\rel}$ valued $\| T \|$ integrable function.} For this
	purpose choose a disjoint sequence of Borel sets $B_i$, sequences
	$M_i$ of $\vdim$ dimensional submanifolds of $\rel^\adim$ of class $1$
	and functions $f_i : M_i \to \rel$ of class $1$ satisfying $\| V \|
	\big ( U \without \bigcup_{i=1}^\infty B_i \big ) = 0$ and
	\begin{gather*}
		B_i \subset M_i, \quad f_i (z) = f(z), \quad Df_i (z) = ( \| V
		\|, \vdim ) \ap Df(z) = F(z) | \Tan^\vdim ( \| V \|, z )
	\end{gather*}
	whenever $i \in \nat$ and $z \in B_i$, see \ref{lemma:approx_diff},
	\cite[2.8.18, 2.9.11, 3.2.17, 3.2.29]{MR41:1976} and Allard
	\cite[3.5\,(2)]{MR0307015}. Let $B = \bigcup_{i=1}^\infty B_i$. If $t
	\in \rel$ satisfies
	\begin{gather*}
		\mathscr{H}^{\vdim-1} ( B  \cap \{ z \with \text{$f(z) = t$
		and $( \| V \|, \vdim ) \ap Df (z)=0$} \}) = 0, \\
		\tint{B \cap K \cap \{ z \with f(z) = t \}}{} \density^\vdim (
		\| V \|, z ) \ud \mathscr{H}^{\vdim-1} z< \infty
	\end{gather*}
	whenever $K$ is a compact subset of $U$, then define $W(t) \in
	\RVar_{\vdim-1} ( U)$ by
	\begin{gather*}
		W(t)(h) = \tint{B \cap \{ z \with f(z) = t \}}{} h (z, \ker \,
		( \| V \|, \vdim ) \ap Df(z) ) \density^{\vdim-1} ( \| V \|, z
		) \ud \mathscr{H}^{\vdim-1} z
	\end{gather*}
	for $h \in \mathscr{K} ( U \times \grass{\adim}{\vdim-1} )$. From
	\cite[3.2.22]{MR41:1976} one infers that
	\begin{gather*}
		\tint{A \cap \{ z \with a < f(z) < b \}}{} |F| \ud \| V \| =
		\tint{a}{b} \tint{A \cap B \cap \{ z \with f(z) = t \}}{}
		\density^\vdim ( \| V \|, z ) \ud \mathscr{H}^{\vdim-1} z \ud
		\mathscr{L}^1 t
	\end{gather*}
	whenever $A$ is $\| V \|$ measurable and $- \infty < a < b < \infty$,
	hence
	\begin{gather*}
		\tint{}{} g(z,f(z)) | F(z) | \ud \| V \| z = \tint{}{}
		\tint{}{} g (z,t) \ud \| W(t) \| z \ud \mathscr{L}^1 t
	\end{gather*}
	whenever $g$ is an $\overline{\rel}$ valued $\| T \|$ integrable
	function. The remaining parts of the assertion now follow by
	considering appropriate choices of $g$ and recalling
	\ref{example:kx_lusin}.
	
	Consequently, \eqref{item:ap_coarea:equation} is implied by
	\ref{thm:distribution_on_product}\,\eqref{item:distribution_on_product:absolute}.
	Finally, notice that
	\begin{gather*}
		S(t)(\theta) = \lim_{\varepsilon \to 0+} \varepsilon^{-1}
		\tint{t}{t+\varepsilon} \tint{}{} \left < \theta, |F|^{-1} F
		\right > \ud \| W(s) \| \ud \mathscr{L}^1 s \quad \text{for
		$\theta \in \mathscr{D} (U, \rel^\adim )$}
	\end{gather*}
	whenever $t \in \rel$, hence \eqref{item:ap_coarea:structure} follows
	using \cite[2.8.17, 2.9.8]{MR41:1976}.
\end{proof}
\begin{corollary} \label{corollary:coarea}
	Suppose $\vdim$, $\adim$, $U$, and $V$ are as in
	\ref{miniremark:situation_general}, $f \in \trunc (V)$, and $E(t) = \{
	z \with f(z) > t \}$ for $t \in \rel$.

	Then there exists an $\mathscr{L}^1$ measurable function $W$ with
	values in $\RVar_{\vdim-1} (U)$ (with the weak topology) such that for
	$\mathscr{L}^1$ almost all $t$ there holds
	\begin{gather*}
		\Tan^{\vdim-1} ( \| W(t) \|, z ) = \Tan^\vdim ( \| V \|, z )
		\cap \ker \derivative{V}{f} (z) \in \grass{\adim}{\vdim-1}, \\
		\density^{\vdim-1} ( \| W(t) \|, z ) = \density^\vdim ( \| V
		\|, z )
	\end{gather*}
	for $\| W(t) \|$ almost all $z$ and
	\begin{gather*}
		\boundary{V}{E(t)} (\theta) = \tint{}{} \left < \theta,
		|\derivative{V}{f}|^{-1} \derivative{V}{f} \right > \ud \|
		W(t) \| \quad \text{for $\theta \in \mathscr{D} (U,\rel^\adim
		)$}.
	\end{gather*}
\end{corollary}
\begin{proof}
	In view of \ref{lemma:level_sets} and \ref{thm:approx_diff}, this is a
	consequence of \ref{thm:ap_coarea}.
\end{proof}
\begin{remark} \label{remark:no_rectifiable_structure}
	The formulation of \ref{corollary:coarea} is modelled on similar
	results for sets of locally finite perimeter, see \cite[4.5.6\,(1),
	4.5.9\,(12)]{MR41:1976}. Observe however that the structural
	description of $\boundary{V}{E(t)}$ given here for $\mathscr{L}^1$
	almost all $t$ does not extend to arbitrary $\| V \| + \| \delta V \|$
	measurable sets $E$ such that $\boundary{V}{E}$ is representable by
	integration. (Using \cite[2.10.28]{MR41:1976} to construct $V \in
	\IVar_1 ( \rel^2 )$ such that $\| \delta V \|$ is a nonzero Radon
	measure with $\| \delta V \| ( \{ z \} ) = 0$ for $z \in \rel^2$ and
	$\| V \| ( \spt \delta V) = 0$, one may take $E = \spt \delta V$.)
\end{remark}
\section{Oscillation estimates}
In this section two situations are studied where the oscillation of a weakly
differentiable function may be controlled by its weak derivative; assuming a
suitable integrability of the mean curvature of the underlying varifold. In
general, such control is necessarily rather weak, see \ref{thm:mod_continuity}
and \ref{remark:mod_continuity}, but under special circumstances one may
obtain H\"older continuity, see \ref{thm:hoelder_continuity}. These main
ingredients are the analysis of the connectedness structure of a varifold, see
\ref{corollary:conn_structure}, and the Sobolev Poincar{\'e} type inequalities
with several medians, see \ref{thm:sob_poin_several_med}.
\begin{theorem} \label{thm:mod_continuity}
	Suppose $\vdim$, $\adim$, $p$, $U$, and $V$ are as in
	\ref{miniremark:situation_general}, $p = \vdim$, $K$ is a compact
	subset of $U$, $0 < \varepsilon < \dist ( K, \rel^\adim \without U )$,
	and either
	\begin{enumerate}
		\item \label{item:mod_continuity:m=1} $1 = \vdim = q$ and
		$\lambda = 2^{\vdim+5} \unitmeasure{\vdim}^{-1}
		\Gamma_{\ref{thm:sob_poin_summary}} ( 2 \adim ) \sup \{
		\density^1 ( \| V \|, a ) \with a \in K \}$, or
		\item \label{item:mod_continuity:m>1} $1 < \vdim < q$ and
		$\lambda = \varepsilon$.
	\end{enumerate}

	Then there exists a positive, finite number $\Gamma$ with the
	following property.

	If $l \in \nat$, $f : \spt \| V \| \to \rel^l$ is a continuous
	function, $f \in \trunc (V,\rel^l)$, and $\gamma = \sup \{
	\eqLpnorm{\| V \| \restrict \cball{a}{\varepsilon}}{q}{
	\derivative{V}{f} } \with a \in K \}$, then
	\begin{gather*}
		|f(z)-f(w)| \leq l^{1/2} \lambda \gamma \quad \text{whenever
		$z,w \in K \cap \spt \| V \|$ and $|z-w| \leq \Gamma^{-1}$}.
	\end{gather*}
\end{theorem}
\begin{proof}
	Let $\psi$ be as in \ref{miniremark:situation_general}.  Abbreviate $S
	= \spt \| V \|$ and $\alpha = 1-\vdim/q$. Abbreviate
	\begin{gather*}
		\Delta_1 = 2^{\vdim+3} \unitmeasure{\vdim}^{-1}, \quad
		\Delta_2 = \Gamma_{\ref{thm:sob_poin_summary}} ( 2 \adim ).
	\end{gather*}
	Notice that $\density_\ast^\vdim ( \| V \|, a ) \geq 1/2$ for $a \in
	S$ by \ref{corollary:density_1d}\,\eqref{item:density_1d:lower_bound}
	and \ref{remark:kuwert_schaetzle}. If \eqref{item:mod_continuity:m=1}
	holds, then
	\begin{gather*}
		\dmn \density^1 ( \| V \|, \cdot ) = U, \quad \Delta_3 = \sup
		\{ \density^1 ( \| V \|, a ) \with a \in K \} < \infty
	\end{gather*}
	by
	\ref{corollary:density_1d}\,\eqref{item:density_1d:upper_bound}\,\eqref{item:density_1d:real}.
	If \eqref{item:mod_continuity:m>1} holds, $\density^{\vdim-\alpha} (
	\| V \|, a ) = 0$ for $a \in S$ by
	\ref{corollary:density_ratio_estimate}.  Moreover, if $a \in S$ and
	$\cball{a}{r} \subset U$ then, by
	\ref{corollary:conn_structure}\,\eqref{item:conn_structure:open},
	there exists $0 < s \leq r$ such that $S \cap \oball{a}{s}$ is a
	subset of the connected component of $S \cap \oball{a}{r/2}$ which
	contains $a$. Since $K \cap S$ is compact, one may therefore construct
	$j \in \nat$ and $a_i$, $s_i$ and $r_i$ for $i = 1, \ldots, j$ such
	that $K \cap S \subset \bigcup_{i=1}^j \oball{a_i}{s_i}$ and
	\begin{gather*}
		a_i \in K \cap S, \quad 0 < s_i \leq r_i \leq \varepsilon,
		\quad S \cap \oball{a_i}{s_i} \subset H_i, \quad
		\cball{a_i}{r_i} \subset U, \\
		\psi ( \oball{a_i}{r_i} \without \{ a_i \} ) \leq
		\Delta_2^{-1}, \quad \measureball{\| V \|}{ \oball{a_i}{r_i} }
		\geq (1/2) \unitmeasure{\vdim} (r_i/2)^\vdim, \\
		r_i^{-1} \measureball{\| V \|}{\cball{a_i}{r_i}} \leq 4
		\Delta_3 \quad \text{if \eqref{item:mod_continuity:m=1}
		holds}, \\
		\Delta_1
		\Gamma_{\ref{thm:sob_poin_several_med}\,\eqref{item:sob_poin_several_med:p=m<q}}
		( 2\adim )^{\vdim/\alpha} r_i^{\alpha-\vdim}
		\measureball{\| V \|}{ \oball{a_i}{r_i} } \leq \varepsilon
		\quad \text{if \eqref{item:mod_continuity:m>1} holds}
	\end{gather*}
	for $i = 1, \ldots, j$, where $H_i$ is the connected component of $S
	\cap \oball{a_i}{r_i/2}$ which contains $a_i$. Since $K \cap S$ is
	compact, there exists a positive, finite number $\Gamma$ with the
	following property. If $z,w \in K \cap S$ and $|z-w| \leq \Gamma^{-1}$
	then $\{ z,w \} \subset \oball{a_i}{s_i}$ for some $i$.

	In order to verify that $\Gamma$ has the asserted property, suppose
	that $l$, $f$, $\gamma$, $z$, and $w$ are related to $\Gamma$ as in
	the body of the theorem.

	In view of \ref{remark:integration_by_parts} one may assume $l = 1$.
	Choose $a_i \in K \cap S$ such that $\{ z,w \} \subset H_i$. Choose $P
	\in \nat$ satisfying
	\begin{gather*}
		(1/2) P \unitmeasure{\vdim} (r_i/2)^\vdim \leq \measureball{\|
		V \|}{ \oball{a_i}{r_i} } \leq (1/2) (P+1) \unitmeasure{\vdim}
		(r_i/2)^\vdim.
	\end{gather*}
	Applying
	\ref{thm:sob_poin_several_med}\,\eqref{item:sob_poin_several_med:p=m=1}\,\eqref{item:sob_poin_several_med:p=m<q}
	with $U$, $M$, $Q$, $r$, and $Z$ replaced by $\oball{a_i}{r_i}$,
	$2 \adim$, $1$, $r_i/2$, and $\{ a_i \}$ yields a subset $Y$ of $\rel$
	such that $\card Y \leq P +1$ and
	\begin{gather*}
		{\textstyle f \lIm \oball{a_i}{r_i/2} \rIm \subset \bigcup \{
		\cball{y}{\delta_i} \with y \in Y \}},
	\end{gather*}
	where $\delta_i = \Delta_2 \gamma$ if \eqref{item:mod_continuity:m=1}
	holds and $\delta_i =
	\Gamma_{\ref{thm:sob_poin_several_med}\,\eqref{item:sob_poin_several_med:p=m<q}}
	( 2 \adim )^{\vdim/\alpha} r_i^\alpha \gamma$ if
	\eqref{item:mod_continuity:m>1} holds. Defining $E$ to be the
	connected component of $\bigcup \{ \cball{y}{\delta_i} \with y \in Y
	\}$ which contains $f(a)$, one infers
	\begin{gather*}
		|f(z)-f(w)| \leq \diam E = \mathscr{L}^1 (E) \leq 2 (P+1)
		\delta_i \leq \Delta_1 \delta_i r_i^{-\vdim} \measureball{\| V
		\|}{ \oball{a_i}{r_i} } \leq \lambda \gamma
	\end{gather*}
	since $\{ f(z), f(w) \} \subset f \lIm H_i \rIm \subset E$ and $E$ is
	an interval.
\end{proof}
\begin{remark} \label{remark:mod_continuity}
	Considering varifolds corresponding to two parallel planes, it is
	clear that $\Gamma$ may not be chosen independently of $V$. \verify
	Also the continuity hypothesis on $f$ is essential as may be seen
	considering $V$ associated to two parallel lines.
\end{remark}
\begin{theorem} \label{thm:hoelder_continuity}
	Suppose $1 \leq M < \infty$.

	Then there exists a positive, finite number $\Gamma$ with the
	following property.

	If $\vdim$, $\adim$, $p$, $U$, $V$, and $\psi$ are as in
	\ref{miniremark:situation_general}, $p = \vdim$, $\adim \leq M$, $0 <
	r < \infty$, $A = \{ z \with \oball zr \subset U \}$, $C = \{
	(z,\cball zs ) \with \cball zs \subset U \}$,
	\begin{gather*}
		\measureball{\psi}{ \oball{a}{r} } \leq \Gamma^{-1}, \qquad
		\measureball{\| V \|}{ \oball as } \leq (2 - M^{-1} )
		\unitmeasure{\vdim} s^\vdim \quad \text{for $0 < s \leq r$}
	\end{gather*}
	whenever $a \in A \cap \spt \| V \|$, $l \in \nat$, $\vdim < q \leq
	\infty$, $\alpha = 1-\vdim/q$, $f \in \trunc (V,\rel^l)$, $B$ is the
	set of $a \in A \cap \spt \| V \|$ such that $f$ is $(\| V \|,C)$
	approximately continuous at $a$, and $\gamma = \sup \{ \eqLpnorm{ \| V
	\| \restrict \oball ar}{q}{ \derivative Vf } \with a \in A \cap \spt
	\| V \| \}$, then
	\begin{gather*}
		| f(z)-f(w) | \leq l^{1/2} \lambda |z-w|^\alpha \gamma \quad
		\text{whenever $z,w \in B$ and $|z-w| \leq r/\Gamma$},
	\end{gather*}
	where $\lambda = \Gamma$ if $\vdim = 1$ and $\lambda =
	\Gamma^{1/\alpha}$ if $\vdim > 1$.
\end{theorem}
\begin{proof}
	Define
	\begin{gather*}
		\Delta_1 = (1-1/(4M-1))^{1/M}, \quad \Delta_2 = \sup \{ 1,
		\Gamma_{\ref{thm:sob_poin_several_med}\,\eqref{item:sob_poin_several_med:p=m<q}}
		(2M) \}^M, \\
		\Gamma = 4 (1-\Delta_1)^{-1} \sup \{ 4
		\Gamma_{\ref{thm:sob_poin_summary}} (2M), \Delta_2 \}
	\end{gather*}
	and notice that $\Gamma \geq 4 (1-\Delta_1)^{-1}$.

	In order to verify that $\Gamma$ has the asserted property, suppose
	that $\vdim$, $\adim$, $p$, $U$, $V$, $\psi$, $M$, $r$, $A$, $C$, $l$,
	$q$, $f$, $B$, $\gamma$, $w$, $z$, and $\lambda$ are related to
	$\Gamma$ as in the body of the theorem.

	In view of \ref{remark:integration_by_parts} and
	\ref{corollary:integrability}\,\eqref{item:integrability:m=1}\,\eqref{item:integrability:m=p<q}
	one may assume $l = 1$. Define $s = 2 |z-w|$ and $t =
	(1-\Delta_1)^{-1} s$ and notice that $0 < s < t \leq r$ and
	\begin{gather*}
		\begin{aligned}
			& \measureball{\| V \|}{ \oball{z}{t}} \leq (2-M^{-1})
			\unitmeasure{\vdim} (1-\Delta_1)^{-\vdim} s^\vdim \\
			& \qquad \leq \big (2- (2M)^{-1} \big )
			\unitmeasure{\vdim} \Delta_1^\vdim
			(1-\Delta_1)^{-\vdim} s^\vdim = \big (2-(2M)^{-1} \big
			) \unitmeasure{\vdim} (t-s)^\vdim,
		\end{aligned} \\
		\measureball{\psi}{ \oball zt } \leq \Gamma^{-1} \leq
		\Gamma_{\ref{thm:sob_poin_summary}} ( 2M )^{-1}.
	\end{gather*}
	Therefore applying
	\ref{thm:sob_poin_several_med}\,\eqref{item:sob_poin_several_med:p=m=1}\,\eqref{item:sob_poin_several_med:p=m<q}
	with $U$, $M$, $P$, $Q$, $r$, and $Z$ replaced by $\oball{z}{t}$,
	$2M$, $1$, $1$, $t-s$, and $\varnothing$ yields a subset $Y$ of $\rel$
	with $\card Y = 1$ such that $\eqLpnorm{\| V \| \restrict
	\oball{z}{s}}{\infty}{f_Y}$ is bounded by
	\begin{gather*}
		\Gamma_{\ref{thm:sob_poin_summary}} ( 2M ) \eqLpnorm{\| V \|
		\restrict \oball zt}{1}{ \derivative{V}{f} } \leq 4
		\Gamma_{\ref{thm:sob_poin_summary}} ( 2M ) t^\alpha \gamma
		\quad \text{if $\vdim = 1$, and by} \\
		\Gamma_{\ref{thm:sob_poin_several_med}\,\eqref{item:sob_poin_several_med:p=m<q}}
		( 2M )^{\vdim/\alpha} t^\alpha \eqLpnorm{\| V \| \restrict
		\oball zt}{q}{ \derivative Vf } \leq \Delta_2^{1/\alpha}
		t^\alpha \gamma \quad \text{if $\vdim > 1$}.
	\end{gather*}
	Noting $|f(z)-f(w)| \leq 2 \eqLpnorm{\| V \| \restrict \oball zs
	}{\infty}{f_Y}$ and $t = 2 (1-\Delta_1)^{-1} |z-w|$, the conclusion
	follows.
\end{proof}
\section{Geodesic distance} \label{sec:geodesic_distance}
Reconsidering the support of the weight measure of a varifold whose mean
curvature satisfies a suitable integrability of the mean curvature, the
oscillation estimate \ref{thm:mod_continuity} is used to established in this
section that its connected components agree with the components induced by its
geodesic distance, see \ref{thm:conn_path_finite_length}.
\begin{example} \label{example:space_fill}
	Whenever $1 \leq p < \vdim < \adim$, $U$ is an open subset of
	$\rel^\adim$ and $G$ is an open subset of $U$, there exists $V$
	related to $\vdim$, $\adim$, $U$, and $p$ as in
	\ref{miniremark:situation_general} such that $\spt \| V \|$ equals the
	closure of $G$ relative to $U$ as is readily seen taking the into
	account the behaviour of $\psi$ and $V$ under homotheties; compare
	\cite[1.2]{snulmenn.isoperimetric}.
\end{example}
\begin{theorem} \label{thm:conn_path_finite_length}
	Suppose $\vdim$, $\adim$, $U$, $V$, and $p$ are as in
	\ref{miniremark:situation_general}, $p = \vdim$, and $C$ is a
	connected component of $\spt \| V \|$, and $w,z \in C$.

	Then there exist $- \infty < a \leq b < \infty$ and a Lipschitzian
	function $g : \{ t \with a \leq t \leq b \} \to \spt \| V \|$ such
	that $g(a) = w$ and $g(b) = z$.
\end{theorem}
\begin{proof}
	In view of
	\ref{corollary:conn_structure}\,\eqref{item:conn_structure:piece}\,\eqref{item:conn_structure:open}
	one may assume $C = \spt \| V \|$. Whenever $c \in C$ denote by $Z(c)$
	the set of $\xi \in C$ such that there exist $- \infty < a \leq b <
	\infty$ and a Lipschitzian function $g : \{ t \with a \leq t \leq b \}
	\to C$ such that $g(a) = c$ and $g(b) = \xi$. Observe that it is
	sufficient to prove that $c$ belongs to the interior of $Z(c)$
	relative to $C$ whenever $c \in C$.
	
	For this purpose suppose $c \in C$, choose $0 < r < \infty$ with
	$\cball{c}{2r} \subset U$. Let $K = \cball{c}{r}$. If $\vdim = 1$,
	define $\lambda$ as in
	\ref{thm:mod_continuity}\,\eqref{item:mod_continuity:m=1} and notice
	that $\lambda < \infty$ by
	\ref{corollary:density_1d}\,\eqref{item:density_1d:upper_bound}.
	Define $q = 1$ if $\vdim = 1$ and $q = 2 \vdim$ if $\vdim > 1$. Choose
	$0 < \varepsilon \leq r$ such that
	\begin{gather*}
		\lambda \sup \{ \| V \| ( \cball{a}{\varepsilon} )^{1/q} \with
		a \in K \} \leq r,
	\end{gather*}
	where $\lambda = \varepsilon$ in accordance with
	\ref{thm:mod_continuity}\,\eqref{item:mod_continuity:m>1} if $\vdim >
	1$, and define
	\begin{gather*}
		s = \inf \{ \Gamma_{\ref{thm:mod_continuity}} ( \vdim, \adim,
		\vdim, U, V, K, \varepsilon, q )^{-1}, r \}.
	\end{gather*}

	Next, define functions $f_\delta : C \to \overline{\rel}$ by
	\begin{gather*}
		f_\delta ( \xi ) = \inf \left \{ \sum_{i=1}^j
		|z_i-z_{i-1}| \with z_0 = c, z_j = \xi, z_i \in C,
		|z_i-z_{i-1}| \leq \delta \right \}
	\end{gather*}
	whenever $\xi \in C$ and $\delta > 0$. Notice that
	\begin{gather*}
		f_\delta (\xi) \leq f_\delta (\zeta) + |\zeta-\xi|
		\quad \text{whenever $\zeta,\xi \in \spt \| V \|$ and
		$|\zeta-\xi| \leq \delta$}.
	\end{gather*}
	Since $C$ is connected and $f_\delta (c) = 0$, it follows that
	$f_\delta$ is a locally Lipschitzian real valued function
	satisfying
	\begin{gather*}
		\Lip ( f_\delta | Z ) \leq 1 \quad \text{whenever $Z
		\subset U$ and $\diam Z \leq \delta$},
	\end{gather*}
	in particular $f_\delta \in \trunc (V)$ and $\Lpnorm{\| V
	\|}{\infty}{\derivative{V}{f_\delta}} \leq 1$ by
	\ref{example:lipschitzian}. One infers
	\begin{gather*}
		f_\delta (\xi) \leq r \quad \text{whenever $\xi \in C
		\cap \cball{c}{s}$ and $\delta > 0$}
	\end{gather*}
	from \ref{thm:mod_continuity}. It follows that $C \cap \cball{c}{s}
	\subset Z(c)$; in fact, if $\xi \in C \cap \cball{c}{s}$  one readily
	constructs $g_\delta : \{ u \with 0 \leq t \leq r \} \to
	\rel^\adim$ satisfying
	\begin{gather*}
		g_\delta (0)=c, \quad g_\delta (r) = \xi, \quad \Lip
		g_\delta \leq 1+\delta, \\
		\dist ( g_\delta (t), C ) \leq \delta \quad
		\text{whenever $0 \leq t \leq r$},
	\end{gather*}
	hence, noting that $\im g_\delta \subset \cball{c}{(1+\delta)r}$, the
	existence of $g : \{ t \with 0 \leq t \leq r \} \to C$ satisfying $g
	(0)=c$, $g (r) = \xi$, and $\Lip g \leq 1$ now is a consequence of
	\cite[2.10.21]{MR41:1976}.
\end{proof}
\begin{remark} \label{remark:metric_spaces}
	The deduction of \ref{thm:conn_path_finite_length} from
	\ref{thm:mod_continuity} is adapted from Cheeger \cite[\S
	17]{MR1708448} who attributes the argument to Semmes, see also David
	and Semmes \cite{MR1132876}.
\end{remark}
\begin{remark} \label{remark:topping}
	In view of \ref{example:space_fill} it is not hard to construct
	examples showing that the hypothesis ``$p = \vdim$'' in
	\ref{thm:mod_continuity} may not be replaced by ``$p \geq q$'' for any
	$1 \leq q < \vdim$ if $\vdim < \adim$. Yet, for indecomposable $V$,
	the study of possible extensions of \ref{thm:conn_path_finite_length}
	as well as related questions seems to be most natural under the
	hypothesis ``$p = \vdim-1$'', see Topping \cite{MR2410779}.
\end{remark}
\section{Curvature varifolds} \label{sec:curvature_varifolds}
In this section Hutchinson's concept of curvature varifold is rephrased in
terms of the concept of weakly differentiable function proposed in the present
paper, see \ref{thm:curvature_varifolds}. To indicate possible benefits of
this perspective, a result on the differentiability of the tangent plane map
is included, see \ref{corollary:curv_var_diff}.
\begin{miniremark} \label{miniremark:trace}
	Suppose $\adim \in \nat$ and $Y = \Hom ( \rel^\adim, \rel^\adim ) \cap
	\{ \tau \with \tau = \tau^\ast \}$. Then $T : \Hom ( \rel^\adim, Y )
	\to \rel^\adim$ denotes the linear map which is given by the
	composition (see \cite[1.1.1, 1.1.2, 1.1.4, 1.7.9]{MR41:1976})
	\begin{align*}
		\Hom ( \rel^\adim, Y ) & \xrightarrow{\subset}
		\Hom ( \rel^\adim, \Hom (\rel^\adim, \rel^\adim)) \\
		& \simeq \Hom ( \rel^\adim, \rel ) \otimes \Hom ( \rel^\adim,
		\rel ) \otimes \rel^\adim \xrightarrow{\Phi \otimes
		\id{\rel^\adim}} \rel \otimes \rel^\adim \simeq \rel^\adim,
	\end{align*}
	where the linear map $\Phi$ is induced by the inner product on $\Hom (
	\rel^\adim, \rel )$, hence $T ( g ) = \sum_{i=1}^\adim \left < v_i,
	g(v_i) \right >$ whenever $g \in \Hom (\rel^\adim, Y )$ and $v_1,
	\ldots, v_\adim$ form an orthonormal basis of $\rel^\adim$.
\end{miniremark}
\begin{miniremark} \label{miniremark:mean_curv}
	Suppose $\adim$, $Y$ and $T$ are as in \ref{miniremark:trace}, $\adim
	\geq \vdim \in \nat$, $M$ is an $\vdim$ dimensional submanifold of
	$\rel^\adim$ of class $2$, and $S : M \to Y$ is defined by $S(z) =
	\project{\Tan ( M,z)}$ for $z \in M$. Then one computes
	\begin{gather*}
		\mathbf{h} ( M,z ) = T ( DS(z) \circ S(z) ) \quad
		\text{whenever $z \in M$};
	\end{gather*}
	in fact, differentiating the equation $S(z) \circ S(z) = S(z)$ for $z
	\in M$, one obtains
	\begin{gather*}
		S(z) \circ \left < v, DS(z) \right > \circ S(z) = 0 \quad
		\text{for $v \in \Tan(M,z)$}, \\
		T(DS(z) \circ S(z)) \in \Nor (M,z)
	\end{gather*}
	for $z \in M$ and, denoting by $v_1, \ldots, v_\adim$ an orthonormal
	base of $\rel^\adim$, one computes
	\begin{align*}
		& S(z) \bullet ( Dg(z) \circ S(z) ) = \tsum{i=1}{\adim} \left
		< S(z)(v_i), Dg(z) \right > \bullet S(z)(v_i) \\
		& \qquad = - g(z) \bullet \tsum{i=1}{\adim} \left < S(z)(v_i),
		DS(z) \right > (v_i) = - g(z) \bullet T ( DS(z) \circ S(z))
	\end{align*}
	whenever $g : M \to \rel^\adim$ is of class $1$ and $g(z) \in \Nor
	(M,z)$ for $z \in M$.
\end{miniremark}
\begin{lemma} \label{lemma:gen_sym_endo}
	Suppose $\adim$ and $Y$ are as in \ref{miniremark:trace} and $\adim >
	\vdim \in \nat$.

	Then the vectorspace $Y$ is generated by $\{\project{S} \with S \in
	\grass{\adim}{\vdim} \}$.
\end{lemma}
\begin{proof}
	If $v_1, \ldots, v_{\vdim+1}$ are orthonormal vectors in $\rel^\adim$,
	\begin{gather*}
		S_i = \Span \big \{ v_j \with \text{$j \in \{ 1,\ldots,\vdim
		+1 \}$ and $j \neq i$} \big \} \in \grass{\adim}{\vdim}
	\end{gather*}
	for $i=1, \ldots, \vdim+1$, and $v_{\vdim+1} \in T \in
	\grass{\adim}{1}$, then $\project{T} +  \eqproject{S_{\vdim+1}} =
	\frac{1}{\vdim} \sum_{i=1}^{\vdim+1} \eqproject{S_i}$. Hence one may
	assume $\vdim=1$ in which case \cite[1.7.3]{MR41:1976} yields the
	assertion.
\end{proof}
\begin{definition} \label{def:curvature_varifold}
	Suppose $\adim$, $Y$ and $T$ are as in \ref{miniremark:trace}, $\adim
	\geq \vdim \in \nat$, $U$ is an open subset of $\rel^\adim$, $V \in
	\IVar_\vdim (U)$, $Z = U \cap \{ z \with \Tan^\vdim ( \| V \|, z ) \in
	\grass{\adim}{\vdim} \}$, and $R : Z \to Y$ satisfies
	\begin{gather*}
		R(z) = \project{\Tan^\vdim ( \| V \|, z )} \quad
		\text{whenever $z \in Z$}.
	\end{gather*}

	Then $V$ is called a \emph{curvature varifold} if and only if there
	exists $A \in \Lploc{1} \big ( \| V \|, \Hom ( \rel^\adim, Y) \big )$
	satisfying
	\begin{gather*}
		\tint{}{} \left < (R(z)(v), A(z)(v)), D \phi (z,R(z)) \right >
		+ \phi (z,R(z)) T(A(z)) \bullet v \ud \| V \| z = 0
	\end{gather*}
	whenever $v \in \rel^\adim$ and $\phi \in \mathscr{D}^0 ( U \times Y
	)$.
\end{definition}
\begin{remark} \label{remark:hutchinson_reformulations}
	Using approximation and the fact that $R$ is a bounded function, one
	may verify that the same definition results if the equation is
	required for every $\phi : U \times Y \to \rel$ of class $1$ such that
	\begin{gather*}
		\Clos ( U \cap \{ (z,\tau) \with \text{$\phi(z,\tau) \neq 0$
		for some $\tau \in Y$} \} )
	\end{gather*}
	is compact. Extending $T$ accordingly, one may also replace $Y$ by
	$\Hom ( \rel^\adim, \rel^\adim )$ in the definition; in fact, if $A
	\in \Lploc{1} \big ( \| V \|, \Hom ( \rel^\adim, \Hom ( \rel^\adim,
	\rel^\adim) ) \big )$ satisfies the condition of the modified
	definition then $\im A(z) \subset Y$ for $\| V \|$ almost all $z$ as
	may be seen by enlarging the class of $\phi$ as before and considering
	$\phi ( z, \tau ) = \gamma ( z ) g ( \tau )$ for $z \in U$, $\tau \in
	\Hom ( \rel^\adim, \rel^\adim )$, $\gamma \in \mathscr{D}^0 (
	\rel^\adim )$ and $g : \Hom ( \rel^\adim, \rel^\adim ) \to \rel$ is
	linear with $Y \subset \ker g$, compare Hutchinson
	\cite[5.2.4\,(i)]{MR825628}.

	Consequently, the definition is equivalent to Hutchinson's definition
	in \cite[5.2.1]{MR825628}. The present formulation is motivated by
	\ref{lemma:gen_sym_endo}. A less obvious reformulation will be
	provided in \ref{thm:curvature_varifolds}.
\end{remark}
\begin{theorem} \label{thm:curvature_varifolds}
	Suppose $\adim$, $Y$ and $T$ are as in \ref{miniremark:trace},
	$\adim \geq \vdim \in \nat$, $U$ is an open subset of $\rel^\adim$, $V
	\in \IVar_\vdim (U)$, $Z = U \cap \{ z \with \Tan^\vdim ( \| V \|, z )
	\in \grass{\adim}{\vdim} \}$, and $R : Z \to Y$ satisfies
	\begin{gather*}
		R(z) = \project{\Tan^\vdim ( \| V \|, z )} \quad
		\text{whenever $z \in Z$}.
	\end{gather*}

	Then $V$ is a curvature varifold if and only if $\| \delta V \|$ is a
	Radon measure absolutely continuous with respect to $\| V \|$ and $R
	\in \trunc(V, Y )$. In this case, there holds
	\begin{gather*}
		\mathbf{h} ( V,z ) = T ( \derivative{V}{R} (z) ) \quad
		\text{for $\| V \|$ almost all $z$}, \\
		( \delta V )_z ( \phi (z,R(z)) v ) = \tint{}{} \left <
		( R(z)(v), \derivative{V}{R}(z)(v) ), D \phi (z,R(z))
		\right > \ud \| V \| z
	\end{gather*}
	whenever $\phi : U \times Y \to \rel$ is of class $1$ such that
	\begin{gather*}
		\Clos ( U \cap \{ (z,\tau) \with \text{$\phi(z,\tau) \neq 0$
		for some $\tau \in Y$} \} )
	\end{gather*}
	is compact and $v \in \rel^\adim$.
\end{theorem}
\begin{proof}
	Suppose $V$ is a curvature varifold and $A$ is as in
	\ref{def:curvature_varifold}. If $\gamma \in \mathscr{D}^0 ( U )$,
	$\eta \in \mathscr{E}^0 ( Y )$ and $v \in \rel^\adim$, one may take
	$\phi (z,\tau) = \gamma(z) \eta (\tau)$ in
	\ref{def:curvature_varifold}, \ref{remark:hutchinson_reformulations}
	to obtain
	\begin{multline*}
		\tint{}{} \eta ( R(z)) \left < (R(z)(v),D\gamma(z) \right > +
		\gamma (z) \left < A(z)(v), D \eta (R(z)) \right > \ud \| V \|
		z \\
		= - \tint{}{} \gamma (z) \eta (R(z)) T (A(z)) \bullet
		v \ud \| V \| z.
	\end{multline*}
	In particular, taking $\eta = 1$, one infers that $\| \delta V \|$ is
	a Radon measure absolutely continuous with respect to $\| V \|$ with
	\begin{gather*}
		T(A(z)) = \mathbf{h}(V,z) \quad \text{$\| V \|$ almost all
		$z$}.
	\end{gather*}
	It follows that
	\begin{gather*}
		( \delta V ) ( ( \eta \circ R ) \gamma \cdot v ) = - \tint{}{}
		\gamma (z) \eta (R(z)) T ( A(z) ) \bullet v \ud \| V \|z
	\end{gather*}
	for $\gamma \in \mathscr{D}^0 ( U )$, $\eta \in \mathscr{E}^0 ( Y )$
	and $v \in \rel^\adim$. Together with the first equation this implies
	that $R \in \trunc(V,Y)$ with $\derivative{V}{R} (z) = A(z)$ for $\| V
	\|$ almost all $z$ by \ref{remark:eq_condition_weak_diff}.
	
	To prove the converse, suppose $\| \delta V \|$ is a Radon measure
	absolutely continuous with respect to $\| V \|$ and $R \in \trunc(V,
	Y)$. Since $R$ is a bounded function, $\derivative{V}{R} \in \Lploc{1}
	\big ( \| V \|, \Hom ( \rel^\adim, Y ) \big )$ by
	\ref{remark:eq_bounded_condition}. In order to prove the equation for
	the generalised mean curvature vector of $V$, in view of
	\ref{thm:approx_diff} and \cite[4.8]{snulmenn.c2}, it is sufficient to
	prove that
	\begin{gather*}
		\mathbf{h} (M,z) = T \big ( ( \| V \|, \vdim ) \ap DR(z) \circ
		R(z) \big ) \quad \text{for $\| V \|$ almost all $z \in U \cap
		M$}
	\end{gather*}
	whenever $M$ is an $\vdim$ dimensional submanifold of $\rel^\adim$ of
	class $2$. The latter equation however is evident from
	\ref{lemma:approx_diff}\,\eqref{item:approx_diff:comp} and
	\ref{miniremark:mean_curv} since
	\begin{gather*}
		\Tan (M,z) = \Tan^\vdim ( \| V \|, z ) \quad \text{for $\| V
		\|$ almost all $z \in U \cap M$}
	\end{gather*}
	by \cite[2.8.18, 2.9.11, 3.2.17]{MR41:1976} and Allard
	\cite[3.5\,(2)]{MR0307015}.

	Next, define $f : Z \to \rel^\adim \times Y$ by
	\begin{gather*}
		f(z) = (z,R(z)) \quad \text{for $z \in \dmn R$}
	\end{gather*}
	and notice that $f \in \trunc ( V, \rel^\adim \times Y )$ with
	\begin{gather*}
		\derivative{V}{f} (z) (v) = ( R(z)(v), \derivative{V}{R}(z)(v)
		) \quad \text{for $v \in \rel^\adim$}
	\end{gather*}
	for $\| V \|$ almost all $z$ by
	\ref{thm:addition}\,\eqref{item:addition:join} and
	\ref{example:lipschitzian}. The proof may be concluded by establishing
	the last equation of the postscript. Using approximation
	and the fact that $R$ is a bounded function yields that is sufficient
	to consider $\phi \in \mathscr{D}^0 ( U \times Y )$. Choose $\gamma
	\in \mathscr{D}^0 (U)$ with
	\begin{gather*}
		\Clos ( U \cap \{ z \with \text{$\phi (z,\tau) \neq 0$ for
		some $\tau \in Y$} \} ) \subset \Int \{ z \with \gamma (z) = 1
		\}.
	\end{gather*}
	Now, applying \ref{remark:eq_condition_weak_diff} with $\eta$ replaced
	by $\phi$ yields equation in question.
\end{proof}
\begin{remark}
	The first paragraph of the proof is essentially contained in
	Hutchinson in \cite[5.2.2, 5.2.3]{MR825628} and is included here for
	completeness.
\end{remark}
\begin{remark}
	If $\vdim = \adim$ then $\density^\vdim ( \| V \|, \cdot )$ is a
	locally constant, integer valued function whose domain is $U$; in
	fact, $R$ is constant, hence $V$ is stationary by
	\ref{thm:curvature_varifolds} and the asserted structure follows from
	Allard \cite[4.6\,(3)]{MR0307015}.
\end{remark}
\begin{corollary} \label{corollary:curv_var_diff}
	Suppose $\vdim, \adim \in \nat$, $1 < \vdim < \adim$, $\beta =
	\vdim/(\vdim-1)$, $U$ is an open subset of $\rel^\adim$, $V \in
	\IVar_\vdim (U)$ is a curvature varifold, $Z = U \cap \{ z \with
	\Tan^\vdim ( \| V \|, z ) \in \grass{\adim}{\vdim} \}$, and $R : Z \to
	Y$ satisfies
	\begin{gather*}
		R(z) = \project{\Tan^\vdim ( \| V \|, z )} \quad
		\text{whenever $z \in Z$}.
	\end{gather*}

	Then $\| V \|$ almost all $a$ satisfy
	\begin{gather*}
		\lim_{r \to 0+} r^{-\vdim} \tint{\cball ar}{} ( |
		R(z)-R(a)- \left <z-a, \derivative VR (a) \right > | /
		| z-a| )^\beta \ud \| V \| z = 0.
	\end{gather*}
\end{corollary}
\begin{proof}
	This is an immediate consequence of \ref{thm:curvature_varifolds} and
	\ref{thm:diff_lebesgue_spaces}\,\eqref{item:diff_lebesgue_spaces:m>1=p}.
\end{proof}
\begin{remark} \label{remark:curv_flatness_m=1}
	Using
	\ref{thm:diff_lebesgue_spaces}\,\eqref{item:diff_lebesgue_spaces:m=p=1},
	one may formulate a corresponding result for the case $\vdim = 1$.
\end{remark}
\begin{remark} \label{remark:curv_flatness}
	Notice that one may deduce decay results for height quantities
	from this result by use of \cite[4.11\,(1)]{snulmenn.poincare}.
\end{remark}
\begin{remark} \label{remark:curv_to_do}
	If $1 < p < \vdim$ and $\derivative{V}{R} \in \Lploc{p} \big ( \| V \|
	, \Hom ( \rel^\adim, Y ) \big )$, see \ref{miniremark:trace}, one may
	investigate whether the conclusion still holds with $\beta$ replaced
	by $\vdim p/(\vdim-p)$.
\end{remark}

{\small \noindent Max Planck Institut for Gravitational Physics (Albert
Einstein Institute) \newline Am M{\"u}hlen\-berg 1, 14476 Potsdam, Germany
\newline \texttt{Ulrich.Menne@aei.mpg.de} \medskip \newline University of
Potsdam, Institute for Mathematics, \newline Am Neuen Palais 10, 14469
Potsdam, Germany \newline \texttt{Ulrich.Menne@math.uni-potsdam.de} }

\newpage
\part{PDEs on varifolds} \label{part:pde}
This part presents the estimates for locally Lipschitzian solutions of second
order linear elliptic partial differential equations to be employed in Part
\ref{part:geom}. Boundary conditions are phrased in terms of compactness
hypotheses on certain superlevel sets.

Whereas this framework is sufficient for the particular applications of Part
\ref{part:geom} it is the purpose of the theory developed in Part
\ref{part:weakly} to provide a more natural (and larger) classes of functions
in which to phrase the results. Therefore the results of this part are
supposed to be accordingly extended before completation of this project.
\section{Sobolev spaces}
In this section Sobolev spaces are introduced by means of suitable
completation procedures starting from Lipschitzian functions, see
\ref{def:strict_local_sobolev_space}, \ref{def:local_sobolev_space},
\ref{def:strict_sobolev_space} and \ref{def:sobolev_space}. All these spaces
are complete linear spaces consisting of generalised weakly differentiable
functions and, by the approximation result \ref{corollary:approximation_lip},
one could have alternately started from smooth rather than Lipschitzian
function in the usual cases, see \ref{remark:strict_local_sobolev_space},
\ref{remark:local_sobolev_space}, \ref{remark:strict_sobolev_space} and
\ref{remark:sobolev_space}. Finally, the notion is compared to a notion of
Sobolev space for general measures introduced by Bouchitt\'e, Buttazzo and
Seppecher in \cite{MR1424348}, see \ref{remark:other_sobolev_space}.
\begin{lemma} \label{lemma:c1_extension}
	Suppose $l, \adim \in \nat$, $A \subset U \subset \rel^\adim$, $U$ is
	open, and $A$ is closed relative to $U$, $f : U \to \rel^l$ is of
	class $1$, and $\varepsilon > 0$.

	Then there exist an open subset $Z$ of $U$ and a function $g : U \to
	\rel^l$ of class $1$ such that $A \subset Z$, $f|Z = g|Z$, and
	\begin{gather*}
		\Lip g \leq \varepsilon + \sup \{ \Lip ( f|A ), \sup \| Df \|
		\lIm A \rIm \}.
	\end{gather*}
	Moreover, if $l = 1$ and $f \geq 0$ then one may require $g \geq 0$.
\end{lemma}
\begin{proof}
	Assume $\lambda = \sup \{ \Lip (f|A), \sup \| D f \| \lIm A \rIm \} <
	\infty$.

	Define $\Phi$ to be the family of all open subsets $V$ of $U$ such
	that
	\begin{gather*}
		\Lip ( f | V ) \leq \varepsilon/2 + \lambda
	\end{gather*}
	and note $A \subset \bigcup \Phi$. Choose compact sets $K_i$
	satisfying
	\begin{gather*}
		K_i \subset \Int K_{i+1} \quad \text{for $i \in \nat$} \qquad
		\text{and} \qquad {\textstyle \bigcup \{ K_i \with i \in \nat
		\} = \bigcup \Phi }
	\end{gather*}
	and inductively a sequence $\delta_i$ with the following properties
	\begin{gather*}
		0 < \delta_{i+1} \leq \delta_i, \quad
		\classification{\rel^\adim}{z}{\dist (z,A \cap K_i) \leq 2
		\delta_i } \subset \Int K_{i+1}, \\
		\text{if $\{z,w\} \subset K_{i+1}$ and $|z-w| \leq \delta_i$
		then $\{z,w\} \subset V$ for some $V \in \Phi$}
	\end{gather*}
	whenever $i \in \nat$. In particular,
	\begin{gather*}
		|f(z)-f(w)| \leq ( \varepsilon/2 + \lambda ) |z-w| \quad
		\text{whenever $\{z,w\} \subset K_{i+1}$, $|z-w| \leq
		\delta_i$}.
	\end{gather*}
	Define $\eta = ( \varepsilon + 4 \lambda )^{-1}
	\varepsilon/2$ and note $\eta \leq 1$. Let
	\begin{gather*}
		W_i = \classification{U}{z}{\dist(z,A \cap K_i) < \eta
		\delta_i} \quad \text{whenever $i \in \nat$}
	\end{gather*}
	and $W = \bigcup \{ W_i \with i \in \nat \}$. Note $A \cap K_i \subset
	W_i$ for $i \in \nat$, hence $A \subset W$.

	Next, it will be shown
	\begin{gather*}
		\Lip ( f | W ) \leq \varepsilon / 2 + \lambda.
	\end{gather*}
	For this purpose suppose $i, j \in \nat$, $i \leq j$, $z \in W_i$, and
	$w \in W_j$. If $|z-w| < \delta_i$ then
	\begin{gather*}
		\{z,w\} \subset \Int K_{i+1}, \quad |f(z)-f(w)| \leq (
		\varepsilon/2 + \lambda ) |z-w|.
	\end{gather*}
	If $|z-w| \geq \delta_i$ then, choosing
	\begin{gather*}
		\text{$a \in A \cap K_i$ with $|z-a| < \eta \delta_i$} \quad
		\text{and} \quad \text{$b \in A \cap K_j$ with $|w-b| <
		\eta \delta_j$},
	\end{gather*}
	one obtains
	\begin{gather*}
		\{z,a\} \subset K_{i+1}, \quad \{ w,b \} \subset K_{j+1},
		\quad |a-b| \leq |z-w| + 2 \eta \delta_i, \\
		\begin{aligned}
			& |f(z)-f(w)| \leq |f(z) - f(a)| + |f(a) - f(b)| +
			|f(b) - f(w)| \\
			& \qquad \leq 2 ( \varepsilon/2 + \lambda ) \eta
			\delta_i + \lambda |a-b| \leq \eta \delta_i
			\varepsilon + 4 \eta \delta_i \lambda + |z-w| \,
			\lambda \\
			& \qquad \leq ( \eta ( \varepsilon + 4 \lambda ) +
			\lambda ) |z-w| \leq ( \varepsilon/2 + \lambda )
			|z-w|.
		\end{aligned}
	\end{gather*}

	Kirszbraun's extension theorem, see \cite[2.10.43]{MR41:1976}, yields
	a function $h : \rel^\adim \to \rel^l$ with $h | W = f | W$ and $\Lip
	h = \Lip ( f | W)$. If $l=1$ and $f \geq 0$ then possibly replacing
	$h$ by $h^+$ one may require $h \geq 0$. Using
	\cite[3.1.13]{MR41:1976} with $\Phi$ replaced by $\{ W, U \without A
	\}$, one constructs nonnegative functions $\phi_0 \in \mathscr{E}^0 (
	U )$ and $\phi_i \in \mathscr{D}^0 (U)$ for $i \in \nat$ such that
	\begin{gather*}
		\card \{ i \with K \cap \spt \phi_i \neq \varnothing \} < \infty
		\quad \text{whenever $K$ is compact subset of $U$},
		\\
		A \subset \Int ( \classification{U}{z}{\phi_0 (z) = 1 }),
		\qquad \spt \phi_0 \subset W, \qquad \spt \phi_i \subset
		U \without A \quad \text{for $i \in \nat$}, \\
		\tsum{i=0}{\infty} \phi_i (z) = 1 \quad \text{for $z \in U$}.
	\end{gather*}
	Now, define $g_0 = h$ and use convolution to obtain for each $i \in
	\nat$ a function $g_i : \rel^\adim \to \rel^l$ of class $1$ such that
	\begin{gather*}
		\text{if $l=1$ and $f \geq 0$ then $g_i \geq 0$}, \\
		\Lip g_i \leq \Lip h, \qquad ( \Lip \phi_i ) | ( g_i - h )(z)
		| \leq 2^{-i-1} \varepsilon \quad \text{for $z \in
		\rel^\adim$}.
	\end{gather*}
	Define $g = \sum_{i=0}^\infty \phi_i g_i$ and observe that $g$ is of
	class $1$. Also for $z,w \in U$
	\begin{gather*}
		g(z)-g(w) = \tsum{i=0}{\infty} \big ( \phi_i (z) (
		g_i(z)-g_i(w)) + ( \phi_i (z) - \phi_i (w) ) ( g_i (w) - h(w)
		) \big ), \\
		\Lip g \leq \varepsilon/2 + \Lip h = \varepsilon/2 + \Lip
		(f|W) \leq \varepsilon + \lambda.
	\end{gather*}
	Therefore one may take $Z = \Int ( \classification{U}{z}{\phi_0 (z) =
	1 })$.
\end{proof}
\begin{lemma} \label{lemma:submanifold_extension}
	Suppose $\vdim, \adim \in \nat$, $\vdim \leq \adim$, $M$ is an
	$\vdim$ dimensional submanifold of class $1$ of $\rel^\adim$, $Y$ is a
	normed space, and $f : M \to Y$ is of class $1$ relative to $M$.

	Then the following two statements hold:
	\begin{enumerate}
		\item \label{item:submanifold_extension:uni_diff} Denote by
		$\varrho (C,\delta)$ the supremum of all numbers
		\begin{gather*}
			| f(z)-f(a)-\left < \project{\Tan ( M,a)}(z-a), D f
			(a) \right > |/|z-a|
		\end{gather*}
		corresponding to $\{ z,a \} \subset C$ with $0 < |z-a| \leq
		\delta$ whenever $C$ is a compact subset of $M$ and $\delta >
		0$. Then $\varrho ( C, \delta ) \to 0$ as $\delta \to 0+$
		whenever $C$ is a compact subset of $M$.
		\item \label{item:submanifold_extension:extension} There
		exist an open subset $U$ of $\rel^\adim$ with $M \subset U$
		and a  function $g : U \to Y$ of class $1$ with $g|M = f$ and
		\begin{gather*}
			Dg(a) = Df(a) \circ \project{\Tan(M,a)} \quad
			\text{for $a \in M$}.
		\end{gather*}
	\end{enumerate}
\end{lemma}
\begin{proof}
	\eqref{item:submanifold_extension:uni_diff} is readily
	verified by use of \cite[3.1.19\,(1), 3.1.11]{MR41:1976}.

	Define $P_a : \rel^\adim \to Y$ by $P_a(z) = f (a) + \left <
	\project{\Tan(M,a)} (z-a), D f(a) \right >$ for $a \in M$ and $z \in
	\rel^\adim$. Noting
	\begin{gather*}
		DP_a(b)-DP_b(b) = D f(a) \circ \project{\Tan(M,a)} - Df(b)
		\circ \project{\Tan(M,b)} \quad \text{for $a,b \in M$}
	\end{gather*}
	and \eqref{item:submanifold_extension:uni_diff}, Whitney's extension
	theorem, see \cite[3.1.14]{MR41:1976}, yields for each closed set $A$
	with $A \subset M$ a function $g_A : \rel^\adim \to Y$ of class
	$1$ with $g_A|A = f |A$ and
	\begin{gather*}
		Dg_A(a) = Df(a) \circ \project{\Tan (M,a)} \quad \text{for $a
		\in A$}.
	\end{gather*}
	Therefore $g$ is constructable by use of a partition of unity.
\end{proof}
\begin{remark}
	A consequence is the following proposition: \emph{Whenever $M$ is a
	submanifold of class $1$ of $\rel^\adim$ there exists a function
	$r$ retracting some open subset of $\rel^\adim$ onto $M$ satisfying
	\begin{gather*}
		Dr(a) = \project{\Tan(M,a)} \quad \text{whenever $a \in M$};
	\end{gather*}}
	in fact, observing that one may assume $M$ to be connected, one
	obtains from \cite[3.1.20]{MR41:1976} a map $h$ of class $1$
	retracting some open subset of $\rel^\adim$ onto $M$ and from
	\eqref{item:submanifold_extension:extension} with $Y = \rel^\adim$ and
	$f = \id{M}$ an open set $U$ and a function $g$ and one may take $r =
	h \circ g$.

	In case $M$ is of class $2$ such $r$ could have alternately been
	constructed using a retraction which maps each point of a
	neighbourhood of $M$ onto its unique nearest point in $M$, cp. Federer
	\cite[4.12, 4.8]{MR0110078}.
\end{remark}
\begin{theorem} \label{thm:approx_stepanoff}
	Suppose $l, \adim \in \nat$, $A \subset \rel^\adim$, $f : A \to
	\rel^l$, and
	\begin{gather*}
		\ap \limsup_{z \to a} |f(z)-f(a)|/|z-a| < \infty
	\end{gather*}
	for $\mathscr{L}^\adim$ almost all $a \in A$, then for each
	$\varepsilon > 0$ there exists a map $g : \rel^\adim \to \rel^l$ of
	class $1$ such that
	\begin{gather*}
		\mathscr{L}^\adim ( A \without \{ z \with f(z) = g(z) \} ) <
		\varepsilon.
	\end{gather*}
	Moreover, if $l=1$ and $f \geq 0$ then one may take $g \geq 0$.
\end{theorem}
\begin{proof}
	With the exception of the postscript this is the content of
	\cite[3.1.13]{MR41:1976}.

	To prove the postscript choose $h : \rel^\adim \to \rel$ of class $1$
	such that
	\begin{gather*}
		\mathscr{L}^\adim ( A \without \{ z \with f(z) = h(z) \} ) <
		\varepsilon/3.
	\end{gather*}
	Since $A$ is $\mathscr{L}^\adim$ measurable by
	\cite[2.9.11]{MR41:1976}, one may choose closed sets $B$ and $C$ such
	that
	\begin{gather*}
		B \subset \classification{A}{z}{f(z) = 0}, \quad
		\mathscr{L}^\adim ( \classification{A}{z}{f(z) = 0} \without B
		) < \varepsilon/3, \\
		C \subset \classification{A}{z}{f(z) = h(z) > 0}, \quad
		\mathscr{L}^\adim ( \classification{A}{z}{f(z) = h (z) > 0}
		\without C ) < \varepsilon /3.
	\end{gather*}
	Noting $\mathscr{L}^\adim ( A \without ( B \cup C ) ) < \varepsilon$,
	one constructs $g$ from $h$ and the function mapping $\rel^\adim$ onto
	$\{0\}$ by use a partition of unity subordinate to $\{ \rel^\adim
	\without B, \rel^\adim \without C \}$, see \cite[3.1.13]{MR41:1976}.
\end{proof}
\begin{theorem} \label{thm:approximation_lip}
	Suppose $l, \vdim, \adim \in \nat$, $\vdim \leq \adim$, $U$ is an open
	subset of $\rel^\adim$, $V \in \RVar_\vdim ( U )$, $f : U \to \rel^l$
	is Lipschitzian, and $\varepsilon > 0$.

	Then there exists $g : U \to \rel^l$ of class $1$ satisfying
	\begin{gather*}
		\spt g \subset \classification{U}{z}{\dist (z, \spt f) \leq
		\varepsilon }, \quad \Lip g \leq \varepsilon + \Lip f, \\
		\| V \| ( \{ z \with f(z) \neq g(z) \} ) \leq
		\varepsilon.
	\end{gather*}
	Moreover, if $l=1$ and $f \geq 0$ then one may require $g \geq 0$.
\end{theorem}
\begin{proof}
	Define $Z = \classification{U}{z}{\dist(z,\spt f) <
	\varepsilon}$ and choose an $\vdim$ dimensional submanifold $M$ of
	class $1$ of $\rel^\adim$ such that
	\begin{gather*}
		M \subset Z, \quad \| V \| ( Z \without M ) \leq \varepsilon /
		2.
	\end{gather*}
	Using \cite[3.1.22]{MR41:1976}, \ref{thm:approx_stepanoff} and a
	partition of unity, one constructs an open subset $H$ of $U$ and a
	function $h : H \to \rel^l$ of class $1$ with
	\begin{gather*}
		\text{if $l=1$ and $f \geq 0$ then $h \geq 0$}, \quad \| V \|
		( M \without \{ z \with f (z) \neq h(z) \} ) < \varepsilon/2.
	\end{gather*}
	Noting \cite[2.9.11, 3.1.22, 3.1.5]{MR41:1976}, there exists a set $B$
	which is closed relative to $U$ satisfying
	\begin{gather*}
		B \subset H \cap M, \quad h|B = f|B, \quad \| V \| ( M
		\without B ) \leq \varepsilon/2, \\
		D (h|M) (z) = D (f|M) (z) \quad \text{for $z \in B$}.
	\end{gather*}
	By
	\ref{lemma:submanifold_extension}\,\eqref{item:submanifold_extension:extension}
	with $f$ replaced by $h$ one may assume
	\begin{gather*}
		Dh(z) = D (f|M) (z) \circ \project{\Tan(M,z)} \quad \text{for
		$z \in B$}.
	\end{gather*}
	Since $B$ and $U \without Z$ are closed relative to $U$ and $B \cap (
	U \without Z ) = \varnothing$, one may also assume $U \without Z \subset
	H$ and
	\begin{gather*}
		h(z) = 0 \quad \text{and} \quad Dh(z)=0 \quad \text{for $z \in
		U \without Z$}.
	\end{gather*}
	Defining $A = B \cup ( U \without Z )$, one notes $h|A = f|A$ and
	\begin{gather*}
		\sup \{ \Lip (h|A), \sup \| Dh \| \lIm A \rIm \} \leq \Lip f.
	\end{gather*}
	Therefore \ref{lemma:c1_extension} yields a function $g : U \to
	\rel^l$ of class $1$ satisfying
	\begin{gather*}
		\text{if $l=1$ and $f \geq 0$ then $g \geq 0$}, \quad g|A =
		h|A, \quad \Lip g \leq \varepsilon + \Lip f.
	\end{gather*}
	One now readily verifies that $g$ has the asserted properties.
\end{proof}
\begin{corollary} \label{corollary:approximation_lip}
	Suppose $l, \vdim, \adim \in \nat$, $\vdim \leq \adim$, $U$ is an open
	subset of $\rel^\adim$, $V \in \RVar_\vdim ( U)$, and $f : U \to
	\rel^l$ is locally Lipschitzian.

	Then there exists a sequence $f_i \in \mathscr{E} ( U, \rel^l )$ such
	that
	\begin{gather*}
		\lim_{i \to \infty} \sup \{ | (f-f_i)(z) | \with z \in K \cap
		\spt \| V \| \} + \eqLpnorm{\| V \| \restrict K}{q}{
		| \ap D (f-f_i)| } = 0
	\end{gather*}
	whenever $K$ is a compact subset of $U$ and $1 \leq q < \infty$, where
	``$\ap$'' denotes approximate differentiation with respect to $( \| V
	\|, \vdim )$. Moreover, if $l=1$ and $f \geq 0$ then one may require
	$f_i \geq 0$.
\end{corollary}
\begin{proof}
	Whenever $K$ is a compact subset of $U$ one may use
	\ref{thm:approximation_lip} to construct a sequence of functions $g_i
	: U \to \rel^l$ of class $\class{1}$ such that
	\begin{gather*}
		\sup \{ \Lip g_i \with i \in \nat \} < \infty, \quad \lim_{i
		\to \infty} \| V \| ( \classification{K}{z}{g_i(z) \neq f(z) }
		) = 0, \\
		\text{if $l=1$ and $f \geq 0$ then $g_i \geq 0$},
	\end{gather*}
	hence, in view of \cite[2.10.19\,(4)]{MR41:1976},
	\begin{gather*}
		\lim_{i \to \infty} \eqLpnorm{\| V \| \restrict K}{q}{
		| \ap D (f-g_i) | } = 0 \quad \text{for $1 \leq q < \infty$}.
	\end{gather*}
	The Arz\'el\`a-Ascoli theorem allows to require additionally that the
	functions $g_i$ converge locally uniformly to some Lipschitzian
	function $g : U \to \rel^l$ as $i \to \infty$ which then satisfies
	\begin{gather*}
		f(z) = g(z) \quad \text{for $z \in  \spt ( \| V \| \restrict K
		)$}
	\end{gather*}
	as the functions $g_i$ converge to $f$ in $\| V \| \restrict K$
	measure as $i \to \infty$.

	Taking a sequence of compact sets $K_i$ with $K_i \subset \Int
	K_{i+1}$ and $\bigcup \{ K_i \with i \in \nat \} = U$, one constructs
	$f_i \in \mathscr{E} ( U, \rel^l )$ satisfying
	\begin{gather*}
		| (f-f_i)(z) | \leq 1/i \quad \text{for $z \in  \spt ( \| V \|
		\restrict K_i )$}, \\
		\eqLpnorm{\| V \| \restrict K_i}{i} { | \ap D (f-f_i) | } \leq
		1/i, \quad \text{if $l=1$ and $f \geq 0$ then $f_i \geq 0$}
	\end{gather*}
	by use of the preceding paragraph and convolution. The conclusion now
	readily follows.
\end{proof}
\begin{definition} \label{def:strict_local_sobolev_space}
	Suppose $l, \vdim, \adim \in \nat$, $\vdim \leq \adim$, $U$ is an open
	subset of $\rel^\adim$, $V \in \RVar_\vdim (U)$, $\| \delta V \|$ is a
	Radon measure, and $1 \leq q \leq \infty$.

	Then the \emph{strict local Sobolev space with respect to $V$ and
	exponent $q$}, denoted by $\SWloc{q} ( V , \rel^l )$, is defined to be
	the vectorspace consisting of all $f \in \Lploc{q} ( \| V \| + \|
	\delta V \|, \rel^l )$ such that for some $F \in \Lploc{q} ( \| V \|,
	\Hom ( \rel^\adim, \rel^l ) )$ there exists a sequence of locally
	Lipschitzian functions $f_i : U \to \rel^l$ satisfying
	\begin{gather*}
		\eqLpnorm{(\| V \| + \| \delta V \|) \restrict K}{q}{f-f_i} +
		\eqLpnorm{\| V \| \restrict K}{q}{F-\derivative{V}{f_i} } \to
		0 \quad \text{as $i \to \infty$}
	\end{gather*}
	whenever $K$ is a compact subset of $U$. Abbreviate $\SWloc{q}
	(V,\rel) = \SWloc{q}(V)$.
\end{definition}
\begin{remark} \label{remark:strict_local_sobolev_space}
	Note $f \in \trunc (V,\rel^l)$ and $\derivative{V}{f} (z) = F(z)$ for
	$\| V \|$ almost all $z$. In view of \ref{example:lipschitzian} and
	\ref{corollary:approximation_lip} one may require the functions $f_i$
	to be of class $\class{\infty}$ if $q < \infty$.

	Also note, if $f \in \Lploc{q} ( \| V \| + \| \delta V \|, \rel^l )$,
	$F \in \Lploc{q} ( \| V \|, \Hom ( \rel^\adim, \rel^l ) )$, and $f_i
	\in \SWloc{q} ( V, \rel^l )$ with
	\begin{gather*}
		\eqLpnorm{(\| V \| + \| \delta V \|) \restrict K}{q}{f-f_i} +
		\eqLpnorm{\| V \| \restrict K}{q}{F-\derivative{V}{f_i} } \to
		0 \quad \text{as $i \to \infty$}
	\end{gather*}
	whenever $K$ is a compact subset of $U$, then $f \in \SWloc{q}
	(V,\rel^l)$ with
	\begin{gather*}
		\derivative{V}{f} (z) = F(z) \quad \text{for $\| V \|$ almost
		all $z$}.
	\end{gather*}
	As a consequence one obtains the following three propositions:

	If $f \in \SWloc{q} ( V, \rel^l )$ and $g : U \to \rel$
	is locally Lipschitzian, then $gf \in \SWloc{q} (V,\rel^l)$ and
	\begin{gather*}
		\derivative{V}{(gf)} (z) = \derivative{V}{g}(z)\,f(z) + g(z)
		\derivative{V}{f} (z) \quad \text{for $\| V \|$ almost all
		$z$}.
	\end{gather*}

	If $f \in \SWloc{q} ( V, \rel^l )$, $k \in \nat$, and $\phi : \rel^l
	\to \rel^k$ is of class $\class{1}$ with $\Lip \phi < \infty$, then
	$\phi \circ f \in \SWloc{q} ( V, \rel^l )$ and
	\begin{gather*}
		\derivative{V}{( \phi \circ f )} (z) = D \phi ( f(z) ) \circ
		\derivative{V}{f} (z) \quad \text{for $\| V \|$ almost all
		$z$}.
	\end{gather*}

	If $f \in \SWloc{q} (V)$ and $q < \infty$, then $\{ f^+, f^-, |f| \}
	\subset \SWloc{q} (V)$; in fact one may use functions $f_i$ as in the
	definition of $\SWloc{q} (V)$ in conjunction with
	\ref{lemma:basic_v_weakly_diff} and
	\ref{example:composite}\,\eqref{item:composite:1d}, noting
	\ref{thm:addition}\,\eqref{item:addition:zero}.
\end{remark}
\begin{definition} \label{def:local_sobolev_space}
	Suppose $l, \vdim, \adim \in \nat$, $\vdim \leq \adim$, $U$ is an open
	subset of $\rel^\adim$, $V \in \RVar_\vdim (U)$, $\| \delta V \|$ is a
	Radon measure which is absolutely continuous with respect to $\| V
	\|$, and $1 \leq q \leq \infty$.

	Then the \emph{local Sobolev space with respect to $V$ and exponent
	$q$}, denoted by $\VSobloc{q} ( V , \rel^l )$, is defined to be the
	vectorspace consisting of all $f \in \Lploc{q} ( \| V \|, \rel^l )$
	such that for some $F \in \Lploc{q} ( \| V \|, \Hom ( \rel^\adim,
	\rel^l ) )$ there exists a sequence of locally Lipschitzian functions
	$f_i : U \to \rel^l$ satisfying
	\begin{gather*}
		\eqLpnorm{\| V \| \restrict K}{q}{f-f_i} + \eqLpnorm{\| V \|
		\restrict K}{q}{F-\derivative{V}{f_i} } \to 0 \quad \text{as
		$i \to \infty$}
	\end{gather*}
	whenever $K$ is a compact subset of $U$. Abbreviate $\VSobloc{q}
	(V,\rel) = \VSobloc{q} (V)$.
\end{definition}
\begin{remark} \label{remark:local_sobolev_space}
	A remark analogous to \ref{remark:strict_local_sobolev_space} holds;
	in fact, it sufficient to replace $\| V \| + \| \delta V \|$ and
	$\SWloc{q}$ by $\| V \|$ and $\VSobloc{q}$.
\end{remark}
\begin{definition} \label{def:strict_sobolev_space}
	Suppose $l$, $\vdim$, $\adim$, $U$, $V$, and $q$ are as in
	\ref{def:strict_local_sobolev_space}.

	Then let
	\begin{gather*}
		\SWnorm{q}{V}{f} = \eqLpnorm{\| V \| + \| \delta V \|}{q}{f} +
		\Lpnorm{\| V \|}{q}{ \derivative{V}{f} } \quad \text{for $f
		\in \trunc (V,\rel^l)$}
	\end{gather*}
	and define the \emph{strict Sobolev space with respect to $V$ and
	exponent $q$} by
	\begin{gather*}
		\SW{q} (V,\rel^l ) = \classification{\SWloc{q} (V,\rel^l)}{f}{
		\SWnorm{q}{V}{f} < \infty}.
	\end{gather*}
	Abbreviate $\SW{q} (V,\rel) = \SW{q}(V)$.
\end{definition}
\begin{remark} \label{remark:strict_sobolev_space}
	Note \emph{$\SW{q} ( V, \rel^l )$ is a $\SWnorm{q}{V}{\cdot}$ complete
	vectorspace}. Moreover, \emph{the vector subspaces
	\begin{gather*}
		\text{$\SW{q} ( V, \rel^l ) \cap \mathscr{E} (U, \rel^l)$ if
		$q < \infty$}, \\
		\text{and $\classification{\SW{q} (V,\rel^l) }{f}{\text{$f$ is
		locally Lipschitzian} }$ if $q = \infty$}
	\end{gather*}
	are $\SWnorm{q}{V}{\cdot}$ dense in $\SW{q} (V,\rel^l)$}; in fact
	suppose $\varepsilon > 0$ and $g \in \SW{q} ( V, \rel^l )$, choose a
	sequence $\phi_i$ forming a partition of unity on $U$ associated with
	$\{ U \}$ as in \cite[3.1.13]{MR41:1976} and $f_i : U \to \rel^l$
	locally Lipschitzian with
	\begin{multline*}
		( 1 + \sup \im | D \phi_i | ) \eqLpnorm{(\| V \| + \| \delta V
		\|) \restrict \spt \phi_i}{q}{g-f_i} \\
		+ \eqLpnorm{\| V \| \restrict \spt \phi_i
		}{q}{\derivative{V}{g} - \derivative{V}{f_i}} \leq \varepsilon
		2^{-i}
	\end{multline*}
	and $f_i \in \mathscr{E} (U,\rel^l)$ if $q < \infty$, note $f =
	\sum_{i=1}^{\infty} \phi_i f_i$ is locally Lipschitzian with $f \in
	\mathscr{E} ( U, \rel^l )$ if $q < \infty$, $f \in \SWloc{q}
	(V,\rel^l)$, and $g-f = \sum_{i=1}^\infty \phi_i (g-f_i)$, and compute
	\begin{gather*}
		\derivative{V}{(g-f)}(z) = \tsum{i=1}{\infty}
		\derivative{V}{\phi_i}(z)\,(g-f_i)(z) + \phi_i(z)
		\derivative{V}{(g-f_i)}(z)
	\end{gather*}
	for $\| V \|$ almost all $z$, hence $\SWnorm{q}{V}{g-f} \leq
	\varepsilon$.
\end{remark}
\begin{definition} \label{def:sobolev_space}
	Suppose $l$, $\vdim$, $\adim$, $U$, $V$, and $q$ are as in
	\ref{def:local_sobolev_space}.

	Then let
	\begin{gather*}
		\VSobnorm{q}{V}{f} = \Lpnorm{\| V \|}{q}{f} + \Lpnorm{\| V
		\|}{q}{ \derivative{V}{f} } \quad \text{for $f \in \trunc
		(V,\rel^l)$}
	\end{gather*}
	and define the \emph{Sobolev space with respect to $V$ and exponent
	$q$} by
	\begin{gather*}
		\VSob{q} (V,\rel^l ) = \classification{\VSobloc{q}
		(V,\rel^l)}{f}{ \VSobnorm{q}{V}{f} < \infty}.
	\end{gather*}
	Abbreviate $\VSob{q} (V,\rel) = \VSob{q} (V)$.
\end{definition}
\begin{remark} \label{remark:sobolev_space}
	Replacing $\| V \| + \| \delta V \|$ and $\SW{q}$ by $\| V \|$ and
	$\VSob{q}$ in \ref{remark:strict_sobolev_space} one obtains:
	\emph{$\VSob{q} ( V, \rel^l )$ is a $\VSobnorm{q}{V}{\cdot}$ complete
	vectorspace such that the vector subspaces
	\begin{gather*}
		\text{$\VSob{q} ( V, \rel^l ) \cap \mathscr{E} ( U, \rel^l )$
		if $q < \infty$}, \\
		\text{and $\classification{\VSob{q} (V, \rel^l)}{f}{ \text{$f$
		is locally Lipschitzian} }$ if $q = \infty$}
	\end{gather*}
	are $\VSobnorm{q}{V}{\cdot}$ dense in $\VSob{q} (V,\rel^l)$}. Observe
	that $\mathscr{E} ( U, \rel^l )$ may be replaced by $\classification{
	\mathscr{E} ( U, \rel^l ) }{f}{ \text{$\spt f$ is bounded} }$ if $q <
	\infty$.
\end{remark}
\begin{remark} \label{remark:other_sobolev_space}
	Here, the relation to the notion of Sobolev space introduced by
	Bouchitt\'e, Buttazzo and Seppecher in \cite{MR1424348} is discussed.

	Suppose $M$ is the set of Radon measures $\mu$ on $\rel^\adim$ with
	$\mu ( \rel^\adim ) < \infty$ and $C = \{ ( a, \cball{a}{r} ) \with
	\text{$a \in \rel^\adim$, $0 < r < \infty$} \}$. Define $X_\mu^p$ to
	be the vectorspace consisting of all $g \in \Lp{p} ( \mu, \rel^\adim
	)$ such that for some $0 \leq \alpha < \infty$ there holds
	\begin{gather*}
		\tint{}{} \left < g, D \theta \right > \ud \mu \leq
		\alpha \, \Lpnorm{\mu}{p/(p-1)}{\theta} \quad \text{for
		$\theta \in \mathscr{D}^0 ( \rel^\adim )$}
	\end{gather*}
	whenever $\mu \in M$ and $1 < p < \infty$. Let $T_\mu^q ( z )$ denote
	the vectorspace of all $y \in \rel^\adim$ such that
	\begin{gather*}
		y = ( \mu, C ) \aplim_{\xi \to z} g ( \xi ) \quad \text{for
		some $g \in X_\mu^{q/(q-1)}$}
	\end{gather*}
	whenever $\mu \in M$, $1 < q < \infty$, and $z \in \rel^\adim$.
	Observe that the function $T_\mu^q$ is a representative of the
	equivalence class bearing that name in \cite[p.~38]{MR1424348} by
	\cite[2.9.13]{MR41:1976}. In \ref{example:tangent_difference} it will
	be shown that it may happen that for some $1 < q < r < \infty$ there
	occurs
	\begin{gather*}
		T_\mu^q (z) \neq T_\mu^r (z) \quad \text{for $\mu$ almost all
		$z$}
	\end{gather*}
	even if $\mu$ is a weight of a rectifiable varifold. However, if
	$\adim \geq \vdim \in \nat$, $V \in \RVar_\vdim ( \rel^\adim )$, $\| V
	\| ( \rel^\adim ) < \infty$, $1 < q < \infty$, $0 \leq \alpha <
	\infty$, and $( \delta V) ( g ) \leq \alpha \Lpnorm{\| V \|}{q}{g}$
	for $g \in \mathscr{D} ( \rel^\adim, \rel^\adim )$, then
	\begin{gather*}
		T^q_{\| V\|} (z) = \Tan^\vdim ( \| V \|, z ) \quad \text{for
		$\| V \|$ almost all $z$},
	\end{gather*}
	cp.~e.g.~Fragal\`a and Mantegazza \cite[Lemma 2.4, Theorem
	3.8]{MR1686704}, hence in this case the Sobolev space $W_{\| V
	\|}^{1,q}$ with notion of derivative $D_{\| V \|}$ and norm $\| \cdot
	\|_{1,q,\| V \|}$ defined by Bouchitt\'e, Buttazzo, and Seppecher in
	\cite[p.~39]{MR1424348} is isomorphic to $\VSob{q} ( V )$ with
	$\derivative{V}{}$ and $\VSobnorm{q}{V}{\cdot}$ by
	\ref{remark:sobolev_space}.
\end{remark}
\begin{example} \label{example:tangent_difference}
	Here, the example announced in \ref{remark:other_sobolev_space} is
	constructed.

	Suppose $1 < q < r < \infty$. Choose an $\mathscr{L}^1$ measurable
	function $f : \rel \to \rel$ such that
	\begin{gather*}
		0 < f(z) \leq 1 \quad \text{for $z \in \oball{0}{1}$}, \qquad
		f(z) = 0 \quad \text{for $z \in \rel \without \oball{0}{1}$},
		\\
		\tint{\oball{0}{1}}{} f^{1/(1-r)} \ud \mathscr{L}^1 z <
		\infty, \quad \tint{U}{} f^{1/(1-q)} \ud \mathscr{L}^1 =
		\infty
	\end{gather*}
	whenever $U$ is a nonempty open subset of $\oball{0}{1}$, let $\mu
	= f \mathscr{L}^1 \restrict \oball{0}{1}$ and note
	$\Lpnorm{\mu}{r/(r-1)}{1/f} < \infty$. One readily verifies
	\begin{gather*}
		\{ g \with \text{$g= h/f$ for some $h \in \mathscr{D}^0 (
		\rel)$ with $\spt h \subset \oball{0}{1}$} \} \subset
		X_\mu^{r/(r-1)},
	\end{gather*}
	hence $T_\mu^r ( z ) = \rel$ for $z \in \rel \cap \oball{0}{1}$.

	Next, suppose $g \in X_\mu^{q/(q-1)}$. Then $fg | \oball{0}{1}$ is
	weakly differentiable, hence by
	\cite[4.5.9\,(30)\,(\printRoman{1})\,(\printRoman{5}),
	4.5.16]{MR41:1976} there exists a continuous function $h : \rel \cap
	\oball{0}{1} \to \rel$ with $f(z) g (z) = h(z)$ for $\mathscr{L}^1$
	almost all $z \in \oball{0}{1}$. This yields
	\begin{gather*}
		\tint{\oball{0}{1}}{} |h|^{q/(q-1)} f^{1/(1-q)} \ud
		\mathscr{L}^1 = \tint{}{} |g|^{q/(q-1)} \ud \mu < \infty, \\
		h(z) = 0 \quad \text{for $z \in \oball{0}{1}$}, \quad g(z) = 0
		\quad \text{for $\mu$ almost all $z$},
	\end{gather*}
	hence $T_\mu^q (z) = \{ 0 \}$ for $z \in \rel$.

	Taking $\nu = \mathscr{L}^1 \restrict \oball{0}{1}$, this example also
	shows that
	\begin{gather*}
		T_\nu^q ( z ) \not \subset T_\mu^q (z) \quad \text{for $z \in
		\rel \cap \oball{0}{1}$}
	\end{gather*}
	despite $\nu$ is absolutely continuous with respect to $\mu$, contrary
	to the assertion in Bouchitt\'e, Buttazzo and Seppecher \cite[Remark
	2.2\,(iii)]{MR1424348}.
\end{example}
\section{Preliminaries}
In this section the dual seminorms $\nuqnorm{r}{V}{\cdot}$ and
$\nuqastnorm{q}{V}{\cdot}$ are defined, see \ref{def:nuqnorm} and
\ref{def:nuqastnorm}.
\begin{lemma} \label{lemma:coarea}
	Suppose $\vdim, \adim \in \nat$, $\vdim \leq \adim$, $U$ is an open
	subset of $\rel^\adim$, $V \in \RVar_\vdim ( U )$, $\| \delta V \|$ is
	a Radon measure, $\eta ( V; \cdot ) : U \to \mathbf{S}^{\adim-1}$ is
	$\| \delta V \|$ measurable,
	\begin{gather*}
		\delta V ( \theta ) = \tint{}{} \theta \bullet \eta ( V ;
		\cdot ) \ud \| \delta V \| \quad \text{whenever $\theta \in
		\mathscr{D} ( U, \rel^\adim )$}
	\end{gather*}
	(see Allard \cite[4.3]{MR0307015}), $f : U \to \rel$ is locally
	Lipschitzian and whenever $t \in \rel$ and $\theta \in \mathscr{D} (
	U, \rel^\adim )$
	\begin{gather*}
		E(t) = \classification{U}{z}{f(z) > t}, \\
		T(t) (\theta) = \tint{E(t)}{} \theta \bullet \eta ( V ; \cdot)
		\ud \| \delta V \| - \tint{E(t) \times \grass{\adim}{\vdim}}{}
		\project{S} \bullet D \theta (z) \ud V(z,S),
	\end{gather*}
	denote by ``$\ap\!$'' approximate differentiation with respect to $(
	\| V \|, \vdim)$ and abbreviate $R(z) = \project{\Tan^\vdim ( \| V \|,
	z)}$ whenever $z \in U$ with $\Tan^\vdim ( \| V \|, z ) \in
	\grass{\adim}{\vdim}$.

	Then for $\mathscr{L}^1$ almost all $t$
	\begin{gather*}
		\begin{aligned}
			& T(t)(\theta) \\
			& \quad = \tint{\classification{U}{z}{f(z)=t}}{} |
			\ap Df(z) |^{-1} \left < R(z) (\theta(z)), \ap Df (z)
			\right > \density^\vdim ( \| V \|, z ) \ud
			\mathscr{H}^{\vdim-1} z
		\end{aligned}
	\end{gather*}
	whenever $\theta \in \mathscr{D} ( U, \rel^\adim )$ and
 	\begin{gather*}
		\tint{}{} T(t)(\theta) \ud \mathscr{L}^1 t = \tint{}{} \left <
		R(z) ( \theta (z)), \ap D f (z) \right > \ud \| V \| z \quad
		\text{for $\theta \in \mathscr{D} ( U, \rel^\adim )$}, \\
		\tint{}{} \| T (t) \| ( g ) \ud \mathscr{L}^1 t = \tint{}{} g
		| \ap Df| \ud \| V \| \quad \text{for $g \in \mathscr{K}
		(U)$}.
	\end{gather*}
\end{lemma}
\begin{proof}
	Whenever $t \in \rel$ define $f_t = \inf \{ f, t \}$. From
	\cite[4.5\,(4)]{snulmenn.decay} one obtains for $\mathscr{L}^1$
	almost all $t \in \rel$, $0 < h < \infty$ and $\theta \in
	\mathscr{D} ( U, \rel^\adim )$
	\begin{gather*}
		\begin{aligned}
			\tint{}{} (f_{t+h}-f_t) \theta \bullet \eta ( V ;
			\cdot ) \ud \| \delta V \| & = \tint{}{} (f_{t+h}-f_t)
			(z) \project{S} \bullet D \theta (z) \ud V (z,S) \\
			& \phantom{=}\ + \tint{E_t \without E_{t+h}}{} \left <
			R(z) ( \theta (z) ), \ap D f (z) \right > \ud \| V \|
			z,
		\end{aligned}
	\end{gather*}
	hence dividing by $h$ and considering the limit $h \to 0+$ with the
	help of the coarea formula \cite[3.2.22]{MR41:1976} in conjunction
	with \cite[2.9.9]{MR41:1976} one infers the asserted expression for
	$T(t)(\theta)$ for $\mathscr{L}^1$ almost all $t$, in particular
	\begin{gather*}
		\| T (t) \| (g) = \tint{\classification{U}{z}{f(z)=t}}{} g (z)
		\density^\vdim ( \| V \|, z ) \ud \mathscr{H}^{\vdim-1} z
		\quad \text{for $g \in \mathscr{K} ( U )$}
	\end{gather*}
	for such $t$. Now, the conclusion follows from
	\cite[3.2.22]{MR41:1976}.
\end{proof}
\begin{definition} \label{def:nuqnorm}
	Suppose $l, \vdim, \adim \in \nat$, $\vdim \leq \adim$, $1 \leq r \leq
	\infty$, $U$ is an open subset of $\rel^\adim$, and $V \in \Var_\vdim
	( U )$.

	Then one defines the seminorm $\nuqnorm{r}{V}{\cdot}$ by
	\begin{gather*}
		\nuqnorm{r}{V}{\theta} = \Lpnorm{V}{r}{f} \quad \text{where
		$f(z,S) = D \theta (z) \circ \project{S}$ for $(z,S) \in U
		\times \grass{\adim}{\vdim}$}
	\end{gather*}
	whenever $\theta : U \to \rel^l$ is of class $1$. In case $V \in
	\RVar_\vdim (U)$ the domain of $\nuqnorm{r}{V}{\cdot}$ may be extended
	by requiring
	\begin{gather*}
		\nuqnorm{r}{V}{\theta} = \Lpnorm{\| V \|}{r}{|\ap D\theta|}
	\end{gather*}
	whenever $\theta : U \to \rel^l$ is a locally Lipschitzian function
	with compact support where ``$\ap D$'' denotes approximate
	differentiation with respect to $( \| V \|, \vdim )$.
\end{definition}
\begin{remark} \label{remark:nuqnorm}
	If $r < \infty$ one readily verifies by means of convolution that
	$\mathscr{D} (U,\rel^l)$ is $\|V\|_{(\infty)} + \nuqnorm{r}{V}{\cdot}$
	dense in the space $C$ of all functions of class $1$ with compact
	support mapping $U$ into $\rel^l$ and that if additionally $l=1$ then
	$\mathscr{D}^0 (U)^+$ is $\| V \|_{(\infty)} + \nuqnorm{r}{V}{\cdot}$
	dense in $\classification{C}{\theta}{\theta \geq 0}$.
\end{remark}
\begin{theorem} \label{corollary:nuq_density}
	Suppose $l, \vdim, \adim \in \nat$, $\vdim \leq \adim$, $U$ is an open
	subset of $\rel^\adim$, $K$ is a compact subset of $U$, $V \in
	\RVar_\vdim ( U )$, $f : U \to \rel^l$ is locally Lipschitzian
	with
	\begin{gather*}
		\spt \| V \| \cap \spt f \subset \Int K,
	\end{gather*}
	$1 \leq r < \infty$, and $\varepsilon > 0$.

	Then there exists $g \in \mathscr{D} ( U, \rel^l )$ with
	\begin{gather*}
		\spt g \subset K, \quad \Lpnorm{\| V \|}{\infty}{f-g} +
		\nuqnorm{r}{V}{f-g} \leq \varepsilon.
	\end{gather*}
	If $l=1$ and $f \geq 0$ then one may require $g \geq 0$.
\end{theorem}
\begin{proof}
	It is sufficient to construct $g : U \to \rel^l$ of class $1$, cp.
	\ref{remark:nuqnorm}. Choose $\eta \in \mathscr{D}^0 ( U )^+$
	satisfying
	\begin{gather*}
		\spt \eta \subset \Int K, \quad \spt \| V \| \cap \spt f
		\subset \Int \{ z \with \eta (z) = 1 \},
	\end{gather*}
	define $h = \eta f$ and note $h$ is Lipschitzian with
	\begin{gather*}
		h(z) = f (z) \quad \text{and} \quad \ap D h (z) = \ap D f (z)
		\qquad \text{for $\| V \|$ almost all $z$}.
	\end{gather*}
	Suppose $\varepsilon_i > 0$ with $\varepsilon_i \to 0$ as $i \to
	\infty$. Then one obtains from \ref{thm:approximation_lip} a sequence
	of functions $g_i : U \to \rel^l$ of class $1$ such that
	\begin{gather*}
		\spt g_i \subset K, \quad \Lip g_i \leq \varepsilon_i + \Lip
		h, \\
		\| V \| ( \classification{U}{z}{g_i (z) \neq h(z)} ) \leq
		\varepsilon_i, \quad \text{if $l=1$ and $f \geq 0$ then $g_i
		\geq 0$}.
	\end{gather*}
	Possibly passing to a subsequence, one may assume that $g_i$ converges
	uniformly to some function $G : U \to \rel^l$. Since $g_i$ converge to
	both $h$ and $f$ in $\| V \|$ measure $\Lpnorm{\| V \|}{\infty}{g_i -
	f} \to 0$ as $i \to \infty$ and the conclusion follows.
\end{proof}
\begin{definition} \label{def:nuqastnorm}
	Suppose $l, \vdim, \adim \in \nat$, $\vdim \leq \adim$, $1 \leq q \leq
	\infty$, $1 \leq r \leq \infty$, $1/q + 1/r = 1$, $U$ is an open
	subset of $\rel^\adim$, and $V \in \Var_\vdim ( U )$.

	Then one defines the seminorm $\nuqastnorm{q}{V}{\cdot}$ by
	\begin{gather*}
		\nuqastnorm{q}{V}{T} = \sup \{ T ( \theta ) \with
		\text{$\theta \in \mathscr{D} (U,\rel^l)$ and
		$\nuqnorm{r}{V}{\theta} \leq 1$} \}
	\end{gather*}
	whenever $T : \mathscr{D} ( U, \rel^l ) \to \rel$ is a linear map.
\end{definition}
\begin{remark} \label{remark:nuqastnorm}
	If $\nuqastnorm{q}{V}{T} < \infty$ then $T$ is $\nuqnorm{r}{V}{\cdot}$
	continuous and $T \in \mathscr{D}' ( U, \rel^l)$. If additionally $q
	> 1$ and $V \in \RVar_\vdim ( U )$ then $T$ admits a unique
	$\nuqnorm{r}{V}{\cdot}$ continuous extension to the space of all
	locally Lipschitzian functions $\theta : U \to \rel^l$ such that $\spt
	\| V \| \cap \spt \theta$ is compact by \ref{corollary:nuq_density}.
\end{remark}
\begin{theorem}
	Suppose $l$, $\vdim$, $\adim$, $q$, $U$, $V$, and $T$ are as in
	\ref{def:nuqastnorm}, $q>1$, and $\nuqastnorm{q}{V}{T}< \infty$.

	Then there exists $g \in \Lp{q} ( V, \Hom ( \rel^\adim, \rel^l ) )$
	such that
	\begin{gather*}
		T(\theta) = \tint{}{} ( D \theta (z) \circ \project{S} )
		\bullet g(z,S) \ud V (z,S) \quad \text{whenever $\theta \in
		\mathscr{D} ( U, \rel^l )$}, \\
		\Lpnorm{V}{q}{g} = \nuqastnorm{q}{V}{T}.
	\end{gather*}
	Moreover, any such $g$ satisfies
	\begin{gather*}
		g (z,S) \circ \project{S} = g(z,S) \quad \text{for $V$ almost
		all $(z,S)$}.
	\end{gather*}
\end{theorem}
\begin{proof}
	Let $1 \leq r \leq \infty$ with $1/q+1/r=1$, abbreviate $X = \Lp{r}(
	V, \Hom ( \rel^\adim, \rel^l ) )$, define $L : \mathscr{D} ( U, \rel^l
	) \to X$ by
	\begin{gather*}
		L ( \theta ) (z,S) = D \theta (z) \circ \project{S} \quad
		\text{whenever $\theta \in \mathscr{D} ( U, \rel^l )$, $z \in
		U$, $S \in \grass{\adim}{\vdim}$},
	\end{gather*}
	and note
	\begin{gather*}
		T ( \theta ) \leq \nuqastnorm{q}{V}{T}
		\Lpnorm{V}{r}{L(\theta)} \quad \text{for $\theta \in
		\mathscr{D} ( U, \rel^l )$}.
	\end{gather*}
	Therefore Hahn-Banach's theorem, see \cite[2.4.12]{MR41:1976},
	furnishes a linear map $\mu : X \to \rel$ with
	\begin{gather*}
		\mu | \im L = T \circ L^{-1}, \quad \mu \leq
		\nuqastnorm{q}{V}{T} \Lpnorm{V}{r}{\cdot}.
	\end{gather*}
	Using \cite[2.5.7]{MR41:1976}, one constructs $g \in \Lp{q} ( V, \Hom (
	\rel^\adim, \rel^l ) )$ with
	\begin{gather*}
		\mu (x) = \tint{}{} x \bullet g \ud V \quad \text{for $x \in
		X$},
	\end{gather*}
	hence
	\begin{gather*}
		T ( \theta ) = \tint{}{} ( D \theta (z) \circ \project{S} )
		\bullet g(z,S) \ud V (z,S) \quad \text{for $\theta \in
		\mathscr{D} (U,\rel^l )$}.
	\end{gather*}
	Adapting \cite[2.4.16]{MR41:1976}, one obtains
	\begin{gather*}
		\Lpnorm{V}{q}{g} = \sup \{ \mu (h) \with \text{$h \in X$ and
		$\Lpnorm{V}{r}{h} \leq 1$} \} \leq \nuqastnorm{q}{V}{T}.
	\end{gather*}

	Noting
	\begin{gather*}
		| \tau |^2 = | \tau \circ \project{S} |^2 + | \tau \circ
		\perpproject{S} |^2, \quad
		( \sigma \circ \project{S} ) \bullet \tau = ( \sigma \circ
		\project{S} ) \bullet ( \tau \circ \project{S} )
	\end{gather*}
	for $\sigma, \tau \in \Hom ( \rel^\adim, \rel^l )$, $S \in
	\grass{\adim}{\vdim}$ and defining $h \in \Lp{q} ( V, \Hom (
	\rel^\adim , \rel^l ) )$ by
	\begin{gather*}
		h(z,S) = g(z,S) \circ \project{S} \quad \text{for $(z,S) \in
		U \times \grass{\adim}{\vdim}$},
	\end{gather*}
	H\"older's inequality implies
	\begin{gather*}
		\nuqastnorm{q}{V}{T} \leq \Lpnorm{V}{q}{h} \leq
		\Lpnorm{V}{q}{g},
	\end{gather*}
	hence
	\begin{gather*}
		g(z,S) \circ \project{S} = g(z,S) \quad \text{for $V$ almost
		all $(z,S)$}.
	\end{gather*}

	Since the argument of the preceding paragraph is applicable to any $g$
	satisfying the main part of the conclusion, the proof is complete.
\end{proof}
\begin{remark}
	If $V \in \RVar_\vdim ( U )$ the value of the extension of $T$
	described in \ref{remark:nuqastnorm} at $\theta$ equals
	\begin{gather*}
		\tint{}{} \ap D \theta (z) \bullet ( g (z,S)|S ) \ud V (z,S)
	\end{gather*}
	whenever $\theta : U \to \rel^l$ is a Lipschitzian function with
	compact support where ``$\ap D$'' denotes approximate differentiation
	with respect to $( \| V \|, \vdim )$.
\end{remark}
\section{Maximum estimates} \label{sec:local_max}
This section concerns varifolds with critical integrability their first
variation ($p=\vdim$ in \ref{miniremark:situation_general}). The main results
of this section are the global and local maximum estimates for subsolutions
obtained in \ref{thm:global_maximum_estimates} and
\ref{thm:local_maximum_estimates}. (The results concerning functions in
\ref{thm:sobolev_embedding}--\ref{corollary:sobolev}, \ref{corollary:emb_loo}
are essentially superseded by the newer results in Section
\ref{sec:embeddings}.)
\begin{theorem} \label{thm:sobolev_embedding}
	Suppose $\vdim$, $\adim$, $p$, $U$, $V$, and ``$\ap\!$'' are as in
	\ref{miniremark:situation_general}, $p=1$, $\beta = \infty$ if $\vdim
	=1$ and $\beta = \vdim/(\vdim-1)$ if $\vdim > 1$, $f : U \to \rel$ is
	locally Lipschitzian, and
	\begin{gather*}
		\text{$\classification{\spt \| V \|}{z}{|f(z)| \geq t}$ is
		compact whenever $0 < t < \infty$}.
	\end{gather*}

	Then
	\begin{gather*}
		\Lpnorm{\| V \|}{\beta}{f} \leq \isoperimetric{\vdim} \big (
		\Lpnorm{\| V \|}{1}{| \ap D f |} + \Lpnorm{\| \delta V
		\|}{1}{f} \big ).
	\end{gather*}
\end{theorem}
\begin{proof}
	This is a slight extension of Allard \cite[7.3]{MR0307015} and Michael
	and Simon \cite[Theorem 2.1]{MR0344978}.

	First, \emph{the case $\spt \| V \| \cap \spt f$ is compact} will be
	treated. Assume $f \geq 0$. Define $E(t)$ and $T(t)$ as in
	\ref{lemma:coarea}. Let $f_t = \inf \{ f, t \}$ and $V_t = V \restrict
	E(t) \times \grass{\adim}{\vdim}$ for $0 < t < \infty$.  There holds
	\begin{gather*}
		\| V_t \| ( U )^{1/\beta} \leq \isoperimetric{\vdim} \| \delta
		V_t \| ( U ) \leq \isoperimetric{\vdim} \big ( \| T (t) \| ( U
		) + \| \delta V \| (E(t)) \big )
	\end{gather*}
	for $0 < t < \infty$ where $0^0 = 0$. If $\vdim = 1$ then the
	assertion of the present case follows from \ref{lemma:coarea}. If
	$\vdim > 1$, defining $\phi : \{ t : 0 < t < \infty \} \to \rel$ by
	$\phi (t) = \Lpnorm{\| V \|}{\beta}{f_t}$ for $0 < t < \infty$,
	Minkowski's inequality implies
	\begin{gather*}
		0 \leq \phi (t+h) - \phi (t) \leq \Lpnorm{\| V
		\|}{\beta}{f_{t+h}-f_t} \leq h \| V \| ( E(t))^{1/\beta}
	\end{gather*}
	for $0 < t < \infty$ and $0 < h < \infty$. Therefore $\phi$ is
	Lipschitzian and
	\begin{gather*}
		0 \leq \phi'(t) \leq \| V_t \| ( U )^{1/\beta} \quad \text{for
		$\mathscr{L}^1$ almost all $0 < t < \infty$},
	\end{gather*}
	hence $\Lpnorm{\| V \|}{\beta}{f} = \lim_{t \to \infty} \phi (t) =
	\tint{0}{\infty} \phi' \ud \mathscr{L}^1$ and \ref{lemma:coarea}
	implies the assertion of the present case.

	In \emph{the general case} suppose $0 < t < \infty$, apply the special
	case with $f(z)$ replaced by $\sup \{ |f(z)|-t, 0 \}$ and consider the
	limit $t \to 0+$.
\end{proof}
\begin{remark} \label{remark:embedding_1d}
	If $\vdim = 1$, $0 \leq \alpha < \isoperimetric{\vdim}^{-1}$, and
	$\psi ( \classification{U}{z}{f(z) \neq 0} ) \leq \alpha$ with $\psi$
	as in \ref{miniremark:situation_general} then
	\begin{gather*}
		\Lpnorm{\| V \|}{\infty}{f} \leq \isoperimetric{\vdim} ( 1 -
		\alpha \isoperimetric{\vdim} )^{-1} \Lpnorm{\| V \|}{1}{| \ap
		Df| };
	\end{gather*}
	in fact assuming $\spt \| V \| \cap \spt f$ to be compact this follows
	from \ref{thm:sobolev_embedding}.
\end{remark}
\begin{corollary} \label{corollary:sobolev}
	Suppose $\vdim$, $\adim$, $p$, $U$, $V$, $\psi$, and ``$\ap\!$'' are
	as in \ref{miniremark:situation_general}, $p = \vdim$, $1 \leq q <
	\vdim$, $0 \leq \alpha < \isoperimetric{\vdim}^{-1}$, $f : U \to \rel$
	is locally Lipschitzian, $\psi ( \classification{U}{z}{f(z) \neq
	0})^{1/\vdim} \leq \alpha$, and
	\begin{gather*}
		\text{$\classification{\spt \| V \|}{z}{ |f(z)| \geq t }$ is
		compact whenever $0 < t < \infty$}.
	\end{gather*}

	Then
	\begin{gather*}
		\Lpnorm{\| V \|}{\vdim q / ( \vdim-q )}{f} \leq \Gamma
		\Lpnorm{\| V \|}{q}{| \ap Df |}
	\end{gather*}
	where $\Gamma = \isoperimetric{\vdim} ( 1-\alpha \isoperimetric{\vdim}
	)^{-1} q ( \vdim-1 )/(\vdim-q)$.
\end{corollary}
\begin{proof}
	One may assume $\spt \| V \| \cap \spt f$ to be compact. \emph{The
	case $q = 1$} is then immediate from \ref{thm:sobolev_embedding} and
	H\"older's inequality and \emph{the case $q > 1$} follows by
	applying the special case with $f$ replaced by $|f|^{\eta/\beta}$
	where $\eta = \vdim q / ( \vdim-q)$ and $\beta = \vdim/ (
	\vdim-1)$, see e.g. \cite[4.5.15]{MR41:1976}.
\end{proof}
\begin{corollary} \label{corollary:dual_embedding}
	Suppose $\vdim$, $\adim$, $p$, $U$, $V$, $\psi$, and ``$\ap\!$'' are
	as in \ref{miniremark:situation_general}, $p = \vdim$, $0 \leq \alpha
	< \isoperimetric{\vdim}^{-1}$, $\psi ( U )^{1/\vdim} \leq \alpha$, $1
	\leq q \leq \vdim$, $q>1$ if $\vdim > 1$, $1 \leq r \leq \infty$,
	$1/q+1/r=1$, $1 \leq s \leq \infty$, $1/s = 1/q-1/\vdim$, and $T \in
	\mathscr{D}^0 (U)$.

	Then
	\begin{gather*}
		\nuqastnorm{s}{V}{T} \leq \Gamma \sup \{ T ( \theta ) \with
		\text{$\theta \in \mathscr{D}^0 ( U )$ and $\Lpnorm{\| V
		\|}{r}{ \theta } \leq 1$} \}
	\end{gather*}
	where $\Gamma = \isoperimetric{\vdim} ( 1 - \alpha
	\isoperimetric{\vdim} )^{-1}$ if $\vdim = 1$ and
	$\Gamma = \isoperimetric{\vdim} ( 1 - \alpha
	\isoperimetric{\vdim} )^{-1} (1-1/\vdim)/(1-1/q)$ if $\vdim > 1$.
\end{corollary}
\begin{proof}
	The assertion is readily verified by means of
	\ref{remark:embedding_1d} if $\vdim=1$ and \ref{corollary:sobolev}
	with $q$ replaced by $s/(s-1)$ if $\vdim > 1$.
\end{proof}
\begin{remark} \label{remark:dual_embedding}
	Suppose $U$ is an open subset of $\rel^\adim$, $1 \leq q \leq \infty$,
	$1 \leq r \leq \infty$, $1/q+1/r=1$, $\phi$ is a Radon measure on $U$,
	$T \in \mathscr{D}_0 ( U )$, and
	\begin{gather*}
		M = \sup \{ T ( \theta ) \with \text{$\theta \in \mathscr{D}^0
		( U )$ and $\Lpnorm{\phi}{r}{\theta} \leq 1$} \}.
	\end{gather*}
	If $q = 1$ then $M < \infty$ if and only if $T$ is
	representable by integration, $\spt T \subset \spt \phi$ and $\| T
	\| (U) < \infty$ by \cite[4.1.5]{MR41:1976}. In this case $M = \| T \|
	( U )$. If $q >1$ then $M < \infty$ if and only if there exists $g \in
	\Lp{q}( \| V \| )$ satisfying
	\begin{gather*}
		T ( \theta ) = \tint{}{} \theta g \ud \phi \quad \text{for
		$\theta \in \mathscr{D}^0 ( U )$}
	\end{gather*}
	by \cite[2.5.7]{MR41:1976} as $\mathscr{D}^0 ( U )$ is $\phi_{(r)}$
	dense in $\Lp{r} ( \phi )$. In this case $g$ is $\| V \|$ almost
	unique and $M = \Lpnorm{\phi}{q}{g}$.

	In particular, one may take $\phi = \| V \|$ in
	\ref{corollary:dual_embedding}.
\end{remark}
\begin{lemma} \label{lemma:stampacchia_iteration}
	Suppose $j \in \nat$, $0 \leq c_i < \infty$, $0 < \alpha_i < \infty$
	and $1 < \beta_i < \infty$ for $i=1, \ldots, j$, $\phi : \{ t \with 0
	\leq t < \infty \} \to \rel$ is a nonnegative and nonincreasing
	function satisfying
	\begin{gather*}
		\phi (t) \leq \tsum{i=1}{j} c_i (t-s)^{-\alpha_i} \phi_i
		(s)^{\beta_i} \quad \text{whenever $0 \leq s < t < \infty$},
	\end{gather*}
	$\gamma = \sup \{ 2^{\alpha_i/(\beta_i-1)} \with i = 1, \ldots, j \}$,
	and $d = 2 \sup \{ (j c_i \gamma )^{1/\alpha_i}
	\phi(0)^{(\beta_i-1)/\alpha_i} \with i = 1, \ldots, j \}$.

	Then $\lim_{t \to 0+} \phi (t) = 0$ if $d=0$ and $\phi (d) = 0$ if $d
	> 0$.
\end{lemma}
\begin{proof}
	Assuming $d > 0$ and defining $t_k = d - 2^{-k} d$, it will be shown
	inductively
	\begin{gather*}
		\phi ( t_k ) \leq \phi (0) \gamma^{-k} \quad \text{whenever $k
		\in \nat \cup \{ 0 \}$}.
	\end{gather*}
	Trivially, the inequality holds for $k=0$. If it holds for some $k$
	then
	\begin{gather*}
		\begin{aligned}
			\phi ( t_{k+1} ) & \leq \tsum{i=1}{j} c_i
			2^{(k+1)\alpha_i} d^{-\alpha_i} \phi ( t_k )^{\beta_i}
			\\
			& \leq \tsum{i=1}{j} ( 2^{\alpha_i} \gamma^{1-\beta_i}
			)^k ( d^{-\alpha_i} 2^{\alpha_i} c_i
			\phi(0)^{\beta_i-1} ) \phi(0) \gamma^{-k} \leq \phi(0)
			\gamma^{-k-1}.
		\end{aligned}
	\end{gather*}
	Since $\gamma > 1$, the conclusion follows.
\end{proof}
\begin{remark} \label{remark:stampacchia_iteration}
	If $j=1$ then $d = 2^{\beta_1/(\beta_1-1)} c^{1/\alpha_1}
	\phi(0)^{(\beta_1-1)/\alpha_1}$ and the assertion can be found in
	\cite[Lemma 4.1\,(i)]{MR0251373} or \cite[Lemma B.1]{MR567696}.
\end{remark}
\begin{definition} \label{def:multilinear}
	Suppose $V$ and $W$ are real vectorspaces and $i$ is a positive
	integer. Denote by $\bigotimes^i ( V,W )$ the vectorspace of all $i$
	linear functions mapping the $i$ fold product $V^i$ into $W$.
	Whenever $f \in \bigotimes^i (V,W)$ there exists a unique linear map
	$g : \bigotimes_i V \to W$ with
	\begin{gather*}
		f (v_1, \ldots, v_i) = g ( v_1 \otimes \cdots \otimes v_i )
		\quad \text{whenever $v_1, \ldots, v_i \in V$}.
	\end{gather*}
	Associating $g$ with $f$ yields a linear isomorphism
	\begin{gather*}
		{\textstyle \bigotimes^i (V,W)} \simeq \Hom \big (
		{\textstyle\bigotimes_i V}, W \big ).
	\end{gather*}
	Whenever $f \in \bigotimes^i (V,W)$ corresponds to $g \in \Hom \big (
	\bigotimes_i V, W \big )$ under this isomorphism $g ( \xi )$ for $\xi
	\in \bigotimes_i V$ will be denoted alternately by $\left < \xi, g
	\right >$ or $\left < \xi, f \right >$. If $V$ and $W$ are normed
	spaces let
	\begin{gather*}
		\| f \| = \sup \{ | f (v_1, \ldots, v_i) | \with \text{$v_j
		\in V$ and $|v_j| \leq 1$ for $j=1,\ldots,i$} \}
	\end{gather*}
	whenever $f \in \bigotimes^i (V,W)$.
\end{definition}
\begin{remark}
	The preceding definition is modelled on the corresponding
	notation for alternating and symmetric forms, see \cite[1.4.1, 1.10.1,
	1.10.5]{MR41:1976}.
\end{remark}
\begin{miniremark} \label{miniremark:bilinear_form}
	Suppose $\vdim, \adim \in \nat$, $\vdim \leq \adim$, $U$ is an open
	subset of $\rel^\adim$, $0 < \nu \leq \mu < \infty$, and $A$ is a
	function whose domain equals $U \times \grass{\adim}{\vdim}$ such that
	whenever $(z,S) \in U \times \grass{\adim}{\vdim}$
	\begin{gather*}
		A(z,S) \in {\textstyle \bigotimes^2 \Hom (S,\rel)}, \quad \| A
		(z,S) \| \leq \mu, \\
		\nu | \sigma |^2 \leq \left < \sigma \otimes \sigma, A(z,S)
		\right > \quad \text{for $\sigma \in \Hom ( S, \rel )$},
	\end{gather*}
	see \ref{def:multilinear}. Whenever $V$ measures $U \times
	\grass{\adim}{\vdim}$ the function $A$ is said to be $V$ measurable if
	and only if
	\begin{gather*}
		\text{$\left < ( \sigma |S ) \otimes ( \tau|S ), A(z,S) \right
		>$ depends $V$ measurably on $(z,S)$}
	\end{gather*}
	whenever $\sigma, \tau \in \Hom ( \rel^\adim, \rel )$.
\end{miniremark}
\begin{example} \label{example:V_subharmonic}
	If $A$ is such that $\left < \sigma \otimes \tau , A(z,S) \right > =
	\sigma \bullet \tau$ whenever $z \in U$, $S \in \grass{\adim}{\vdim}$,
	$\sigma, \tau \in \Hom (S,\rel)$ one may take $\nu = \mu = 1$ and $f
	\in \mathscr{E}^0 ( U )$ is $V$ subharmonic in the sense of Allard
	\cite[7.5\,(1)]{MR0307015} if and only if 
	\begin{gather*}
		\tint{}{} \left < (D \theta (z)|S) \otimes ( D f (z)|S) ,
		A(z,S) \right > \ud V (z,S) \leq 0 \quad \text{for $\theta \in
		\mathscr{D}^0(U)^+$}.
	\end{gather*}
\end{example}
\begin{example} \label{ex:some_lap}
	Suppose $\vdim, \adim \in \nat$, $\vdim \leq \adim$, $U$ is an open
	subset of $\rel^\vdim$, $V \in \Var_\vdim ( U )$, $\| \delta V \|$ is
	a Radon measure, and $f : U \to \rel$ is a Lipschitzian function.

	Then the following two statements hold:
	\begin{enumerate}
		\item \label{item:some_lap:c2} If $f$ is of class
		$\class{2}$ and $\eta ( V ; \cdot ) : U \to
		\mathbf{S}^{\adim-1}$ is a $\| \delta V \|$ measurable
		function such that
		\begin{gather*}
			( \delta V ) ( g ) = \tint{}{} g (z) \bullet \eta ( V;
			z ) \ud \| \delta V \| z \quad \text{whenever $g \in
			\mathscr{D} ( U, \rel^\adim)$},
		\end{gather*}
		see Allard \cite[4.3]{MR0307015}, then
		\begin{gather*}
			\begin{aligned}
				& \tint{}{} ( D \theta (z) |S ) \bullet ( Df
				(z) | S ) \ud V (z,S) \\
				& \qquad = \tint{}{} \theta \left < \eta
				(V;\cdot), D f \right > \ud \| \delta V \| -
				\tint{}{} \theta (z) \trace ( D^2 f (z)| S
				\times S ) \ud V (z,S)
			\end{aligned}
		\end{gather*}
		whenever $\theta \in \mathscr{D}^0 ( U )$.
		\item \label{item:some_lap:convex} If $U$ and $f$ are convex
		and $V \in \RVar_\vdim ( U )$ then
		\begin{gather*}
			\tint{}{} \ap D \theta \bullet \ap Df \ud \| V \| \leq
			( \Lip f ) \| \delta V \| (\theta) \quad
			\text{whenever $\theta \in \mathscr{D}^0 ( U )^+$}
		\end{gather*}
		where ``$\ap\!$'' is as in \ref{miniremark:situation_general}.
	\end{enumerate}

	Similar to Allard \cite[7.5]{MR0307015}, noting $\trace ( D^2 f (z) |
	S \times S ) = D ( \grad f) (z) \bullet \project{S}$ for $(z,S) \in U
	\times \grass{\adim}{\vdim}$, part \eqref{item:some_lap:c2} may be
	obtained by computing $\delta V ( \theta \grad f )$. Part
	\eqref{item:some_lap:convex} may be reduced to the case that $f$ is of
	class $\infty$, which is a consequence of \eqref{item:some_lap:c2}, by
	means of convolution and \cite[4.5\,(3)]{snulmenn.decay}.
\end{example}
\begin{theorem} \label{thm:global_maximum_estimates}
	Suppose $\vdim$, $\adim$, $p$, $U$, $V$, $\psi$, and
	``$\ap\!$'' are as in \ref{miniremark:situation_general}, $p = \vdim$,
	$\vdim < q \leq \infty$, $q \geq 2$, $0 \leq \alpha <
	\isoperimetric{\vdim}^{-1}$, $0 \leq M < \infty$, $f : U \to \rel$ is
	locally Lipschitzian,
	\begin{gather*}
		\| V \| ( \classification{U}{z}{f(z) > 0} ) \leq M, \quad
		\psi ( \classification{U}{z}{f(z)>0} )^{1/\vdim} \leq \alpha,
		\\
		\text{$\spt \| V \| \cap \{ z \with f(z) \geq t \}$ is compact
		for $0 < t < \infty$},
	\end{gather*}
	$\nu$, $\mu$, and $A$ are related to $\vdim$, $\adim$, and $U$ as in
	\ref{miniremark:bilinear_form}, $A$ is $V$ measurable, and $T \in
	\mathscr{D}_0 ( U )$ satisfies
	\begin{gather*}
		\tint{}{} \left < \ap D \theta (z) \otimes \ap Df(z), A(z,S)
		\right > \ud V (z,S) \leq T ( \theta ) \quad \text{for $\theta
		\in \mathscr{D}^0 ( U )^+$}.
	\end{gather*}

	Then
	\begin{gather*}
		\Lpnorm{\| V \|}{\infty}{f^+} \leq \Gamma M^{1/\vdim-1/q}
		\nu^{-1} \nuqastnorm{q}{V}{T}
	\end{gather*}
	where $\Gamma = \isoperimetric{\vdim}(1-\alpha
	\isoperimetric{\vdim})^{-1} 2^{(1-1/q)/(1/\vdim-1/q)}$.
\end{theorem}
\begin{proof}
	Assume $\gamma = \nuqastnorm{q}{V}{T} < \infty$, and replacing
	$A(z,S)$ and $T$ by $\nu^{-1} A(z,S)$ and $\nu^{-1}T$ also $\nu=1$.

	Let $1 \leq r \leq 2$ such that $1/q+1/r=1$. Define $f_t : U
	\to \rel$ by $f_t(z) = \sup \{ f(z)-t, 0 \}$ for $z \in U$, $0 \leq t
	< \infty$ and $\phi : \{ t \with 0 \leq t < \infty \} \to \rel$ by
	\begin{gather*}
		\phi (t) = \| V \| ( \classification{U}{z}{f_t(z) > 0} ) \quad
		\text{whenever $0 \leq t < \infty$}.
	\end{gather*}
	For $0 < s < t < \infty$ one estimates, noting
	\ref{corollary:nuq_density} and \ref{remark:nuqastnorm},
	\begin{gather*}
		\begin{aligned}
			& \Lpnorm{\| V \|}{2}{|\ap Df_s|}^2 \leq \tint{}{}
			\left < \ap D f_s (z) \otimes \ap Df(z), A(z,S) \right
			> \ud V (z,S) \\
			& \qquad \leq \gamma \nuqnorm{r}{V}{f_s} \leq \gamma
			\phi(s)^{1/2-1/q} \Lpnorm{\| V \|}{2}{| \ap Df_s|} <
			\infty,
		\end{aligned}
	\end{gather*}
	hence, using H\"older's inequality,
	\begin{gather*}
		\Lpnorm{\| V \|}{1}{|\ap Df_s|} \leq \gamma \phi (s)^{1-1/q}.
	\end{gather*}
	Considering the limit $s \to 0+$, one obtains the same inequality for
	$s=0$.

	\emph{The case $\vdim=1$} now follows from \ref{thm:sobolev_embedding}
	and H\"older's inequality. To treat \emph{the case $\vdim>1$}, define
	\begin{gather*}
		\beta = \vdim/(\vdim-1), \quad \Delta =
		\isoperimetric{\vdim}(1-\alpha \isoperimetric{\vdim})^{-1},
	\end{gather*}
	and use \ref{corollary:sobolev} with $\xi$ replaced by $1$ to obtain
	\begin{gather*}
		\phi (t)^{1/\beta} (t-s) \leq \Lpnorm{\| V \|}{\beta}{f_s}
		\leq \Delta \Lpnorm{\| V \|}{1}{ | \ap Df_s |} \leq \Delta
		\phi(s)^{1-1/q} \gamma, \\
		\phi (t) \leq ( \Delta \gamma)^\beta (t-s)^{-\beta}
		\phi(s)^{\beta(1-1/q)}
	\end{gather*}
	whenever $0 \leq s < t < \infty$. The conclusion then follows from
	\ref{lemma:stampacchia_iteration} and
	\ref{remark:stampacchia_iteration} with $j$, $c_1$, $\alpha_1$, and
	$\beta_1$ replaced by $1$, $(\Delta \gamma)^\beta$, $\beta$, and
	$\beta (1-1/q)$.
\end{proof}
\begin{example}
	Suppose $\adim \in \nat$, $B \subset U \subset \rel^\adim$, $U$ is
	open, $B$ is bounded and closed relative to $U$, and $g : \Clos B \to
	\rel$ is continuous with
	\begin{gather*}
		g(z) \leq 0 \quad \text{whenever $z \in ( \Clos B ) \without
		B \subset \Bdry U$}.
	\end{gather*}
	Then $\classification{B}{z}{g(z) \geq t}$ is compact for $0 < t <
	\infty$.

	In particular, taking $B = \spt \| V \|$ yields a sufficient condition
	for the compactness hypotheses in \ref{thm:global_maximum_estimates}.
\end{example}
\begin{corollary} \label{corollary:emb_loo}
	Suppose $\vdim$, $\adim$, $p$, $U$, $V$, $\psi$, and ``$\ap\!$'' are
	as in \ref{miniremark:situation_general}, $p = \vdim$, $1 < \vdim < q
	\leq \infty$, $0 \leq \alpha < \isoperimetric{\vdim}^{-1}$, $0 \leq M
	< \infty$, $f : U \to \rel$ is locally Lipschitzian,
	\begin{gather*}
		\| V \| ( \classification{U}{z}{f(z) \neq 0} ) \leq M, \quad
		\psi ( \classification{U}{z}{f(z) \neq 0})^{1/\vdim} \leq
		\alpha, \\
		\text{$\classification{\spt \| V \|}{z}{|f(z)| \geq t}$ is
		compact for $0 < t < \infty$}.
	\end{gather*}

	Then
	\begin{gather*}
		\Lpnorm{\| V \|}{\infty}{f} \leq \Gamma M^{1/\vdim-1/q}
		\Lpnorm{ \| V \|}{q}{| \ap Df |}
	\end{gather*}
	where $\Gamma = \isoperimetric{\vdim}(1-\alpha
	\isoperimetric{\vdim})^{-1} 2^{(1-1/q)/(1/\vdim-1/q)}$.
\end{corollary}
\begin{proof}
	Assuming $f \geq 0$ and taking $\nu$, $\mu$, and $A$ as in
	\ref{example:V_subharmonic} the conclusion follows from
	\ref{thm:global_maximum_estimates}.
\end{proof}
\begin{theorem} \label{thm:local_maximum_estimates}
	Suppose $\vdim$, $\adim$, $p$, $U$, $V$, $\psi$, and ``$\ap$'' are as
	in \ref{miniremark:situation_general}, $p = \vdim$, $\vdim < q \leq
	\infty$, $q \geq 2$, $1 \leq r < \infty$, $0 \leq \alpha <
	\isoperimetric{\vdim}^{-1}$, $f : U \to \rel$ is locally Lipschitzian,
	$d : U \to \rel$ is a nonnegative function, $\Lip d \leq 1$, $0 <
	\delta < \infty$,
	\begin{gather*}
		\psi ( \classification{U}{z}{\text{$f(z)>0$ and
		$d(z)>0$}} )^{1/\vdim} \leq \alpha, \\
		\text{$\classification{\spt \| V \|}{z}{\text{$f(z) \geq t$
		and $d(z) \geq t$} }$ is compact for $0 < t < \infty$},
	\end{gather*}
	$\nu$, $\mu$, and $A$ are related to $\vdim$, $\adim$, and $U$ as in
	\ref{miniremark:bilinear_form}, $A$ is $V$ measurable, and $T \in
	\mathscr{D}_0 ( U )$ satisfies
	\begin{gather*}
		\tint{}{} \left < \ap D \theta (z) \otimes \ap Df(z), A (z,S)
		\right > \ud V (z,S) \leq T ( \theta ) \quad \text{for $\theta
		\in \mathscr{D}^0 ( U )^+$}.
	\end{gather*}

	Then there holds
	\begin{gather*}
		\eqLpnorm{\| V \| \restrict \{z\with d(z) \geq
		\delta\}}{\infty}{f^+} \leq \Gamma \big ( \delta^{-\vdim/r}
		\Lpnorm{\| V \|}{r}{f^+} + \delta^{1-\vdim/q} \nu^{-1}
		\nuqastnorm{q}{V}{T} \big )
	\end{gather*}
	where $\Gamma$ is a positive, finite number depending only on $\vdim$,
	$q$, $\alpha$, and $\mu/\nu$.
\end{theorem}
\begin{proof}
	The problem may reduced to the case $\nu=1$ by replacing $A(z,S)$ and
	$T$ by $\nu^{-1} A(z,S)$ and $\nu^{-1}T$ and to the case $\delta=1$
	via rescaling. Moreover, defining $\gamma = \Lpnorm{\| V \|}{r}{f^+} +
	\nuqastnorm{q}{V}{T}$ one may assume $0 < \gamma < \infty$.

	Define $\xi = 1$ if $\vdim = 1$ and $\xi = 2 \vdim/ ( \vdim+2 )$ if
	$\vdim > 1$, and
	\begin{gather*}
		\alpha_1 = 2 + 4 / \vdim, \quad \alpha_2 = 2 + 4 / \vdim - 4 /
		q, \quad \beta_1 = 1 + 2 / \vdim, \quad \beta_2 = 1 + 2 /
		\vdim - 2/q, \\
		\Delta_1 = \big ( \isoperimetric{\vdim} (1-\alpha
		\isoperimetric{\vdim} )^{-1} \xi ( \vdim-1 )/(\vdim-\xi) \big
		)^2, \\
		\Delta_2 = \sup \{ 2^{\alpha_1/(\beta_1-1)},
		2^{\alpha_2/(\beta_2-1)} \}, \\
		\Delta_3 = \sup \big \{ ( 2^{\alpha_1+1} \Delta_1 \Delta_2 (
		16 \mu^2 + 12 ) )^{\vdim/4}, 2^{1+1/\alpha_2} ( 36 \Delta_1
		\Delta_2 )^{1/\alpha_2} \big \}, \\
		\Gamma_1 = \Delta_3, \quad \Gamma_2 = 1 + \Gamma_1,
		\quad \Gamma_3 = 2^{6\vdim+2}(1+\Gamma_1)^2, \quad \Gamma =
		\sup \{ \Gamma_1, \Gamma_2, \Gamma_3 \}.
	\end{gather*}
	Moreover, note $\Delta_3 > 0$ and let
	\begin{gather*}
		c_1 = \gamma^{-4/\vdim} \Delta_1 \Delta_3^{-4/\vdim} ( 16
		\mu^2 + 12 ), \quad c_2 = \gamma^{-4/\vdim+4/q} 36 \Delta_1
		\Delta_3^{-2-4/\vdim+4/q}, \\
		F(t) = \classification{U}{z}{d(z) \geq t} \quad \text{for $0
		\leq t < \infty$}.
	\end{gather*}

	First, the \emph{case $r=2$} is considered. Define $f_t : U \to \rel$
	by $f_t(z) = \sup \{ f(z)-t, 0 \}$ for $z \in U$. Moreover, let $\phi
	: \{ t \with 0 \leq t < \infty \} \to \rel$ be defined by
	\begin{gather*}
		\phi (t) = \tint{F(t)}{} (f_{\Delta_3 \gamma t})^2 \ud \| V \|
		\quad \text{whenever $0 \leq t < \infty$}.
	\end{gather*}
	Suppose $0 < s < t < \infty$ and $\eta = (t-s)^{-1} ( \inf \{ d,t \} -
	\inf \{ d,s \} )$ and note
	\begin{gather*}
		0 \leq \eta (z) \leq 1 \quad \text{for $z \in U$}, \qquad \Lip
		\eta \leq (t-s)^{-1}, \\
		\eta(z)=1 \quad \text{for $z \in F(t)$}, \qquad \spt \eta
		\subset F(s).
	\end{gather*}
	Abbreviate $N = \| V \| (\classification{F(s)}{z}{f(z)> \Delta_3
	\gamma t})$. Define $g = f_{\Delta_3 \gamma t}$, $\theta = \eta^2 g$
	and note that $\spt \| V \| \cap \spt \theta$ is compact. One computes
	\begin{gather*}
		\begin{aligned}
			& \ap D \theta (z) \otimes \ap Df(z) \\
			& \qquad = 2 \eta (z) g (z) \ap D \eta (z) \otimes
			\ap D g (z) + \eta^2 (z) \ap Dg (z) \otimes \ap
			Dg (z),
		\end{aligned} \\
		\eta^2 (z) | \ap Dg(z) |^2 \leq 4 \mu^2 | \ap D \eta(z)|^2 g
		(z)^2 + 2 \left < \ap D \theta (z) \otimes \ap Df(z), A(z,S)
		\right >
	\end{gather*}
	for $V$ almost all $(z,S)$. Define $h = \eta g$ and estimate, noting
	$\theta = \eta h$ and $h \leq g$,
	\begin{gather*}
		| \ap D h (z) |^2 \leq ( 8 \mu^2 + 2 ) | \ap D \eta (z) |^2
		g(z)^2 + 4 \left < \ap D \theta (z) \otimes \ap Df(z),
		A(z,S) \right >, \\
		| \ap D \theta (z) |^2 \leq 2 | \ap D h (z) |^2 + 2 | \ap D
		\eta (z) |^2 g (z)^2
	\end{gather*}
	for $V$ almost all $(z,S)$, hence, using \ref{corollary:nuq_density},
	\begin{gather*}
		\| V \| ( | \ap D h|^2 ) \leq ( 8 \mu^2 + 2 ) \| V \| ( | \ap D
		\eta |^2 g^2 ) + 4 \gamma N^{1/2-1/q} \Lpnorm{\| V \|}{2}{
		| \ap D \theta | }, \\
		\| V \| ( | \ap D h |^2 ) \leq ( 16 \mu^2 + 12 ) \| V \| ( | \ap
		D \eta |^2 g^2 ) + 36 \gamma^2 N^{1-2/q}.
	\end{gather*}
	Combining the last inequality with
	\begin{gather*}
		N \leq \Delta_3^{-2} \gamma^{-2} (t-s)^{-2} \phi (s), \quad \|
		V \| ( | \ap D \eta |^2 g^2 ) \leq (t-s)^{-2} \phi (s), \\
		\phi (t) \leq N \Lpnorm{\| V \|}{\infty}{h}^2 \leq \Delta_1
		N^2 \| V \| ( | \ap D h |^2 ) \quad \text{if $\vdim=1$}, \\
		\phi (t) \leq \| V \| ( h^2 ) \leq \Delta_1 \Lpnorm{\| V
		\|}{\xi}{|\ap D h|}^2 \leq \Delta_1 N^{2/\vdim} \| V \| ( |
		\ap D h |^2 ) \quad \text{if $\vdim > 1$},
	\end{gather*}
	by \ref{thm:sobolev_embedding} if $\vdim=1$ and by
	\ref{corollary:sobolev} if $\vdim>1$ implies
	\begin{gather*}
		\phi (t) \leq \tsum{i=1}{2} c_i (t-s)^{-\alpha_i}
		\phi(s)^{\beta_i}.
	\end{gather*}
	Taking the limit $s \to 0+$ the inequality also holds for $s=0$.
	Verifying
	\begin{gather*}
		2 \sup \{ (2 c_i \Delta_2 )^{1/\alpha_i}
		\phi(0)^{(\beta_i-1)/\alpha_i} \with i \in \{1,2\} \} \leq
		1,
	\end{gather*}
	one obtains $\phi(1) = 0$ from \ref{lemma:stampacchia_iteration}.
	Therefore in the present case the conclusion is satisfied with
	$\Gamma$ replaced by $\Gamma_1$.

	Second, \emph{the case $r > 2$} is treated. In fact, assuming
	$\Lpnorm{\| V \|}{r}{f^+} =1$ and noting
	\begin{gather*}
		\tint{}{} \sup \{ f(z)-1, 0 \}^2 \ud \| V \| z \leq \Lpnorm{\|
		V \|}{r}{f^+}^r = 1,
	\end{gather*}
	one may apply the preceding case with $f(z)$ replaced by $f(z)-1$ to
	verify that the conclusion is satisfied with $\Gamma$ replaced by
	$\Gamma_2$ in the present case.

	Third and last, \emph{the case $r < 2$} is considered. Applying the
	case $r=2$ with $U$, $d$ and $\delta$ replaced by
	$\classification{U}{z}{d(z)>s}$, the function mapping $z \in U$ with
	$d(z)>s$ onto $\sup \{ d(z)-s,0 \}$ and $t-s$ yields
	\begin{gather*}
		\begin{aligned}
			& \eqLpnorm{\| V \| \restrict F(t)}{\infty}{f^+} \\
			& \qquad \leq \Gamma_1 \big ( (t-s)^{-\vdim/2}
			\eqLpnorm{\| V \| \restrict F(s)}{2}{f^+} +
			(t-s)^{1-\vdim/q} \nuqastnorm{q}{V}{T} \big )
		\end{aligned}
	\end{gather*}
	whenever $0 < s < t \leq 1$. Define
	\begin{gather*}
		a (t) = \eqLpnorm{\| V \| \restrict
		F(t+1/2)}{\infty}{f^+}, \quad b
		(t) = \eqLpnorm{\| V \| \restrict F(t+1/2)}{r}{f^+}
	\end{gather*}
	whenever $0 < t \leq 1/2$. Let
	\begin{gather*}
		\sigma = \sup \{ t^{\vdim/2+\vdim/r} a(t) \with 0 < t \leq 1/2
		\}
	\end{gather*}
	and note $\sigma \leq \eqLpnorm{\| V \| \restrict F(1/2)}{\infty}{f^+}
	< \infty$. One estimates for $0 < t \leq 1/2$
	\begin{gather*}
		\begin{aligned}
			t^{\vdim/2+\vdim/r} a (t) & \leq \Gamma_1
			\big ( 2^\vdim t^{\vdim/r} a(t/2)^{1-r/2}
			b(t/2)^{r/2} + \nuqastnorm{q}{V}{T} \big )
			\\
			& \leq 2^\vdim \Gamma_1 \big ( t^{\vdim/r} (
			\lambda^{2/(2-r)} a(t/2) + \lambda^{-2/r} b(t/2) ) +
			\nuqastnorm{q}{V}{T} \big )
		\end{aligned}
	\end{gather*}
	whenever $0 < \lambda < \infty$. Taking $\lambda = (
	2^{\vdim+1+\vdim/r+\vdim/2} \Gamma_1 )^{(r-2)/2} t^{(2-r) \vdim/4}$
	yields
	\begin{gather*}
		\begin{aligned}
			t^{\vdim/2+\vdim/r} a(t) & \leq \sigma/2 + 2^\vdim
			\Gamma_1 \big ( 2^{3\vdim+1} ( 1 + \Gamma_1 )
			t^{\vdim/2} b (t/2) + \nuqastnorm{q}{V}{T} \big ) \\
			& \leq \sigma/2 + 2^{4\vdim+1} ( 1 + \Gamma_1 )^2
			\gamma
		\end{aligned}
	\end{gather*}
	hence $2^{-\vdim/2-\vdim/r} \eqLpnorm{\| V \| \restrict
	F(1)}{\infty}{f^+} \leq \sigma \leq 2^{4\vdim+2}(1+\Gamma_1)^2 \gamma$
	and the conclusion follows with $\Gamma$ replaced by $\Gamma_3$ in the
	present case.
\end{proof}
\begin{remark}
	Defining $Z = \classification{U}{z}{\text{$f(z)>0$ and $d(z)>0$}}$ and
	choosing $1 \leq s \leq \infty$ with $1/q+1/s=1$, the assertion may be
	sharpened by replacing the sum in the conclusion by
	\begin{gather*}
		\delta^{-\vdim/r} \eqLpnorm{\| V \| \restrict Z}{r}{f^+} +
		\delta^{1-\vdim/q} \sup \{ T ( \theta ) \with \theta \in
		\mathscr{D}^0 ( U ), \spt \theta \subset Z,
		\nuqnorm{s}{V}{\theta} \leq 1 \}
	\end{gather*}
	as may verified by applying \ref{thm:local_maximum_estimates} with $U$
	replaced by $Z$.
\end{remark}
\begin{example}
	Suppose $\adim \in \nat$, $B \subset U \subset \rel^\adim$, $U$ is
	open, $B$ is bounded and closed relative to $U$ hence $( \Clos B )
	\without B \subset \Bdry U$, $d : U \to \rel$ is a nonnegative
	Lipschitzian function, $C = \eqclassification{\Clos B}{c}{\text{$c
	\notin B$ and $\lim_{z \to c} d(z)>0$}}$ hence $C$ is open $U$
	relative to $( \Clos B ) \without B$, and $g : \Clos B \to \rel$ is
	continuous with
	\begin{gather*}
		g(c) \leq 0 \quad \text{whenever $c \in C$}.
	\end{gather*}
	Then $\classification{B}{z}{\text{$g(z) \geq t$ and $d(z) \geq t$}}$
	is compact for $0 < t < \infty$.

	In particular, taking $B = \spt \| V \|$ yields a sufficient condition
	for the compactness hypotheses in \ref{thm:local_maximum_estimates}.

	Given $a \in \rel^\adim$ and $0 < s < \infty$ a typical choice for $d$
	is
	\begin{gather*}
		d(z) = \dist (z, \rel^\adim \without \oball{a}{s} ) \quad
		\text{for $z \in U$}
	\end{gather*}
	in which case $\classification{U}{z}{d(z) \geq \delta} = U \cap
	\cball{a}{s-\delta}$ for $0 < \delta \leq s$.

	Finally, if $U$ is related to $d$ and $\delta$ via $d(z) = \inf \{
	\delta, \dist ( z, \rel^\adim \without U ) \}$ for $z \in U$ then the
	compactness hypotheses in \ref{thm:local_maximum_estimates} may be
	omitted as may be seen by applying \ref{thm:local_maximum_estimates}
	with $U$ replaced by $U \cap \oball{0}{R}$ for $0 < R < \infty$.
\end{example}
\begin{remark}
	The case $r=2$, $f$ is of class $\infty$, $d(z) = \inf \{ \delta,
	\dist ( z, \rel^\adim \without U ) \}$ for $z \in U$ and $T=0$ is
	treated by Allard in \cite[7.5\,(6)]{MR0307015} using Moser's
	iteration method. However, adapting this approach to the case of
	nonvanishing $T$ along the lines of \cite[\S 8.6]{MR1814364} seems to
	require bounds on $\| V \| ( U )$. Therefore Stampacchia's iteration
	method, see e.g.  \cite[Th\'eor\`eme 5.1]{MR0192177}, is used instead.

	In case $A$ is given by \ref{example:V_subharmonic} and $T=0$ a
	similar situation (with $p=\infty$) is analysed by Michael and Simon
	in \cite[Theorem 3.4]{MR0344978} where an explicit bound on $\Gamma$
	is obtained.
\end{remark}
\begin{corollary} \label{corollary:interpolation_sum}
	Suppose $\vdim$, $\adim$, $p$, $U$, $V$, $\psi$, and ``$\ap$'' are as
	in \ref{miniremark:situation_general}, $p=\vdim$, $1 < \vdim < q \leq
	\infty$, $1 \leq r < \infty$, $0 \leq \alpha <
	\isoperimetric{\vdim}^{-1}$, $f : U \to \rel$ is locally Lipschitzian,
	$d : U \to \rel$ is a nonnegative function, $\Lip d \leq 1$, $0 <
	\delta < \infty$, and
	\begin{gather*}
		\psi ( \classification{U}{z}{\text{$f(z)>0$ and
		$d(z)>0$}} )^{1/\vdim} \leq \alpha, \\
		\text{$\classification{\spt \| V \|}{z}{\text{$f(z) \geq t$
		and $d(z) \geq t$} }$ is compact for $0 < t < \infty$}.
	\end{gather*}

	Then
	\begin{gather*}
		\eqLpnorm{\| V \| \restrict \{z\with d(z) \geq
		\delta\}}{\infty}{f} \leq \Gamma \big ( \delta^{-\vdim/r}
		\Lpnorm{\| V \|}{r}{f} + \delta^{1-\vdim/q} \Lpnorm{\| V
		\|}{q}{| \ap Df|} \big )
	\end{gather*}
	where $\Gamma$ is a positive, finite number depending only on $\vdim$,
	$q$, and $\alpha$.
\end{corollary}
\begin{proof}
	Assuming $f \geq 0$ and taking $\nu$, $\mu$, and $A$ as in
	\ref{example:V_subharmonic} the conclusion follows from
	\ref{thm:local_maximum_estimates}.
\end{proof}
\begin{miniremark} \label{miniremark:schaukel}
	If $a \geq 0$, $b \geq 0$, $c > 0$ and $d > 0$ then
	\begin{gather*}
		\inf \{ a t^c + b t^{-d} \with 0 < t < \infty \} = \big (
		(d/c)^{c/(c+d)} + (d/c)^{-d/(c+d)} \big ) a^{d/(c+d)}
		b^{c/(c+d)}
	\end{gather*}
\end{miniremark}
\begin{corollary}
	Suppose $\vdim$, $\adim$, $p$, $U$, $V$, $\psi$, and ``$\ap$'' are as
	in \ref{miniremark:situation_general}, $p=\vdim$, $1 < \vdim < q \leq
	\infty$, $1 \leq r < \infty$, $\beta = (1/r) / (1/\vdim - 1/q + 1/r
	)$, $0 \leq \alpha < \isoperimetric{\vdim}^{-1}$, $f : U \to \rel$ is
	locally Lipschitzian, and
	\begin{gather*}
		\Lpnorm{\| V \|}{r}{f} + \Lpnorm{\| V \|}{q}{| \ap Df |} <
		\infty, \quad \psi ( \classification{U}{z}{f(z)>0} )^{1/\vdim}
		\leq \alpha, \\
		\text{$\classification{\spt \| V \|}{z}{f(z) \geq t}$ is
		compact for $0 < t < \infty$}.
	\end{gather*}

	Then
	\begin{gather*}
		\Lpnorm{\| V \|}{\infty}{f} \leq \Gamma \Lpnorm{\| V
		\|}{r}{f}^{1-\beta} \Lpnorm{\| V \|}{q}{| \ap Df|}^\beta
	\end{gather*}
	where $\Gamma$ is a positive, finite number depending only on $\vdim$,
	$q$, and $\alpha$.
\end{corollary}
\begin{proof}
	Define $\Delta_1 = \Gamma_{\ref{corollary:interpolation_sum}}
	( \vdim, q, \alpha )$, observe that $\Delta_2$ defined to be the
	supremum of all numbers
	\begin{gather*}
		\left ( \frac{1/r}{1/\vdim-1/q} \right )
		^{(1/\vdim - 1/q)/(1/\vdim - 1/q + 1/r)} + \left (
		\frac{1/r}{1/\vdim - 1/q} \right )^{-(1/r)/(1/\vdim - 1/q +
		1/r)}
	\end{gather*}
	corresponding $1 \leq r < \infty$ is finite and let $\Gamma = \Delta_1
	\Delta_2$. Whenever $z \in \spt \| V \|$ and $0 < \delta < \infty$
	applying \ref{corollary:interpolation_sum} for every $\varepsilon > 0$
	with $d(z) = \dist ( z, \rel^\adim \without \oball{z}{\delta +
	\varepsilon} )$ yields
	\begin{gather*}
		|f(z)| \leq \Delta_1 \big ( \delta^{-\vdim/r} \Lpnorm{\| V
		\|}{r}{f} + \delta^{1-\vdim/q} \Lpnorm{\| V \|}{q}{| \ap Df|}
		\big ).
	\end{gather*}
	Therefore \ref{miniremark:schaukel} implies the conclusion.
\end{proof}
\begin{remark}
	In case $U = \rel^\adim$ the compactness hypotheses may be omitted.
\end{remark}
\section{Second order flatness in Lebesgue spaces for integral varifolds with
subcritical integrability of the mean curvature}
In the present section a substitute for the results of Section
\ref{sec:local_max} for integral varifolds with subcritical integrability (the
case $1 < p < \vdim$ in \ref{miniremark:situation_general}) is developed to
the extend necessary to deduce an optimal second flatness result in Lebesgue
spaces from second order rectifiability results of the author previously
obtained in \cite[4.8]{snulmenn.c2}, see \ref{thm:second_order_lq_flatness}.
\begin{theorem} \label{thm:subcritical_sobolev_embedding}
	Suppose $\vdim$, $\adim$, $p$, $U$, $V$, and $\psi$ are as in
	\ref{miniremark:situation_general}, $p < \vdim$, $1 \leq \xi < \vdim$,
	$f : U \to \rel$ is Lipschitzian, $\spt f$ is a compact subset of $U$,
	and $\alpha = ( \vdim-p ) \xi / ( \vdim-\xi )$.

	Then
	\begin{gather*}
		\Lpnorm{\| V \|}{\vdim \xi/(\vdim-\xi)}{f} \leq \Gamma \big (
		\Lpnorm{\| V \|}{\xi}{ | \ap Df | } + \psi ( |f|^\alpha
		)^{1/\alpha} \big )
	\end{gather*}
	where ``$\ap$'' denotes approximate differentiation with respect to $(
	\| V \|, \vdim )$ and $\Gamma$ is a positive, finite number depending
	only on $\vdim$, $p$, and $\xi$.
\end{theorem}
\begin{proof}
	Let $c = (\vdim-1) \xi / ( \vdim-\xi )$, note $c \geq 1$ and define
	\begin{gather*}
		\Delta_1 = 1 \quad \text{and} \quad \Delta_2 = 1 \quad
		\text{if $c=1$}, \\
		\Delta_1 = ( 3 \isoperimetric{\vdim} (c-1) )^{(1-c)/c} \quad
		\text{and} \quad \Delta_2 = c^{-1} ( \Delta_1 )^{-c} \quad
		\text{if $c > 1$}.
	\end{gather*}
	Moreover, let $\beta = (1-1/p)/(1-1/\vdim)$, note $0 \leq \beta < 1$ and
	define
	\begin{gather*}
		\Delta_3 = 1 \quad \text{and} \quad \Delta_4 = 1 \quad
		\text{if $\beta = 0$}, \\
		\Delta_3 = ( 3 \isoperimetric{\vdim} \beta )^{-\beta} \quad
		\text{and} \quad \Delta_4 = (1-\beta) ( \Delta_3
		)^{1/(\beta-1)} \quad \text{if $\beta > 0$}, \\
		\Gamma = (3\isoperimetric{\vdim})^{1/c} \sup \{ c \Delta_2,
		\Delta_4 \}^{1/c}.
	\end{gather*}
	Assume $f \geq 0$ and apply \ref{thm:sobolev_embedding} with $f$
	replaced $f^c$ to obtain
	\begin{gather*}
		\Lpnorm{\| V\|}{\vdim \xi/(\vdim-\xi)}{f}^c \leq
		\isoperimetric{\vdim} \big ( c \| V \| ( f^{c-1} | \ap Df | )
		+ \| \delta V \| ( f^c ) \big ),
	\end{gather*}
	here $0^0=1$.

	If $c = 1$ then $\xi=1$. If $c > 1$ then, noting $\xi > 1$ and $(c-1)
	\xi/(\xi-1) = \vdim \xi / ( \vdim - \xi)$, one uses H\"older's
	inequality and the inequality relating arithmetic and geometric means
	to infer
	\begin{gather*}
		\begin{aligned}
			& \| V \| ( f^{c-1} | \ap D f | ) \leq \Lpnorm{\| V
			\|}{\vdim\xi/(\vdim-\xi)}{f}^{c-1} \Lpnorm{\| V
			\|}{\xi}{| \ap Df|} \\
			& \qquad \leq (1-1/c) ( \Delta_1 )^{c/(c-1)}
			\Lpnorm{\| V \|}{\vdim \xi/(\vdim-\xi)}{f}^c + c^{-1}
			( \Delta_1 )^{-c}\Lpnorm{\| V \|}{\xi}{|\ap Df|}^c \\
			& \qquad = ( 3 \isoperimetric{\vdim} c )^{-1}
			\Lpnorm{\| V \|}{\vdim \xi/(\vdim-\xi)}{f}^c +
			\Delta_2 \Lpnorm{\| V \|}{\xi}{| \ap D f |}^c.
		\end{aligned}
	\end{gather*}
	If $\beta = 0$ then $p=1$ and $c = \alpha$. If $\beta > 0$ then,
	noting
	\begin{gather*}
		p > 1, \quad c \beta p / (p-1) = \vdim \xi / ( \vdim-\xi ),
		\quad c (1-\beta) p = \alpha,
	\end{gather*}
	one deduces similarly as before
	\begin{gather*}
		\begin{aligned}
			\| \delta V \| (f^c) & \leq \Lpnorm{\| V
			\|}{p/(p-1)}{f^{c \beta}}
			\Lpnorm{\| V \|}{p}{f^{c(1-\beta)} \mathbf{h}
			(V;\cdot)} \\
			& = \Lpnorm{\| V \|}{\vdim\xi/(\vdim-\xi)}{f}^{c
			\beta} \psi ( f^\alpha )^{(c/\alpha)(1-\beta)}
			\\
			& \leq \beta ( \Delta_3 )^{1/\beta} \Lpnorm{\| V
			\|}{\vdim\xi/(\vdim-\xi)}{f}^c + (1-\beta) ( \Delta_3
			)^{1/(\beta-1)} \psi ( f^\alpha )^{c/\alpha} \\
			& \leq (3 \isoperimetric{\vdim})^{-1} \Lpnorm{\| V
			\|}{\vdim \xi/(\vdim-\xi)}{f}^c + \Delta_4 \psi (
			f^\alpha )^{c/\alpha}.
		\end{aligned}
	\end{gather*}
	Now, the conclusion readily follows.
\end{proof}
\begin{definition}
	Suppose $\phi$ measures $X$.

	Then $\phi$ is called \emph{nonatomic} if and only if for every $\phi$
	measurable set $A$ with $0 < \phi (A) < \infty$ there exists a $\phi$
	measurable subset $B$ of $A$ with $0 < \phi (B) < \phi(A)$.
\end{definition}
\begin{remark} \label{remark:nonatomic_measures}
	Note $0 < \inf \{ \phi (B), \phi ( A \without B ) \} \leq \phi (A)/2$.
	
	Among the basic properties of nonatomic measures are the
	following: \emph{Whenever $t \in \rel$ and $A$ is a $\phi$ measurable
	set with $0 < t < \phi (A) < \infty$ there exists a $\phi$ measurable
	subset $B$ of $A$ with $\phi (B) = t$;} in fact denoting by $F$ the
	family of $\phi$ measurable subsets $C$ of $A$ with $\phi (C) > 0$ and
	observing
	\begin{gather*}
		\inf \{ \phi (C) \with \text{$C \in F$ and $C \subset D$} \} =
		0 \quad \text{for $D \in F$},
	\end{gather*}
	one may take $B$ to be the union of a maximal member of the family of
	all disjointed subfamilies $G$ of $F$ with $\sum_{C \in G} \phi (C)
	\leq t$. Therefore, \emph{if $f : X \to \rel$ is a $\phi$ measurable
	function, $1 \leq p < \infty$, and $0 < \lambda \leq \phi (X) <
	\infty$ then there exists a $\phi$ measurable set $B$ with $\phi (B)
	\leq \lambda$ minimising
	\begin{gather*}
		\eqLpnorm{\phi \restrict X \without B}{p}{f}
	\end{gather*}
	among all such sets;} in fact it is sufficient to construct a $\phi$
	measurable set $B$ such that $\phi (B) = \lambda$ and
	\begin{gather*}
		\classification{X}{x}{|f(x)|>s} \subset B \subset
		\classification{X}{x}{|f(x)| \geq s} \quad \text{for some $s
		\in \rel$}.
	 \end{gather*}
\end{remark}
\pagebreak
\begin{example} \label{example:nonatomic}
	If $\phi$ measures $\rel^\adim$, $\phi ( \rel^\adim ) < \infty$ and
	all closed sets are $\phi$ measurable then $\phi$ is nonatomic if and
	only if $\phi ( \{ z \} ) = 0$ for $z \in \rel^\adim$.
\end{example}
\begin{lemma} \label{lemma:basic_interpolation}
	Suppose $1 \leq p < q < r \leq \infty$, $0 < \alpha < 1$, $1/q =
	\alpha / p + (1-\alpha)/r$, $\phi$ measures $X$, $\phi$ is nonatomic,
	$f$ is a real valued $\phi$ measurable function, $0 < \lambda <
	\infty$, and $B$ is a $\phi$ measurable set with $\phi ( B ) \leq
	\lambda^{1/(1/q-1/r)} \leq \phi (X) < \infty$.

	Then
	\begin{gather*}
		\Lpnorm{\phi}{q}{f} \leq 2 \lambda^{1-1/\alpha} \eqLpnorm{\phi
		\restrict X \without B}{p}{f} + 2 \lambda \Lpnorm{\phi}{r}{f}.
	\end{gather*}
\end{lemma}
\begin{proof}
	Assume $\Lpnorm{\phi}{r}{f} < \infty$. By
	\ref{remark:nonatomic_measures} one may assume that $\phi (B) =
	\lambda^{1/(1/q-1/r)}$ and
	\begin{gather*}
		\classification{X}{x}{|f(x)|>s} \subset B \subset
		\classification{X}{x}{|f(x)| \geq s} \quad \text{for some $s
		\in \rel$}.
	 \end{gather*}
	Clearly, this implies $s \lambda^{1/(1-q/r)} = s \phi (B)^{1/q} \leq
	\Lpnorm{\phi}{q}{f}$. Defining $\chi$ to be the characteristic
	function of $X \without B$, $g = \chi f$ and $h = f-g$, note
	\begin{gather*}
		\Lpnorm{\phi}{\infty}{g} \leq s, \quad (1-p/q)/(1-q/r) = (p/q)
		(1/\alpha-1)
	\end{gather*}
	and estimate
	\begin{gather*}
		\begin{aligned}
			\Lpnorm{\phi}{q}{f} & \leq \Lpnorm{\phi}{q}{g} +
			\Lpnorm{\phi}{q}{h} \leq s^{1-p/q}
			\Lpnorm{\phi}{p}{g}^{p/q} + \phi (B)^{1/q-1/r}
			\Lpnorm{\phi}{r}{f} \\
			& \leq \lambda^{(p/q)(1-1/\alpha)}
			\Lpnorm{\phi}{q}{f}^{1-p/q} \Lpnorm{\phi}{p}{g}^{p/q}
			+ \lambda \Lpnorm{\phi}{r}{f}.
		\end{aligned}
	\end{gather*}
	Using the inequality relating the arithmetic and geometric mean, one
	notes
	\begin{gather*}
		\Lpnorm{\phi}{q}{f}^{1-p/q} \Lpnorm{\phi}{p}{g}^{p/q} \leq
		\delta^{q/(q-p)} \Lpnorm{\phi}{q}{f} + \delta^{-q/p}
		\Lpnorm{\phi}{p}{q} \quad \text{for $0 < \delta < \infty$}.
	\end{gather*}
	Upon choosing $\delta = 2^{{p/q}-1}
	\lambda^{(p/q)(1/\alpha-1)(1-p/q)}$ one infers from the two preceding
	estimates that
	\begin{gather*}
		\Lpnorm{\phi}{q}{f} \leq \textstyle{\frac{1}{2}}
		\Lpnorm{\phi}{q}{f} + 2^{-1+p/q} \lambda^{1-1/\alpha}
		\Lpnorm{\phi}{p}{g} + \lambda \Lpnorm{\phi}{r}{f},
	\end{gather*}
	hence the conclusion as $\Lpnorm{\phi}{q}{f} < \infty$.
\end{proof}
\begin{remark} \label{remark:basic_interpolation}
	Note $1-1/\alpha = (1/q-1/p)/(1/q-1/r)$.
\end{remark}
\begin{lemma} \label{lemma:subsolutions}
	Suppose $\vdim, \adim \in \nat$, $\vdim < \adim$, $U$ is an open
	subset of $\rel^\adim$, $V \in \RVar_\vdim ( U )$, $T_i \in
	\mathscr{D}_0 ( U )$ are representable by integration for $i \in \{ 1,
	2 \}$, and $u_i : U \to \rel$ for $i \in \{1,2\}$ are Lipschitzian
	functions satisfying
	\begin{gather*}
		\tint{}{} \ap D \theta (z) \bullet \ap D u_i (z) \ud \| V \| z
		\leq T_i ( \theta ) \quad \text{whenever $\theta \in
		\mathscr{D}^0 ( U )^+$, $i \in \{1,2\}$},
	\end{gather*}
	$A = \classification{U}{z}{ u_1 (z) > u_2 (z) }$, and $T = T_1
	\restrict A + T_2 \restrict U \without A$.

	Then $u = \sup \{ u_1, u_2 \}$ satisfies
	\begin{gather*}
		\tint{}{} \ap D \theta (z) \bullet \ap D u(z) \ud \| V \| z
		\leq T ( \theta ) \quad \text{whenver $\theta \in
		\mathscr{D}^0 ( U )^+$}.
	\end{gather*}
\end{lemma}
\begin{proof}
	Using approximation by means of \cite[4.5\,(3)]{snulmenn.decay},
	one infers that the inequality in the hypotheses remains valid for any
	nonnegative, Lipschitzian function $\theta$ with compact support.

	Define $B (f,g) = \tint{}{} f (z) \bullet g (z) \ud \| V \| z$
	whenever $f$ and $g$ are $\|V\|$ measurable functions with $f(z),g(z)
	\in \Hom ( \Tan^\vdim ( \| V \|, z ), \rel )$ for $\| V \|$ almost all
	$z$ and $f \bullet g \in \Lp{1} ( \| V \| )$. (Here, $\| V \|$
	measurability of $f$ means $\| V \|$ measurability of the function
	mapping $z$ onto $f (z) \circ \project{\Tan^\vdim ( \| V \|, a)} \in
	\Hom ( \rel^\adim, \rel^\adim )$ whenever $z \in U$ with $\Tan^\vdim (
	\| V \|, a ) \in \grass{\adim}{\vdim}$ and $(f \bullet g) (z) = f(z)
	\bullet g (z)$ for $z \in \dmn f \cap \dmn g$.)

	Suppose $\theta \in \mathscr{D}^0 ( U )^+$.  Define $\eta_j : U \to
	\rel$ by $\eta_j (z) = \inf \{ j (u_1 - u_2)^+ (z), 1 \}$ for $z \in
	U$ and $j \in \nat$. Noting $0 \leq \eta_j (z) \leq 1$ and
	\begin{gather*}
		\eta_j (z) = 0 \quad \text{if $u_1 (z) \leq u_2 (z)$}, \qquad
		\eta_j (z) = 1 \quad \text{if $u_1 (z) \geq u_2 (z) +
		j^{-1}$}
	\end{gather*}
	whenever $z \in U$ and $j \in \nat$, one estimates
	\begin{gather*}
		\begin{aligned}
			& T_1 ( \eta_j \theta ) + T_2 ( (1-\eta_j) \theta )
			\\
			& \qquad \geq B ( \ap D ( \eta_j \theta ), \ap D u_1 )
			+ B ( \ap D ((1-\eta_j)\theta) , \ap Du_2 ) \\
			& \qquad = B ( \eta_j \ap D \theta, \ap D u_1 ) + B (
			\theta \ap D \eta_j, \ap Du_1) \\
			& \qquad \phantom{=} \ + B ( (1-\eta_j) \ap D \theta,
			\ap D u_2) - B ( \theta \ap D \eta_j, \ap Du_2 ) \\
			& \qquad \geq B ( \eta_j \ap D \theta , \ap D u_1 ) +
			B ( (1-\eta_j) \ap D \theta , \ap D u_2 ).
		\end{aligned}
	\end{gather*}
	Letting $j \to \infty$, this yields the conclusion.
\end{proof}
\begin{lemma} \label{lemma:lebesgue_space_estimate}
	Suppose $\vdim$, $\adim$, $p$, $U$, $V$, and $\psi$ are as in
	\ref{miniremark:situation_general}, $p < \vdim$, $1 \leq M < \infty$,
	$\| V \| ( U ) \leq M$, $1 < s < \infty$, $1 \leq \zeta < \infty$,
	$\zeta \leq \vdim/ (\vdim-2)$ if $\vdim>2$, $t = s \zeta ( \vdim-p
	)/\vdim$, $q = ( 1 - 1 / \zeta + 1 / (\zeta s ) )^{-1}$, $N =
	M^{1/\vdim+1/(2\zeta)-1/2}$, $u : U \to \rel$ Lipschitzian, $f \in
	\Lp{q} ( \| V \| )$,
	\begin{gather*}
		\tint{}{} \ap D \theta \bullet \ap Du \ud \| V \| \leq
		\tint{}{} \theta f \ud \| V \| \quad \text{whenever $\theta
		\in \mathscr{D}^0 ( U )^+$}
	\end{gather*}
	where ``ap'' denotes approximate differentiation with respect to $( \|
	V \|, \vdim )$, $K$ is a compact subset of $U$, and $r = \dist ( K,
	\rel^\adim \without U ) < \infty$.

	Then
	\begin{gather*}
		\eqLpnorm{\| V \| \restrict K}{\zeta s}{ u^+ } \leq \Gamma
		\big ( r^{-2/s} \Lpnorm{\| V \|}{s}{u^+} + \Lpnorm{\| V
		\|}{q}{f} + \psi ( ( u^+ )^t )^{1/t} \big )
	\end{gather*}
	where $\Gamma$ is a positive, finite number depending only on $\vdim$,
	$p$, $s$, $\zeta$, and $N$.
\end{lemma}
\begin{proof}
	Define $\xi = 2 \zeta \vdim / ( 2 \zeta + \vdim )$ and note $1
	\leq \xi < \vdim$ and $\xi \leq 2$. Let $\alpha = ( \vdim-p ) \xi / (
	\vdim-\xi )$ and note $t = \alpha s / 2$. Define
	\begin{gather*}
		\Delta_1 = \sup \big \{ 8 + 8 s^2 (s-1)^{-2}, s^2 (s-1)^{-1}
		\big \}^{1/s}, \quad \Delta_2 = 2 ( N
		\Gamma_{\ref{thm:subcritical_sobolev_embedding}} ( \vdim, p ,
		\xi ) )^{2/s}, \\
		\Delta_3 = ( 2 \Delta_1 \Delta_2 )^{1-s}, \quad \Gamma = 2
		\Delta_2 \sup \{ \Delta_1, \Delta_1/\Delta_3, 1 \}
	\end{gather*}
	and let $h : \{ y : y > 0 \} \to \rel$ by $h(y) = y^{s-1}$ for $y>0$.
	Suppose $\eta \in \mathscr{D}^0 (U)$ with
	\begin{gather*}
		0 \leq \eta (z) \leq 1 \quad \text{and} \quad | D \eta (z) |
		\leq 2 r^{-1} \quad \text{for $z \in U$}, \\
		\eta (z) = 1 \quad \text{for $z \in K$}.
	\end{gather*}

	Assume $f \geq 0$. First, \emph{the case $\inf \im u > 0$} is treated.

	Let $\theta = \eta^2 h \circ u$. Note $\theta$ is Lipschitzian and
	\begin{gather*}
		\ap D \theta (z) = \eta^2 (z) h ' (u(z)) \ap D u (z) + 2 \eta
		(z) h ( u(z)) \ap D \eta (z)
	\end{gather*}
	for $\| V \|$ almost all $z \in U$. Via approximation by
	\cite[4.5\,(3)]{snulmenn.decay}, one infers that $\theta$ is
	admissible in the hypotheses. Therefore one estimates
	\begin{gather*}
		\tint{}{} \eta^2 h' \circ u | \ap D u |^2 \ud \| V \| \leq
		\tint{}{} 2 \eta h \circ u | \ap D \eta | | \ap Du | + \eta^2
		h \circ u f \ud \| V \|, \\
		2 \eta h \circ u | \ap D \eta | | \ap Du | \leq
		{\textstyle\frac{1}{2}} (s-1) \eta^2 u^{s-2} | \ap Du |^2 + 2
		(s-1)^{-1} | \ap D \eta |^2 u^s, \\
		\begin{aligned}
			& \tint{}{} \eta^2 u^{s-2} | \ap D u |^2 \ud \| V \|
			\\
			& \qquad \leq 4 (s-1)^{-2} \tint{}{} | \ap D \eta |^2
			u^s \ud \| V \| + 2 (s-1)^{-1} \tint{}{} \eta^2
			u^{s-1} f \ud \| V \|.
		\end{aligned}
	\end{gather*}
	Defining $v = u^{s/2}$ and noting that $v$ is Lipschitzian with
	\begin{gather*}
		| \ap D v(z) |^2 = {\textstyle\frac{1}{4}} s^2 u^{s-2} (z) |
		\ap D u (z) |^2 \quad \text{for $\| V \|$ almost all $z$},
	\end{gather*}
	one obtains
	\begin{gather*}
		\begin{aligned}
			& \tint{}{} \eta^2 | \ap D v |^2 \ud \| V \| \\
			& \qquad \leq 4 s^2 (s-1)^{-2} r^{-2} \tint{}{} u^s
			\ud \| V \| + \textstyle\frac{1}{2} s^2 (s-1)^{-1}
			\tint{}{} \eta^2 u^{s-1} f \ud \| V \|,
		\end{aligned} \\
		\tint{}{} | \ap D ( \eta v ) |^2 \ud \| V \| \leq (
		\Delta_1)^s \big ( r^{-2} \tint{}{} u^s \ud \| V \| +
		\tint{}{} \eta^2 u^{s-1} f \ud \| V \| \big )
	\end{gather*}
	In order to estimate the last term, note $1-1/q=(1/\zeta)
	(1-1/s)$, hence $1 < q < \infty$, and compute by use of H\"older's
	inequality and the inequality relating arithmetic and geometric means
	\begin{gather*}
		\begin{aligned}
			& \tint{}{} \eta^2 u^{s-1} f \ud \| V \| \leq
			\Lpnorm{\| V \|}{\zeta s}{\eta^{2/s} u}^{s-1}
			\Lpnorm{\| V \|}{q}{\eta^{2/s} f} \\
			& \qquad \leq (1-1/s) (\Delta_3)^{s/(s-1)} \Lpnorm{\|
			V \|}{\zeta s}{\eta^{2/s} u}^s + (1/s) (\Delta_3)^{-s}
			\Lpnorm{\| V \|}{q}{f}^s.
		\end{aligned}
	\end{gather*} Defining $w = \eta v$, this implies
	\begin{multline*}
		\Lpnorm{\| V \|}{2}{| \ap D w|}^{2/s} \\
		\leq \Delta_1 \big ( r^{-2/s} \Lpnorm{\| V \|}{s}{u} +
		(\Delta_3)^{1/(s-1)} \Lpnorm{\| V \|}{2\zeta}{w}^{2/s} + (
		\Delta_3 )^{-1} \Lpnorm{\| V \|}{q}{f} \big ).
	\end{multline*}
	By \ref{thm:subcritical_sobolev_embedding} and
	H\"older's inequality there holds
	\begin{gather*}
		\Lpnorm{\| V \|}{2\zeta}{w} \leq
		\Gamma_{\ref{thm:subcritical_sobolev_embedding}} ( \vdim, p,
		\xi ) \big ( \Lpnorm{\| V \|}{\xi}{| \ap Dw |} + \psi (
		w^\alpha)^{1/\alpha} \big ), \\
		\Lpnorm{\| V \|}{2 \zeta}{w}^{2/s} \leq \Delta_2 \big (
		\Lpnorm{\| V \|}{2}{| \ap Dw|}^{2/s} + ( \psi \restrict \spt
		\eta ) ( u^t)^{1/t} \big ).
	\end{gather*}
	Since $\Delta_1 \Delta_2 (\Delta_3)^{1/(s-1)} \leq 1/2$, combining
	this with the preceding estimate yields the conclusion in the
	present case with $\psi$ replaced by $\psi \restrict \spt \eta$.

	To treat the \emph{general case}, note that for $\delta \in \rel$ the
	hypotheses of the lemma are also satisfied with $u$ replaced by $\sup
	\{ u, \delta \}$ by \ref{lemma:subsolutions} and consider the limit
	$\delta \to 0+$.
\end{proof}
\begin{remark} \label{remark:lebesgue_space_estimate}
	In the special case that for some $a \in \rel^\adim$ and some $0 < r <
	\infty$, $U = \oball{a}{2r}$ and $K = \cball{a}{r}$ the conclusion
	takes the form
	\begin{multline*}
		r^{-1-\vdim/(\zeta s)} \eqLpnorm{\| V \| \restrict
		\cball{a}{r}}{\zeta s}{u^+} \\
		\leq \Gamma \big ( 2 r^{-1-\vdim/s} \Lpnorm{\| V \|}{s}{u^+} +
		r^{1-\vdim/q} \Lpnorm{\| V \|}{q}{f} + r^{-1+p/t-\vdim/t} \psi
		( ( u^+ )^t)^{1/t} \big )
	\end{multline*}
	as may be verified first in the case $r=1$ and then via rescaling in
	the general case.
\end{remark}
\begin{theorem} \label{thm:second_order_lq_flatness}
	Suppose $\vdim$, $\adim$, $p$, $U$, $V$ are as in
	\ref{miniremark:situation_general}, $1 < p < \vdim$, $\beta = \vdim
	p/(\vdim-p)$, and $V \in \IVar_\vdim ( U )$.

	Then for $V$ almost all $(a,T)$
	\begin{gather*}
		\limsup_{r \to 0+} r^{-2-\vdim/\beta} \big (
		\tint{\cball{a}{r}}{} \dist (z-a,T)^\beta \ud \| V \| z \big
		)^{1/\beta} < \infty.
	\end{gather*}
\end{theorem}
\begin{proof}
	Suppose $(a,T)$ is such that for some $\| V \|$ measurable set $A$
	there holds
	\begin{gather*}
		0 < \density^\vdim ( \| V \|, a ) < \infty, \quad
		\density^{\ast \vdim} ( \psi, a ) < \infty, \quad
		\density^\vdim ( \| V \| \restrict U \without A, a ) = 0, \\
		\limsup_{r \to 0+} r^{-2} \sup \{ \dist (z-a,T) \with z \in A
		\cap \cball{a}{r} \} < \infty
	\end{gather*}
	where $\psi$ is as in \ref{miniremark:situation_general}. Note that
	these conditions are met by $V$ almost all $(a,T)$ by
	\cite[2.10.19\,(3)\,(4)]{MR41:1976} and \cite[4.8]{snulmenn.c2}.
	Defining $f = | \mathbf{h} ( V ; \cdot ) |$ and $u : U \to \rel$ by
	$u(z) = \dist (z-a,T)$ for $z \in U$, there holds
	\begin{gather*}
		\tint{}{} \ap D \theta \bullet \ap D u \ud \| V \| \leq
		\tint{}{} \theta f \ud \| V \| \quad \text{for $\theta \in
		\mathscr{D}^0 ( U )^+$}
	\end{gather*}
	where ``ap'' denotes approximate differentiation with respect to $( \|
	V \|, \vdim )$ by \ref{ex:some_lap}\,\eqref{item:some_lap:convex}.

	Define $\zeta = \vdim/(\vdim-1)$, $s = \beta / \zeta$. Choose $1 \leq
	M < \infty$ and $0 < R \leq 1$ such that
	\begin{gather*}
		\cball{a}{2R} \subset U, \quad M^{-1} r^\vdim \leq
		\measureball{\| V \|}{\cball{a}{r}} \leq M r^\vdim, \\
		\measureball{\psi}{\cball{a}{r}} \leq M r^\vdim, \quad \sup u
		\lIm A \cap \cball{a}{r} \rIm \leq M r^2
	\end{gather*}
	for $0 < r \leq 2R$. Let
	\begin{gather*}
		N = 2^{1/2} M^{1/(2\vdim)}, \quad \Delta_1 =
		\Gamma_{\ref{lemma:lebesgue_space_estimate}} ( \vdim, p, s,
		\zeta, N ) \sup \{ 2^{2+\vdim/s} , 2^{2+\vdim/p} M^{1/p} \},
		\\
		\lambda = \inf \big \{ M^{1/(\beta(1-\vdim))}, ( 8
		\Delta_1)^{-1} \big \}, \quad \gamma = (1-1/s)/(1/s-1/\beta),
		\\
		\Delta_2 = 2 \lambda^{-\gamma} M^2 + 2 \Delta_1.
	\end{gather*}
	Possibly replacing $R$ by a smaller number, one may assume that
	\begin{gather*}
		\| V \| ( \cball{a}{r} \without A ) \leq
		\lambda^{\beta(\vdim-1)} r^\vdim \quad \text{for $0 < r \leq
		2R$}.
	\end{gather*}
	Define $g : \{ r \with 0 < r \leq 2R \} \to \rel$ and $h : \{ r \with
	0 < r \leq 2R \} \to \rel$ by
	\begin{gather*}
		g(r) = r^{-1-\vdim/\beta} \eqLpnorm{\| V \| \restrict
		\cball{a}{r}}{\beta}{u}, \quad h(r) = r^{-1-\vdim/s}
		\eqLpnorm{\| V \| \restrict \cball{a}{r}}{s}{u}
	\end{gather*}
	for $0 < r \leq 2R$.

	Noting that $r^{-\vdim} \| V \| \restrict \cball{a}{r}$ is nonatomic
	by \ref{example:nonatomic} and $s>1$, $s/(\beta-s) = \vdim-1$, and
	\ref{remark:basic_interpolation},
	\begin{gather*}
		\| V \| ( \cball{a}{r} \without A ) \leq \lambda^{\beta
		(\vdim-1)} r^\vdim \leq M^{-1} r^\vdim \leq \measureball{\| V
		\|}{\cball{a}{r}} < \infty, \\
		r^{-\vdim} \eqLpnorm{\| V \| \restrict \cball{a}{r} \cap
		A}{1}{u} \leq M^2 r^2,
	\end{gather*}
	one applies \ref{lemma:basic_interpolation} with $p$, $q$, $r$,
	$\phi$, $X$, $f$, $B$ replaced by $1$, $s$, $\beta$, $r^{-\vdim} \| V
	\| \restrict \cball{a}{r}$, $U$, $u$, $U \without A$ to infer that
	\begin{gather*}
		r h(r) \leq 2 \lambda^{-\gamma} M^2 r^2 + 2 \lambda r g (r)
		\quad \text{for $0 < r \leq R$}.
	\end{gather*}
	Also, \ref{lemma:lebesgue_space_estimate} and
	\ref{remark:lebesgue_space_estimate} imply, noting $s \zeta ( \vdim-p
	) / \vdim = p$, $1-1/\zeta+1/(\zeta s) = 1/p$, and $1/\vdim + 1/
	(2\zeta) - 1/2 = 1/(2\vdim)$,
	\begin{gather*}
		g(r) \leq \Gamma_{\ref{lemma:lebesgue_space_estimate}} (
		\vdim, p, s, \zeta, N ) \big ( 2^{2+\vdim/s} h (2r) +
		2^{2+\vdim/p} M^{1/p} r \big ) \leq \Delta_1 ( h(2r) + r )
	\end{gather*}
	for $0 < r \leq R$. Combining these two inequalities, one obtains
	\begin{gather*}
		r^{-1} h(r) \leq  2 \Delta_1 \lambda r^{-1} h (2r) + \Delta_2
		\quad \text{for $0 < r \leq R$}.
	\end{gather*}
	Defining $\Gamma = \sup \{ 2^{2+\vdim/s} R^{-1} h (R), 2 \Delta_2 \}$,
	\emph{it will be shown that
	\begin{gather*}
		h(r) \leq \Gamma r \quad \text{whenever $0 <r \leq R$};
	\end{gather*}}
	in fact this clearly holds if $r \geq R/2$ and if $h(2r) \leq 2 \Gamma
	r$ for some $0 < r \leq R/2$ then
	\begin{gather*}
		r^{-1} h(r) \leq 4 \Delta_1 \lambda \Gamma + \Delta_2 \leq
		\Gamma.
	\end{gather*}
	Therefore $g(r) \leq \Delta_1 ( 2 \Gamma + 1 ) r$ for $0 < r \leq R/2$.
\end{proof}
\section{Second order diffentiability of the support of integral
varifolds with critical integrability of the mean curvature} In
\ref{def:diffferentiability}--\ref{remark:twice_diff} a concept of second
order differentiability of closed sets is introduced. It is applied in
combination with the results of Section \ref{sec:local_max} to the study of
the support of varifolds with critical integrability of the first variation
($p=\vdim$ in \ref{miniremark:situation_general}), see
\ref{thm:locality_mean_curvature}. In case of integral varifolds, combining
these results with the author's second order rectifiability result in
\cite[4.8]{snulmenn.c2}, one obtains that the support is twice differentiable
with the mean curvature of the support agreeing with the variationally defined
mean curvature of the varifold at almost all points, see
\ref{corollary:second_order_agreement}.
\begin{miniremark} \label{miniremark:hausdorff_distance}
	Suppose $\varrho$ is the metric of a nonempty, boundedly compact
	metric space $X$, $F$ is the family of nonempty closed subsets of $X$
	equipped with the topology induced from the cartesian product topology
	on $\rel^X$ via the univalent map $\Omega : F \to
	\classification{\rel^X}{f}{\Lip f \leq 1}$ defined by
	\begin{gather*}
		\Omega_A (x) = \dist (x,A) = \inf \varrho \lIm \{x\} \times A
		\rIm \quad \text{for $A \in F$ and $x \in X$}.
	\end{gather*}
	The same topology results if uniform convergence on bounded subsets of
	$X$ is considered on $\rel^X \cap \{ f \with \Lip f \leq 1 \}$, see
	\cite[2.10.21]{MR41:1976}. The resulting notion of convergence in $F$
	is termed \emph{locally in Hausdorff distance}. One notes for $A, B, C
	\in F$
	\begin{gather*}
		\sup \Omega_A \lIm B \cap C \rIm \cup \Omega_B \lIm A \cap C
		\rIm \leq \sup | \Omega_A - \Omega_B | \lIm C \rIm, \\
		\sup | \Omega_A - \Omega_B | \lIm C \rIm \leq \sup \Omega_A
		\lIm C(s) \cap B \rIm \cup \Omega_B \lIm C(s) \cap A \rIm
	\end{gather*}
	where $s = \sup \Omega_A \lIm C \rIm \cup \Omega_B \lIm C \rIm$ and
	$C(s) = \bigcup \{ \cball{c}{s} \with c \in C \}$. If $C$ is bounded
	so is $C(s)$. In particular, $A_i$ converges to $A$ locally in
	Hausdorff distance if and only if for some $x \in X$
	\begin{gather*}
		\sup \Omega_{A_i} \lIm A \cap \cball{x}{r} \rIm \cup \Omega_A
		\lIm A_i \cap \cball{x}{r} \rIm \to 0 \quad \text{as $i \to
		\infty$ for $0 < r < \infty$}.
	\end{gather*}
\end{miniremark}
\begin{miniremark} \label{miniremark:distance_comparison}
	If $\vdim, \adim \in \nat$, $\vdim < \adim$, $f : \rel^\vdim \to
	\rel^\codim$ is locally Lipschitzian, and $(x,y) \in \rel^\vdim \times
	\rel^\codim$ then (see \ref{miniremark:hausdorff_distance})
	\begin{gather*}
		\Omega_f (x,y) \leq | f(x)-y | \leq ( 1 + \Lip f |
		\cball{x}{\Omega_f (x,y)} ) \Omega_f (x,y);
	\end{gather*}
	in fact, abbreviating $L = \Lip f | \cball{x}{\Omega_f(x,y)}$ and
	choosing $v \in \rel^\vdim$ with with $\Omega_f (x,y) =
	|(v,f(v))-(x,y)|$, one notes $|x-v| \leq \Omega_f (x,y)$ and
	estimates
	\begin{gather*}
		\begin{split}
			|f(x)-y| & \leq |f(x)-f(v)| + |f(v)-y| \\
			& \leq L |x-v| + | (v,f(v))
			- (x,y) | \leq (1+L) \Omega_f (x,y).
		\end{split}
	\end{gather*}
\end{miniremark}
\begin{lemma} \label{lemma:equivalent_blowups}
	Suppose $\vdim, \adim \in \nat$, $\vdim < \adim$, $1 \leq \mu <
	\infty$, $A \subset \rel^\vdim \times \rel^\codim$, $f : \rel^\vdim
	\to \rel^\codim$ is locally Lipschitzian, $h : \rel^\vdim \times
	\rel^\codim \to \rel$, $\phi_r : \rel^\vdim \times \rel^\codim \to
	\rel^\vdim \times \rel^\codim$, and
	\begin{gather*}
		f(tx)=t^\mu f(x), \quad h(x,y)=|y-f(x)|, \quad \phi_r(x,y) =
		(r^{-1} x,r^{-\mu} y)
	\end{gather*}
	whenever $(x,y) \in \rel^\vdim \times \rel^\codim$, $0 \leq t <
	\infty$ and $0 < r < \infty$.

	Then the following conditions are equivalent (see
	\ref{miniremark:hausdorff_distance}):
	\begin{enumerate}
		\item \label{item:equivalent_blowups:g} $r^{-\mu} \sup
		\Omega_f \lIm
		A \cap \cball{0}{r} \rIm \to 0$ as $r \to 0+$,
		\item \label{item:equivalent_blowups:h} $r^{-\mu} \sup h \lIm
		A \cap \cball{0}{r} \rIm \to 0$ as $r \to 0+$,
		\item \label{item:equivalent_blowups:g_scaling} $\sup \Omega_f
		\lIm \phi_r \lIm A \rIm \cap \cball{0}{R} \rIm \to 0$ as $r
		\to 0+$ whenever $0 < R < \infty$,
		\item \label{item:equivalent_blowups:h_scaling} $\sup h \lIm
		\phi_r \lIm A \rIm \cap \cball{0}{R} \rIm \to 0$ as $r \to 0+$
		whenever $0 < R < \infty$.
	\end{enumerate}
	Moreover, the list may be amplified by replacing ``whenever'' by ``for
	some'' in \eqref{item:equivalent_blowups:g_scaling} and
	\eqref{item:equivalent_blowups:h_scaling}. If the conditions
	hold and
	\begin{gather*}
		\sup \Omega_{\boldsymbol{\tau}_{1/r} \lIm A \rIm} \lIm (
		\rel^\vdim \times \{0\} ) \cap \cball{0}{R} \rIm \to 0 \quad
		\text{as $r \to 0+$ for $0 < R < \infty$},
	\end{gather*}
	then
	\begin{gather*}
		\phi_r \lIm \Clos A \rIm \to f \quad \text{locally in
		Hausdorff distance as $r \to 0+$}.
	\end{gather*}
\end{lemma}
\begin{proof}
	Denote the conditions resulting from
	\eqref{item:equivalent_blowups:g_scaling} and
	\eqref{item:equivalent_blowups:h_scaling} by replacing ``whenever'' by
	``for some'' by \eqref{item:equivalent_blowups:g_scaling}$'$ and
	\eqref{item:equivalent_blowups:h_scaling}$'$. Clearly,
	\eqref{item:equivalent_blowups:g_scaling} implies
	\eqref{item:equivalent_blowups:g_scaling}$'$ and
	\eqref{item:equivalent_blowups:h_scaling} implies
	\eqref{item:equivalent_blowups:h_scaling}$'$. Note also that each of
	the six conditions implies $0 \in \Clos A$. Abbreviating $L = \Lip ( f
	| \cball{0}{1} ) < \infty$, one obtains
	\begin{gather*}
		\Lip ( f | \cball{x}{\Omega_f(x,y)} ) \leq \Lip ( f |
		\cball{0}{2R} ) \leq ( 2R )^{\mu-1} L, \\
		\Omega_f(x,y) \leq h(x,y) \leq ( 1 + (2R)^{\mu-1} L ) \Omega_f
		(x,y)
	\end{gather*}
	whenever $(x,y) \in \cball{0}{R} \subset \rel^\vdim \times
	\rel^\codim$ and $0 < R < \infty$ by
	\ref{miniremark:distance_comparison}. Hence
	\eqref{item:equivalent_blowups:g} is equivalent to
	\eqref{item:equivalent_blowups:h},
	\eqref{item:equivalent_blowups:g_scaling} is equivalent to
	\eqref{item:equivalent_blowups:h_scaling} and
	\eqref{item:equivalent_blowups:g_scaling}$'$ is equivalent to
	\eqref{item:equivalent_blowups:h_scaling}$'$. Therefore in order to
	prove the equivalences it is sufficient to show
	\eqref{item:equivalent_blowups:h} implies
	\eqref{item:equivalent_blowups:h_scaling} and
	\eqref{item:equivalent_blowups:g_scaling}$'$ implies
	\eqref{item:equivalent_blowups:h}.

	The first implication is a consequence of the facts
	\begin{gather*}
		h \circ \phi_r^{-1} = r^\mu h, \quad \phi_r^{-1} \lIm
		\cball{0}{R} \rIm \subset \cball{0}{rR}
	\end{gather*}
	whenever $0 < r \leq 1$ and $0 < R < \infty$.

	If the second implication were false there would exist $0 < R <
	\infty$ with
	\begin{gather*}
		\sup \Omega_f \lIm \phi_r \lIm A \rIm \cap \cball{0}{R} \rIm
		\to 0 \quad \text{as $r \to 0+$},
	\end{gather*}
	$\varepsilon > 0$, and a sequence $(x_i,y_i) \in A$ tending to $0$ as
	$i \to \infty$ with $h(x_i,y_i) > \varepsilon |x_i|^\mu$. Choosing $0
	< r(i) < \infty$ with $|\phi_{r(i)} ( x_i,y_i )| = R$, one could
	assume, possibly passing to a subsequence, the existence of $(x,y) \in
	\rel^\vdim \times \rel^\codim$ such that
	\begin{gather*}
		\phi_{r(i)} (x_i,y_i) \to (x,y) \quad \text{as $i \to
		\infty$} \qquad \text{and} \qquad |(x,y)| = R.
	\end{gather*}
	Since $\Omega_f ( \phi_{r(i)} ( x_i, y_i ) ) \to 0$ as $i \to \infty$,
	one would infer $f(x)=y$ and, since
	\begin{gather*}
		\varepsilon |x|^\mu \gets \varepsilon | r(i)^{-1} x_i |^\mu <
		r(i)^{-\mu} h (x_i,y_i) = h ( \phi_{r(i)} (x_i,y_i) ) \to
		h(x,y) = 0
	\end{gather*}
	as $i \to \infty$, also $x=0$, hence $y=f(0)=0$ in contradiction to
	$|(x,y)|=R$.

	To prove the postscript, it is sufficient to show
	\begin{gather*}
		\sup \Omega_{\phi_r \lIm A \rIm} \lIm f \cap \cball{0}{R} \rIm
		\to 0 \quad \text{as $r \to 0+$ whenever $0 < R < \infty$}
	\end{gather*}
	by \ref{miniremark:hausdorff_distance}. For this purpose suppose $0 <
	R < \infty$, $L = \Lip ( f | \cball{0}{2R} )$ and $0 < \varepsilon
	\leq R$ and choose $0 < s < \infty$ such that
	\begin{gather*}
		\sup \Omega_{\boldsymbol{\tau}_{1/r} \lIm A \rIm} \lIm (
		\rel^\vdim \times \{ 0 \} ) \cap \cball{0}{R} \rIm <
		\varepsilon, \quad \sup h \lIm A \cap \cball{0}{3Rr} \rIm \leq
		\varepsilon r^\mu
	\end{gather*}
	whenever $0 < r \leq s$. If $(x,y) \in f \cap \cball{0}{R}$ and $0 < r
	\leq s$ there exists $(v,w) \in A$ with $| \boldsymbol{\tau}_{1/r}
	(v,w) - (x,0) | \leq \varepsilon$ hence
	\begin{gather*}
		|v-rx| \leq \varepsilon r, \quad |v| \leq 2Rr, \quad |w| \leq
		Rr, \quad (v,w) \in \cball{0}{3Rr}, \\
		|\phi_r(v,w) - (x,y)| \leq |r^{-1}v-x| + r^{-\mu} |w-f(v)| +
		|f(r^{-1}v)-f(x)| \leq (2+L) \varepsilon
	\end{gather*}
	and the postscript follows.
\end{proof}
\begin{definition} \label{def:diffferentiability}
	Suppose $a \in A \subset \rel^\adim$ and $U \cap \Clos A = U \cap A$
	for some neighbourhood $U$ of $a$.
	
	Then $A$ is called \emph{differentiable at $a$} if and only if for
	some linear subspace $T$ of $\rel^\adim$ there holds
	\begin{gather*}
		\boldsymbol{\mu}_{1/r} \circ \boldsymbol{\tau}_{-a} \lIm \Clos
		A \rIm \to T \quad \text{locally in Hausdorff distance as $r
		\to 0+$}.
	\end{gather*}
\end{definition}
\begin{remark} \label{remark:once_diff}
	Note $T = \Tan (A,a)$ in case $A$ is differentiable at $a$ but
	$A$ needs not to be differentiable at $a$ if $\Tan(A,a)$ is a linear
	subspace of $\rel^\adim$. If $A$ is a submanifold of
	$\rel^\adim$ then $A$ is differentiable at $a$ by
	\ref{miniremark:distance_comparison}.
\end{remark}
\begin{definition} \label{def:twice_diff}
	Suppose $a \in A \subset \rel^\adim$ and $U \cap \Clos A = U \cap A$
	for some neighbourhood $U$ of $a$.
	
	Then $A$ is called \emph{twice differentiable at $a$} if and only if
	$A$ is differentiable at $a$ and, in case $0 < \vdim = \dim \Tan(A,a)
	< \adim$, there exists a homogeneous polynomial function $Q : \Tan
	(A,a) \to \Nor (A,a)$ of degree $2$ such that with
	\begin{gather*}
		\tau : \Tan (A,a) \times \Nor (A,a) \to \rel^\adim, \\
		\tau (v,w)=v+w \quad \text{for $v \in \Tan (A,a)$, $w \in \Nor
		(A,a)$}, \\
		\phi_r = r^{-1} \project{\Tan(A,a)} + r^{-2} \project{\Nor
		(A,a)} \quad \text{for $0 < r < \infty$}
	\end{gather*}
	there holds
	\begin{gather*}
		\phi_r \circ \boldsymbol{\tau}_{-a} \lIm \Clos A \rIm \to \tau
		\lIm Q \rIm \quad \text{locally in Hausdorff distance as $r
		\to 0+$.}
	\end{gather*}
	Since $\tau$ is an isometry, $Q$ is uniquely determined by $A$ and
	$a$, hence the \emph{second fundamental form $\mathbf{b} (A;a)$ of $A$
	at $a$} may be defined by $\mathbf{b} (A;a) = D^2 Q (0)$.
\end{definition}
\begin{remark} \label{remark:twice_diff}
	If $A$ is a submanifold of class $2$ of $\rel^\adim$ then $A$ is twice
	differentiable at $a$ and the definition of $\mathbf{b}(A;a)$ in
	\ref{def:twice_diff} agrees with classical definitions by
	\ref{remark:once_diff} and
	\ref{lemma:equivalent_blowups}\,\eqref{item:equivalent_blowups:h}.
\end{remark}
\begin{lemma} \label{lemma:differentiation}
	Suppose $\vdim$, $\adim$, $p$, $U$, $V$, and are as in
	\ref{miniremark:situation_general}, $p=\vdim$, $f : U \to \{ t \with
	0 \leq t \leq \infty \}$, $0 \leq \nu < \infty$, $0 < \mu < \infty$,
	and $A$ is the set of all $a \in \spt \| V \|$ such that
	\begin{gather*}
		f (z) \leq \nu |z-a|^\mu \quad \text{whenever $z \in \spt \| V
		\|$ with $|z-a| < \nu^{-1}$}.
	\end{gather*}

	Then $A$ is closed relative to $U$ and if $a \in \spt \| V \|$ with
	$\density^\vdim ( \| V \| \restrict U \without A, a ) = 0$ and, in
	case $\vdim=1$, $\| \delta V \| ( \{a\} ) < \isoperimetric{1}^{-1}$,
	then
	\begin{gather*}
		\lim_{r \to 0+} \sup \{ |z-a|^{-\mu} f(z) \with z \in \spt \|
		V \| \cap \cball{a}{r} \without \{ a \} \} = 0.
	\end{gather*}
\end{lemma}
\begin{proof}
	One readily verifies that $A$ is closed relative to $U$.

	Note $f (z) = 0$ for $z \in A$, suppose $a$ satisfies the hypotheses
	of the second part of the conclusion, let $\varepsilon > 0$ and choose
	$0 \leq \alpha < \isoperimetric{\vdim}^{-1}$ and $0 < r < \nu^{-1}$
	such that $\cball{a}{2r} \subset U$ and
	\begin{gather*}
		\psi(\cball{a}{2r})^{1/\vdim} < \alpha \qquad \| V \| (
		\cball{a}{2s} \without A ) \leq \varepsilon^\vdim s^\vdim
		\quad \text{whenever $0 < s \leq r$}
	\end{gather*}
	with $\psi$ as in \ref{miniremark:situation_general}. Observe $a \in A
	\cap \Clos ( \spt \| V \| \without \{a \} )$ since
	$\density_\ast^\vdim ( \| V \|, a) > 0$ by
	\cite[2.5]{snulmenn.isoperimetric}. Suppose $z \in \spt \| V \| \cap
	\cball{a}{r} \without A$ and $w \in A$ such that $|z-w| = \dist
	(z,A)$. Note $0 < |z-w| \leq |z-a| \leq r < \nu^{-1}$. One infers with
	$\Delta = ( \isoperimetric{\vdim}^{-1} - \alpha ) / \vdim$
	\begin{gather*}
		\measureball{\| V \|}{\oball{z}{|z-w|}} \geq \Delta^\vdim
		|z-w|^\vdim
	\end{gather*}
	from \cite[2.5]{snulmenn.isoperimetric}, hence, noting
	$\oball{z}{|z-w|} \subset \cball{a}{2|z-a|} \without A$,
	\begin{gather*}
		| z-w |^\vdim \leq \Delta^{-\vdim} \| V \| ( \cball{a}{2|z-a|}
		\without A ) \leq (\varepsilon |z-a|/\Delta)^\vdim, \\
		f (z) \leq \nu |z-w|^\mu \leq \nu ( \varepsilon |z-a|/ \Delta
		)^\mu. \qedhere
	\end{gather*}
\end{proof}
\begin{definition}
	Suppose $V$ is a finite dimensional inner product space, $W$ is a
	vector space, and $f : V \times V \to W$ is a bilinear map.

	Then $\trace f \in W$ is characterised by
	\begin{gather*}
		q ( \trace f ) = \trace ( q \circ f )
	\end{gather*}
	whenever $q : W \to \rel$ is a linear map.
\end{definition}
\begin{lemma} \label{lemma:virtual_tangent_planes}
	Suppose $\vdim, \adim \in \nat$, $\vdim < \adim$, $Z$ is an open
	subset of $\rel^\adim$,
	\begin{gather*}
		G = \classification{\Hom ( \rel^\adim,
		\rel^\adim)}{h}{\text{$h^\ast = h = h \circ h$ and $\trace h =
		\vdim$} },
	\end{gather*}
	$Q : Z \to G$ is of class $\class{1}$ and $N$ is the set of all
	functions $g$ of class $\class{1}$ mapping $Z$ into $\rel^\adim$ such
	that
	\begin{gather*}
		Q (z)(g(z)) = 0 \quad \text{for $z \in Z$}.
	\end{gather*}
	
	Then there exists a function $h : Z \to \rel^\adim$ uniquely
	characterised by
	\begin{gather*}
		Q(z)(h(z)) = 0, \quad Q(z) \bullet Dg(z) = - h (z) \bullet
		g(z)
	\end{gather*}
	whenever $z \in Z$ and $g \in N$. The function $h$ is continuous.
	Moreover, if $M$ is an $\vdim$ dimensional submanifold of $\rel^\adim$
	of class $\class{2}$ satisfying
	\begin{gather*}
		\Tan ( M,z ) = \im Q (z) \quad \text{whenever $z \in Z \cap
		M$}
	\end{gather*}
	then $h(z) = \mathbf{h} ( M;z)$ whenever $z \in Z \cap M$.
\end{lemma}
\begin{proof}
	Since $\ker Q(z) = \{ g (z) \with g \in N \}$ for $z \in Z$, the
	uniqueness of $h$ follows.

	Define $b(z) : \rel^\adim \times \rel^\adim \to \rel^\adim$ for $z \in
	Z$ by
	\begin{gather*}
		b(z)(v,w) = \left < Q(z)(v), \left < Q(z)(w), DQ(z) \right >
		\right > \quad \text{for $v,w \in \rel^\adim$}
	\end{gather*}
	and $h : Z \to \rel^\adim$ by
	\begin{gather*}
		h(z) = \trace b(z) \quad \text{for $z \in Z$}.
	\end{gather*}
	Differentiating the identity $Q(z) \circ Q(z) = Q(z)$ for $z \in Z$ in
	direction $v$ and composing the result with $Q(z)$ yields
	\begin{gather*}
		Q(z) \circ \left < v, DQ(z) \right > \circ Q(z) = 0 \quad
		\text{for $z \in Z$, $v \in \rel^\adim$},
	\end{gather*}
	hence $\im b(z) \subset \ker Q(z)$ and $Q(z)(h(z)) = 0$ for $z \in Z$.
	Whenever $z \in Z$, $v \in \rel^\adim$, $g \in N$ and $e_1, \ldots,
	e_\adim$ form an orthonormal base of $\rel^\adim$, one uses $Q(\xi) =
	Q(\xi)^\ast$ for $\xi \in Z$ hence $\left < v, DQ(z) \right
	> = \left < v, DQ (z) \right >^\ast$ for $v \in \rel^\adim$ to compute
	\begin{gather*}
		\begin{aligned}
			g(z) & = \left < g(z), \id{\rel^\adim} - Q(z) \right
			>, \\
			\left < v, Dg(z) \right > & = \left < \left < v, Dg(z)
			\right > , \id{\rel^\adim} - Q(z) \right > - \left < g
			(z), \left < v, DQ(z) \right > \right >, \\
			Q (z) \bullet Dg(z) & = - \tsum{i=1}{\adim} \left < g
			(z), \left < Q(z)(e_i), DQ(z) \right > \right >
			\bullet Q(z)(e_i) \\
			& = - g(z) \bullet \tsum{i=1}{\adim} \left <
			Q(z)(e_i), \left < Q(z)(e_i), DQ(z) \right > \right >
			= - h (z) \bullet g(z).
		\end{aligned}
	\end{gather*}

	The postscript now follows from the analoguous characterisation of
	$\mathbf{h} ( M;\cdot )$ cp. e.g. Allard \cite[2.5\,(2)]{MR0307015}.
\end{proof}
\begin{theorem} \label{thm:locality_mean_curvature}
	Suppose $\vdim$, $\adim$, $p$, $U$, $V$, and ``$\ap\!$'' are as in
	\ref{miniremark:situation_general}, $p = \vdim < \adim$, $M$ is an
	$\vdim$ dimensional submanifold of $\rel^\adim$ of class $\class{2}$,
	$\Omega_M : U \to \rel$ satisfies $\Omega_M (z) = \dist (z,M)$ for $z
	\in \rel^\adim$ and $R(z) = \project{\Tan^\vdim ( \| V \|, z )}$
	whenever $\Tan^\vdim ( \| V \|, z ) \in \grass{\adim}{\vdim}$.

	Then for $\mathscr{H}^\vdim$ almost all $a \in M \cap \spt \| V \|$
	there holds
	\begin{enumerate}
		\item \label{item:locality_mean_curvature:flatness} $\lim_{r
		\to 0+} r^{-2} \sup \Omega_M \lIm \cball{a}{r} \cap \spt \| V
		\| \rIm = 0$.
		\item \label{item:locality_mean_curvature:differentiability}
		$\spt \| V \|$ is twice differentiable at $a$ with
		\begin{gather*}
			\Tan ( \spt \| V \|, a ) = \Tan ( M, a) = \Tan^\vdim (
			\| V \|, a ), \\
			\mathbf{b} ( \spt \| V \|; a ) = \mathbf{b} ( M; a ),
			\quad \project{\Nor ( M, a )} ( \mathbf{h}
			(V;a) ) = \mathbf{h} (M;a).
		\end{gather*}
		\item \label{item:locality_mean_curvature:tangent_map} $R$ is
		$( \| V \|, \vdim )$ approximately differentiable at $a$ with
		\begin{gather*}
			\tfint{\cball{a}{r}}{} ( | R(z)-R(a)-
			\left < R(a)(z-a), \ap DR (a) \right > |/|z-a| )^2 \ud
			\| V \| z \to 0
		\end{gather*}
		as $r \to 0+$.
	\end{enumerate}
\end{theorem}
\begin{proof}
	Abbreviate $S = \spt \| V \|$ and suppose $\psi$ and $\eta$
	are as in \ref{miniremark:situation_general} and $\beta = \infty$ if
	$\vdim = 1$ and $\beta = \vdim/(\vdim-1)$ if $\vdim>1$. Note
	$\density_\ast^\vdim ( \| V \|, a ) \geq 1$ for $a \in S$ with $\psi (
	\{ a \} ) = 0$ by \cite[2.7]{snulmenn.isoperimetric}, hence $\| V \|
	(B) = 0$ implies $\mathscr{H}^\vdim ( B \cap S ) = 0$ for $B \subset
	U$ by \cite[2.10.19\,(3)]{MR41:1976}.

	\emph{\emph{Step 1:~} For $\mathscr{H}^\vdim$ almost all $a \in M \cap
	S$ there holds
	\begin{gather*}
		\Tan (S,a) = \Tan (M,a) = \Tan^\vdim ( \| V \|, a ), \\
		\limsup_{r \to 0+} r^{-2} \sup \{ \dist (z-a, \Tan(S,a)) \with
		z \in \cball{a}{r} \cap S \} < \infty.
	\end{gather*}}
	One can use \cite[2.5]{snulmenn.isoperimetric} to deduce
	\begin{gather*}
		\boldsymbol{\mu}_{1/r} \circ \boldsymbol{\tau}_{-a} \lIm \Clos
		S \rIm \to \Tan^\vdim ( \| V \|, a ) \quad \text{locally in
		Hausdorff distance as $r \to 0+$}
	\end{gather*}
	for $\| V \|$ almost all $a$, cp. \cite[Lemma 17.11]{MR756417}.
	Hence the assertion on the tangent spaces follows e.g. from Allard
	\cite[3.5\,(1)]{MR0307015}.

	To prove the inequality consider $(a,T) \in ( M \cap S ) \times
	\grass{\adim}{\vdim}$ such that
	\begin{gather*}
		0 < \density^\vdim ( \| V \|, a) < \infty, \quad
		\density^{\ast \vdim} ( \psi, a ) < \infty, \quad
		\density^\vdim ( \| V \| \restrict U \without M, a ) = 0, \\
		T = \Tan (S,a), \quad \limsup_{r \to 0+} r^{-2} \sup \{ \dist
		(z-a,T) \with z \in M \cap \cball{a}{r} \} < \infty
	\end{gather*}
	and note that these conditions are met by $V$ almost all $(a,T) \in (
	M \cap S ) \times \grass{\adim}{\vdim}$ by Allard
	\cite[3.5\,(1)]{MR0307015} and \cite[2.10.19\,(3)\,(4)]{MR41:1976}.
	Defining $f : U \to \rel$ by $f(z) = \dist (z-a, T )$ for $z \in U$,
	there holds
	\begin{gather*}
		\tint{}{} \ap D \theta \bullet \ap Df \ud \| V || \leq \|
		\delta V \| ( \theta ) \quad \text{whenever $\theta \in
		\mathscr{D}^0 ( U )^+$}
	\end{gather*}
	by \ref{ex:some_lap}\,\eqref{item:some_lap:convex}. Choose $0 \leq
	\alpha < \isoperimetric{\vdim}^{-1}$, $1 \leq N < \infty$ and $0 < R
	\leq 1$ such that
	\begin{gather*}
		\cball{a}{2R} \subset U, \quad N^{-1} r^\vdim \leq
		\measureball{\| V \|}{\cball{a}{r}} \leq N r^\vdim, \\
		\measureball{\psi}{\cball{a}{r}} \leq N r^\vdim \leq
		\alpha^\vdim, \quad \sup f \lIm M \cap \cball{a}{r} \rIm \leq
		N r^2
	\end{gather*}
	for $0 < r \leq 2R$. Let
	\begin{gather*}
		\Delta_1 = \isoperimetric{\vdim} ( 1 - \alpha
		\isoperimetric{\vdim} )^{-1}, \quad \Delta_2 = 2
		\Gamma_{\ref{thm:local_maximum_estimates}} ( \vdim, \infty,
		\alpha, 1 ) \sup \big \{ 2^\vdim, \Delta_1 N^{1/\vdim} \big
		\}, \\
		\lambda = \inf \big \{ N^{-1/2} , ( 8 \Delta_2 )^{-1} \big \},
		\quad \Delta_3 = 2 \lambda^{-1} N^2 + 2 \Delta_2.
	\end{gather*}
	Possibly replacing $R$ by a smaller number, one may assume that
	\begin{gather*}
		\| V \| ( \cball{a}{r} \without M ) \leq \lambda^2 r^\vdim
		\quad \text{for $0 < r \leq 2R$}.
	\end{gather*}
	Define $g : \{ r \with 0 < r \leq 2R \} \to \rel$ and $h : \{ r \with
	0 < r \leq 2R \} \to \rel$ by
	\begin{gather*}
		g(r) = r^{-1} \eqLpnorm{\| V \| \restrict
		\cball{a}{r}}{\infty} {f}, \quad h(r) = r^{-1-\vdim/2}
		\eqLpnorm{\| V \| \restrict \cball{a}{r}}{2}{f}
	\end{gather*}
	for $0 < r \leq 2R$. Noting that $r^{-\vdim} \| V \| \restrict
	\cball{a}{r}$ is nonatomic by \ref{example:nonatomic} and
	\begin{gather*}
		\| V \| ( \cball{a}{r} \without M ) \leq \lambda^2 r^\vdim
		\leq N^{-1} r^\vdim \leq \measureball{\| V \|}{\cball{a}{r}} <
		\infty, \\
		r^{-\vdim} \eqLpnorm{\| V \| \restrict \cball{a}{r} \cap
		M}{1}{f} \leq N^2 r^2,
	\end{gather*}
	one then applies \ref{lemma:basic_interpolation} with $p$, $q$, $r$,
	$\alpha$, $\phi$, $X$, and $B$ replaced by $1$, $2$, $\infty$, $1/2$,
	$r^{-\vdim} \| V \| \restrict \cball{a}{r}$, $U$, and $U \without M$,
	to infer that
	\begin{gather*}
		rh(r) \leq 2 \lambda^{-1} N^2 r^2 + 2 \lambda r g (r) \quad
		\text{for $0 < r \leq R$}.
	\end{gather*}
	Also, \ref{thm:local_maximum_estimates} with $q$, $r$, $d$, $\delta$,
	and $T(\theta)$ replaced by $\infty$, $2$, $\dist (z,\rel^\adim
	\without \oball{a}{2r} )$, $r$, and $\| \delta V \| ( \theta )$ and
	$\nu$, $\mu$, and $A$ as in \ref{example:V_subharmonic} implies in
	conjuction with \ref{corollary:dual_embedding},
	\ref{remark:dual_embedding} with $q$, $r$, and $s$ replaced by
	$\vdim$, $\beta$, and $\infty$ that
	\begin{gather*}
		g (r) \leq \Gamma_{\ref{thm:local_maximum_estimates}} ( \vdim,
		\infty, \alpha, 1 ) \big ( 2^{1+\vdim} h (2r) + 2 \Delta_1
		N^{1/\vdim} r \big ) \leq \Delta_2 ( h(2r) + r )
	\end{gather*}
	for $0 < r \leq R$. Combining these two inequalities, one obtains
	\begin{gather*}
		r^{-1} h(r) \leq 2\Delta_2 \lambda r^{-1} h (2r) + \Delta_3
		\quad \text{for $0 < r \leq R$}.
	\end{gather*}
	Defining $\gamma = \sup \big \{ 2^{2+\vdim} R^{-1} h(R), 2 \Delta_3
	\big \}$, \emph{it will be shown that
	\begin{gather*}
		h(r) \leq \gamma r \quad \text{whenever $0 < r \leq R$};
	\end{gather*}}
	in fact this clearly holds if $r \geq R/2$ and if $h(2r) \leq 2 \gamma
	r$ for some $0 < r \leq R/2$ then
	\begin{gather*}
		r^{-1} h(r) \leq 4 \Delta_2 \lambda \gamma + \Delta_3 \leq
		\gamma.
	\end{gather*}
	Therefore $g(r) \leq \Delta_2 ( 2 \gamma + 1 ) r$ for $0 < r \leq
	R/2$.

	\emph{\emph{Step 2:~} Proof of
	\eqref{item:locality_mean_curvature:flatness}.} Observing that
	\ref{miniremark:distance_comparison} may used to infer
	\begin{gather*}
		\limsup_{r \to 0+} r^{-2} \sup | \Omega_M -
		\Omega_{\boldsymbol{\tau}_a \lIm \Tan ( M, a ) \rIm} | \lIm
		\cball{a}{r} \rIm < \infty \quad \text{whenever $a \in M$}
	\end{gather*}
	one infers from step 1 that
	\begin{gather*}
		\Tan (M,a) = T \quad \text{for $V$ almost all $(a,T) \in ( U
		\cap M ) \times \grass{\adim}{\vdim}$}, \\
		\limsup_{r \to 0+} r^{-2} \sup \Omega_M \lIm \cball{a}{r} \cap
		S \rIm < \infty \quad \text{for $\| V \|$ almost
		all $a \in U \cap M$}.
	\end{gather*}
	Then applying \ref{lemma:differentiation} in conjunction with
	\cite[2.10.19\,(4)]{MR41:1976} for each $\nu \in \nat$ with $f=
	\Omega_M$ and $\mu=2$ yields the assertion of step 2.

	\emph{\emph{Step 3:~} Proof of
	\eqref{item:locality_mean_curvature:tangent_map}.} From
	\cite[4.14]{snulmenn.decay} with $q$ replaced by $\beta$ one infers
	whenever $a \in U$, $0 < r < \infty$, $\cball{a}{2r} \subset U$, $T
	\in \grass{\adim}{\vdim}$ and $h : U \to \rel$ satisfies $h(z) = \dist
	(z-a,T)$ for $z \in U$ that
	\begin{gather*}
		\begin{aligned}
			& \tint{\cball{a}{r} \times \grass{\adim}{\vdim}}{} |
			R(z)-\project{T}|^2 \ud \| V \| z \\
			& \qquad \leq \Delta \big ( \psi ( \cball{a}{2r}
			)^{1/\vdim} \eqLpnorm{\| V \| \restrict
			\cball{a}{2r}}{\beta}{h} + r^{-2} (\| V \|
			\restrict \cball{a}{2r})(h^2) \big )
		\end{aligned}
	\end{gather*}
	whenever $\Delta$ is a positive, finite number depending only on
	$\vdim$. Multiplying by $r^{-\vdim}$ one readily deduces with the help
	of step 1 and \cite[2.10.19\,(3)]{MR41:1976} that
	\begin{gather*}
		\limsup_{r \to 0+} r^{-1} \big ( \tfint{\cball{a}{r}}{} | R(z)
		- R(a) |^2 \ud \| V \| z \big )^{1/2} < \infty
	\end{gather*}
	for $\| V \|$ almost all $z \in U \cap M$. The assertion of step 3 is
	then a consequence of \cite[3.7,\,9]{snulmenn.isoperimetric}.

	\emph{\emph{Step 4:~} $S$ is twice differentiable at $a$ with
	$\mathbf{b} ( S; a ) = \mathbf{b} ( M;a )$ for $\| V \|$ almost all $a
	\in U \cap M$.} If $a$ satisfies the conclusion of
	\eqref{item:locality_mean_curvature:flatness}, and $Q$, $\tau$ are
	related to $a$ as in \ref{def:twice_diff} with $A$ replaced by $M$
	then one readily verifies by use of
	\ref{miniremark:distance_comparison} that
	\begin{gather*}
		r^{-2} \sup \Omega_{\tau \lIm Q \rIm} \lIm \cball{a}{r} \cap S
		\rIm \to 0 \quad \text{as $r \to 0+$}.
	\end{gather*}
	If additionally $0 < \density^\vdim ( \| V \|, a ) < \infty$ and $\psi
	( \{ a \} ) = 0$ then
	\begin{gather*}
		r^{-\vdim} \tint{}{} f ( r^{-1} (z-a) ) \ud \| V \| z \to
		\density^\vdim ( \| V \|, a ) \tint{\Tan(M,a)}{} f \ud
		\mathscr{H}^\vdim \quad \text{as $r \to 0+$}
	\end{gather*}
	whenever $f \in \mathscr{K} ( \rel^\adim )$ by Allard \cite[6.5,
	4.6\,(3)]{MR0307015} and $S$ is twice differentiable at $a$ with
	$\mathbf{b}(S;a)= \mathbf{b}(M;a)$ by \ref{lemma:equivalent_blowups}.
	
	\emph{\emph{Step 5:~} $\project{\Nor ( M, a )} ( \mathbf{h} ( V;z ) )
	= \mathbf{h} ( M;z )$ for $\| V \|$ almost all $z \in U \cap M$.}
	Assume $M$ to be connected and let $G$ be as in
	\ref{lemma:virtual_tangent_planes}. Since the function mapping $z \in
	M$ onto $\project{\Tan (M,a)} \in G$ is of class $\class{1}$ relative
	to $M$, one uses \cite[3.1.20,\,22]{MR41:1976} to construct an open
	subset $Z$ of $U$ and a map $Q : Z \to G$ of class $\class{1}$ such
	that
	\begin{gather*}
		Q(z) = \project{\Tan(M,z)} \quad \text{whenever $z \in M$}.
	\end{gather*}
	The equality will be shown to hold at a point $a \in U
	\cap M$ satisfying
	\begin{gather*}
		a \in \dmn R, \quad 0 < \density^\vdim ( \| V \|, a ) <
		\infty, \quad
		\density^\vdim ( \| \delta V \| - | \mathbf{h} ( V;\cdot ) |
		\| V \|, a ) = 0, \\
		r^{-\vdim} \tint{}{} f (r^{-1}(z-a)) \ud \| V \| z \to
		\density^\vdim (\| V \|, a ) \tint{\im R(a)}{} f \ud
		\mathscr{H}^\vdim \quad \text{for $f \in
		\ccspace{\rel^\adim}$}, \\
		r^{-\vdim-1} \tint{\cball{a}{r}}{} | R-Q | \ud \| V \|
		\to 0, \quad
		r^{-\vdim} \tint{\cball{a}{r}}{} | \mathbf{h} (V;z) -
		\mathbf{h} (V;a) | \ud \| V \| z \to 0
	\end{gather*}
	as $r \to 0+$. These conditions are met by $\| V \|$ almost all $z \in
	U \cap M$ by Allard \cite[3.5\,(1)]{MR0307015},
	\cite[2.9.9,\,10]{MR41:1976}, step 3 and
	\cite[3.7\,(i)]{snulmenn.isoperimetric}. Take $h$ as in
	\ref{lemma:virtual_tangent_planes}, choose $0 < R < \infty$ with
	$\cball{a}{R} \subset Z$ and $\phi \in \mathscr{D}^0 ( \rel^\adim )$
	with $\spt \phi \subset \cball{0}{1}$ with $\int_{\im R(a)} \phi \ud
	\mathscr{H}^\vdim = 1$. Suppose $v \in \rel^\adim$ and define $g_r : U
	\to \rel^\adim$ by
	\begin{gather*}
		g_r(z) = \phi ( r^{-1} (z-a) ) ( \id{\rel^\adim} - Q(z) ) (v)
		\quad \text{if $z \in Z$}, \qquad g_r(z) = 0 \quad \text{if $z
		\in U\without Z$}
	\end{gather*}
	whenever $z \in Z$ and $0 < r \leq R$. Note that $g_r$ is of class
	$\class{1}$ and compute
	\begin{gather*}
		\begin{aligned}
			& - \density^\vdim ( \| V \|, a ) h(a) \bullet (
			\id{\rel^\adim} - Q (a) ) (v) \\
			& \qquad = \lim_{r \to 0+} - r^{-\vdim} \tint{}{} h
			\bullet g_r \ud \| V \| \\
			& \qquad = \lim_{r \to 0+} r^{-\vdim} \tint{}{} Q
			\bullet Dg_r \ud \| V \| \\
			& \qquad = \lim_{r \to 0+} r^{-\vdim} \tint{}{} R
			\bullet Dg_r \ud \| V \| \\
			& \qquad = \lim_{r \to 0+} r^{-\vdim}
			\tint{\cball{a}{R}}{} \phi
			(r^{-1} (z-a)) \eta ( V;z ) \bullet ( \id{\rel^\adim}
			- Q(z) ) (v) \ud \| \delta V \| z \\
			& \qquad = - \density^\vdim ( \| V \|, a ) \mathbf{h}
			( V;a ) \bullet ( \id{\rel^\adim}
			- Q(a) ) (v)
		\end{aligned}
	\end{gather*}
	to obtain $\project{\Nor ( M, a )} ( h(a)-\mathbf{h}
	(V;a) ) = 0$, hence the conclusion as $h(a) = \mathbf{h} (M;a) \in
	\Nor ( M, a )$ by \ref{lemma:virtual_tangent_planes}.
\end{proof}
\begin{remark}
	In the special case $V \in \IVar_\vdim ( U )$ and $\vdim = \adim-1$
	and $p$ satisfies $p > \vdim$ and $p \geq 2$ instead of $p = \vdim$
	the inequality of step 1 could have been alternately be deduced from
	Sch\"atzle \cite[Prop.  4.1]{MR2064971}. The idea to use the
	subsolution property of the distance function to an $\vdim$
	dimensional plane is taken from Ecker \cite[1.6, 1.7]{MR1384839}, see
	also \cite[Corollary D.3]{MR2024995}.

	The proof of step 5 essentially follows Sch\"atzle, see e.g.
	\cite[Theorem 4.1]{rsch:willmore}.
\end{remark}
\begin{corollary} \label{corollary:second_order_agreement}
	Suppose $\vdim$, $\adim$, $p$, $U$, and $V$ are as in
	\ref{miniremark:situation_general}, $p = \vdim< \adim$, and $V \in
	\IVar_\vdim ( U )$.

	Then $\spt \| V \|$ is twice differentiable at $a$ with
	\begin{gather*}
		\trace \mathbf{b} ( \spt \| V \|; a ) = \mathbf{h} ( V; a)
	\end{gather*}
	for $\mathscr{H}^\vdim$ almost all $a \in \spt \| V \|$.
\end{corollary}
\begin{proof}
	This is a consequence of
	\ref{thm:locality_mean_curvature}\,\eqref{item:locality_mean_curvature:differentiability}
	and \cite[4.8]{snulmenn.c2}.
\end{proof}
\section{Harnack inequality} \label{sec:harnack}
In this section based on the results of Section \ref{sec:local_max} and a
relative isoperimetric inequality, a Harnack type inequality for
supersolutions is established, see \ref{thm:weak_harnack_inequality}. This
result entails a strong maximum principle, see
\ref{corollary:strong_maximum_principle}. (The preliminary results
\ref{lemma:mydiss}--\ref{remark:role} are essentially contained in the newer
results of Sections \ref{sec:rel_iso} and \ref{sec:embeddings}.)
\begin{lemma} \label{lemma:mydiss}
	Suppose $\vdim, \adim \in \nat$, $\vdim \leq \adim$, $a \in \rel^\adim$,
	$0 < r < \infty$, $V \in \Var_\vdim (\oball{a}{r})$, $\| \delta V\|$
	is a Radon measure, $\density^\vdim ( \| V \| , z ) \geq 1$ for $\| V
	\|$ almost all $z$, $a \in \spt \| V \|$, and $f : \{ s \with 0 < s <
	r \} \to \rel$ satisfies
	\begin{gather*}
		f(s) = \measureball{\| V \|}{\cball{a}{s}} \quad
		\text{whenever $0 < s < r$}.
	\end{gather*}

	Then
	\begin{gather*}
		\isoperimetric{\vdim}^{-1} \leq f(s)^{1/\vdim-1} (
		\measureball{\| \delta V \|}{\cball{a}{s}} + f'(s)) \quad
		\text{for $\mathscr{L}^1$ almost all $0 < s < r$}.
	\end{gather*}
\end{lemma}
\begin{proof}
	Note $V \in \RVar_\vdim ( \oball{a}{r} )$ by Allard
	\cite[5.5\,(1)]{MR0307015}; hence the statement is \cite[A.7]{mydiss}
	if $\vdim<\adim$ and the case $\vdim = \adim$ may be proved
	identically.
\end{proof}
\begin{theorem} \label{thm:relative_iso}
	Suppose $\vdim, \adim \in \nat$, $\vdim \leq \adim$, $0 \leq M <
	\infty$, $U$ is an open subset of $\rel^\adim$, $W \in \Var_\vdim ( U
	)$, $S, T \in \mathscr{D}' ( U, \rel^\adim )$ are representable by
	integration, $\delta W = S + T$, $\density_\ast^\vdim ( \| W \|, z )
	\geq 1$ for $\| W \|$ almost all $z$,
	\begin{gather*}
		B = \Clos ( \spt \| W \| ) \without \spt \| W \|, \\
		\beta = \infty \quad \text{if $\vdim = 1$}, \quad \beta =
		\vdim / (\vdim-1) \quad \text{if $\vdim > 1$}, \\
		\alpha = \sup \{ S(f) \with f \in \text{$\mathscr{D} ( U,
		\rel^\adim )$ and $\Lpnorm{\| W \|}{\beta}{f} \leq 1$} \},
	\end{gather*}
	$A$ is the set of all $z \in \spt \| W \|$ such that
	\begin{gather*}
		\| W \| ( U \cap \oball{z}{r}) \leq M r^\vdim \quad \text{for
		some $0 < r < \infty$ with $\oball{z}{r} \cap B = \emptyset$},
	\end{gather*}
	$0 < \Gamma < \infty$, and one of the following two conditions is
	satisfied:
	\begin{enumerate}
		\item
		\label{item:relative_iso:some_bound}
		There holds $\alpha < \isoperimetric{\vdim}^{-1}$, $M \leq
		\Gamma_{\eqref{item:relative_iso:some_bound}}^{-1}$
		and $\Gamma =
		\Gamma_{\eqref{item:relative_iso:some_bound}}$
		where
		$\Gamma_{\eqref{item:relative_iso:some_bound}}$
		is a positive, finite number depending only on $\vdim$ and
		$\alpha$.
		\item
		\label{item:relative_iso:sharp_bound} For
		some $\delta$ there holds $\delta > 0$, $\alpha \leq
		\Gamma_{\eqref{item:relative_iso:sharp_bound}}^{-1}$,
		$M \leq (1-\delta) \unitmeasure{\vdim}$ and $\Gamma =
		\Gamma_{\eqref{item:relative_iso:sharp_bound}}$
		where
		$\Gamma_{\eqref{item:relative_iso:sharp_bound}}$
		is a positive, finite number depending only on $\adim$ and
		$\delta$.
	\end{enumerate}

	Then
	\begin{gather*}
		\| W \| ( A )^{1-1/\vdim} \leq \Gamma \| T \| ( U ).
	\end{gather*}
	where $0^0=0$.
\end{theorem}
\begin{proof}
	First, it is proven that one may assume $U = \rel^\adim \without B$.
	In fact, $B$ is closed, $\spt S \subset \spt \| W \|$, hence also
	$\spt T \subset \spt \| W \|$, and one may use the inclusion $\iota :
	U \to \rel^\adim \without B$ to replace $W$, $S$, and $T$ by
	appropriate elements of $\Var_\vdim ( \rel^\adim \without B)$ and
	$\mathscr{D}' ( \rel^\adim \without B, \rel^\adim )$ since $\iota |
	\spt \| W \|$ is proper.

	Define $A_i$ for $i \in \nat$ to be the set of all $z \in \spt \| W
	\|$ such that
	\begin{gather*}
		\measureball{\| W \|}{\oball{z}{r}} \leq M r^\vdim \quad
		\text{for some $0 < r < i$ with $\oball{z}{r} \subset U$}
	\end{gather*}
	and note $A = \bigcup \{ A_i \with i \in \nat \}$ hence $\| W \| ( A )
	= \lim_{i \to \infty} \| W \| ( A_i )$ by \cite[2.1.5\,(1)]{MR41:1976}.
	Clearly, \emph{if $\vdim = 1$ then $\| S \| ( U ) \leq \alpha$}.
	Observe \emph{if $\vdim > 1$ then
	\begin{gather*}
		\| S \| (C) \leq \alpha \| W \| ( C )^{1-1/\vdim} \quad
		\text{whenever $C \subset U$};
	\end{gather*}}
	in fact one may assume $C$ to be open and in this case
	\begin{gather*}
		\| S \| ( C ) = \sup \{ S(f) \with \text{$f \in \mathscr{D} (
		U , \rel^\adim )$ with $\spt f \subset C$ and $|f(z)| \leq 1$
		for $z \in U$} \}
	\end{gather*}
	by \cite[4.1.5]{MR41:1976}. Also note that $W \in \RVar_\vdim (U)$ by
	Allard \cite[5.5\,(1)]{MR0307015}.

	In order to prove the alternative
	\eqref{item:relative_iso:some_bound}, define
	\begin{gather*}
		\Delta_1 = 2 (50)^{\vdim-1} ( \isoperimetric{\vdim}^{-1} -
		\alpha)^{-1}, \quad \Delta_2 = ( \isoperimetric{\vdim}^{-1} -
		\alpha) / (2 \vdim), \\
		\Delta_3 = \inf \{ \unitmeasure{\vdim}/2, ( \Delta_2 )^\vdim
		\}, \quad
		\Gamma_\eqref{item:relative_iso:some_bound}
		= \sup \big \{ 2 (20)^\vdim ( \Delta_3 )^{-1}, \Delta_1 \big \}
	\end{gather*}
	and note $\Delta_2 + ( \vdim \Delta_1 )^{-1} ( 50 )^{\vdim-1} \leq (
	\isoperimetric{\vdim}^{-1} - \alpha )/ \vdim$.

	Next, the following assertion will be shown: \emph{Whenever $i \in
	\nat$, $z \in A_i$ and $\density_\ast^\vdim ( \| W \|, z ) \geq 1$
	there exists $0 < s < \frac{1}{5} \inf \{ i, \dist (z,\rel^\adim
	\without U) \}$ such that $\| W \| ( \cball{z}{5s} )^{1-1/\vdim} \leq
	\Delta_1 \measureball{\| T \|}{\cball{z}{s}}$}. For this purpose
	choose $r$ such that
	\begin{gather*}
		0 < r < \inf \{ i, \dist (z,\rel^\adim \without U) \},
		\quad \measureball{\| W \|}{\cball{z}{r}} \leq (20)^{-\vdim}
		\Delta_3 r^\vdim,
	\end{gather*}
	let $P$ denote the set of all $0 < t < r$ such that
	\begin{gather*}
		\measureball{\| W \|}{\cball{z}{t}} \leq \Delta_3 t^\vdim
	\end{gather*}
	and $Q$ the set of all $0 < t < \frac{r}{10}$ such that $\{ s \with t
	\leq s \leq 10t \} \subset P$. One notes for $\frac{r}{20} \leq s < r$
	\begin{gather*}
		s^{-\vdim} \measureball{\| W \|}{\cball{z}{s}} \leq ( 20
		)^\vdim r^{-\vdim} \measureball{\| W \|}{\cball{z}{r}} \leq
		\Delta_3,
	\end{gather*}
	hence $\frac{r}{20} \in Q$. Let $t = \inf Q$ and note $t \in Q$ since
	$\Delta_3 < \unitmeasure{\vdim}$. Also, whenever $t \leq s \leq 10 t$
	\begin{gather*}
		s^{-\vdim} \measureball{\| W \|}{\cball{z}{s}} \geq
		(10)^{-\vdim} t^{-\vdim} \measureball{\| W
		\|}{\cball{z}{t}} = (10)^{-\vdim} \Delta_3
	\end{gather*}
	because $t \in \Clos ( \{ s \with s < t \} \without P )$.
	Defining $\phi : \{ s : 0 < s < r \} \to \rel$ by $\phi (s) =
	\measureball{\| W \|}{\cball{z}{s}}$ for $0 < s < r$, one infers from
	\ref{lemma:mydiss} and the preceding paragraph that
	\begin{gather*}
		\isoperimetric{\vdim}^{-1} \leq \phi (s)^{1/\vdim-1} \big (
		\measureball{\| \delta W \|}{\cball{z}{s}} + \phi'(s) \big ),
		\\
		\isoperimetric{\vdim}^{-1} - \alpha \leq \phi (s)^{1/\vdim-1}
		\big ( \measureball{\| T \|}{\cball{z}{s}} + \phi ' (s) \big )
	\end{gather*}
	for $\mathscr{L}^1$ almost all $0 < s < r$. If the assertion were
	false then for $\mathscr{L}^1$ almost all $t < s < 2 t$
	\begin{gather*}
		\begin{aligned}
			( \isoperimetric{\vdim}^{-1} - \alpha ) / \vdim & <
			( \vdim \Delta_1 )^{-1} \phi (5s)^{1-1/\vdim} \phi
			(s)^{1/\vdim-1} + ( \phi^{1/\vdim} )' (s) \\
			& \leq ( \vdim \Delta_1 )^{-1} ( 50)^{\vdim-1} + (
			\phi^{1/\vdim} ) ' (s),
		\end{aligned} \\
		( \Delta_3 )^{1/\vdim} \leq \Delta_2 < ( \phi^{1/\vdim} ) '
		(s),
	\end{gather*}
	hence, using $\phi^{1/\vdim} (t) = ( \Delta_3 )^{1/\vdim}
	t$ and \cite[2.9.19]{MR41:1976}, one would obtain for $ t
	< s < 2 t$
	\begin{gather*}
		\phi^{1/\vdim}(s) \geq \phi^{1/\vdim} ( t ) + \tint{t}{s} (
		\phi^{1/\vdim} ) ' \ud \mathscr{L}^1
		> ( \Delta_3 )^{1/\vdim} s, \quad s \notin P, \quad t
		\notin Q
	\end{gather*}
	a contradiction.

	By the assertion of the preceding paragraph and
	\cite[2.8.5]{MR41:1976} for $i \in \nat$ there exists a countable
	disjointed family $G$ of closed balls such that
	\begin{gather*}
		\classification{A_i}{z}{\density_\ast^\vdim ( \| W \|, z )
		\geq 1 } \subset {\textstyle\bigcup} \{ \widehat{D} \with D
		\in G \} \subset U, \\
		\| W \| ( \widehat{D} )^{1-1/\vdim} \leq \Gamma \| T \| (D)
		\quad \text{for $D \in G$}.
	\end{gather*}
	Therefore, if $\vdim > 1$, then
	\begin{gather*}
		\begin{aligned}
			& \| W \| (A_i) = \| W \| (
			\classification{A_i}{z}{\density_\ast^\vdim ( \| W \|,
			z) \geq 1 } ) \leq \tsum{D \in G}{} \| W \| (
			\widehat{D} ) \\
			& \qquad \leq \Gamma^\beta \tsum{D \in G}{} \| T \|
			(D)^\beta \leq \Gamma^\beta \big ( \tsum{D \in G}{} \|
			T \| (D) \big)^\beta \leq \Gamma^\beta \| T \| ( U
			)^\beta.
		\end{aligned}
	\end{gather*}
	If $\vdim = 1$ and $\| W \| (A_i)>0$ then $\Gamma^{-1} \leq \| T \| (D)
	\leq \| T \| ( U )$ for some $D \in G$.

	In order to prove the alternative
	\eqref{item:relative_iso:sharp_bound}, define
	\begin{gather*}
		\Delta_4 = \inf {\textstyle\bigcup} \big \{ \{ ( 2
		\isoperimetric{j})^{-1},
		\varepsilon_{\text{\cite[2.6]{snulmenn.isoperimetric}}} (
		\adim-j, j, \delta/2 ) \} \with \adim \geq j \in \nat \big \},
		\\
		\Delta_5 = 2 ( \Delta_4 )^{-1}, \quad
		\Gamma_{\eqref{item:relative_iso:sharp_bound}}
		= \Delta_5 \besicovitch{\adim}.
	\end{gather*}

	If $i \in \nat$, $z \in A_i$ there exists $0 < s < \inf \{ i, \dist
	(z, \rel^\adim \without U ) \}$ such that
	\begin{gather*}
		\| W \| ( \cball{z}{s} )^{1-1/\vdim} \leq \Delta_5
		\measureball{\| T \|}{\cball{z}{s}};
	\end{gather*}
	in fact there exists $0 < s < \inf \{ i, \dist (z, \rel^\adim
	\without U ) \}$ such that
	\begin{gather*}
		\Delta_4 \| W \| ( \cball{z}{s} )^{1-1/\vdim} <
		\measureball{\| \delta W \|}{\cball{z}{s}}
	\end{gather*}
	by \cite[2.6]{snulmenn.isoperimetric} and
	\begin{gather*}
		\measureball{\| \delta W \|}{\cball{z}{s}} \leq \alpha \| W \|
		( \cball{z}{s} )^{1-1/\vdim} + \measureball{\| T
		\|}{\cball{z}{s}}.
	\end{gather*}
	Therefore for $i \in \nat$ there exist countable disjointed families
	of closed balls $F_1, \ldots, F_{\besicovitch{\adim}}$ such that
	\begin{gather*}
		A_i \subset {\textstyle\bigcup\bigcup} \{ F_j \with j = 1,
		\ldots, \besicovitch{\adim} \}, \\
		\| W \| (D)^{1-1/\vdim} \leq \Delta_5 \| T \| ( D ) \quad
		\text{for $D \in {\textstyle\bigcup} \{ F_j \with j =
		1,\ldots, \besicovitch{\adim} \}$}.
	\end{gather*}
	If $\vdim > 1$ then
	\begin{gather*}
		\begin{aligned}
			\| W \| (A_i) & \leq
			\tsum{j=1}{\besicovitch{\adim}}\tsum{D \in F_j}{} \| W
			\| (D) \leq ( \Delta_5 )^\beta
			\tsum{j=1}{\besicovitch{\adim}} \tsum{D \in F_j}{} \|
			T \| (D)^\beta \\
			& \leq ( \Delta_5 )^\beta
			\tsum{j=1}{\besicovitch{\adim}} \big ( \tsum{D \in
			F_j}{} \| T \| (D) \big )^\beta \leq ( \Delta_5
			)^\beta \besicovitch{\adim} \| T \| ( U )^\beta.
		\end{aligned}
	\end{gather*}
	If $\vdim = 1$ and $\| W \| (A_i) > 0$ then $\Gamma^{-1} \leq \| T \|
	( D ) \leq \| T \| ( U )$ for some $D \in \bigcup \{ F_j \with j = 1,
	\ldots, \besicovitch{\adim} \}$.
\end{proof}
\begin{corollary} \label{corollary:relative_iso}
	Suppose $\vdim$, $\adim$, $p$, $U$, $V$, and $\psi$ are as in
	\ref{miniremark:situation_general}, $p = \vdim$, $0 \leq \alpha <
	\infty$, $0 \leq M < \infty$, $E$ is of locally bounded $V$ perimeter
	and $\| \delta V \|$ measurable, $0 \leq r < \infty$,
	\begin{gather*}
		\| V \| ( E ) \leq M r^\vdim, \quad \psi ( E )^{1/\vdim} \leq
		\alpha, \quad
		B = \Clos ( \spt \| V \| \restrict E ) \without \spt ( \| V \|
		\restrict E ),
	\end{gather*}
	and $0 < \Gamma < \infty$ satisfies
	\begin{gather*}
		\text{either} \quad \text{$\alpha <
		\isoperimetric{\vdim}^{-1}$, $M \leq
		\Gamma_{\ref{thm:relative_iso}\eqref{item:relative_iso:some_bound}}
		( \vdim, \alpha )^{-1}$ and $\Gamma =
		\Gamma_{\ref{thm:relative_iso}\eqref{item:relative_iso:some_bound}}
		( \vdim, \alpha )$}, \\
		\text{or} \quad \text{$\delta > 0$, $\alpha \leq 
		\Gamma_{\ref{thm:relative_iso}\eqref{item:relative_iso:sharp_bound}}
		( \adim, \delta )^{-1}$, $M \leq (1-\delta)
		\unitmeasure{\vdim}$ and $\Gamma = 
		\Gamma_{\ref{thm:relative_iso}\eqref{item:relative_iso:sharp_bound}}
		( \adim, \delta )$}.
	\end{gather*}

	Then 
	\begin{gather*}
		\| V \| ( \classification{E}{z}{ \oball{z}{r} \cap B =
		\emptyset })^{1-1/\vdim} \leq \Gamma \| \boundary{V}{E} \| ( U
		)
	\end{gather*}
	where $0^0=0$.
\end{corollary}
\begin{proof}
	Let $W = V \restrict E \times \grass{\adim}{\vdim}$ and note $\| W \|
	= \| V \| \restrict E$ and $\density_\ast^\vdim ( \| W \|, z) \geq 1$
	for $\| W \|$ almost all $z$ by \cite[2.9.11]{MR41:1976}. If $\vdim >
	1$ then approximation shows
	\begin{gather*}
		\tint{E}{} \theta \bullet \eta ( V ; \cdot ) \ud \| \delta V
		\| = - \tint{E}{} \theta \bullet \mathbf{h} ( V; \cdot ) \ud
		\| V \| \quad \text{for $\theta \in \mathscr{D} ( U,
		\rel^\adim )$}.
	\end{gather*}
	Therefore applying \ref{thm:relative_iso} with
	$S(\theta)$ and $T$ replaced by $\tint{E}{} \theta \bullet \eta
	(V;\cdot) \ud \| \delta V \|$ and $-\boundary{V}{E}$ yields the
	conclusion.
\end{proof}
\begin{theorem} \label{thm:sobolev_poincare_inequality}
	Suppose $\vdim$, $\adim$, $p$, $U$, $V$, and $\psi$ are as in
	\ref{miniremark:situation_general}, $p = \vdim$, $0 \leq \alpha <
	\infty$, $0 \leq M < \infty$, $f : U \to \rel$ is locally
	Lipschitzian, $0 \leq r < \infty$,
	\begin{gather*}
		B = \eqclassification{\Clos (\spt \| V \| ) \without \spt \| V
		\|}{a}{ {\textstyle\limsup_{z \to a}} | ( f | \spt \| V \| )
		(z) | > 0}, \\
		\| V \| ( \classification{U}{z}{f(z) \neq 0} ) \leq M r^\vdim,
		\quad \psi ( \classification{U}{z}{f(z) \neq 0} )^{1/\vdim}
		\leq \alpha, \\
		\beta = \infty \quad \text{if $\vdim=1$}, \qquad \beta =
		\vdim/(\vdim-1) \quad \text{if $\vdim > 1$}
	\end{gather*}
	and $0 < \Gamma < \infty$ satisfies
	\begin{gather*}
		\text{either} \quad \text{$\alpha <
		\isoperimetric{\vdim}^{-1}$, $M \leq
		\Gamma_{\ref{thm:relative_iso}\eqref{item:relative_iso:some_bound}}
		( \vdim, \alpha )^{-1}$ and $\Gamma =
		\Gamma_{\ref{thm:relative_iso}\eqref{item:relative_iso:some_bound}}
		( \vdim, \alpha )$}, \\
		\text{or} \quad \text{$\delta > 0$, $\alpha \leq 
		\Gamma_{\ref{thm:relative_iso}\eqref{item:relative_iso:sharp_bound}}
		( \adim, \delta )^{-1}$, $M \leq (1-\delta)
		\unitmeasure{\vdim}$ and $\Gamma = 
		\Gamma_{\ref{thm:relative_iso}\eqref{item:relative_iso:sharp_bound}}
		( \adim, \delta )$}.
	\end{gather*}

	Then
	\begin{gather*}
		\eqLpnorm{\| V \| \restrict \{ z \with \oball{z}{r} \cap B =
		\emptyset \}}{\beta}{f} \leq \Gamma \Lpnorm{\| V \|}{1}{|\ap
		Df|}.
	\end{gather*}
\end{theorem}
\begin{proof}
	Assume $f \geq 0$ and define $A = \classification{U}{z}{\oball{z}{r}
	\cap B = \emptyset}$ and $E(t)$, $T(t)$ as in \ref{lemma:coarea}. Let
	$A(t)$ for $0 < t < \infty$ denote the set of all $z \in U$ such that
	\begin{gather*}
		\oball{z}{r} \cap \Clos ( \spt ( \| V \| \restrict E(t) ) )
		\without \spt ( \| V \| \restrict E(t) ) = \emptyset
	\end{gather*}
	and note $A \subset A(t)$ since
	\begin{gather*}
		( \Bdry U ) \cap \Clos ( \spt \| V \| \restrict E(t) ) \subset
		B.
	\end{gather*}

	If $\vdim = 1$ then $1 \leq \Gamma \| T(t) \| (U)$ whenever $0 < t <
	\eqLpnorm{\| V \| \restrict A}{\beta}{f}$ by
	\ref{corollary:relative_iso} and the conclusion
	follows from \ref{lemma:coarea}. If $\vdim > 1$, defining $f_t = \inf
	\{ f, t \}$ for $0 < t < \infty$ and $\phi : \{ t : 0 < t < \infty \}
	\to \overline{\rel}$ by $\phi (t) = \eqLpnorm{\| V \| \restrict
	A}{\beta}{f_t}$ for $0 < t < \infty$, noting
	\begin{gather*}
		\phi (t) \leq t \| V \| ( \classification{U}{z}{f(z) \neq 0}
		)^{1/\beta} < \infty \quad \text{for $0 < t < \infty$},
	\end{gather*}
	Minkowski's inequality implies
	\begin{gather*}
		\phi (t+h) - \phi (t) \leq \eqLpnorm{\| V \| \restrict
		A}{\beta}{f_{t+h}-f_t} \leq h \| V \| ( A \cap E(t)
		)^{1-1/\vdim}
	\end{gather*}
	for $0 < t < \infty$ and $0 < h < \infty$. Therefore $\phi$ is
	Lipschitzian and from
	\ref{corollary:relative_iso}
	\begin{gather*}
		\phi'(t) \leq \| V \| ( A \cap E(t) )^{1-1/\vdim} \leq \Gamma
		\| T (t) \| ( U ) \quad \text{for $\mathscr{L}^1$ almost all
		$0 < t < \infty$},
	\end{gather*}
	hence $\eqLpnorm{\| V \| \restrict A}{\beta}{f} = \lim_{t \to \infty}
	\phi (t) = \tint{0}{\infty} \phi' \ud \mathscr{L}^1$ and
	\ref{lemma:coarea} implies the conclusion.
\end{proof}
\begin{remark} \label{remark:role}
	The role of \ref{corollary:relative_iso} in the
	preceding proof is identical to the one of Allard
	\cite[7.1]{MR0307015} in Allard \cite[7.3]{MR0307015}.
\end{remark}
\begin{lemma} \label{lemma:log_bound}
	Suppose $\vdim, \adim \in \nat$, $\vdim \leq \adim$, $U$ is an open
	subset of $\rel^\adim$, $V \in \RVar_\vdim (U)$, ``$\ap D\!$'' denotes
	approximate differentiation with respect to $( \| V \|, \vdim )$, $0 <
	M < \infty$, $\| V \| ( U ) \leq M$, $f : U \to \rel$ is a
	nonnegative, locally Lipschitzian function, $t = \sup f \lIm \spt \| V
	\| \rIm$, $d : U \to \rel$ is nonnegative, $\Lip d \leq 1$, $0 <
	\delta < \infty$,
	\begin{gather*}
		\text{$\classification{\spt \| V \|}{z}{ \text{$f(z) \leq s$ and
		$d(z) \geq \varepsilon$} }$ is compact for $0 < s < t$ and
		$\varepsilon > 0$},
	\end{gather*}
	$1 \leq q \leq \infty$, $1 \leq r \leq 2$, $N : \mathscr{D}^0 ( U )
	\to \rel$,
	\begin{gather*}
		N ( \theta ) = \nu \delta^{-1} M^{1-1/q} \Lpnorm{\| V
		\|}{q}{\theta} + \nu M^{1-1/r} \Lpnorm{\| V \|}{r}{| \ap D
		\theta |} \quad \text{for $\theta \in \mathscr{D}^0 ( U )$},
	\end{gather*}
	$\nu$, $\mu$, and $A$ are related to $\vdim$, $\adim$ and $U$ as in
	\ref{miniremark:bilinear_form}, $A$ is $V$ measurable, $T \in
	\mathscr{D}_0 ( U )$,
	\begin{gather*}
		\tint{}{} \left < \ap D \theta (z) \otimes \ap Df(z), A(z,S)
		\right > \ud V(z,S) \geq T ( \theta ) \quad \text{whenever
		$\theta \in \mathscr{D}^0 ( U )^+$},
	\end{gather*}
	$0 < \gamma < \infty$ and $\sup \{ T ( \theta ) \with \text{$\theta
	\in \mathscr{D}^0 (U)$ and $N(\theta) \leq 1$} \} \leq \gamma$.

	Then
	\begin{gather*}
		\tint{\{ z \with d(z) \geq \delta \}}{} ( \delta^{-1} f(z) +
		\gamma)^{-2} | \ap Df(z) |^2 \ud \| V \| z \leq 2^6 \mu^2
		\nu^{-2} M.
	\end{gather*}
\end{lemma}
\begin{proof}
	First, \emph{the case $M = \delta = \nu = 1$} is treated. Suppose $0
	< s < t$ and let $E = \classification{U}{z}{f(z) \leq s}$. Define
	$\eta : U \to \rel$, $\theta : U \to \rel$ and $g : U \to \rel$ by
	\begin{gather*}
		\begin{aligned}
			\eta (z) & = 2 ( \inf \{ d(z), 1 \} - \inf \{ d(z),
			1/2 \} ), \\
			\theta (z) & = \eta(z)^2 \sup \{ ( f(z) + \gamma
			)^{-1} - ( s + \gamma )^{-1}, 0 \}, \\
			g (z) & = \sup \{ \log ( s + \gamma ) - \log ( f(z) +
			\gamma ), 0 \}
		\end{aligned}
	\end{gather*}
	whenever $z \in U$. Note $\Lip \eta \leq 2$ and
	\begin{gather*}
		0 \leq \eta (z) \leq 1, \qquad \eta(z) =1 \quad \text{if $d(z)
		\geq 1$}, \qquad \eta (z) = 0 \quad \text{if $d(z) \leq
		1/2$}, \\
		\theta (z) = 0 \quad \text{and} \quad g(z) = 0 \qquad \text{if
		$z \notin E$}
	\end{gather*}
	whenever $z \in U$. In particular $\spt \| V \| \cap \spt \theta$ is
	compact. One computes
	\begin{gather*}
		\begin{aligned}
			\ap D \theta (z) & = 2 \big ( ( f(z) + \gamma )^{-1} -
			( s + \gamma)^{-1} \big ) \eta (z) \ap D \eta (z) \\
			& \phantom{=} \ - \eta(z)^2 ( f(z) + \gamma)^{-2} \ap
			Df(z), \\
			\ap D g (z) & = - ( f(z)+\gamma )^{-1} \ap Df(z), \\
			\ap D \theta (z) \otimes \ap Df(z) & = 2 \big (( f(z)
			+ \gamma )^{-1} - (s+\gamma)^{-1} \big ) \eta (z) \ap
			D\eta (z) \otimes \ap Df(z) \\
			& \phantom{=} \ - \eta(z)^2 \ap D g (z) \otimes \ap D
			g (z),
		\end{aligned}
	\end{gather*}
	for $\| V \|$ almost all $z \in E$. Therefore
	\begin{gather*}
		0 \leq ( f(z) + \gamma )^{-1} - ( s + \gamma )^{-1} \leq ( f
		(z) + \gamma )^{-1} \leq \gamma^{-1}, \\
		\eta(z)^2 | \ap D g (z) |^2 \leq 4 \mu^2 | \ap D \eta (z) |^2
		- 2 \left < \ap D \theta (z) \otimes \ap Df (z), A(z,S) \right
		>, \\
		\gamma | \ap D \theta (z) | \leq 2 | \ap D \eta(z)| + \eta(z)
		| \ap Dg(z) |
	\end{gather*}
	for $V$ almost all $(z,S) \in E \times \grass{\adim}{\vdim}$. One
	obtains, noting \ref{corollary:nuq_density} and using H\"older's
	inequality,
	\begin{gather*}
		\Lpnorm{\| V \|}{r}{\eta | \ap Dg|} \leq \Lpnorm{\| V
		\|}{2}{\eta | \ap Dg|} \leq 1 + {\textstyle\frac{1}{4}} \| V
		\| ( \eta^2 | \ap Dg|^2 ), \\
		\begin{aligned}
			\| V \| ( \eta^2 | \ap D g |^2 ) & \leq 16 \mu^2 + 2
			\gamma ( \Lpnorm{\| V \|}{q}{\theta} + \Lpnorm{\| V
			\|}{r}{| \ap D \theta|} ) \\
			& \leq 16 \mu^2 + 12 + {\textstyle\frac{1}{2}} \| V \|
			( \eta^2 | \ap D g|^2 ).
		\end{aligned}
	\end{gather*}
	Since $\spt \| V \| \cap \spt \eta \cap \spt g$ is compact $\| V \| (
	\eta^2 | \ap D g|^2 ) < \infty$. Therefore
	\begin{gather*}
		\tint{\{ z \with \text{$f(z) \leq s$ and $d(z) \geq 1$}
		\}}{} ( f(z) + \gamma )^{-2} | \ap Df(z) |^2 \ud \| V \| z
		\leq 32 \mu^2 + 24
	\end{gather*}
	and considering the limit $s \to t-$ yields the conclusion in the
	present case noting $\ap Df(z) = 0$ for $\| V \|$ almost all $z$ with
	$f(z) \geq t$.

	In \emph{the general case} define $W$, $B$, and $R$ by
	\begin{gather*}
		\begin{aligned}
			W(g) & = M^{-1} \tint{}{} g ( \delta^{-1} z, S ) \ud V
			(z,S) && \quad \text{for $g \in \mathscr{K} (
			\boldsymbol{\mu}_{1/\delta} \lIm U \rIm \times
			\grass{\adim}{\vdim}  )$}, \\
			B (z,S) & = \nu^{-1} A ( \delta z, S ) && \quad
			\text{for $(z,S) \in \boldsymbol{\mu}_{1/\delta}
			\lIm U \rIm \times \grass{\adim}{\vdim}$}, \\
			R ( \theta ) & = \delta \nu^{-1} M^{-1} T ( \theta
			\circ \boldsymbol{\mu}_{1/\delta} ) && \quad
			\text{for $\theta \in \mathscr{D}^0 (
			\boldsymbol{\mu}_{1/\delta} \lIm U \rIm )$}.
		\end{aligned}
	\end{gather*}
	Applying the special case with
	\begin{gather*}
		\text{$U$, $V$, $M$, $f$, $t$, $d$, $\delta$, $\nu$, $\mu$,
		$A$, and $T$} \\
		\text{replaced by $\boldsymbol{\mu}_{1/\delta} \lIm U \rIm$,
		$W$, $1$, $\delta^{-1} f \circ \boldsymbol{\mu}_{\delta}$,
		$\delta^{-1} t$, $\delta^{-1} d \circ
		\boldsymbol{\mu}_{\delta}$, $1$, $\mu/\nu$, $B$, and $R$}
	\end{gather*}
	then readily yields the conclusion.
\end{proof}
\begin{lemma} \label{lemma:slicing_pde}
	Suppose $\vdim, \adim \in \nat$, $\vdim \leq \adim$, $U$ is an open
	subset of $\rel^\adim$, $V \in \RVar_\vdim ( U )$, ``$\ap D$'' denotes
	approximate differentiation with respect to $( \| V \|, \vdim )$, $g :
	U \to \rel$ and $f : U \to \rel$ are locally Lipschitzian functions,
	$A$, $\nu$, and $\mu$ are as in \ref{miniremark:bilinear_form}, $T \in
	\mathscr{D}_0 ( U )$ is representable by integration, and
	\begin{gather*}
		\tint{}{} \left < \ap D \theta (z) \otimes \ap Dg (z), A (z,S)
		\right > \ud V (z,S) = T ( \theta ) \quad \text{whenever
		$\theta \in \mathscr{D}^0 (U)$}, \\
		F(z) = | \ap D f (z) |^{-1} \ap D f (z) \quad \text{for $z \in
		\dmn \ap D f$ with $\ap Df(z) \neq 0$}.
	\end{gather*}

	Then for $\mathscr{L}^1$ almost all $t$ there holds
	\begin{gather*}
		\begin{aligned}
			& ( T \restrict \{ z \with f(z) > t \} ) ( \theta ) =
			\tint{\{ z \with f(z) > t \}}{} \left < \ap D \theta
			(z) \otimes \ap D g (z), A(z,S) \right > \ud V (z,S)
			\\
			& \qquad + \tint{}{} \theta (z) \left < F(z) \otimes \ap
			D g (z), A (z,\Tan^\vdim ( \| V \| , z ) ) \right >
			\density^\vdim ( \| V \|, z ) \ud
			\mathscr{H}^{\vdim-1} z
		\end{aligned}
	\end{gather*}
	whenever $\theta \in \mathscr{D}^0 ( U )$.
\end{lemma}
\begin{proof}
	Suppose $t \in \rel$, $h > 0$ and $\theta \in \mathscr{D}^0 ( U )$,
	define $f_s = \inf \{ f, s \}$ for $s \in \rel$ and and $\xi = \theta
	( f_{t+h} - f_t )$. Approximating $\xi$ by convolution noting
	\cite[4.5\,(3)]{snulmenn.decay}, one obtains
	\begin{gather*}
		\begin{aligned}
			& T ( \theta (f_{t+h}-f_t) ) = \tint{}{}
			(f_{t+h}-f_t) (z) \left < \ap D \theta (z) \otimes
			\ap D g (z) , A (z,S) \right > \ud V (z,S) \\
			& \qquad + \tint{ \{ z \with s < f (z) < s + h
			\}}{} \theta (z) \left < \ap D f (z) \otimes \ap D
			g(z), A (z,S) \right > \ud V (z,S).
		\end{aligned}
	\end{gather*}
	Dividing by $h$ and considering the limit $h \to 0+$ with the help of
	the coarea formula \cite[3.2.22]{MR41:1976} in conjunction with
	\cite[2.9.9]{MR41:1976} one infers the conclusion.
\end{proof}
\begin{remark}
	The first assertion of \ref{lemma:coarea} can be obtained from
	\ref{lemma:slicing_pde} by taking $g$ to be the restriction of a
	standard coordinate function of $\rel^\adim$ to $U$ and $A$ as in
	\ref{example:V_subharmonic}.
\end{remark}
\begin{theorem} \label{thm:weak_harnack_inequality}
	Suppose $\vdim$, $\adim$, $p$, $U$, $V$, $\psi$, and ``$\ap$'' are as
	in \ref{miniremark:situation_general}, $p = \vdim$, $\vdim < q \leq
	\infty$, $q \geq 2$, $0 \leq \alpha < \infty$, $0 \leq M < \infty$,
	$f : U \to \rel$ is a nonnegative, locally Lipschitzian function, $0
	\leq t \leq \infty$, $0 < r < \infty$, $B \subset \Bdry U$,
	\begin{gather*}
		\eqclassification{\Clos ( \spt \| V \| ) \without \spt \|
		V \| }{a}{ { \textstyle\liminf_{z \to a}} ( f | \spt \| V \| )
		(z) < t } \subset B, \\
		\| V \| ( \classification{U}{z}{f(z)<t} ) \leq M r^\vdim,
		\quad \psi ( \classification{U}{z}{f(z)<t} )^{1/\vdim} \leq
		\alpha,
	\end{gather*}
	$\nu$, $\mu$ and $A$ are related to $\vdim$, $\adim$, $U$ as in
	\ref{miniremark:bilinear_form}, $A$ is $V$ measurable, $T \in
	\mathscr{D}_0 ( U )$,
	\begin{gather*}
		\tint{}{} \left < \ap D \theta (z) \otimes \ap Df (z), A (z,S)
		\right > \ud V (z,S) \geq T ( \theta ) \quad \text{whenever
		$\theta \in \mathscr{D}^0 ( U )^+$},
	\end{gather*}
	$0 < \varepsilon < \infty$, and exactly one of the following
	two conditions is satisfied:
	\begin{enumerate}
		\item \label{item:weak_harnack_inequality:some_bound} $\alpha
		< \isoperimetric{\vdim}^{-1}$, $M \leq \varepsilon$, and
		$\varepsilon = 3^{-\vdim}
		\Gamma_{\ref{thm:relative_iso}\eqref{item:relative_iso:some_bound}}
		( \vdim , \alpha )^{-1}$,
		\item \label{item:weak_harnack_inequality:sharp_bound} for
		some $\delta$ there holds
		$\delta > 0$, $\alpha \leq \varepsilon$, $M \leq (1-\delta)
		\unitmeasure{\vdim}$, and $\varepsilon = \inf \big ( \{
		\Gamma_{\ref{thm:relative_iso}\eqref{item:relative_iso:sharp_bound}}
		( \adim, \delta/2 )^{-1} \} \cup \bigcup \{ ( 2
		\isoperimetric{j} )^{-1} \with \adim \geq j \in \nat \} \big )$.
	\end{enumerate}

	Then
	\begin{gather*}
		t \leq \Gamma \big ( \inf f \lIm \classification{\spt \| V
		\|}{z}{ \dist (z,B) > r } \rIm + r^{1-\vdim/q} \nu^{-1}
		\nuqastnorm{q}{V}{T} \big )
	\end{gather*}
	where $\Gamma$ is a positive, finite number which, in case
	\eqref{item:weak_harnack_inequality:some_bound}, depends only on
	$\vdim$, $q$, $\alpha$, and $\mu/\nu$ and, in case
	\eqref{item:weak_harnack_inequality:sharp_bound}, depends only on
	$\adim$, $q$, $\mu/\nu$, and $\delta$.
\end{theorem}
\begin{proof}
	The problem may be reduced to the case $\nu=1$ by replacing
	$A(z,S)$ and $T$ by $\nu^{-1} A (z,S)$ and $\nu^{-1}T$ and to the case
	$r=1$ via rescaling.

	Define $1 \leq s \leq \infty$ such that $1/s+1/q = 1$ and
	\begin{gather*}
		\beta = \infty \quad \text{if $\vdim = 1$}, \qquad \beta =
		\vdim/(\vdim-1) \quad \text{if $\vdim > 1$}.
	\end{gather*}
	Define, in case \eqref{item:weak_harnack_inequality:some_bound},
	\begin{gather*}
		\lambda = 1/3, \quad \Delta_1 = 8 \mu \sup \{ 1, \varepsilon
		\}, \quad \Delta_2 = \lambda^{-1} \varepsilon^{1/\vdim-1/q}
		\Gamma_{\ref{corollary:sobolev}} ( \vdim, \alpha, s ), \\
		\Delta_3 = \lambda^{1-\vdim} \Gamma_{\ref{thm:local_maximum_estimates}} ( \vdim,
		q, \alpha, \mu ).
	\end{gather*}
	and, in case \eqref{item:weak_harnack_inequality:sharp_bound},
	\begin{gather*}
		\begin{aligned}
			\lambda & = \big (1- (1-\delta)^{1/\adim}
			(1-\delta/2)^{-1/\adim} \big )/2, \\
			\Delta_1 & = 8 \mu \sup \big ( \{ 1 \} \cup
			{\textstyle\bigcup} \{ \unitmeasure{i} \with \adim
			\geq i \in \nat \} \big ), \\
			\Delta_2 & = \lambda^{-1} \sup \{
			\Gamma_{\ref{corollary:sobolev}} ( i, \varepsilon, s )
			\unitmeasure{i}^{1/i-1/q} \with \adim \geq i \in \nat,
			i < q \}, \\
			\Delta_3 & = \lambda^{1-\adim} \sup \{
			\Gamma_{\ref{thm:local_maximum_estimates}} ( i, q,
			\varepsilon, \mu ) \with \adim \geq i \in \nat, i < q
			\}.
		\end{aligned}
	\end{gather*}
	Moreover, let
	\begin{gather*}
		\Delta_4 = \varepsilon^{-1} \lambda^{-1} \Delta_1, \quad
		\Delta_5 = \lambda^{-1} ( 1 + \Delta_2 ), \\
		\Delta_6 = \sup \{ \Delta_4, \Delta_3 ( \Delta_4 + \Delta_5 )
		\}, \quad \Gamma = \exp ( \Delta_6 ).
	\end{gather*}
	Define $d_R : U \to \rel$ by
	\begin{gather*}
		d_R (z) = \dist ( z, B \cup ( \rel^\adim \without \oball{0}{R}
		) ) \quad \text{whenever $z \in U$ and $0 < R < \infty$}
	\end{gather*}
	and abbreviate $S = \spt \| V \|$ and $W ( \zeta ) =
	\classification{U}{z}{ \dist (z,B) > \zeta }$ for $0 \leq \zeta <
	\infty$.

	First, it will be shown that one can assume $f(z) < t$ for $z \in U$
	by replacing $U$ by $\classification{U}{z}{f(z) < t}$. In fact,
	defining $U' = \classification{U}{z}{f(z) < t}$, it is sufficient to
	observe
	\begin{gather*}
		\begin{aligned}
			& \eqclassification{\Clos(S) \without S}{a}{
			\liminf_{z \to a} (f|S)(z) < t} \\
			& \qquad =
			\eqclassification{ \Clos ( U' \cap S ) \without ( U'
			\cap S )}{a}{ \liminf_{z \to a} ( f | U' \cap S ) (z)
			< t }
		\end{aligned}
	\end{gather*}
	since $\Bdry U \cap \Clos U' \subset \Bdry U' \subset ( \Bdry U ) \cup
	\{ z \with f(z) = t \}$. Defining $Z (\zeta) = S \cap W (\zeta)$
	whenever $0 \leq \zeta < \infty$, one may also assume $\inf f \lIm Z (
	1 ) \rIm + \nuqastnorm{q}{V}{T} < \infty$, in particular $\| V \| ( Z
	( 1 ) ) > 0$.

	Next, the following two statements are proven.
	\begin{enumerate}
		\item \label{item:weak_harnack_inequality:p1}
		$\eqclassification{\Clos (Z(\zeta)) \without Z ( \zeta )}{a}{
		\liminf_{z \to a} ( f | Z ( \zeta ) ) (z) < t }$ is contained
		in $\classification{S}{z}{\dist (z,B) = \zeta}$ for $0 < \zeta
		< \infty$.
		\item \label{item:weak_harnack_inequality:p2}
		$\classification{Z( \zeta )}{z}{\text{$f(z) \leq \tau$ and $d_R(z)
		\geq \zeta + \eta$} }$ is compact whenever $0 \leq
		\zeta < \infty$, $0 \leq \tau < t$, $0 < R < \infty$ and
		$\eta > 0$.
	\end{enumerate}
	If $a \in \Clos (Z(\zeta)) \without Z ( \zeta )$ with $\liminf_{z \to
	a} ( f | Z ( \zeta ) ) (z) < t$ then there exist $z_i \in Z ( \zeta)$
	with $z_i \to a$ as $i \to \infty$ and $\lim_{i \to \infty} f (z_i) <
	t$ hence
	\begin{gather*}
		\dist (a,B) \geq \zeta, \quad a \notin B, \quad a \notin
		( \Clos S ) \without S, \quad a \in S, \quad \dist (a,B) \leq
		\zeta
	\end{gather*}
	and
	\eqref{item:weak_harnack_inequality:p1} follows. Since
	\begin{gather*}
		\begin{aligned}
			& \classification{Z(\zeta)}{z}{\text{$f(z) \leq \tau$
			and $d_R(z) \geq \zeta + \eta$}} \\
			& \qquad = \classification{S}{z}{\text{$f(z) \leq \tau$
			and $d_R ( z) \geq \zeta + \eta$} } \subset Z (
			\zeta + \eta/2 ),
		\end{aligned}
	\end{gather*}
	any point $a$ in the closure of the set in
	\eqref{item:weak_harnack_inequality:p2} belongs to $S$ either by
	\eqref{item:weak_harnack_inequality:p1} with $\zeta$ replaced by
	$\zeta + \eta /2$ if $a \notin Z ( \zeta + \eta/2 )$ or
	trivially else, hence \eqref{item:weak_harnack_inequality:p2} follows.

	Suppose $\nuqastnorm{q}{V}{T} < \gamma < \infty$. Define $H (z) = (
	\lambda^{-1} f(z) + \gamma )^{-1} |\ap Df(z)|$ whenever $z \in \dmn
	\ap Df$. Recalling \eqref{item:weak_harnack_inequality:p2} with $\zeta
	= 0$ and applying \ref{lemma:log_bound} for $0 < R < \infty$ with $M$,
	$t$, $d$, $\delta$, and $r$ replaced $\sup \{ M, 1 \}$, $f \lIm \spt
	\| V \| \rIm$, $d_R$, $\lambda$, and $s$ and letting $R$ tend to
	$\infty$, one obtains
	\begin{gather*}
		\tint{Z(\lambda)}{} H \ud \| V \| \leq M^{1/2} \big (
		\tint{Z(\lambda)}{} H^2 \ud \| V \| \big)^{1/2} \leq \Delta_1.
	\end{gather*}
	Let $E ( \tau ) = \classification{U}{z}{f(z) < \tau}$ for $0 \leq \tau
	< \infty$, define $g_\tau : U \to \rel$ by
	\begin{gather*}
		g_\tau(z) = \sup \{ \log ( \lambda^{-1} \tau + \gamma ) - \log
		( \lambda^{-1} f (z) + \gamma ), 0 \} \quad \text{for $z \in
		U$, $0 \leq \tau < \infty$},
	\end{gather*}
	note that $g_\tau$ is locally Lipschitzian and compute
	\begin{gather*}
		\begin{aligned}
			\ap D g_\tau (z) & = - \lambda^{-1} ( \lambda^{-1}
			f(z) + \gamma )^{-1} \ap Df (z) && \quad \text{if $z
			\in E(\tau)$}, \\
			\ap D g_\tau (z) & = 0 && \quad \text{if $z \in U
			\without E(\tau)$}
		\end{aligned}
	\end{gather*}
	for $\| V \|$ almost all $z$. Recalling
	\eqref{item:weak_harnack_inequality:p1} with $\zeta = \lambda$,
	applying \ref{thm:sobolev_poincare_inequality} with $U$, $M$, $f$ and
	$r$ replaced by $W ( \lambda )$, $(1-2\lambda)^{-\vdim} M$, $g_\tau |
	W ( \lambda )$, and $1-2\lambda$ and, in case
	\eqref{item:weak_harnack_inequality:p2} also $\delta$ replaced by
	$\delta/2$, yields
	\begin{gather*}
		\eqLpnorm{\| V \| \restrict Z ( 1-\lambda)}{\beta}{g_\tau}
		\leq \varepsilon^{-1}
		\eqLpnorm{\| V \| \restrict Z ( \lambda )}{1}{| \ap Dg_\tau|}
		\leq \Delta_4
	\end{gather*}
	whenever $0 \leq \tau < \infty$ and $\tau \leq t$. Letting $\tau \to
	t-$, one infers $t < \infty$ since $\| V \| ( Z (1) ) > 0$ and,
	abbreviating $h = g_t$,
	\begin{gather*}
		\eqLpnorm{\| V \| \restrict Z ( 1-\lambda)}{\beta}{h} \leq
		\Delta_4.
	\end{gather*}

	Next, the following assertion will be shown: \emph{If $\vdim>1$ then
	\begin{gather*}
		\tint{}{} \left < \ap D \theta (z) \otimes \ap D h (z), A(z,S)
		\right > \ud V (z,S) \leq \Delta_5 \Lpnorm{\| V\|}{s}{ | \ap D
		\theta |}
	\end{gather*}
	whenever $\theta \in \mathscr{D}^0 ( U )^+$}. For this purpose define
	$\xi : U \to \rel$ by
	\begin{gather*}
		\xi (z) = \theta (z) \lambda^{-1} ( \lambda^{-1} f(z) + \gamma
		)^{-1} \quad \text{for $z \in U$},
	\end{gather*}
	note that $\xi$ is a Lipschitzian function with compact support and
	compute
	\begin{gather*}
		\begin{aligned}
			\ap D \xi (z) & = \lambda^{-1} ( \lambda^{-1} f (z) +
			\gamma )^{-1} \ap D \theta (z) \\
			& \phantom{=} \ - \theta (z) \lambda^{-2} (
			\lambda^{-1} f(z) + \gamma )^{-2} \ap D f(z), \\
			\gamma | \ap D \xi (z) | & \leq \lambda^{-1} ( | \ap D
			\theta (z) | + \theta (z) | \ap D h (z) | ), \\
			\ap D \xi (z) \otimes \ap D f(z) & =
			- \ap D \theta (z) \otimes \ap Dh (z) - \theta(z) \ap
			D h(z) \otimes \ap D h (z),
		\end{aligned} \\
		\begin{aligned}
			& \left < \ap D \theta (z) \otimes \ap D h(z), A(z,S)
			\right > \\
			& \qquad \leq - \left < \ap D \xi (z) \otimes \ap D
			f(z), A(z,S) \right > - \theta (z) | \ap D h (z) |^2,
		\end{aligned}
	\end{gather*}
	for $V$ almost all $(z,S)$. Then, noting
	\begin{gather*}
		(2-s)/s = 1/s -1/q > 1/s - 1/\vdim, \quad s < 2 \leq \vdim,
	\end{gather*}
	one estimate by use of H\"older's
	inequality and \ref{corollary:sobolev}
	\begin{gather*}
		\begin{aligned}
			& \Lpnorm{\| V \|}{s}{ \theta | \ap Dh|} \leq
			\Lpnorm{\| V \|}{s/(2-s)}{\theta}^{1/2} \Lpnorm{\| V
			\|}{2}{\theta^{1/2} |\ap Dh|} \\
			& \qquad \leq \lambda^{-1} \Lpnorm{\| V
			\|}{s/(2-s)}{\theta} + \lambda \| V \| ( \theta | \ap
			D h |^2 ) \\
			& \qquad \leq \Delta_2 \Lpnorm{\| V \|}{s} {|\ap
			D\theta|} + \lambda \| V \| ( \theta | \ap Dh|^2 ).
		\end{aligned}
	\end{gather*}
	Now, the assertion readily follows noting \ref{corollary:nuq_density}.
	Recalling \eqref{item:weak_harnack_inequality:p2} with $\zeta =
	1-\lambda$, the assertion implies using
	\ref{thm:local_maximum_estimates} with $U$, $r$, $f$, $d(z)$ and
	$\delta$ replaced by $W( 1-\lambda )$, $\beta$, $h| W( 1-\lambda)$,
	$\sup \{ d_R (z) - ( 1-\lambda ), 0 \}$ and $\lambda$ for $0 < R <
	\infty$ and letting $R$ tend to $\infty$ that
	\begin{gather*}
		\eqLpnorm{\| V \| \restrict Z( 1)}{\infty}{h} \leq \Delta_3
		\big ( \eqLpnorm{\| V \| \restrict Z ( 1-\lambda )}{\beta}{h}
		+ \Delta_5 \big ) \quad \text{if $\vdim>1$}.
	\end{gather*}
	Combining this estimate with the preceding paragraph yields
	\begin{gather*}
		\eqLpnorm{\| V \| \restrict Z ( 1 )}{\infty} {h} \leq
		\Delta_6, \\
		(\lambda^{-1}t + \gamma )/(\lambda^{-1} f(z) + \gamma ) \leq
		\Gamma, \quad t \leq \Gamma \big ( f (z) + \lambda
		\nuqastnorm{q}{V}{T} \big )
	\end{gather*}
	whenever $z \in Z ( 1 )$ and the conclusion follows.
\end{proof}
\begin{corollary} \label{corollary:strong_maximum_principle}
	Suppose $\vdim$, $\adim$, $p$, $U$, $V$, $\psi$, and ``$ap$'' are as
	in \ref{miniremark:situation_general}, $p = \vdim$, $f : U \to \rel$
	is a locally Lipschitzian function, $\psi ( \{ z \} ) <
	\isoperimetric{\vdim}^{-1}$ for $z \in U$,
	\begin{gather*}
		a \in \spt \| V \|, \qquad f(z) \geq f(a) \quad \text{for $z
		\in \spt \| V \|$}, \\
		W = V \restrict \{ (z,S) \with f(z) = f(a) \},
	\end{gather*}
	$\nu$, $\mu$ and $A$ are related to $\vdim$, $\adim$ and $U$ as in
	\ref{miniremark:bilinear_form}, $A$ is $V$ measurable, and
	\begin{gather*}
		\tint{}{} \left < \ap D \theta (z) \otimes \ap D f (z), A
		(z,S) \right > \ud V (z,S) \geq 0 \quad \text{whenever $\theta
		\in \mathscr{D}^0 ( U )^+$}.
	\end{gather*}

	Then
	\begin{gather*}
		a \in \spt \| W \| \quad \text{and} \quad \spt \| W \| \cap
		\spt \| V-W \| = \emptyset.
	\end{gather*}
\end{corollary}
\begin{proof}
	Assume $f(a)=0$ and $f$ is nonnegative. Define
	\begin{gather*}
		B = \classification{U}{z} {f(z)=0}, \quad X = V-W, \\
		E(t) = \classification{U}{z}{f(z)<t}, \quad W_t = V \restrict
		E(t) \times \grass{\adim}{\vdim}, \\
		T(t) = ( \delta V ) \restrict E(t) -  \delta W_t
	\end{gather*}
	whenever $0 < t < \infty$. Note
	\begin{gather*}
		W_t(g) \to W(g) \quad \text{as $t \to 0+$ whenever $g \in
		\mathscr{K} ( U \times \grass{\adim}{\vdim} )$}.
	\end{gather*}

	First, \emph{it is proven $\spt \| W \| = B \cap \spt \| V \|$}.
	Clearly, $\spt \| W \| \subset B \cap \spt \| V \|$.  To prove the
	opposite inclusion suppose $b \in B \cap \spt \| V \|$.  There exist
	$0 < r < \infty$ and $0 \leq \alpha < \isoperimetric{\vdim}^{-1}$ such
	that
	\begin{gather*}
		\cball{b}{3r} \subset U, \quad \psi ( \oball{b}{3r} )^{1/\vdim}
		\leq \alpha.
	\end{gather*}
	For $0 < s \leq r$ define with $\Delta =
	\Gamma_{\ref{thm:relative_iso}} ( \vdim, \alpha
	)^{-1} 2^{-\vdim}$
	\begin{gather*}
		t(s) = \sup \{ u \with \| V \| ( \oball{b}{3s} \cap E(u) )
		\leq \Delta s^\vdim \}
	\end{gather*}
	hence $t(s) \geq 0$ and
	\begin{gather*}
		\| V \| ( \oball{b}{3s} \cap \{ z \with f(z) \leq t(s) \} )
		\geq \Delta s^\vdim.
	\end{gather*}
	Since $t(s) = 0$ for $0 < s \leq r$ by
	\ref{thm:weak_harnack_inequality}, one obtains
	\begin{gather*}
		\measureball{\| W \|}{\oball{b}{3s}} \geq \Delta s^\vdim \quad
		\text{for $0 < s \leq r$}, \qquad b \in \spt \| W \|
	\end{gather*}
	as claimed.

	Next, \emph{it is shown $\delta X = ( \delta V ) \restrict U \without
	B$}. For this purpose suppose $K$ is a compact subset of $U$, $0 <
	\delta < \dist ( K, \rel^\adim \without U )$ and $Z =
	\classification{U}{z}{\dist (z,K) < \delta }$. Using
	\ref{lemma:coarea}, one chooses numbers $u_i$ such that $2^{-i-1} <
	u_i < 2^{-i}$ and
	\begin{gather*}
		\| T (u_i) \| (K) \leq 2^{i+1}
		\tint{\classification{K}{z}{2^{-i-1} < f(z) < 2^{-i}}}{} | \ap
		Df | \ud \| V \|.
	\end{gather*}
	Since \ref{lemma:log_bound} with $U$, $M$, $d(z)$, $q$ and $r$
	replaced by $Z$, $\| V \|(Z)$, $\dist (z,\rel^\adim \without Z )$, $1$
	and $1$ implies $\tint{K \without B}{} f^{-1} | \ap Df| \ud \| V \| <
	\infty$, one infers
	\begin{gather*}
		\| T ( u_i ) \| (K) \leq 2 \tint{K \cap E (2^{-i}) \without
		B}{} f^{-1} | \ap Df | \ud \| V \| \to 0 \quad \text{as $i \to
		\infty$}
	\end{gather*}
	hence $\delta W = ( \delta V ) \restrict B$ and the claim follows. As
	$\ap Df(z) = 0$ for $\| V \|$ almost all $z \in B$, one notes
	\begin{gather*}
		\tint{}{} \left < \ap D \theta (z) \otimes \ap Df(z), A(z,S)
		\right > \ud X (z,S) \geq 0 \quad \text{whenever $\theta \in
		\mathscr{D}^0 ( U )^+$}.
	\end{gather*}
	Therefore if there would exist $b \in \spt \| V \| \cap \spt \| V - W
	\|$ the hypotheses would be satisfied with $V$ and $a$ replaced by $X$
	and $b$ implying by the assertion of the preceding paragraph that $b
	\in \spt \| X \| \restrict B$ in contradiction to $\| X \| (B) = 0$.
\end{proof}
\part{The generalised Gauss map} \label{part:geom}
In this part the validity of an area formula for the generalised Gauss map of
integral varifolds whose mean curvature is integrable to the critical power
($2 \leq p = \vdim = \adim-1$ in \ref{miniremark:situation_general}) is
established, see \ref{thm:area_formula_codim1}. As corollary one obtains a
sharp lower bound on the mean curvature integral with critical power for
integral varifolds, see \ref{thm:lowerbound_hm}.

In view of the author's second order rectifiability results in
\cite[4.8]{snulmenn.c2} this boils down to establishing a Lusin property for
the generalised Gauss map.
\section{Points of finite lower density} \label{sec:finite_density}
In this section the Harnack inequality of Section \ref{sec:harnack} is
utilised to establish the necessary Lusin property for points whose lower
density is finite, see \ref{lemma:zero_sets_to_zero_sets}.
\begin{lemma} \label{lemma:radius_from_calculus}
	Suppose $0 < r < \infty$, $0 < \vdim < \infty$, $I = \{ s \with 0 < s
	< r \}$, $j \in \nat$, $f : I \to \rel$ is a nonnegative function,
	$f_i : I \to \rel$ are nondecreasing functions for $i \in \{ 1,
	\ldots, j \}$ with
	\begin{gather*}
		f(s) \leq f_i (s) \quad \text{whenever $s \in I$},
	\end{gather*}
	$0 < \lambda \leq 1$ and $\liminf_{s \to 0+} s^{-\vdim} f(s) > 0$.

	Then there exists $s \in I$ with
	\begin{gather*}
		f_i ( s ) \leq 2 \lambda^{-j\vdim} f_i (\lambda s) \quad
		\text{whenever $i \in \{ 1, \ldots, j \}$}.
	\end{gather*}
\end{lemma}
\begin{proof}
	Abbreviate $\Gamma = 2 \lambda^{-j\vdim}$, choose $\varepsilon
	> 0$, $t \in I$ and $0 \leq M < \infty$ such that
	\begin{gather*}
		f(s) \geq \varepsilon s^\vdim \quad \text{for $0 < s \leq t$},
		\qquad f_i (t) \leq M \quad \text{for $i \in \{ 1, \ldots, j
		\}$}.
	\end{gather*}
	Define $t_l = \lambda^l t$ for $l \in \nat \cup \{ 0 \}$, choose $k
	\in \nat$ such that $M 2^{-k} \leq \varepsilon \lambda^{j\vdim}
	t^\vdim$ hence
	\begin{gather*}
		\Gamma^{-k} M \leq \varepsilon \lambda^{(jk+1)\vdim} t^\vdim.
	\end{gather*}
	Denote by $P(i)$ the set of all positive integers $l$ with
	\begin{gather*}
		l \leq jk+1 \quad \text{and} \quad f_i ( t_{l-1} ) \geq \Gamma
		f_i (t_l)
	\end{gather*}
	whenever $i \in \{ 1, \ldots, j \}$. Since $f_i$ is nondecreasing one
	estimates
	\begin{gather*}
		\varepsilon \lambda^{(jk+1)\vdim} t^\vdim = \varepsilon
		(t_{jk+1})^\vdim \leq f_i (t_{jk+1}) \leq \Gamma^{-\card P(i)}
		f_i (t) \leq \Gamma^{-\card P(i)} M
	\end{gather*}
	hence $\card P(i) \leq k$ for $i \in \{ 1, \ldots, j \}$. Therefore
	there exists a positive integer $l$ with $l \leq jk+1$ and $l \notin
	P(i)$ for $i \in \{ 1, \ldots, j \}$ and one may take $s = t_{l-1}$.
\end{proof}
\begin{definition} [cf.~\protect{\cite[5.4.12]{MR41:1976}}]
	Suppose $\vdim, \adim \in \nat$, $\vdim < \adim$, $M$ is an $\vdim$
	dimensional submanifold of $\rel^\adim$ of class $2$ and $U$ is an
	open subset of $\rel^\adim$.

	Then $f_1, \ldots, f_n$ are called a \emph{Cartan frame for $M$ in
	$U$} if and only if $f_i : U \to \rel$ are of class $1$ for $i \in
	\{1,\ldots,\adim\}$ and
	\begin{gather*}
		\text{$f_1(z), \ldots, f_\adim(z)$ are orthogonal for $z \in
		U$}, \\
		f_i (z) \in \Tan ( M, z ) \quad \text{for $z \in U \cap M$ and
		$i \leq \vdim$}, \\
		f_i(z) \in \Nor ( M, z ) \quad \text{for $z \in U \cap M$ and
		$i > \vdim$}.
	\end{gather*}
	Moreover, such a Cartan frame is called \emph{osculating to $M$ at
	$a$} if and only if $a \in U \cap M$ and
	\begin{gather*}
		\begin{aligned}
			& \left < v, Df_i (a) \right > \bullet f_j (a) = 0
			\quad \text{whenever $v \in \rel^\adim$} \\
			& \text{and $\{i,j\} \subset \{ 1, \ldots, \vdim \}$
			or $\{ i,j \} \subset \{ \vdim+1, \ldots, \adim \}$}.
		\end{aligned}
	\end{gather*}
\end{definition}
\begin{lemma} \label{lemma:unit_sphere_bundle}
	Suppose $\vdim, \adim \in \nat$, $\vdim < \adim$, $M$ is an $\vdim$
	dimensional submanifold of $\rel^\adim$ of class $2$, $N =
	\eqclassification{\rel^\adim \times \mathbf{O}^\ast ( \adim, 1 )
	}{(a,p)}{\text{$a \in M$ and $\Tan (M,a) \subset \ker p$}}$,
	\begin{gather*}
		\begin{aligned}
			X(a,p) & = \{ (v, -p \circ \mathbf{b} ( M; a )
			(v,\cdot) \circ \project{\Tan (M,a)} ) \with v \in
			\Tan ( M, a ) \}, \\
			Y(a,p) & = \eqclassification{\rel^\adim \times \Hom (
			\rel^\adim, \rel)}{(0,h)}{\text{$\Tan (M,a) \subset
			\ker h$ and $h \bullet p = 0$}}
		\end{aligned}
	\end{gather*}
	for $(a,p) \in N$ and $Q : \rel^\adim \times \Hom ( \rel^\adim, \rel )
	\to \Hom ( \rel^\adim, \rel )$ satisfies
	\begin{gather*}
		Q (z,h) = h \quad \text{for $(z,h) \in \rel^\adim \times \Hom
		( \rel^\adim, \rel)$}.
	\end{gather*}

	Then the following three statements hold:
	\begin{enumerate}
		\item \label{item:unit_sphere_bundle:manifold} $N$ is an
		$\adim-1$ dimensional submanifold of $\rel^\adim \times \Hom (
		\rel^\adim, \rel )$ of class $1$.
		\item \label{item:unit_sphere_bundle:tangent_space} $X(a,p)$
		and $Y(a,p)$ are mutually orthogonal vector subspaces of
		$\rel^\adim \times \Hom ( \rel^\adim, \rel )$ of dimension
		$\vdim$ and $\codim-1$ respectively and
		\begin{gather*}
			\Tan ( N, (a,p) ) = X (a,p) \oplus Y (a,p) \quad
			\text{for $(a,p) \in N$}.
		\end{gather*}
		\item \label{item:unit_sphere_bundle:fubini} If $A$ is
		an $\mathscr{H}^{\adim-1}$ measurable subset of $N$ then
		\begin{gather*}
			\begin{aligned}
				& \tint{\mathbf{O}^\ast ( \adim,1 )}{}
				\mathscr{H}^0 ( A \cap Q^{-1} \lIm \{ p \}
				\rIm ) \ud \mathscr{H}^{\adim-1} p \\
				& \qquad = \tint{A}{} \| \Lambda_{\adim-1} ( Q
				| \Tan ( N, (a,p) ) \| \ud
				\mathscr{H}^{\adim-1} p \\
				& \qquad = \tint{M}{}
				\tint{\classification{\mathbf{O}^\ast ( \adim,
				1 )}{p}{(a,p) \in A}}{} | \discr ( p \circ
				\mathbf{b} ( M ; a ) ) | \ud
				\mathscr{H}^{\codim-1} p \ud \mathscr{H}^\vdim
				a.
			\end{aligned}
		\end{gather*}
	\end{enumerate}
\end{lemma}
\begin{proof}
	Define $P : \rel^\adim \times \Hom ( \rel^\adim, \rel ) \to
	\rel^\adim$ by
	\begin{gather*}
		P (z,h) = z \quad \text{for $(z,h) \in \rel^\adim \times \Hom
		( \rel^\adim, \rel )$}
	\end{gather*}
	and suppose $(a,p) \in N$.

	Choose an open neighbourhood $U$ of $a$ and a Cartan frame for $M$ in
	$U$ osculating to $M$ at $a$ such that $f_n (a) = p^\ast (1)$, cp.
	\cite[5.4.11,\,12]{MR41:1976}. Define $g_i : U \to \Hom ( \rel^\adim,
	\rel )$ by
	\begin{gather*}
		g_i (z) (v) = f_i(z) \bullet v \quad \text{whenever $z \in U$,
		$v \in \rel^\adim$ and $i \in \{ 1, \ldots, \adim \}$}.
	\end{gather*}
	Note $g_n (a) = p$ and observe
	\begin{gather*}
		\left < v, D g_n (a) \right > = -p \circ \mathbf{b}(M;a) ( v,
		\cdot ) \circ \project{\Tan (M,a)} \quad \text{for $v \in \Tan
		(M,a)$};
	\end{gather*}
	in fact $\left < v, Df_n (a) \right > \in \Tan ( M,a )$ and
	\begin{gather*}
		\begin{aligned}
			& \left < v, Dg_n (a) \right > (u) = \left < v, Df_n
			(a) \right > \bullet \project{\Tan(M,a)} (u) \\
			& \qquad = - \mathbf{b} (M;a)
			(v,\project{\Tan(M,a)}(u)) \bullet f_n (a) = - p \circ
			\mathbf{b} (M;z) ( v,\project{\Tan(M,a)}(u) )
		\end{aligned}
	\end{gather*}
	for $u \in \rel^\adim$. If $(v,g) \in X(a,p)$ and $(w,h) \in Y(a,p)$
	then
	\begin{gather*}
		w=0, \quad \Nor (M,a) \subset \ker g , \quad \Tan (M,a)\subset
		\ker h, \\
		g \bullet h = \tsum{i=1}{\adim} g ( f_i(a)) h (f_i(a)) = 0,
	\end{gather*}
	hence $X(a,p)$ and $Y(a,p)$ as well as $Q \lIm X (a,p) \rIm$ and $Q
	\lIm Y (a,p) \rIm$ are mutually orthogonal vector subspaces. Also
	note $Y(a,p) = \{0\}$ if $\codim=1$ and
	\begin{gather*}
		Y (a,p) = \big \{ \big ( 0,
		\tsum{i=1}{\codim-1} w_i g_{\vdim+i} (a) \big ) \with w_i \in
		\rel^{\codim-1} \big \}  \quad \text{if $\codim > 1$}.
	\end{gather*}

	Denote by $\upsilon_1, \ldots, \upsilon_\codim$ the standard base
	vectors of $\rel^\codim$ and define $f : ( U \cap M ) \times
	\mathbf{S}^{\codim-1} \to \rel^\adim \times \Hom ( \rel^\adim, \rel )$
	by
	\begin{gather*}
		f(b,y) = \big ( b, \tsum{i=1}{\codim} y_i g_{\vdim+i} (b) \big
		)
	\end{gather*}
	whenever $b \in U \cap M$ and $y = (y_1,\ldots,y_\codim) \in
	\mathbf{S}^{\codim-1}$. Note
	\begin{gather*}
		f ( a, \upsilon_\codim ) = (a,p), \quad \im f = ( U \times
		\Hom ( \rel^\adim , \rel ) ) \cap N, \\
		f^{-1} (b,q) = \big ( b, ( g_{\vdim+1} (b) \bullet q, \ldots,
		g_\adim (b) \bullet q ) \big ) \quad \text{for $(b,q) \in \im
		f$}.
	\end{gather*}
	One computes for $b \in U \cap M$, $y = (y_1,\ldots,y_\codim) \in
	\mathbf{S}^{\codim-1}$
	\begin{gather*}
		\left < (v,w), Df(b,y) \right > = \big ( v, \tsum{i=1}{\codim}
		y_i \left < v, Dg_{\vdim+i} (b) \right > + w_i g_{\vdim+i} (b)
		\big )
	\end{gather*}
	whenever $v \in \Tan (M,b)$ and $w \in \Tan ( \mathbf{S}^{\codim-1},
	y)$, in particular, $Df(b,y)$ is univalent and
	\eqref{item:unit_sphere_bundle:manifold} follows from
	\cite[3.1.19\,(3)]{MR41:1976}. Comparing the formulae for $Dg_n$ and
	$Df$, one infers \eqref{item:unit_sphere_bundle:tangent_space}.

	Next, the following assertion will be shown:
	\begin{gather*}
		| \discr ( p \circ \mathbf{b} ( M ; a ) ) | = \| \Lambda_\vdim
		( P | \Tan ( N, (a,p) ) ) \|^{-1} \| \Lambda_{\adim-1} ( Q |
		\Tan ( N, (a,p) ) ) \|.
	\end{gather*}
	For this purpose abbreviate $Z = \Tan ( N, (a,p) )$ and let $L : \Tan
	(M,a) \to \Hom ( \rel^\adim, \rel )$ and $K : \Tan (M,a) \to X$ be
	defined by
	\begin{gather*}
		L(v) = - p \circ \mathbf{b} (M;a) (v,\cdot) \circ
		\project{\Tan (M,a)} \quad \text{and} \quad K(v) = (v,L(v))
	\end{gather*}
	whenever $v \in \Tan (M,a)$. Clearly, $K$ is univalent. First,
	denoting by $\iota$ the inclusion of $X(a,p)$ into $Z$, the orthogonal
	projection of $Z$ onto $X(a,p)$ is given by $\iota^\ast$ by
	\eqref{item:unit_sphere_bundle:tangent_space}, and, see
	\cite[1.7.6]{MR41:1976},
	\begin{gather*}
		\iota^\ast \circ \iota = \id{X}, \quad P|Z = (P|Z) \circ
		\iota^\ast, \quad P \circ \iota = K^{-1}, \\
		\iota \circ ( P | Z )^\ast = ( P | Z )^\ast, \quad (P|Z)^\ast
		= \iota^\ast \circ ( P | Z)^\ast = ( P \circ \iota )^\ast = (
		K^{-1} )^\ast \\
		\| \Lambda_\vdim ( P | Z ) \| = \| \Lambda_\vdim K^{-1} \| =
		\| \Lambda_\vdim K \|^{-1}.
	\end{gather*}
	Second, recalling that $X(a,p)$ and $Y(a,p)$ as well as $Q \lIm X
	(a,p) \rIm$ and $Q \lIm Y (a,p) \rIm$ are mutually orthogonal, one
	obtains
	\begin{gather*}
		\| \Lambda_{\adim-1} (Q|Z) \| = \| \Lambda_\vdim ( Q | X(a,p) )
		\|
	\end{gather*}
	since $Q|Y(a,p)$ is an isometry. Therefore, noting $Q \circ K = L$ and
	$\im K = X(a,p)$, one infers
	\begin{gather*}
		\| \Lambda_\vdim L \| = \| \Lambda_{\adim-1} ( Q | Z ) \| \|
		\Lambda_\vdim K \| = \| \Lambda_\vdim ( P |Z ) \|^{-1} \|
		\Lambda_{\adim-1} ( Q | Z ) \|.
	\end{gather*}
	Since $\im L \subset \classification{\Hom ( \rel^\adim, \rel
	)}{h}{\Nor (M,a) \subset \ker h} \simeq \Tan (M,a)$, the assertion now
	readily follows.

	Combining the assertion of the preceding paragraph with the coarea
	formula \cite[3.2.22]{MR41:1976} yields
	\eqref{item:unit_sphere_bundle:fubini}.
\end{proof}
\begin{miniremark} \label{miniremark:det_trace_inq}
	If $V$ is an $\vdim$ dimensional inner product space and $B: V \times
	V \to \rel$ is a symmetric bilinear form with $B(v,v) \geq 0$ for $v
	\in V$ then
	\begin{gather*}
		0 \leq \discr B \leq \vdim^{-\vdim} ( \trace B )^\vdim;
	\end{gather*}
	in fact $B$ admits a representation $B = \sum_{i=1}^\vdim \kappa_i
	\omega_i \odot \omega_i/2$ with $0 \leq \kappa_i < \infty$ and
	$\omega_i$ an orthogonal base of $\Hom ( V, \rel )$, see
	\cite[1.7.3]{MR41:1976}, hence $\discr B =
	\prod_{i=1}^\vdim \kappa_i$ and $\trace B = \sum_{i=1}^\vdim \kappa_i$
	and the assertion follows from the inequality relating arithmetic and
	geometric means.
\end{miniremark}
\begin{theorem} [Classical] \label{thm:classical}
	Suppose $\vdim, \adim \in \nat$, $\vdim < \adim$, $M$ is a nonempty
	compact $\vdim$ dimensional submanifold of $\rel^\adim$ of class $2$.

	Then
	\begin{gather*}
		( \vdim + 1 ) \unitmeasure{\vdim+1} \vdim^\vdim \leq
		\tint{M}{} | \mathbf{h} ( M ; a ) |^\vdim \ud
		\mathscr{H}^\vdim a.
	\end{gather*}
\end{theorem}
\begin{proof}
	Suppose $Q$ is as in \ref{lemma:unit_sphere_bundle}.

	Define $C$ to be the closed set of all $(a,p) \in M \times
	\mathbf{O}^\ast ( \adim, 1 )$ such that
	\begin{gather*}
		p (z-a) \geq 0 \quad \text{whenever $z \in M$}.
	\end{gather*}
	Since for $p \in \mathbf{O}^\ast ( \adim, 1 )$ there exists $a \in M$
	with $p (a) = \inf p \lIm M \rIm$ hence $(a,p) \in C$, one notes
	\begin{gather*}
		Q \lIm C \rIm =\mathbf{O}^\ast ( \adim, 1 ).
	\end{gather*}
	Computing the first and second covariant derivative of $p$ at $a$, one
	infers
	\begin{gather*}
		\Tan (M,a) \subset \ker p, \qquad p \circ \mathbf{b}(M;a)
		(v,v) \geq 0 \quad \text{for $v \in \Tan (M,a)$}
	\end{gather*}
	whenever $(a,p) \in C$. From \ref{miniremark:det_trace_inq} one then
	obtains
	\begin{gather*}
		0 \leq \discr ( p \circ \mathbf{b} (M;a) ) \leq \vdim^{-\vdim}
		p ( \mathbf{h} (M;a) )^\vdim \quad \text{for $(a,p) \in C$}.
	\end{gather*}
	Since members of $\mathbf{O}(\adim)$ induce isometries on
	$\mathbf{O}^\ast ( \adim, 1 )$, one may construct a number $0 < c <
	\infty$ determined only by $\vdim$ and $\adim$ such that
	\begin{gather*}
		c | w |^\vdim = \tint{\classification{\mathbf{O}^\ast ( \adim,
		1 )}{p}{\text{$V \subset \ker p$ and $p (w)\geq 0$}}}{}
		p(w)^\vdim \ud \mathscr{H}^{\codim-1} p
	\end{gather*}
	whenever $V$ is an $\vdim$ dimensional vector subspace of $\rel^\adim$
	and $w \in \rel^\adim$ is orthogonal to $V$. Therefore one obtains
	with the help of
	\ref{lemma:unit_sphere_bundle}\,\eqref{item:unit_sphere_bundle:fubini}
	\begin{gather*}
		\begin{aligned}
			& \mathscr{H}^{\adim-1} ( \mathbf{O}^\ast ( \adim, 1 )
			) \leq \tint{\mathbf{O}^\ast ( \adim, 1 )}{}
			\mathscr{H}^0 ( C \cap Q^{-1} \lIm \{ p \} \rIm ) \ud
			\mathscr{H}^{\adim-1} p \\
			& \qquad = \tint{M}{}
			\tint{\classification{\mathbf{O}^\ast ( \adim, 1
			)}{p}{(a,p) \in C}}{} \discr ( p \circ \mathbf{b} ( M;
			a ) ) \ud \mathscr{H}^{\codim-1} p \ud
			\mathscr{H}^\vdim a \\
			& \qquad \leq \vdim^{-\vdim} c \tint{M}{} | \mathbf{h}
			( M; a )|^\vdim \ud \mathscr{H}^\vdim a.
		\end{aligned}
	\end{gather*}

	Next, consider the special case $M = \iota \lIm \mathbf{S}^\vdim
	\rIm$ for some isometric injection $\iota : \rel^{\vdim+1} \to
	\rel^\adim$. Define $B = \classification{\mathbf{O}^\ast ( \adim, 1
	)}{p}{\iota \lIm \rel^{\vdim+1} \rIm \subset \ker p}$ and note
	$\mathscr{H}^{\adim-1} ( B ) = 0$.  Observe that
	\begin{gather*}
		\mathscr{H}^0 ( C \cap Q^{-1} \lIm \{ p \} \rIm ) = 1 \quad
		\text{for $p \in \mathbf{O}^\ast ( \adim, 1 ) \without B$}, \\
		\discr ( p \circ \mathbf{b} ( M; a ) ) = \vdim^{-\vdim} p (
		\mathbf{h} ( M;a ) ) \quad \text{for $(a,p) \in C$}, \\
		\begin{aligned}
			& \classification{\mathbf{O}^\ast ( \adim, 1)}{p}{
			(a,p) \in C} \\
			& \qquad = \classification{\mathbf{O}^\ast ( \adim, 1
			)}{p}{\text{$\Tan (M,a) \subset \ker p$ and $p (
			\mathbf{h} ( M;a ) ) \geq 0$}} \quad \text{for $a \in
			M$}.
		\end{aligned}
	\end{gather*}
	Therefore one computes, using
	\ref{lemma:unit_sphere_bundle}\,\eqref{item:unit_sphere_bundle:fubini},
	\begin{gather*}
		\begin{aligned}
			& \mathscr{H}^{\adim-1} ( \mathbf{O}^\ast ( \adim,1 )
			) = \tint{\mathbf{O}^\ast ( \adim,1 )}{} \mathscr{H}^0
			( C \cap Q^{-1} \lIm \{ p \} \rIm ) \ud
			\mathscr{H}^{\adim-1} p \\
			& \qquad = \tint{M}{}
			\tint{\classification{\mathbf{O}^\ast ( \adim, 1)}{p}{
			(a,p) \in C}}{} \discr ( p \circ \mathbf{b} ( M;a ) )
			\ud \mathscr{H}^{\codim-1} p \ud \mathscr{H}^\vdim a
			\\
			& \qquad = \vdim^{-\vdim} c \tint{M}{} | \mathbf{h} ( M;
			a ) |^\vdim \ud \mathscr{H}^\vdim a = ( \vdim + 1 )
			\unitmeasure{\vdim+1} c.
		\end{aligned}
	\end{gather*}
	Comparing this with the previous estimate for general $M$ yields the
	conclusion.
\end{proof}
\begin{lemma} \label{lemma:shifting}
	Suppose $\vdim, \adim \in \nat$, $\vdim < \adim$, $U$ is an open
	subset  of $\rel^\adim$, $V \in \Var_\vdim ( U )$, $\| \delta V \|$ is
	a Radon measure, $v \in T \in
	\grass{\adim}{\vdim}$, $0 < r < \infty$, $A$ is a compact subset of
	$U$, $B = \{ z + v + w \with \text{$z \in A$ and $w \in \cball{0}{r}$}
	\}$, and
	\begin{gather*}
		K = \{ z + tv+ w \with \text{$z \in A$, $0 \leq t \leq 1$ and
		$w \in \cball{0}{r}$} \} \subset U.
	\end{gather*}
	\pagebreak

	Then
	\begin{gather*}
		\| V \| ( B ) \geq \| V \| ( A ) - |v| \big ( \| \delta V \| (
		K ) + r^{-1} \tint{K \times \grass{\adim}{\vdim}}{} \|
		\project{S} - \project{T} \| \ud V (z,S) \big ).
	\end{gather*}
\end{lemma}
\begin{proof}
	Suppose $1 < \tau < \infty$ and choose $\theta \in
	\mathscr{D}^0 ( \rel^\adim )$ such that
	\begin{gather*}
		\spt \theta \subset \{ z + w \with \text{$z \in A$ and $w \in
		\cball{0}{r}$} \}, \qquad \theta (z) = 1 \quad \text{for $z
		\in A$}, \\
		0 \leq \theta (z) \leq 1 \quad \text{and} \quad \| D \theta
		(z) \| \leq \tau r^{-1} \qquad \text{for $z \in \rel^\adim$}.
	\end{gather*}
	Define $C(t) = \{ z + tv + w \with \text{$z \in A$ and $w \in
	\cball{0}{r}$} \}$ for $t \in \rel$ and note $\spt ( \theta \circ
	\boldsymbol{\tau}_{-tv} ) \subset C (t)$. Moreover, define $I = \{ t
	\with C(t) \subset U \}$ and $\phi : I \to \rel$ by $\phi (t) =
	\tint{}{} \theta (z-tv) \ud \| V \| z$ for $t \in I$ and note that $I$
	is open, $\{ t \with 0 \leq t \leq 1 \} \subset I$ and $\phi$ is of
	class $1$. One computes
	\begin{gather*}
		\begin{aligned}
			\phi'(t) & = - \tint{}{} \left < v, D \theta (z-tv)
			\right > \ud \| V \| z \\
			& = - ( \delta V )_z ( \theta (z-tv) v ) + \tint{}{}
			\left < ( \project{S}-\project{T} )(v), D \theta
			(z-tv) \right > \ud V (z,S), \\
			| \phi'(t) | & \leq |v| \big ( \| \delta V \| ( K ) +
			\tau r^{-1} \tint{K \times \grass{\adim}{\vdim}}{} \|
			\project{S} - \project{T} \| \ud V (z,S) \big )
		\end{aligned}
	\end{gather*}
	whenever $0 \leq t \leq 1$. Since $\| V \| (A) \leq \phi (0)$ and
	$\phi (1)\leq \| V \| ( B )$, the conclusion readily follows.
\end{proof}
\begin{miniremark} \label{miniremark:codim1_planes}
	If $1 < \adim \in \nat$, $P \in \mathbf{O}^\ast ( \adim, \adim-1)$, $Q
	\in \mathbf{O}^\ast ( \adim, 1 )$, $Q \circ P^\ast = 0$ and $S \in
	\grass{\adim}{\adim-1}$, then
	\begin{gather*}
		| Q \circ \project{S} | = 2^{-1/2} | P^\ast \circ P -
		\project{S} |;
	\end{gather*}
	in fact $\id{\rel^\adim} = P^\ast \circ P + Q^\ast \circ Q$ and
	\begin{gather*}
		\begin{aligned}
			| Q \circ \project{S} |^2 & = \trace ( \project{S}
			\circ Q^\ast \circ Q ) = \project{S} \bullet ( Q^\ast
			\circ Q ) \\
			& = \project{S} \bullet ( \id{\rel^\adim} - P^\ast
			\circ P ) = \adim -1 - \project{S} \bullet ( P^\ast
			\circ P ) = {\textstyle\frac{1}{2}} | P^\ast \circ P -
			\project{S} |^2.
		\end{aligned}
	\end{gather*}
\end{miniremark}
\begin{lemma} \label{lemma:touching}
	Suppose $\vdim \in \nat$ and $0 \leq M < \infty$.

	Then there exists a positive, finite number $\Gamma$ with the
	following property.

	If $\adim$, $p$, $U$, $V$, and $\psi$ are related to $\vdim$ as in
	\ref{miniremark:situation_general}, $p = \vdim = \adim-1$, $a \in
	\rel^\adim$, $0 < r < \infty$, $U = \oball{a}{r}$, $P \in
	\mathbf{O}^\ast ( \adim, \vdim )$, $Q \in \mathbf{O}^\ast ( \adim,
	1)$, $Q \circ P^\ast = 0$,
	\begin{gather*}
		\| V \| ( U ) \leq M r^\vdim, \quad h = \inf Q \lIm \spt \| V
		\| \rIm < \infty, \quad 0 \leq \gamma \leq \Gamma^{-1}, \\
		r^{-1} \big ( \inf Q \lIm \oball{a}{r/8} \cap \spt \| V \|
		\rIm \big ) + \psi ( U )^{1/\vdim} \leq \gamma + r^{-1} h,
	\end{gather*}
	and $\zeta \in \cball{P(a)}{r/2}$ then there exists $\xi \in
	\cball{a}{3r/4} \cap \spt \| V \|$ such that
	\begin{gather*}
		|P (\xi) - \zeta | \leq r/4 \quad \text{and} \quad Q ( \xi )
		\leq \Gamma \gamma r + h.
	\end{gather*}
\end{lemma}
\begin{proof}
	Assume $M \geq 1$. Define $\alpha = ( 2 \isoperimetric{\vdim})^{-1}$,
	\begin{gather*}
		\beta = \infty \quad \text{if $\vdim = 1$}, \qquad \beta =
		\vdim/(\vdim-1) \quad \text{if $\vdim > 1$}, \\
		\varepsilon = \inf \big \{ 2^{-6\vdim}
		\Gamma_{\ref{thm:relative_iso}\eqref{item:relative_iso:some_bound}}
		( \vdim, \alpha )^{-1}, \textstyle{\frac{1}{2}} ( 16 \vdim
		\isoperimetric{\vdim} )^{-\vdim} ( 1 - \alpha
		\isoperimetric{\vdim} )^{-1} \big \}, \\
		\Delta_1 = \isoperimetric{\vdim} ( 1- \alpha
		\isoperimetric{\vdim} )^{-1}, \quad \Delta_2 =
		\Gamma_{\ref{thm:weak_harnack_inequality}} ( \vdim, \infty,
		\alpha, 1 ), \quad \Delta_3 = \Delta_2 ( 1 + \Delta_1 ), \\
		\Delta_4 = ( ( 5 \Delta_3 + 1 ) \exp ( 120M/\varepsilon ) - 1
		)/(5\Delta_3), \\
		\Delta_5 = \sup \big \{ \alpha^{-1}, 3 M^{1/\beta}
		\varepsilon^{-1}, 2^{10} M \varepsilon^{-1} ( 5 \Delta_3
		\Delta_4 + 1 ) \big \}, \quad \Gamma = \sup \{ \Delta_3
		\Delta_4, \Delta_5 \}.
	\end{gather*}

	In order to prove that $\Gamma$ has the asserted property, suppose
	$\vdim$, $M$, $\adim$, $p$, $U$, $V$, $\psi$, $a$, $r$, $P$, $Q$,
	$h$, $\gamma$, and $\zeta$ satisfy the hypotheses in the body of the
	lemma.

	For this purpose one may assume $r = 1$ and $\gamma > 0$. Since
	$\oball{a}{1/8} \cap \spt \| V \| \neq \emptyset$ as $\gamma + h <
	\infty$, one infers from \cite[2.5]{snulmenn.isoperimetric} that
	\begin{gather*}
		\measureball{\| V \|}{\oball{a}{{\textstyle\frac{3}{16}}}}
		\geq (16\vdim \isoperimetric{\vdim} )^{-\vdim} ( 1- \alpha
		\isoperimetric{\vdim} )^{-1}.
	\end{gather*}
	Define $T \in \mathscr{D}_0 ( U )$ by
	\begin{gather*}
		T ( \theta ) = \tint{U}{} \theta (z) \left < \eta ( V; z ), Q
		\right > \ud \| \delta V \| z \quad \text{for $\theta \in
		\mathscr{D}^0 ( U )$}
	\end{gather*}
	with $\eta$ as in \ref{miniremark:situation_general}
	and $f : U \to \rel$ by
	\begin{gather*}
		f (z) = \sup \{ Q (z) - h, 0 \} \quad \text{for $z \in U$}.
	\end{gather*}
	Note that $f$ is a nonnegative Lipschitzian function with
	\begin{gather*}
		f(z) = Q(z)-h \quad \text{and} \quad \ap D f (z) = \ap D Q (z)
	\end{gather*}
	whenever $z \in \spt \| V \|$, hence
	\begin{gather*}
		\tint{U}{} \ap D \theta \bullet \ap Df \ud \| V \| = T (
		\theta ) \quad \text{whenever $\theta \in \mathscr{D}^0 ( U
		)$}
	\end{gather*}
	by \ref{ex:some_lap}\,\eqref{item:some_lap:c2} applied to $Q$. Also
	note that \ref{corollary:dual_embedding} with $q$, $r$, and $s$
	replaced by $\vdim$, $\beta$, and $\infty$ yields in conjunction with
	\ref{remark:dual_embedding} that
	\begin{gather*}
		\nuqastnorm{\infty}{V}{T} \leq \Delta_1 \psi ( U )^{1/\vdim}.
	\end{gather*}
	Defining $t = \sup \{ u \with \| V \| (
	\classification{\oball{a}{\frac{3}{16}}}{z}{f(z) < u} ) \leq
	\varepsilon \}$, one obtains
	\begin{gather*}
		\| V \| (
		\classification{\oball{a}{\textstyle\frac{3}{16}}}{z}{f(z) <
		t} ) \leq \varepsilon \leq \| V \| (
		\classification{\oball{a}{\textstyle\frac{3}{16}}}{z}{f(z)
		\leq t} )
	\end{gather*}
	since $\varepsilon < \measureball{\| V \|}{\oball{a}{\frac{3}{16}}} <
	\infty$. One infers from
	\ref{thm:weak_harnack_inequality}\,\eqref{item:weak_harnack_inequality:some_bound}
	with $U$, $q$, $M$, $r$, and $B$ replaced by
	$\oball{a}{\frac{3}{16}}$, $\infty$, $3^{-\vdim}
	\Gamma_{\ref{thm:relative_iso}\eqref{item:relative_iso:some_bound}} (
	\vdim, \alpha )^{-1}$, $\frac{1}{16}$, $\Bdry \oball{a}{\frac{3}{16}}$
	and $\nu$, $\mu$ and $A$ as in \ref{example:V_subharmonic} that
	\begin{gather*}
		t \leq \Delta_2 \big ( \inf f \lIm \oball{a}{1/8} \cap \spt \|
		V \| \rIm + \Delta_1 \psi (U)^{1/\vdim} \big ) \leq \Delta_3
		\gamma.
	\end{gather*}

	Define $g : U \to \rel$ by $g(z) = \log ( 5 f(z) + \gamma )$ for $z
	\in U$, note that $g$ is a Lipschitzian function and
	\begin{gather*}
		\ap D g (z) = 5 ( 5 f(z) + \gamma )^{-1} \ap D f(z) \quad
		\text{for $\| V \|$ almost all $z$}.
	\end{gather*}
	From \ref{lemma:log_bound} with $d(z)$,  $\delta$, $q$, and $r$
	replaced by $\dist (z, \rel^\adim \without U )$, $1/5$, $\beta$, and
	$1$ and $\nu$, $\mu$ and $A$ as in \ref{example:V_subharmonic} one
	obtains in conjunction with H\"older's inequality
	\begin{gather*}
		\tint{\oball{a}{4/5}}{} | \ap D g | \ud \| V \| \leq 40
		M.
	\end{gather*}
	Let $s_1 = \log ( 5 \Delta_3 + 1 ) + \log \gamma$ and $s_2 = \log ( 5
	\Delta_3 \Delta_4 + 1 ) + \log \gamma $ and note $s_2-s_1= 120
	M/\varepsilon$ by the definition of $\Delta_4$. Defining $E(s) =
	\classification{U}{z}{g(z) < s}$ and $V_s = V \restrict E(s) \times
	\grass{\adim}{\vdim}$ for $s \in \rel$, one infers from
	\ref{lemma:coarea} that
	\begin{gather*}
		\tint{s_1}{s_2} \measureball{\| (\delta V) \restrict E(s) -
		\delta V_s \|} {\oball{a}{4/5}} \ud \mathscr{L}^1 s \leq 40 M.
	\end{gather*}
	Therefore there exists $s$ such that
	\begin{gather*}
		s_1 < s < s_2 \quad \text{and} \quad \measureball{\| ( \delta
		V ) \restrict E(s) - \delta V_s \|}{\oball{a}{4/5}} \leq 40 M
		( s_2-s_1 )^{-1} = \varepsilon/3.
	\end{gather*}
	Since $\psi ( U )^{1/\vdim} \leq \Gamma^{-1} \leq M^{-1/\beta}
	\varepsilon/3$, this implies
	\begin{gather*}
		\measureball{\| \delta V_s \|}{\oball{a}{4/5}} \leq 2
		\varepsilon/3.
	\end{gather*}
	Using \ref{miniremark:codim1_planes}, one also obtains
	\begin{gather*}
		\begin{aligned}
			& 16 \tint{\oball{a}{4/5} \times
			\grass{\adim}{\vdim}}{} \| P^\ast \circ P -
			\project{S} \| \ud V_s (z,S) \\
			& \qquad \leq 2^5
			\tint{\classification{\oball{a}{4/5}}{z}{f(z) <
			\Delta_3 \Delta_4 \gamma}}{} | \ap Df | \ud \| V \|
			\leq 2^8 M ( 5 \Delta_3 \Delta_4 + 1) \gamma \leq
			\varepsilon/3.
		\end{aligned}
	\end{gather*}
	Recalling $t \leq \Delta_3 \gamma$, one notes
	\begin{gather*}
		\classification{U}{z}{f(z)\leq t} \subset E(s), \quad
		\measureball{\| V_s \|}{\cball{a}{\textstyle\frac{3}{16}}}
		\geq \varepsilon.
	\end{gather*}
	Therefore, combining the preceding estimates with \ref{lemma:shifting}
	with $U$, $v$, $T$, $r$, $A$, and $B$ replaced by $\oball{a}{4/5}$,
	$P^\ast ( \zeta -P(a))$, $\im P^\ast$, $\frac{1}{16}$,
	$\cball{a}{\frac{3}{16}}$, and $\cball{a+P^\ast (\zeta-P(a))}{1/4}$,
	one infers
	\begin{gather*}
		\measureball{\| V_s \|}{\cball{a+P^\ast ( \zeta-P(a) )}{1/4}}
		\geq \varepsilon/2.
	\end{gather*}
	Since $\spt \| V_s \| \subset \eqclassification{\spt \| V \|}{z}{f(z)
	\leq \Delta_3 \Delta_4 \gamma}$, this implies the existence of $\xi$
	such that
	\begin{gather*}
		\xi \in \cball{a+P^\ast ( \zeta-P(a))}{1/4} \cap \spt \| V \|,
		\quad Q ( \xi ) - h \leq \Gamma \gamma,
	\end{gather*}
	hence $| P ( \xi ) - \zeta | = | P ( \xi - ( a + P^\ast ( \zeta - P (
	a ) ) ) ) | \leq 1/4$ and $| \xi - a| \leq 3/4$.
\end{proof}
\begin{miniremark} \label{miniremark:planes_1codim}
	Suppose $1 < \adim \in \nat$, $q,Q \in \mathbf{O}^\ast ( \adim, 1 )$
	and $q^\ast ( 1 ) \bullet Q^\ast ( 1 ) \geq 0$.

	Then by Allard \cite[3.2\,(9)]{MR840267}
	\begin{gather*}
		\| \eqproject{\ker q} - \eqproject{\ker Q} \| \leq |q-Q| \leq
		2^{1/2} \| \eqproject{\ker q} - \eqproject{\ker Q} \|.
	\end{gather*}
\end{miniremark}
\begin{lemma} \label{lemma:closed_set_tilt}
	Suppose $\vdim \in \nat$, $a \in \rel^\adim$, $0 < r < \infty$, $S
	\subset \oball{a}{r}$, $0 < \delta < \infty$, $0 \leq \gamma \leq
	\delta$, and $C$ is the set of $(b,Q) \in S \times \mathbf{O}^\ast (
	\vdim+1, 1 )$ satisfying the following two conditions:
	\begin{enumerate}
		\item \label{item:closed_set_tilt:nonnegative} $Q(z-b) \geq -
		\gamma r$ whenever $z \in S$.
		\item \label{item:closed_set_tilt:new_points} Whenever $v \in
		\mathbf{S}^\vdim \cap \ker Q$ there exists $\xi \in S$ such
		that
		\begin{gather*}
			( \xi - b ) \bullet v \geq 4 \delta r \quad \text{and}
			\quad Q ( \xi - b ) \leq \gamma r.
		\end{gather*}
	\end{enumerate}
	Denote by $\phi_\infty$ the size $\infty$ approximating measure
	occuring in the construction of $\mathscr{H}^\vdim$ on
	$\mathbf{O}^\ast ( \vdim + 1, 1 )$.

	Then
	\begin{gather*}
		\phi_\infty ( C \lIm \cball{a}{\delta r} \rIm ) \leq \Gamma
		\delta^{-\vdim} \gamma^\vdim
	\end{gather*}
	where $\Gamma$ is a positive, finite number depending only on $\vdim$.
\end{lemma}
\begin{proof}
	Rescaling, one may assume $a = 0$ and $r = 1$.

	Abbreviate $\adim = \vdim + 1$ and $A = C \lIm \cball{a}{\delta r}
	\rIm$. Using the action of $\mathbf{O} ( \adim )$ on $\mathbf{O}^\ast
	( \adim, 1 )$, one obtains $0 < \Delta < \infty$ depending only on
	$\vdim$ such that
	\begin{gather*}
		\mathscr{H}^\vdim ( \mathbf{O}^\ast ( \adim, 1 ) \cap
		\oball{q}{1/2} ) = \Delta \quad \text{whenever $q \in
		\mathbf{O}^\ast ( \adim, 1 )$}.
	\end{gather*}
	Suppose $B$ is a maximal subset of $A$ (with respect to inclusion)
	such that
	\begin{gather*}
		\text{$q_1, q_2 \in B$ with $\oball{q_1}{1/2} \cap
		\oball{q_2}{1/2} \neq \emptyset$ implies $q_1 = q_2$}.
	\end{gather*}
	One notes
	\begin{gather*}
		A \subset {\textstyle\bigcup} \{ \oball{q}{1} \with q \in B
		\}, \quad \card B \leq \Delta^{-1} \mathscr{H}^\vdim (
		\mathbf{O}^\ast ( \adim, 1 ) ).
	\end{gather*}

	Next, it will be shown: \emph{If $(c,q) \in C| \cball{a}{\delta}$,
	$(b,Q) \in C| \cball{a}{\delta}$, $q \in B$ and $|q-Q| < 1$ then
	\begin{gather*}
		|q-Q| \leq 4 \delta^{-1} \gamma.
	\end{gather*}}
	For this purpose choose $p \in \mathbf{O}^\ast ( \adim, 1 )$ with $q
	\circ p^\ast = 0$. Since $|q-Q| < 1$, one notes $Q(q^\ast (y) ) = y +
	(Q-q) ( q^\ast (y)) \to - \infty$ as $y \to - \infty$. Hence there
	exists an affine function $f : \rel^\vdim \to \rel$ such that
	\begin{gather*}
		\{ p^\ast (x) + q^\ast (y) \with \text{$(x,y) \in \rel^\vdim
		\times \rel$ and $y \geq f(x)$} \} =
		\classification{\rel^\adim}{z}{Q(z-b) \geq 0}.
	\end{gather*}
	Assume $\grad f (0) \neq 0$. Letting $u = | \grad f(0) |^{-1} \grad
	f(0)$ and $v = p^\ast ( u )$, by
	\eqref{item:closed_set_tilt:new_points} with $(b,Q)$ replaced by
	$(c,q)$, there exists $\xi$ such that
	\begin{gather*}
		\xi \in S, \quad q ( \xi-c) \leq \gamma, \quad (\xi-c) \bullet
		v \geq 4 \delta.
	\end{gather*}
	Since $Q ( \xi - b ) \geq - \gamma$ by
	\eqref{item:closed_set_tilt:nonnegative}, one obtains, defining $d =
	\xi + Q^\ast ( \gamma )$,
	\begin{gather*}
		|d-\xi| \leq \gamma, \quad Q(d-b) \geq 0, \quad q (d-c) \leq 2
		\gamma.
	\end{gather*}
	Defining $\kappa = p (b) \bullet u$ and $\lambda = p ( d ) \bullet u$,
	one notes
	\begin{gather*}
		\kappa \leq \delta, \quad \lambda \geq 2 \delta, \quad
		\lambda-\kappa \geq \delta
	\end{gather*}
	since $\sup \{ |b|, |c| \} \leq \delta$ and $\lambda = ( \xi - c )
	\bullet v + ( c + d-\xi ) \bullet v$, $|c+d-\xi| \leq 2 \delta$.
	Moreover, using $f(x) = f(0) + | \grad f(0)| x \bullet u$ for $x \in
	\rel^\vdim$,
	\begin{gather*}
		f ( \lambda u ) = f ( p ( d ) ) \leq q ( d ), \quad f ( \kappa
		u ) = f ( p (b) ) = q (b).
	\end{gather*}
	Combining this with the preceding estimate for $q(d-c)$, one infers
	\begin{gather*}
		| \grad f(0)| = ( \lambda - \kappa )^{-1} ( f ( \lambda u ) -
		f ( \kappa u ) ) \leq \delta^{-1} q ( d-b ) \leq 2 \delta^{-1}
		\gamma.
	\end{gather*}
	Since $q^\ast (1) \bullet Q^\ast (1) = 1 + (Q-q) (q^\ast (1)) > 0$,
	the assertion now follows from Allard \cite[8.9\,(5)]{MR0307015} in
	conjunction with \ref{miniremark:planes_1codim}.

	The assertion of the preceding paragraph implies
	\begin{gather*}
		\diam ( A \cap \oball{q}{1} ) \leq 8 \delta^{-1} \gamma \quad
		\text{whenever $q \in B$}
	\end{gather*}
	hence the conclusion with $\Gamma = \unitmeasure{\vdim} 2^{2\vdim}
	\Delta^{-1} \mathscr{H}^\vdim ( \mathbf{O}^\ast ( \vdim + 1 , 1 ) )$.
\end{proof}
\begin{miniremark} \label{miniremark:touching_ball}
	Suppose $1 < \adim \in \nat$, $p \in \orthproj{\adim}{\adim-1}$, $q
	\in \orthproj{\adim}{1}$, $p \circ q^\ast = 0$, $a \in \rel^\adim$, $0
	< R < \infty$ and $z \in \rel^\adim$ with $|z-a+q^\ast (R)| \geq R$,
	$|p (z-a) | \leq R$ and $q(z-a) \geq -R$.

	Then
	\begin{gather*}
		q(z-a) \geq - ( R^2 - | p (z-a) |^2 )^{-1/2} | p (z-a)|^2;
	\end{gather*}
	in fact
	\begin{align*}
		q(z-a) & = ( |z-a+q^\ast (R) |^2 - | p(z-a) |^2 )^{1/2} - R \\
		& \geq (R^2-|p(z-a)|^2)^{1/2} -R \geq - ( R^2 - | p(z-a)|^2
		)^{-1/2} | p (z-a)|^2.
	\end{align*}
\end{miniremark}
\begin{lemma} \label{lemma:local_estimate_contact_set}
	Suppose $\vdim \in \nat$ and $0 \leq M < \infty$.

	Then there exists a positive, finite number $\Gamma$ with the
	following property.

	If $\adim$, $p$, $U$, $V$, and $\psi$ are related to $\vdim$ as in
	\ref{miniremark:situation_general}, $p = \vdim = \adim-1$, $a \in
	\rel^\adim$, $0 < r < \infty$, $U = \oball{a}{r}$, $\| V \| ( U ) \leq
	M r^\vdim$, $0 < R < \infty$,
	\begin{gather*}
		C = \eqclassification{\spt \| V \| \times \mathbf{O}^\ast (
		\adim, 1 )}{(b,Q)}{ \oball{b-Q^\ast ( R )}{R} \cap \spt \| V
		\| = \emptyset },
	\end{gather*}
	and $\phi_\infty$ denotes the size $\infty$ approximating measure
	occuring in the construction of $\mathscr{H}^\vdim$ on
	$\mathbf{O}^\ast ( \adim, 1 )$, then
	\begin{gather*}
		\phi_\infty ( C \lIm \cball{a}{2^{-5}r} \rIm ) \leq \Gamma
		\big ( ( r/R )^\vdim + \psi ( U ) \big ).
	\end{gather*}
\end{lemma}
\begin{proof}
	Define $\Delta = 8 \sup \{ 1, \Gamma_{\ref{lemma:touching}} ( \vdim,
	2^\vdim M ) \}$ and
	\begin{gather*}
		\Gamma = 2^{6\vdim} \Delta^\vdim \sup \{
		\Gamma_{\ref{lemma:closed_set_tilt}} ( \vdim ),
		\mathscr{H}^\vdim ( \mathbf{O}^\ast ( \vdim + 1, 1 ) \}.
	\end{gather*}

	In order to prove that $\Gamma$ has the asserted property, suppose
	$\vdim$, $M$, $\adim$, $p$, $U$, $V$, $\psi$, $a$, $r$, $R$, $C$, and
	$\phi_\infty$ satisfy the hypotheses in the body of the lemma.

	Abbreviate $\gamma = r/R + \psi ( U )^{1/\vdim}$. If $\gamma \geq
	2^{-5} \Delta^{-1}$ then $\sup \{ r/R, \psi ( U )^{1/\vdim} \} \geq
	2^{-6} \Delta^{-1}$ and
	\begin{gather*}
		\phi_\infty ( C \lIm \cball{a}{2^{-5}r} \rIm ) \leq
		\mathscr{H}^\vdim ( \mathbf{O}^\ast ( \vdim+1, 1 ) ) \leq
		\Gamma \big ( (r/R)^\vdim + \psi ( U ) \big ).
	\end{gather*}
	Therefore one may assume $\gamma \leq 2^{-5} \Delta^{-1}$.

	Define $S = \spt \| V \|$ and let $D$ denote the set of all $(b,Q) \in
	S \times \mathbf{O}^\ast ( \adim,1 )$ such that the following two
	conditions hold:
	\begin{enumerate}
		\item \label{item:local_estimate_contact_set:1} $Q(z-b) \geq -
		\Delta \gamma r$ whenever $z \in S$.
		\item \label{item:local_estimate_contact_set:2} Whenever $v
		\in \mathbf{S}^\vdim \cap \ker Q$ there exists $\xi \in S$
		such that
		\begin{gather*}
			( \xi-b ) \bullet v \geq r/8 \quad \text{and} \quad Q
			( \xi-b ) \leq \Delta \gamma r.
		\end{gather*}
	\end{enumerate}
	Next, \emph{it will be shown $C | \cball{a}{r/2} \subset D$}. For this
	purpose suppose $(b,Q) \in C$ and $b \in \cball{a}{r/2}$ and choose $P
	\in \mathbf{O}^\ast ( \adim, \vdim )$ with $Q \circ P^\ast = 0$. By
	\ref{miniremark:touching_ball}, one estimates
	\begin{gather*}
		Q (z-b) \geq - ( R^2 - | P (z-b ) |^2 )^{-1/2} | P (z-b) |^2
		\geq - 8 R^{-1} r^2 \geq - \Delta \gamma r
	\end{gather*}
	whenever $z \in S$, hence $(b,Q)$ satisfies condition
	\eqref{item:local_estimate_contact_set:1}. In order to verify
	condition \eqref{item:local_estimate_contact_set:2}, let
	\begin{gather*}
		\zeta = (r/4) P(v) + P (b) \quad \text{and} \quad h = \inf Q
		\lIm \oball{b}{r/2} \cap \spt \| V \| \rIm.
	\end{gather*}
	Noting $| \zeta - P (b) | = r/4$ and, by the preceding estimate,
	\begin{gather*}
		\inf Q \lIm \oball{b}{2^{-4}r} \cap \spt \| V \| \rIm - h \leq
		8 R^{-1} r^2,
	\end{gather*}
	one infers from \ref{lemma:touching} with $M$, $a$, $r$, and $\gamma$
	replaced by $2^\vdim M$, $b$, $r/2$, and $2^4 \gamma$ the existence
	of $\xi \in \spt \| V \|$ such that
	\begin{gather*}
		|P(\xi)-\zeta| \leq r/8, \quad Q(\xi) \leq 8
		\Gamma_{\ref{lemma:touching}} ( \vdim , 2^\vdim M ) \gamma r +
		Q (b).
	\end{gather*}
	Since $P ( \xi-b ) - (r/4) P (v) = P ( \xi ) - \zeta$ and $|P (\xi-b)
	| \geq r/8$, noting
	\begin{gather*}
		2^{-6} r^2 \geq | P ( \xi-b ) - (r/4) P ( v ) |^2 = | P
		(\xi-b ) |^2 - (r/2) P (v) \bullet P ( \xi - b ) + 2^{-4} r^2,
		\\
		v = P^\ast ( P (v) ), \quad ( \xi - b ) \bullet v = P ( v )
		\bullet P ( \xi - b ) \geq r/8,
	\end{gather*}
	the condition \eqref{item:local_estimate_contact_set:2} is now
	evident. Therefore the conclusion follows by applying
	\ref{lemma:closed_set_tilt} with $\delta$, $\gamma$, and $C$ replaced
	by $2^{-5}$, $\Delta \gamma$, and $D$.
\end{proof}
\begin{lemma} \label{lemma:ode_comparison}
	Suppose $1 \leq \vdim < \infty$, $0 \leq s < r < \infty$, $0 \leq
	\beta < \infty$, $I = \{ t \with s < t \leq r \}$, and $f : I \to
	\rel$ is a nonnegative function satisfying
	\begin{gather*}
		\limsup_{u \to t-} f(u) \leq f(t) \leq f (r) + \beta
		\tint{t}{r} f(u)^{1-1/\vdim} \ud \mathscr{L}^1 u
	\end{gather*}
	whenever $t \in I$ where $0^0=1$.

	Then
	\begin{gather*}
		f(t)^{1/\vdim} \leq f(r)^{1/\vdim} + ( \beta/\vdim ) (r-t)
		\quad \text{for $t \in I$}.
	\end{gather*}
\end{lemma}
\begin{proof}
	The case $\vdim=1$ is trivial. Suppose $f(r) < \alpha < \infty$ and
	consider the set $J$ of all $t \in I$ such that
	\begin{gather*}
		f(u)^{1/\vdim} \leq \alpha^{1/\vdim} + ( \beta/\vdim ) (r-u)
		\quad \text{whenever $t \leq u \leq r$}.
	\end{gather*}
	Clearly, $J$ is an interval and $r$ belongs to the interior of $J$
	relative to $I$. The same holds for $t$ with $s < t \in \Clos J$ since
	\begin{gather*}
		\begin{aligned}
			f(t) & \leq f(r) + \beta \tint{t}{r} (
			\alpha^{1/\vdim} + ( \beta/\vdim ) (r-u) )^{\vdim-1}
			\ud \mathscr{L}^1 u \\
			& \leq f(r) - \alpha + ( \alpha^{1/\vdim} + (
			\beta/\vdim ) (r-t) )^\vdim < ( \alpha^{1/\vdim} + (
			\beta/\vdim ) (r-t) )^\vdim.
		\end{aligned}
	\end{gather*}
	Therefore $I$ equals $J$.
\end{proof}
\begin{lemma} \label{lemma:upper_mass_bound}
	Suppose $\vdim$, $\adim$, $p$, $U$, $V$, and $\psi$ are as in
	\ref{miniremark:situation_general}, $p = \vdim$, $a \in \rel^\adim$,
	$0 < r < \infty$, $U = \oball{a}{r}$, $0 < \lambda \leq 1$, $0 \leq s
	\leq r$, and
	\begin{gather*}
		\measureball{\psi}{\oball{a}{\lambda t}} \leq \lambda^\vdim
		\measureball{\psi}{\oball{a}{t}} \quad \text{whenever $s < t
		\leq r$}.
	\end{gather*}

	Then
	\begin{gather*}
		t^{-1} \| V \| ( \oball{a}{t} )^{1/\vdim} \leq r^{-1} \| V \|
		( U )^{1/\vdim} + ( \vdim \lambda r )^{-1} (r-t) \psi ( U
		)^{1/\vdim}
	\end{gather*}
	whenever $\lambda s < t \leq r$.
\end{lemma}
\begin{proof}
	Assume $\psi ( U ) < \infty$. Observe that
	\begin{gather*}
		t^{-\vdim} \measureball{\psi}{\oball{a}{t}} \leq
		\lambda^{-\vdim} r^{-\vdim} \psi ( U ) \quad \text{whenever
		$\lambda s < t \leq r$};
	\end{gather*}
	in fact this evidently true if $\lambda r \leq t$ and if it is true
	for some $t$ with $s < t$ then it is true with $t$ replaced by
	$\lambda t$ according to the hypothesis on $\psi$. Therefore defining
	$f(t) = t^{-\vdim} \measureball{\| V \|}{\oball{a}{t}}$ for $0 < t
	\leq r$, the monotonicity identity, see \cite[(17.3)]{MR756417},
	and H\"older's inequality imply
	\begin{gather*}
		\begin{aligned}
			f(t) & \leq f(r) +
			\tint{t}{r} u^{-1} \psi ( \oball{a}{u}
			)^{1/\vdim} f(u)^{1-1/\vdim} \ud \mathscr{L}^1 u \\
			& \leq f (r) + \lambda^{-1} r^{-1} \psi ( U )^{1/\vdim}
			\tint{t}{r} f(u)^{1-1/\vdim} \ud \mathscr{L}^1 u
		\end{aligned}
	\end{gather*}
	whenever $\lambda s < t \leq r$ and the conclusion follows by applying
	\ref{lemma:ode_comparison}.
\end{proof}
\begin{lemma} \label{lemma:choice_of_radius}
	Suppose $\vdim$, $\adim$, $p$, $U$, $V$, and $\psi$ are as in
	\ref{miniremark:situation_general}, $p = \vdim$, $a \in \rel^\adim$,
	$0 < r < \infty$, $U = \oball{a}{r}$, $0 < \lambda \leq 1$, $0 \leq M
	< \infty$, $\psi ( \{ a \} ) < \isoperimetric{1}^{-1}$, and
	\begin{gather*}
		\unitmeasure{\vdim}^{1/\vdim} \density^\vdim_\ast ( \| V \|, a
		)^{1/\vdim} + ( \vdim \lambda)^{-1} \psi ( \{a\} )^{1/\vdim} <
		M^{1/\vdim}.
	\end{gather*}
	
	Then there exists $0 < s \leq r$ such that
	\begin{gather*}
		s^{-\vdim} \measureball{\| V \|}{\oball{a}{s}} \leq M, \quad
		s^\vdim + \measureball{\psi}{\oball{a}{s}} \leq 2
		\lambda^{-\vdim} \big ( ( \lambda s )^\vdim +
		\measureball{\psi}{\oball{a}{\lambda s}} \big ).
	\end{gather*}
\end{lemma}
\begin{proof}
	Assume $a \in \spt \| V \|$ and choose $0 < t \leq r$ such that
	\begin{gather*}
		t^{-1} \| V \| ( \oball{a}{t} )^{1/\vdim} + ( \vdim \lambda
		)^{-1} \psi ( \oball{a}{t} )^{1/\vdim} \leq M^{1/\vdim}.
	\end{gather*}
	Define $I$ to be the set of all $0 < u \leq t$ such that
	\begin{gather*}
		\measureball{\psi}{\oball{a}{\lambda u}} > \lambda^\vdim
		\measureball{\psi}{\oball{a}{u}}.
	\end{gather*}

	First, consider \emph{the case $I = \emptyset$}. Then by
	\ref{lemma:upper_mass_bound} with $s$ and $r$ replaced by $0$ and $t$
	one obtains
	\begin{gather*}
		u^{-\vdim} \measureball{\| V \|}{\oball{a}{u}} \leq M \quad
		\text{whenever $0 < u \leq t$}.
	\end{gather*}
	Applying \ref{lemma:radius_from_calculus} with $r$, $I$, $j$, $f(s)$
	and $f_1 (s)$ replaced by $t$, $\{ u \with 0 < u < t \}$, $1$,
	$s^\vdim$, and $s^\vdim + \measureball{\psi}{\oball{a}{s}}$ yields the
	conclusion in the present case.

	Second, consider \emph{the case $I \neq \emptyset$}. Taking $s = \sup
	I$, one notes
	\begin{gather*}
		\measureball{\psi}{\oball{a}{\lambda s}} \geq \lambda^\vdim
		\measureball{\psi}{\oball{a}{s}}, \quad s^{-\vdim}
		\measureball{\| V \|}{\oball{a}{s}} \leq M.
	\end{gather*}
	by \ref{lemma:upper_mass_bound} with $r$ replaced by $t$ and the
	conclusion is evident.
\end{proof}
\begin{theorem} \label{lemma:zero_sets_to_zero_sets}
	Suppose $\vdim$, $\adim$, $p$, $U$, and $V$ are as in
	\ref{miniremark:situation_general}, $p = \vdim = \adim-1$, $\| \delta
	V \|$ is absolutely continuous with respect to $\| V \|$, and $C$ is
	the relation consisting of all $(b,Q) \in \spt \| V \| \times
	\mathbf{O}^\ast ( \adim, 1 )$ such that for some $0 < R < \infty$
	\begin{gather*}
		\oball{b - Q^\ast (R)}{R} \cap \spt \| V \| = \emptyset.
	\end{gather*}

	Then
	\begin{gather*}
		\mathscr{H}^\vdim \big ( C \lIm
		\classification{A}{z}{\density_\ast^\vdim ( \| V \|, z ) <
		\infty} \rIm ) = 0 \quad \text{whenever $A \subset U$ and $\|
		V \| (A) = 0$}.
	\end{gather*}
\end{theorem}
\begin{proof}
	Suppose $A \subset U$ with $\| V \| ( A ) = 0$ and whenever $i \in
	\nat$ define $C_i$ to be relation of all $(a,Q)$ satisfying
	\begin{gather*}
		a \in A \cap \spt \| V \|, \quad \unitmeasure{\vdim}
		\density_\ast^\vdim ( \| V \|, a ) < i, \quad Q \in
		\mathbf{O}^\ast ( \adim, 1 ), \\
		\oball{a - Q^\ast (1/i)}{1/i} \cap \spt \| V \| = \emptyset.
	\end{gather*}
	Note $C \lIm \classification{A}{z}{ \density_\ast^\vdim ( \| V \|, z )
	< \infty } \rIm \subset \bigcup_{i=1}^\infty \im C_i$.

	In order to show $\mathscr{H}^\vdim ( \im C_i ) = 0$ for $i \in \nat$,
	fix such $i$ and define $\lambda = 2^{-5} \cdot 5^{-1}$ and $\Delta =
	2 \Gamma_{\ref{lemma:local_estimate_contact_set}} ( \vdim, i ) i^\vdim
	\lambda^{-\vdim} \sup \{ 1, 2\unitmeasure{\vdim}^{-1} \}$ and suppose
	$\varepsilon > 0$. Choose an open subset $Z$ of $U$ with $A \subset Z$
	and $( \| V \| + \psi ) (Z) \leq \varepsilon$ and, denoting by
	$\phi_\infty$ the size $\infty$ approximating measure occuring in the
	construction of $\mathscr{H}^\vdim$ on $\mathbf{O}^\ast ( \adim, 1 )$,
	define $F$ to the family of closed balls $\cball{a}{r}$ such that
	\begin{gather*}
		r \leq 1, \quad \cball{a}{r/\lambda} \subset Z, \quad
		\phi_\infty ( C_i \lIm \cball{a}{5r} \rIm ) \leq \Delta
		\measureball{( \| V \| + \psi )}{\cball{a}{r}}.
	\end{gather*}

	Next, \emph{it will be shown that $\dmn C_i \subset \bigcup F$}. For
	this purpose suppose $a \in \dmn C_i$. Noting $\density_\ast^\vdim (
	\| V \|, a ) \geq 1$ by \cite[2.7]{snulmenn.isoperimetric} and
	applying \ref{lemma:choice_of_radius} with $M=i$ and $r$ replaced by a
	suitably small number, one obtains $r$ such that
	\begin{gather*}
		0 < r \leq 1, \quad \cball{a}{r/\lambda} \subset Z, \quad
		\unitmeasure{\vdim}/2 \leq r^{-\vdim} \measureball{\| V
		\|}{\cball{a}{r}}, \\
		(r/\lambda)^{-\vdim} \measureball{\| V
		\|}{\oball{a}{r/\lambda}} \leq i, \quad (r/\lambda)^\vdim +
		\measureball{\psi}{\oball{a}{r/\lambda}} \leq 2
		\lambda^{-\vdim} \big ( r^\vdim +
		\measureball{\psi}{\oball{a}{r}} \big ).
	\end{gather*}
	Then \ref{lemma:local_estimate_contact_set} with $r$, $M$, and $R$
	replaced by $r/\lambda$, $i$, and $1/i$ yields
	\begin{gather*}
		\begin{aligned}
			& \phi_\infty ( C_i \lIm \cball{a}{5r} \rIm ) \leq
			\Gamma_{\ref{lemma:local_estimate_contact_set}} (
			\vdim, i ) i^\vdim \big ( ( r/\lambda )^\vdim +
			\measureball{\psi}{\oball{a}{r/\lambda}} \big ) \\
			& \qquad \leq 2
			\Gamma_{\ref{lemma:local_estimate_contact_set}} (
			\vdim, i ) i^\vdim \lambda^{-\vdim} ( r^\vdim +
			\measureball{\psi}{\oball{a}{r}} ) \leq \Delta
			\measureball{(\| V \| + \psi)}{\cball{a}{r}}.
		\end{aligned}
	\end{gather*}

	From \cite[2.8.5]{MR41:1976} one obtains a countable disjointed
	subfamily $G$ of $F$ with $\bigcup F \subset \bigcup \{ \widehat{S}
	\with S \in G \}$, here $\widehat{S} = \cball{a}{5r}$ if $S =
	\cball{a}{r}$, hence
	\begin{gather*}
		\begin{aligned}
			\phi_\infty ( \im C_i ) & \leq \tsum{S \in G}{}
			\phi_\infty \big ( C_i \big \lIm \widehat{S} \big \rIm
			\big ) \\
			& \leq \Delta \tsum{S \in G}{} ( \| V \| + \psi ) ( S
			) \leq \Delta ( \| V \| + \psi ) (Z) \leq \Delta
			\varepsilon
		\end{aligned}
	\end{gather*}
	and the conclusion follows.
\end{proof}
\section{Points of infinite density} \label{sec:infinite_density}
In this section barycentre estimates based on the monotonicity identity, see
\ref{thm:barycenter_estimate}, are used in order to obtain the necessary Lusin
condition for points of infinite density, see \ref{thm:phantom_estimate}.
Notice that it is unknown whether points of infinite density do exist under
the present hypotheses, see \ref{remark:nonempty}. The afore-mentioned result
\ref{thm:phantom_estimate} at least yields that the set of these points is
mapped to a set of measure zero under the generalised Gauss map.
\begin{lemma} \label{lemma:almost_barycenter_estimate}
	Suppose \ref{miniremark:situation}\,\eqref{item:situation:basic}
	holds, $0 < s < r$, $0 \leq \beta < \infty$,
	\begin{gather*}
		\measureball{\| \delta V \|}{\cball{a}{t}} \leq \beta r^{-1}
		\mu (r)^{1/\vdim} \mu (t)^{1-1/\vdim} \quad \text{for $s \leq
		t \leq r$},
	\end{gather*}
	and $g : \{ \lambda : s/r \leq \lambda \leq 1 \} \to \rel$ is a
	nonnegative, continuous function.

	Then
	\begin{gather*}
		\big | \tint{( \cball{a}{r} \without \cball{a}{s} ) \times
		\grass{\adim}{\vdim}}{} g ( |z-a|/r ) |z-a|^{-\vdim-1}
		\project{S} (z-a) \ud V (z,S) \big | \\
		\leq \big ( \tint{s/r}{1} g ( \lambda ) \lambda^{-1} ( 1 +
		\vdim^{-1} \beta \log (1/\lambda) )^{\vdim-1} \ud
		\mathscr{L}^1 \lambda \big ) r^{-\vdim} \mu (r) \beta.
	\end{gather*}
\end{lemma}
\begin{proof}
	Assume $a=0$ and let $\varepsilon
	> 0$ with $\cball{a}{r+\varepsilon} \subset U$. For
	$\theta \in \mathscr{E}^0 ( \rel )$ with $\spt \theta \subset \{ s
	\with s < r+\varepsilon \}$ define $g_p : U \to \rel^\adim$ by $g_p(z)
	= \theta (|z|) p^\ast (1)$ for $z \in U$, $p \in \orthproj{\adim}{1}$
	and compute
	\begin{gather*}
		\project{S} \bullet Dg_p (z) = \theta' (|z|) |z|^{-1}
		\project{S} ( p^\ast (1) ) \bullet z = \theta' (|z|) |z|^{-1}
		p ( \project{S} (z))
	\end{gather*}
	for $z \in U \without \{0 \}$, $S \in \grass{\adim}{\vdim}$. If $0
	\notin \spt \theta'$ then $g_p \in \mathscr{D} ( U, \rel^\adim )$ and
	computing $\delta V (g_p)$ for $p \in \orthproj{\adim}{1}$ yields
	\begin{gather*}
		\tint{( U \without \{0 \}) \times \grass{\adim}{\vdim} }{}
		\theta' (|z|) |z|^{-1} \project{S} (z) \ud V (z,S) = \tint{}{}
		\theta ( |z| ) \eta (V;z) \ud \| \delta V \| z \in \rel^\adim.
	\end{gather*}

	Define $f : \rel \to \rel$ by
	\begin{gather*}
		f (t) = - \tint{\sup\{s,t\}}{r} g(u/r) u^{-\vdim} \ud
		\mathscr{L}^1 u \quad \text{if $t < r$}, \qquad
		f(t)=0 \quad \text{else}
	\end{gather*}
	for $t \in \rel$. Approximating $f$ by functions $\theta$ as above,
	one infers
	\begin{gather*}
		\tint{( \cball{0}{r} \without \cball{0}{s} ) \times
		\grass{\adim}{\vdim}}{} g ( |z|/r ) |z|^{-\vdim-1}
		 \project{S} (z) \ud V (z,S) \\
		= \tint{\cball{0}{r}}{} f ( |z| ) \eta ( V;z ) \ud \| \delta V
		\| z.
	\end{gather*}
	The modulus of this vector does not exceed, using Fubini's theorem and
	\ref{corollary:density_ratio_estimate},
	\begin{gather*}
		\begin{aligned}
			& - \tint{\cball{0}{r}}{} f (|z|) \ud \| \delta V \| z
			= \tint{s}{r} g(t/r) t^{-\vdim} \measureball{\| \delta
			V \|}{\cball{0}{t}} \ud \mathscr{L}^1 t \\
			& \qquad \leq \beta r^{-1} \mu (r)^{1/\vdim}
			\tint{s}{r} g (t/r) t^{-\vdim} \mu (t)^{1-1/\vdim} \ud
			\mathscr{L}^1 t \\
			& \qquad \leq \beta r^{-\vdim} \mu (r) \tint{s}{r} g
			(t/r) t^{-1} ( 1 + \vdim^{-1} \beta \log (r/t)
			)^{\vdim-1} \ud \mathscr{L}^1 t
		\end{aligned}
	\end{gather*}
	and the conclusion follows.
\end{proof}
\begin{lemma} \label{lemma:auxiliary_estimate}
	Suppose
	\ref{miniremark:situation}\,\eqref{item:situation:basic}\,\eqref{item:situation:l2}
	hold, $0 < s < r$, $0 \leq \gamma < \infty$ and
	\begin{gather*}
		\phi ( t ) \leq \gamma^2 r^{-2} \mu (r)^{2/\vdim} \mu
		(t)^{1-2/\vdim} \quad \text{for $s \leq t \leq r$}.
	\end{gather*}

	Then
	\begin{gather*}
		\tint{s}{r} t^{2-\vdim} \ud_t \phi (t) \leq \big ( 1 +
		(\vdim-2) ( 1 + \vdim^{-1} \gamma \log (r/s))^{\vdim-2} \log
		(r/s) \big ) r^{-\vdim} \mu (r) \gamma^2.
	\end{gather*}
\end{lemma}
\begin{proof}
	By H\"older's inequality and \ref{corollary:density_ratio_estimate}
	\begin{gather*}
		t^{-\vdim} \mu (t) \leq \big ( 1 + \vdim^{-1} \gamma \log
		(r/t) \big )^{\vdim} r^{-\vdim} \mu (r) \quad \text{for $s
		\leq t \leq r$}.
	\end{gather*}
	By partial integration one obtains
	\begin{gather*}
		\tint{s}{r} t^{2-\vdim} \ud_t \phi (t) = r^{2-\vdim} \phi (r)
		- s^{2-\vdim} \phi (s) + (\vdim-2) \tint{s}{r} t^{1-\vdim}
		\phi (t) \ud \mathscr{L}^1 t.
	\end{gather*}
	Finally, one estimates
	\begin{gather*}
		\begin{aligned}
			\tint{s}{r} t^{1-\vdim} \phi ( t ) \ud \mathscr{L}^1 t
			& \leq \gamma^2 r^{-2} \mu(r)^{2/\vdim} \tint{s}{r}
			t^{1-\vdim} \mu (t)^{1-2/\vdim} \ud \mathscr{L}^1 t \\
			& \leq c \tint{s}{r} t^{-1} ( 1 + \vdim^{-1} \gamma
			\log (r/t))^{\vdim-2} \ud \mathscr{L}^1 t \\
			& \leq c ( 1 + \vdim^{-1} \gamma \log (r/s)
			)^{\vdim-2} \log (r/s)
		\end{aligned}
	\end{gather*}
	with $c = r^{-\vdim} \mu (r) \gamma^2$.
\end{proof}
\begin{lemma} \label{lemma:ks_monotonicity}
	Suppose \ref{miniremark:situation} holds.

	Then for $0 < \delta \leq 1$
	\begin{gather*}
		s^{-\vdim} \zeta (s) + (1-\delta) \tint{s}{t}
		u^{-\vdim-2} \ud_u \xi (u) \leq t^{-\vdim} \zeta (t) + ( 4
		\delta \vdim^2 )^{-1} \tint{s}{t} u^{2-\vdim} \ud_u \phi (u)
	\end{gather*}
	whenever $0 < s \leq t \leq r$.
\end{lemma}
\begin{proof}
	Suppose $0 < s \leq t \leq r$. By partial integration
	\begin{gather*}
		- \tint{s}{t} u^{-\vdim-1} \nu (u) \ud \mathscr{L}^1 u =
		  \vdim^{-1} \big ( t^{-\vdim} \nu (t) - s^{-\vdim} \nu (s) -
		  \tint{s}{t} u^{-\vdim} \ud_u \nu (u) \big )
	\end{gather*}
	Moreover,
	\begin{gather*}
		\begin{aligned}
			& \vdim^{-1} \big | \tint{\cball{a}{t} \without
			\cball{a}{s}}{} | z-a |^{-\vdim} (z-a) \bullet
			\mathbf{h} (V;z) \ud \| V \| z \big | \\
			& \qquad \leq \delta \tint{(\cball{a}{t} \without
			\cball{a}{s}) \times \grass{\adim}{\vdim}}{}
			|z-a|^{-\vdim-2} | \perpproject{S} (z-a) |^2 \ud V
			(z,S) \\
			& \qquad \phantom{\leq}~ + ( 4 \delta \vdim^2 )^{-1}
			\tint{\cball{a}{t} \without \cball{a}{s}}{}
			|z-a|^{2-\vdim} | \mathbf{h} ( V;z ) |^2 \ud \| V \|
			z.
		\end{aligned}
	\end{gather*}
	Therefore the conclusion follows from
	\ref{thm:monotonicity_identity}.
\end{proof}
\begin{remark}
	The argument is a variant of Kuwert and Sch\"atzle \cite[Appendix
	A]{MR2119722}.
\end{remark}
\begin{lemma} \label{lemma:cone_estimate_integral}
	Suppose \ref{miniremark:situation} holds, $0 < s < r$, $0 \leq \gamma
	< \infty$,
	\begin{gather*}
		\phi (t) \leq \gamma^2 r^{-2} \mu(r)^{2/\vdim}
		\mu(t)^{1-2/\vdim}, \quad t^{-\vdim} \zeta (t) \geq r^{-\vdim}
		\zeta (r)
	\end{gather*}
	whenever $s \leq t \leq r$, here $0^0=1$, $f : \{ \lambda : s/r <
	\lambda \leq 1 \} \to \rel$ is defined by
	\begin{gather*}
		f ( \lambda ) = \vdim^{-2} \big ( 1 + ( \vdim-2) ( 1 +
		\vdim^{-1} \gamma \log ( 1/\lambda ) )^{\vdim-2} \log ( 1/
		\lambda ) \big ) \quad \text{for $s/r < \lambda \leq 1$},
	\end{gather*}
	and $g : \{ \lambda \with s/r \leq \lambda \leq 1 \} \to \rel$ is a
	nonnegative, nondecreasing, continuous function.

	Then
	\begin{gather*}
		\tint{s}{r} g(u/r) u^{-\vdim-2} \ud_u \xi (u) \leq \big (
		g(s/r)f(s/r) + \tint{s/r}{1} f \ud g \big ) r^{-\vdim} \mu (r)
		\gamma^2.
	\end{gather*}
\end{lemma}
\begin{proof}
	From \ref{lemma:ks_monotonicity} with $\delta = 1/2$ one infers
	\begin{gather*}
		\tint{t}{r} u^{-\vdim-2} \ud_u \xi (u) \leq \vdim^{-2}
		\tint{t}{r} u^{2-\vdim} \ud_u \phi (u) \quad \text{for $s \leq
		t \leq r$}.
	\end{gather*}
	Choose Radon measures $\psi_1$ and $\psi_2$ on $\rel$ such that
	\begin{align*}
		\psi_1 ( \{ u \with u \leq t \} ) & = \tint{s}{\sup\{s,t\}}
		u^{-\vdim-2} \ud_u \xi (u) \quad \text{for $-\infty < t \leq
		r$}, \\
		\psi_2 ( \{ v \with v < t \} ) & = g (t/r) \quad \text{for
		$s \leq t \leq r$}.
	\end{align*}
	One computes, using Fubini's theorem and
	\ref{lemma:auxiliary_estimate},
	\begin{gather*}
		\begin{aligned}
			& \tint{s}{r} g(u/r) u^{-\vdim-2} \ud_u \xi (u) =
			\tint{\{u:s<u\leq r\}}{} g (u/r) \ud \psi_1 u \\
			& \qquad = \tint{\{ u : s < u \leq r\}}{} \psi_2 ( \{ v
			: v < u \} ) \ud \psi_1 u \\
			& \qquad = \tint{\{v : v <r \}}{} \psi_1 ( \{ u
			: \sup \{s,v\} < u \leq r \} ) \ud \psi_2 v \\
			& \qquad = \tint{\{ v : v < r \}}{} \tint{\sup
			\{s,v\}}{r} u^{-\vdim-2} \ud_u \xi (u) \ud \psi_2 v \\
			& \qquad \leq c \tint{\{v: v < r \}}{} f ( \sup
			\{ s, v \} /r ) \ud \psi_2 v = c \big ( g (s/r) f(s/r)
			+ \tint{s/r}{1} f \ud g \big ).
		\end{aligned}
	\end{gather*}
	where $c = r^{-\vdim} \mu (r) \gamma^2$.
\end{proof}
\begin{theorem} \label{thm:barycenter_estimate}
	Suppose \ref{miniremark:situation} holds, $0 \leq \gamma \leq \kappa <
	\infty$,
	\begin{gather*}
		\phi (s) \leq \gamma^2 r^{-2} \mu (r)^{2/\vdim} \mu
		(s)^{1-2/\vdim}, \quad s^{-\vdim} \zeta(s) \geq r^{-\vdim}
		\zeta (r)
	\end{gather*}
	whenever $0 < s \leq r$, here $0^0=1$, and $0 < \delta < \infty$.

	Then
	\begin{gather*}
		\big | \tint{\cball{a}{r}}{} |z-a|^{-\vdim-1+\delta} (z-a) \ud
		\| V \| z \big | \leq \Gamma r^{-\vdim+\delta} \mu (r) \gamma
	\end{gather*}
	where $\Gamma$ is a positive, finite number depending only on $\vdim$,
	$\delta$ and $\kappa$.
\end{theorem}
\begin{proof}
	Note $\| V \| ( \{ a \} ) = 0$ by \ref{corollary:density_ratio_estimate}.
	Noting additionally \ref{remark:basic_integrability}, one defines
	\begin{align*}
		b_1 & = \tint{\cball{a}{r} \times \grass{\adim}{\vdim}}{}
		|z-a|^{-\vdim-1+\delta} \project{S} (z-a) \ud V (z,S) \in
		\rel^\adim, \\
		b_2 & = \tint{\cball{a}{r} \times \grass{\adim}{\vdim}}{}
		|z-a|^{-\vdim-1+\delta} \perpproject{S} (z-a) \ud V (z,S) \in
		\rel^\adim.
	\end{align*}
	By \ref{lemma:almost_barycenter_estimate}, one infers
	\begin{gather*}
		| b_1 | \leq \Delta_1 r^{-\vdim+\delta} \mu (r) \gamma
	\end{gather*}
	where $\Delta_1 = \tint{0}{1} \lambda^{-1+\delta} ( 1 + \vdim^{-1}
	\kappa \log (1/\lambda) )^{\vdim-1} \ud \mathscr{L}^1 \lambda <
	\infty$. By H\"older's inequality, \ref{remark:basic_integrability}
	and \ref{lemma:cone_estimate_integral} one obtains
	\begin{gather*}
		\begin{aligned}
			| b_2 | & \leq \big ( \tint{\cball{a}{r}}{}
			|z-a|^{-\vdim+\delta} \ud \| V \| z \big )^{1/2} \\
			& \phantom{\leq}~ \cdot \big ( \tint{\cball{a}{r}
			\times \grass{\adim}{\vdim}}{} | z-a
			|^{-\vdim-2+\delta} | \perpproject{S} (z-a) |^2 \ud V
			(z,S) \big )^{1/2} \\
			& \leq ( \Delta_2 \Delta_3 )^{1/2} r^{-\vdim+\delta}
			\mu (r) \gamma
		\end{aligned}
	\end{gather*}
	where $\Delta_2 = 1 + \sup \{ 0, \vdim-\delta \} \tint{0}{1}
	t^{-1+\delta} (1 + \vdim^{-1} \kappa \log ( 1/ t ) )^\vdim \ud
	\mathscr{L}^1 t < \infty$ and $\Delta_3 = \delta \tint{0}{1}
	\lambda^{-1+\delta} \vdim^{-2} \big ( 1 + ( \vdim-2) ( 1 + \vdim^{-1}
	\kappa \log ( 1/ \lambda))^{\vdim-2} \log ( 1 / \lambda ) \big ) \ud
	\mathscr{L}^1 \lambda < \infty$. Hence one may take $\Gamma =
	\Delta_1 + (\Delta_2\Delta_3)^{1/2}$.
\end{proof}
\begin{lemma} \label{lemma:abstract_oscillation_estimate}
	Suppose $0 < \delta \leq 1$, $0 < \varepsilon \leq 1$, $1 < \adim \in
	\nat$, $a \in \rel^\adim$, $0 < s \leq t \leq r < \infty$, $s + (
	\delta + \varepsilon ) r \leq t$, $\mu$ is a Radon measure on
	$\cball{a}{r}$, $q,Q \in \orthproj{\adim}{1}$ with $|q-Q|<1$, $b \in
	\cball{a}{s} \cap \spt \mu$, and
	\begin{gather*}
		\inf \{ q (z-a), Q (z-b) \} \geq -\varepsilon r \quad
		\text{for $\mu$ almost all $z$}, \\
		\mu ( \classification{\cball{a}{r}}{z}{(z-a) \bullet v \geq
		t} ) \geq \delta \measureball{\mu}{\cball{a}{r}} \quad
		\text{for $v \in \mathbf{S}^{\adim-1} \cap \ker q$}.
	\end{gather*}

	Then
	\begin{gather*}
		|q-Q| \leq 8 \delta^{-2} \big ( r^{-1} \tfint{}{}q (z-a)
		\ud \mu z + \varepsilon \big ).
	\end{gather*}
\end{lemma}
\begin{proof}
	Rescaling, one may assume $a = 0$, $r = 1$ and
	$\measureball{\mu}{\cball{0}{1}} = 1$.

	Choose $p \in \orthproj{\adim}{{\adim-1}}$ with $p \circ q^\ast = 0$.
	Since $|q-Q| < 1$, one notes $Q ( q^\ast (y) ) = y + (Q-q)(q^\ast (y))
	\to -\infty$ as $y \to - \infty$.  Hence there exists an affine
	function $g : \rel^{\adim-1} \to \rel$ such that
	\begin{gather*}
		\{ p^\ast ( x ) + q^\ast (y) \with \text{$(x,y) \in
		\rel^{\adim-1} \times \rel$ and $y \geq g (x)$} \} =
		\classification{\rel^\adim}{z}{Q(z-b) \geq 0}.
	\end{gather*}

	Define $h = 2 \delta^{-1} \big ( \int q (z) \ud \mu z + \varepsilon
	\big) \geq 0$ and assume $Dg(0) \neq 0$. This implies
	\begin{gather*}
		\mu ( \classification{\cball{0}{1}}{z}{q(z) > h} ) \leq
		\delta /2, \\
		\mu ( \classification{\cball{0}{1}}{z}{\text{$q(z) \leq h$
		and $z \bullet v \geq t$}}) > 0 \quad \text{for $v \in
		\mathbf{S}^{\adim-1} \cap \ker q$}.
	\end{gather*}
	Letting $u = | \grad g(0) |^{-1} \grad g(0)$ and $v = p^\ast (u)$,
	there exists $c \in \cball{0}{1}$ with
	\begin{gather*}
		c \in \spt \mu, \quad q (c) \leq h, \quad c \bullet v \geq
		t.
	\end{gather*}
	Since $Q(c-b) \geq - \varepsilon$, there exists $d \in
	\cball{c}{\varepsilon}$ with
	\begin{gather*}
		Q (d-b) \geq 0, \quad q (d) \leq h + \varepsilon,
		\quad d \bullet v \geq t-\varepsilon.
	\end{gather*}

	Defining $\kappa = p (b) \bullet u$ and $\lambda = p (d) \bullet
	u$, one notes
	\begin{gather*}
		\kappa \leq s, \quad t-\varepsilon \leq \lambda,
		\quad \lambda-\kappa \geq \delta, \\
		g ( \lambda u ) = g ( p (d) ) \leq q (d) \leq h + \varepsilon,
		\quad g ( \kappa u ) = g ( p (b)) = q ( b) \geq - \varepsilon
	\end{gather*}
	as $b \in \spt \mu$. Since $(\lambda-\kappa) | \grad g (0) | = g (
	\lambda u ) - g ( \kappa u )$, this implies
	\begin{gather*}
		| \grad g (0) | \leq \delta^{-1} ( h + 2\varepsilon )
	\end{gather*}
	from which the conclusion can be inferred for example by use of the
	inequalities noted in Allard \cite[8.9\,(5)]{MR0307015} and Allard
	\cite[3.2\,(9)]{MR840267}.
\end{proof}
\begin{lemma} \label{lemma:diameter_estimate}
	Suppose \ref{miniremark:situation} holds, $\vdim = \adim-1$, $0 < R <
	\infty$,
	\begin{gather*}
		D = \eqclassification{\spt \| V \| \times
		\orthproj{\adimI}{1}}{(b,Q)}{ \text{$|z-b+Q^\ast (R)| \geq R$
		for $z \in \spt \| V \|$}},
	\end{gather*}
	$a \in \dmn D$, $\density^\vdim ( \| V \|, a ) = \infty$, $0 < \lambda
	\leq 1$, and $\gamma : I \to \overline{\rel}$ defined by
	\begin{gather*}
		\gamma (s) = \sup \{ t \mu (t)^{-1/\vdim} \phi (u)^{1/2} \mu
		(u)^{1/\vdim-1/2} \with 0 < u \leq t \leq s\} \quad \text{for
		$s \in I$}
	\end{gather*}
	satisfies $\lim_{s \to 0+} \gamma (s) = 0$.

	Then $\card D \lIm \{a\} \rIm \leq 2$ and there exists $0 < s \leq r$
	such that $D \lIm \cball{a}{s} \rIm$ can be covered by at most two
	sets whose diameters do not exceed
	\begin{gather*}
		\Gamma ( \gamma ( \lambda s ) + s / R)
	\end{gather*}
	where $\Gamma$ is a positive, finite number depending only on $\vdim$.
\end{lemma}
\begin{proof}
	Define $\Delta_1 = ( 1 + \vdim^{-1} \log ( 3/ \lambda ))^{\vdim-1}
	\log ( 3/\lambda )$ and choose $2 \leq \Delta_2 < \infty$ such that
	\begin{gather*}
		(1+u)(1+\Delta_1u)(1-u)^{-1} \leq 1 + \Delta_2 u/2 \quad
		\text{for $0 \leq u \leq 1/2$}.
	\end{gather*}

	First, it will be shown that there exists a sequence $s_i \in \rel$
	with $s_{i+1} < s_i \leq \inf \big \{ \frac{R}{66}, r/3 \big \}$ for
	$i \in \nat$ and $s_i \to 0$ as $i \to \infty$ such that
	\begin{gather*}
		\Delta_2 \gamma (3s_i) \leq 1, \quad (1 + \Delta_2 \gamma (t)
		) t^{-\vdim} \zeta (t) \geq ( 1 + \Delta_2 \gamma (s_i) )
		s_i^{-\vdim} \zeta (s_i)
	\end{gather*}
	whenever $0 < t \leq s_i$ and $i \in \nat$; in fact one verifies that
	$\lim_{t \to s+} \zeta (t) = \zeta (s)$ and $\lim_{t \to s+} \gamma
	(t) = \gamma (s)$ for $0 < s < r$ since the same holds for $\mu$,
	$\nu$ and $\phi$, hence the existence of the sequence is a consequence
	of $\lim_{s \to 0+} \gamma (s) =0$ and $\lim_{s \to 0+} s^{-\vdim}
	\zeta (s) = \infty$ as $| \nu | \leq \gamma \mu$.

	Next, it will be shown that whenever $T = \ker q$ for some $(a,q) \in
	D$
	\begin{gather*}
		\mu (s_i)^{-1} \tint{\oball{a}{s_i}}{} f (s_i^{-1} (z-a), S )
		\ud V (z,S) \to \unitmeasure{\vdim}^{-1} \tint{\oball{0}{1}
		\cap T}{} f (z,T) \ud \mathscr{H}^\vdim z
	\end{gather*}
	as $i \to \infty$ for $f \in \ccspace{\oball{0}{1} \times
	\grass{\adimI}{\vdim}}$; in fact every $C \in \Var_\vdim (
	\oball{0}{1} )$ occuring as limit of a subsequence of the rescaled
	varifolds satisfies
	\begin{gather*}
		\measureball{\| C \|}{\oball{0}{1}} \leq 1, \quad \delta C = 0
	\end{gather*}
	as $\lim_{s \to 0+} \gamma (s)= 0$ and, since, recalling $|\nu| \leq
	\gamma \mu$ and noting \ref{miniremark:touching_ball},
	\begin{gather*}
		(1 + \Delta_2 \gamma (t) ) ( 1 + \gamma (t) ) t^{-\vdim} \mu (t)
		\geq ( 1 + \Delta_2 \gamma (s_i) ) (1 - \gamma(s_i))
		s_i^{-\vdim} \mu (s_i), \\
		q(z-a) \geq -2 s_i^2/R
	\end{gather*}
	whenever $0 < t \leq s_i \leq R/2$ and $z \in \cball{a}{s_i} \cap \spt
	\| V \|$,
	\begin{gather*}
		\kappa^{-\vdim} \measureball{\| C \|}{\cball{0}{\kappa}} \geq
		1 \quad \text{for $0 < \kappa < 1$}, \quad \spt \| C \|
		\subset \classification{\rel^\adimI}{z}{q(z) \geq 0}
	\end{gather*}
	hence $\measureball{\| C \|}{\oball{0}{1}} \leq 1 \leq
	\density^\vdim ( \| C \|, 0 ) \unitmeasure{\vdim}$ and $\spt \| C \|
	\subset \ker q$ by Allard \cite[5.2\,(e)]{MR0307015} and $C =
	\unitmeasure{\vdim}^{-1} \var(T \cap \oball{0}{1})$ by Allard
	\cite[4.6\,(3)]{MR0307015}. Therefore $T$ is unique, in particular
	$\card D \lIm \{ a \} \rIm \leq 2$.

	With $\delta = \inf \{ (2 \unitmeasure{\vdim})^{-1} \mathscr{L}^\vdim (
	\classification{\oball{0}{1}}{(x_1,\ldots,x_\vdim)}{x_\vdim \geq 2/3})
	, 1/6 \}$ the assertions of the preceding paragraph and the fact that
	$D$ is closed relative to $U \times \orthproj{\adimI}{1}$ imply the
	existence of $0 < s \leq r$ with $s=s_i/3$ for some $i \in \nat$ such
	that
	\begin{gather*}
		\| V \| ( \classification{\oball{a}{3s}}{z}{ (z-a) \bullet v
		\geq 2 s} ) \geq \delta \measureball{\| V \|}{\cball{a}{3s}}
		\quad \text{for $v \in \mathbf{S}^{\vdim} \cap T$}, \\
		D \lIm \cball{a}{s} \rIm \subset
		\classification{\orthproj{\adimI}{1}}{Q}{\dist (Q, D \lIm \{ a
		\} \rIm )<1}.
	\end{gather*}
	Therefore, noting
	\begin{gather*}
		Q (z-b) \geq -32 s^2/R \quad \text{for $\| V
		\|$ almost all $z \in \cball{a}{3s}$}
	\end{gather*}
	whenever $(b,Q) \in D$ and $b \in \cball{a}{s}$ by
	\ref{miniremark:touching_ball}, one applies
	\ref{lemma:abstract_oscillation_estimate} for each $(a,q) \in D$ with
	$\varepsilon$, $t$, $r$, and $\mu$ replaced by $11s/R$, $2s$, $3s$,
	and $\| V \|$ to infer that $D \lIm \cball{a}{s} \rIm$ can be covered
	by at most two sets whose diameters do not exceed
	\begin{gather*}
		16 \delta^{-2} \Big ( (3s)^{-1} \big | \tfint{\cball{a}{3s}}{}
		z-a \ud \| V \| z \big | + 11s /R \Big ).
	\end{gather*}
	Since $\gamma$ is nondecreasing, \ref{thm:barycenter_estimate} yields
	\begin{gather*}
		(3s)^{-1} \big | \tfint{\cball{a}{3s}}{} z-a \ud \| V \| z
		\big | \leq \Gamma_{\ref{thm:barycenter_estimate}} ( \vdim,
		\vdim+1, 1 ) \gamma (3s).
	\end{gather*}
	Finally, defining $t=\lambda s$ and noting
	\begin{gather*}
		t^{-\vdim} \mu (t) \leq ( 1 + \Delta_1 \gamma (3s) )
		(3s)^{-\vdim} \mu (3s)
	\end{gather*}
	by \ref{corollary:density_ratio_estimate}, one estimates, recalling the
	choice of the $s_i$,
	\begin{gather*}
		( 1 + \Delta_2 \gamma (3s) ) ( 1 - \gamma (3s) ) (3s)^{-\vdim}
		\mu (s) \leq ( 1 + \Delta_2 \gamma (t) ) ( 1 + \gamma (t) )
		t^{-\vdim} \mu (t), \\
		\begin{aligned}
			1 + \Delta_2 \gamma (3s) & \leq ( 1 + \Delta_2
			\gamma (t) ) ( 1 + \gamma (3s) ) ( 1 + \Delta_1
			\gamma(3s)) ( 1 - \gamma (3s) )^{-1} \\
			& \leq ( 1 + \Delta_2 \gamma (t) ) ( 1 + (\Delta_2/2)
			\gamma (3s) ), \\
			& \leq 1 + 2 \Delta_2 \gamma (t) + ( \Delta_2/2 )
			\gamma ( 3s),
		\end{aligned}
	\end{gather*}
	hence $\gamma ( 3s ) \leq 4\gamma (t)$ and one may take $\Gamma = 16
	\delta^{-2} \sup \{ 4 \Gamma_{\ref{thm:barycenter_estimate}}( \vdim,
	\vdim+1, 1 ) , 11 \}$.
\end{proof}
\begin{remark} \label{remark:diameter_estimate}
	Note by H\"older's inequality
	\begin{gather*}
		\gamma (s) \leq \sup \{ t \mu (t)^{-1/\vdim} \with 0 < t \leq
		s \} \big ( \tint{\cball{a}{s}}{} | \mathbf{h} (V;z) |^\vdim
		\| V \| z \big )^{1/\vdim} \quad \text{for $s \in I$}.
	\end{gather*}
\end{remark}
\begin{theorem} \label{thm:phantom_estimate}
	Suppose $1 < \vdim \in \nat$, $U$ is an open subset of $\rel^\adimI$,
	$V \in \Var_\vdim ( U )$, $\| \delta V \|$ is a Radon measure,
	\begin{gather*}
		\delta V ( g ) = - \tint{}{} g(z) \bullet \mathbf{h} (V;z) \ud
		\| V \| z \quad \text{whenever $g \in \mathscr{D} (U,
		\rel^\adimI )$}, \\
		\mathbf{h} (V;\cdot ) \in \Lp{\vdim} ( \| V \| \restrict K ,
		\rel^\adimI ) \quad \text{whenever $K$ is a compact subset of
		$U$}, \\
		\project{S} ( \mathbf{h} (V;z)) = 0 \quad \text{for $V$ almost
		all $(z,S)$},
	\end{gather*}
	and $C$ is the relation consisting of all $(b,Q) \in \spt \| V \|
	\times \orthproj{\adimI}{1}$ such that for some $0 < R < \infty$
	\begin{gather*}
		| z-b+Q^\ast (R) | \geq R \quad \text{whenever $z \in \spt \|
		V \|$}.
	\end{gather*}

	Then
	\begin{gather*}
		\mathscr{H}^\vdim ( C \lIm
		\classification{U}{z}{\density^\vdim ( \| V \|, z ) = \infty}
		\rIm ) = 0.
	\end{gather*}
\end{theorem}
\begin{proof}
	Suppose $0 < R < \infty$ and $\varepsilon > 0$. Let $A =
	\classification{U}{z}{ \density^\vdim ( \| V \|, z ) = \infty}$ and
	\begin{gather*}
		D = \eqclassification{\spt \| V \| \times
		\orthproj{\adimI}{1}}{(b,Q)}{ \text{$|z-b + Q^\ast (R) | \geq
		R$ for $z \in \spt \| V \|$}}.
	\end{gather*}
	Since $\| V \| (A) = 0$ by Allard \cite[5.1\,(4)]{MR0307015}, there
	exists an open subset $T$ of $U$ such that
	\begin{gather*}
		A \subset T, \quad \tint{T}{} 1 + | \mathbf{h} (V;z) |^\vdim
		\ud \| V \| z \leq \varepsilon.
	\end{gather*}
	Denoting by $\phi_\infty$ the size $\infty$ approximating measure
	occuring in the construction of $\mathscr{H}^\vdim$ on
	$\orthproj{\adimI}{1}$, consider the family $F$ of closed balls
	$\cball{a}{s} \subset T$ such that $0 < s \leq 1$ and
	\begin{gather*}
		\phi_\infty ( D \lIm A \cap \cball{a}{s} \rIm ) \leq
		\tint{\cball{a}{s/5}}{} 1 + | \mathbf{h} ( V;z ) |^\vdim \ud
		\| V \| z.
	\end{gather*}
	One verifies $A \cap \dmn D \subset \bigcup F$ by use of
	\ref{lemma:diameter_estimate} with $\lambda=1/5$ and
	\ref{remark:diameter_estimate}. Then \cite[2.8.5]{MR41:1976} yields
	the existence of a countable, disjointed subfamily $G$ of $F$ such
	that $\bigcup F \subset \bigcup_{S \in G} \widehat{S}$, hence
	\begin{gather*}
		\phi_\infty ( D \lIm A \rIm ) \leq \tsum{S \in G}{} \phi_\infty
		( D \lIm A \cap \widehat{S} \rIm ) \leq \tsum{S \in
		G}{} \tint{S}{} 1 + | \mathbf{h} ( V;z ) |^\vdim \ud \| V \| z
		\leq \varepsilon
	\end{gather*}
	and the conclusion follows.
\end{proof}
\begin{remark} \label{remark:perpendicularity}
	The condition $\project{S} ( \mathbf{h} (V;z)) = 0$ for $V$ almost all
	$(z,S)$ is satisfied if $V \in \IVar_\vdim ( U )$ by Brakke
	\cite[5.8]{MR485012}, see alternately \cite[4.8,\,9]{snulmenn.c2}.
\end{remark}
\begin{remark} \label{remark:nonempty}
	The author does not know of any example $V$ satisfying the hypotheses
	of the theorem such that $\classification{U}{z}{\density^\vdim ( \| V
	\|, z ) = \infty }$ is nonempty. In fact, the set is empty provided
	$\vdim =2$ and $V \in \IVar_\vdim (U)$ as shown by
	Kuwert and Sch\"atzle \cite[Appendix A]{MR2119722} even without the
	hypothesis $\vdim=\adim-1$.
\end{remark}
\section{Area formula for the generalised Gauss map}
Having established the relevant Lusin properties in Sections
\ref{sec:finite_density} and \ref{sec:infinite_density}, the area formula for
the generalised Gauss map (see \ref{thm:area_formula_codim1}) readily follows
from the author's previous second order rectifiability result
\cite[4.8]{snulmenn.c2}. It is worth noting that this result entirely concerns
the geometry of the support of the varifold in relation to Hausdorff measure.
In particular, the presence of points of density larger than one which is a
major source of the complexity of the geometry of such varifolds can be
ignored fully in applying the area formula \ref{thm:area_formula_codim1}.
\begin{theorem} \label{thm:area_formula_codim1}
	Suppose $\vdim$, $\adim$, $p$, $U$, and $V$ are as in
	\ref{miniremark:situation_general}, $2 \leq p = \vdim = \adim-1$, $V
	\in \IVar_\vdim ( U )$, $C$ is the relation consisting of all $(a,q)
	\in \spt \| V \| \times \mathbf{O}^\ast ( \adim, 1 )$ such that
	\begin{gather*}
		\oball{a-q^\ast (R)}{R} \cap \spt \| V \| = \emptyset,
	\end{gather*}
	and $B$ is an $\mathscr{H}^\vdim$ measurable subset of $C$.

	Then (see \ref{def:twice_diff})
	\begin{gather*}
		\begin{aligned}
			& \tint{\mathbf{O}^\ast ( \adim, 1 )}{} \mathscr{H}^0
			( \{ a \with (a,q) \in B \} ) \ud \mathscr{H}^\vdim q
			\\
			& \qquad = \tint{\spt \| V \|}{} \tint{
			\classification{\mathbf{O}^\ast ( \adim, 1)}{q}{(a,q)
			\in B}}{} | \discr ( q \circ \mathbf{b} ( \spt \| V
			\|; a ) ) | \ud \mathscr{H}^0 q \ud \mathscr{H}^\vdim
			a.
		\end{aligned}
	\end{gather*}
\end{theorem}
\begin{proof}
	Abbreviate $S = \spt \| V \|$. Note $\density_\ast^\vdim ( \| V \|, a)
	\geq 1$ for $a \in S$ by \cite[2.7]{snulmenn.isoperimetric}, hence $\|
	V \| (A) = 0$ if and only if $\mathscr{H}^\vdim ( S \cap A ) = 0$ by
	\cite[2.10.19\,(3)]{MR41:1976} and Allard \cite[3.5\,(1)]{MR0307015}.

	From \cite[4.8]{snulmenn.c2} one obtains a sequence of $\vdim$
	dimensional submanifolds $M_i$ of $\rel^\adim$ of class $\class{2}$
	such that $\| V \| ( U \without \bigcup_{i=1}^\infty M_i ) = 0$.
	Define $A_i$ to be the set of $a \in S \cap M_i$ such that
	\begin{gather*}
		\mathbf{b} ( M_i ; a ) = \mathbf{b} (S;a),
	\end{gather*}
	and let $A = S \without \bigcup_{i=1}^\infty A_i$. From
	\ref{thm:locality_mean_curvature}\,\eqref{item:locality_mean_curvature:differentiability}
	and \ref{corollary:second_order_agreement} one infers $\| V \| ( ( S
	\cap M_i) \without A_i ) = 0$, hence $\mathscr{H}^\vdim (A) =0$.
	Moreover, denote by $N_i$ the set of $(a,q) \in M_i \times
	\mathbf{O}^\ast ( \adim, 1 )$ such that $\Tan (M_i,a) \subset \ker q$.

	Next, two special cases are considered. First, \emph{the case $B
	\subset C | A$}. Then the right integral equals $0$. From
	\ref{lemma:zero_sets_to_zero_sets} and \ref{thm:phantom_estimate} in
	conjunction with \ref{remark:perpendicularity} one obtains
	$\mathscr{H}^\vdim ( C \lIm A \rIm ) = 0$, hence also the left
	integral vanishes. Second, \emph{the case $B \subset C |A_i$ for some
	$i$}. If $(a,q) \in B$ then the differentiability of $S$ at $a$
	implies, since $\Tan (S,a) = \Tan (M_i,a)$ as $a \in A_i$, that
	\begin{gather*}
		\Tan (S,a) \subset \{ z \with q(z) \geq 0 \}, \quad \Tan (
		M_i, a ) = \ker q, \quad (a,q) \in N_i.
	\end{gather*}
	Therefore
	\ref{lemma:unit_sphere_bundle}\,\eqref{item:unit_sphere_bundle:fubini}
	implies
	\begin{gather*}
		\begin{aligned}
			& \tint{\mathbf{O}^\ast ( \adim,1 )}{} \mathscr{H}^0 (
			\{ a \with (a,q) \in B \} ) \ud \mathscr{H}^\vdim q \\
			& \qquad = \tint{M_i}{} \tint{
			\classification{\mathbf{O}^\ast ( \adim, 1 )}{q}{(a,q)
			\in B} }{} | \discr ( q \circ \mathbf{b} ( M_i ; a ) )
			| \ud \mathscr{H}^0 q \ud \mathscr{H}^\vdim a \\
			& \qquad = \tint{S}{} \tint{
			\classification{\mathbf{O}^\ast ( \adim, 1 )}{q}{(a,q)
			\in B} }{} | \discr ( q \circ \mathbf{b} ( S ; a ) ) |
			\ud \mathscr{H}^0 q \ud \mathscr{H}^\vdim a.
		\end{aligned}
	\end{gather*}

	Since $A_i$ is $\mathscr{H}^\vdim$ almost equal to $S \cap M_i$, the
	set $A_i$ is $\mathscr{H}^\vdim$ measurable. Using the structure of
	$N_i$ described in
	\ref{lemma:unit_sphere_bundle}\,\eqref{item:unit_sphere_bundle:manifold}\,\eqref{item:unit_sphere_bundle:tangent_space},
	one observes that this implies the $\mathscr{H}^\vdim$ measurability
	of $\classification{N_i}{(a,q)}{a \in A_i}$. Now, one readily combines
	the preceding two cases to obtain the conclusion.
\end{proof}
\section{A sharp lower bound on the mean curvature integral with critical
power for integral varifolds}
The area formula, see \ref{thm:area_formula_codim1}, permits to prove the
sharp lower bound on the mean curvature integral with critical power for
integral varifolds in \ref{thm:lowerbound_hm} along the lines of the
classic argument presented in \ref{thm:classical}.
\begin{theorem} \label{thm:lowerbound_hm}
	Suppose $\vdim$, $\adim$, $p$, $U$, and $V$ are as in
	\ref{miniremark:situation_general}, $2 \leq p = \vdim = \adim-1$, $0
	\neq V \in \IVar_\vdim ( U )$ and $\spt \| V \|$ is compact.

	Then
	\begin{gather*}
		( \vdim+1 ) \unitmeasure{\vdim+1} \vdim^\vdim \leq \tint{}{} |
		\mathbf{h} ( V ; \cdot )|^\vdim \ud \| V \|.
	\end{gather*}
\end{theorem}
\begin{proof}
	Abbreviate $S = \spt \| V \|$. Define $B$ to be the closed set
	of all $(a,q) \in S \times \mathbf{O}^\ast ( \adim,1 )$ such that
	\begin{gather*}
		q (z-a) \geq 0 \quad \text{whenever $z \in S$}.
	\end{gather*}
	Since for $q \in \mathbf{O}^\ast ( \adim, 1 )$ there exists $a \in S$
	with $q (a) = \inf q \lIm S \rIm$ hence $(a,q) \in B$, one notes
	\begin{gather*}
		\mathscr{H}^0 ( \{ a \with (a,q) \in B \} ) \geq 1 \quad
		\text{whenever $q \in \mathbf{O}^\ast ( \adim, 1 )$}.
	\end{gather*}
	If $(a,q) \in B$ and $S$ is twice differentiable at $a$ with $\dim
	\Tan ( S,a ) = \vdim$ then
	\begin{gather*}
		\Tan (S,a) \subset \ker q, \qquad q \circ \mathbf{b} ( S;a )
		(v,v) \geq 0 \quad \text{for $v \in \Tan (S,a)$};
	\end{gather*}
	in fact with $\tau$, $\phi_r$, and $Q$ as \ref{def:twice_diff} with
	$A$ replaced by $S$ one infers for $0 < r < \infty$
	\begin{gather*}
		\boldsymbol{\mu}_{1/r} \circ \boldsymbol{\tau}_{-a} \lIm S \rIm
		\subset \classification{\rel^\adim}{z}{q(z) \geq 0}, \quad
		\Tan (S,a) = \ker q, \\
		\phi_r \circ \boldsymbol{\tau}_{-a} \lIm S \rIm \subset
		\classification{\rel^\adim}{z}{q(z) \geq 0}, \quad \tau \lIm Q
		\rIm \subset \classification{\rel^\adim}{z}{q(z) \geq 0},  \\
		q \circ \mathbf{b} (S;a) (v,v) = D^2 (q \circ Q ) (0) (v,v)
		\geq 0 \quad \text{for $v \in \Tan (S,a)$}.
	\end{gather*}
	Recalling $\density_\ast^\vdim ( \| V \|, a ) \geq 1$ for $a \in S$
	from \cite[2.7]{snulmenn.isoperimetric} and noting
	\cite[2.10.19\,(3)]{MR41:1976} one then obtains from
	\ref{miniremark:det_trace_inq} in conjunction with
	\ref{corollary:second_order_agreement}
	\begin{gather*}
		0 \leq \discr ( q \circ \mathbf{b} ( S; a ) ) \leq
		\vdim^{-\vdim} | \mathbf{h} (V;a) |^\vdim \quad \text{whenever
		$(a,q) \in B$}
	\end{gather*}
	for $\mathscr{H}^\vdim$ almost all $a \in S$. Next, it will be shown
	\begin{gather*}
		\mathscr{H}^0 ( \{ q \with (a,q) \in B \} ) \leq 1 \quad
		\text{for $\mathscr{H}^\vdim$ almost all $a \in S$}.
	\end{gather*}
	In fact since $\Tan (S,a) \subset \classification{\rel^\adim}{z}{q(z)
	\geq 0}$ for $(a,q) \in B$ and $\Tan (S,a)$ is an $\vdim$ dimensional
	subspace for $\mathscr{H}^\vdim$ almost all $a \in S$ one only needs
	to exclude that for some $a \in S$ both $(a,q)$ and $(a,-q)$ belong to
	$B$. If this were the case $S$ would be contained in an $\vdim$
	dimensional affine subspace of $\rel^\adim$, hence $\mathbf{h} (V;z) =
	0$ for $\| V \|$ almost all $z$ by \cite[4.8]{snulmenn.c2} and $V$
	would be stationary in contradiction to the constancy theorem, see
	Allard \cite[4.6\,(3)]{MR0307015}.

	Therefore one obtains with the help of \ref{thm:area_formula_codim1}
	\begin{gather*}
		\begin{aligned}
			\mathscr{H}^\vdim ( \mathbf{O}^\ast ( \adim, 1 ) )
			& \leq \tint{\mathbf{O}^\ast ( \adim, 1 )}{}
			\mathscr{H}^0 ( \{ a \with (a,q) \in B \} ) \ud
			\mathscr{H}^\vdim q \\
			& \leq \tint{S}{}
			\tint{\classification{\mathbf{O}^\ast ( \adim, 1
			)}{q}{ (a,q) \in B } }{}  | \discr ( q \circ \mathbf{b}
			(S;a) ) | \ud \mathscr{H}^0 q \ud \mathscr{H}^\vdim a
			\\
			& \leq \vdim^{-\vdim} \tint{}{} | \mathbf{h} ( V ; \cdot
			) |^\vdim \ud \| V \|.
		\end{aligned}
	\end{gather*}
	Observing that equality holds in the preceding estimate in case $V =
	\var ( \mathbf{S}^\vdim ) \in \IVar_\vdim ( \rel^\adim )$, one readily
	infers the conclusion.
\end{proof}
\appendix
\section{Monotonicity identity}
This appendix gathers some simple consequences of the monotonicity identity.
\begin{theorem} [Monotonicity identity, second version] \label{thm:monotonicity_identity}
	Suppose $\vdim$, $\adim$, $U$, $V$, and $\eta$ are as in
	\ref{miniremark:situation_general_varifold}, and
	\begin{gather*}
		\mu (s) = \measureball{\| V \|}{\cball{a}{s}}, \quad \nu (s) =
		\tint{\cball{a}{s}}{} (z-a) \bullet \eta ( V ; z ) \ud \|
		\delta V \| z, \\
		\xi (s) = \tint{\cball{a}{s} \times \grass{\adim}{\vdim}}{} |
		\perpproject{S} (z-a) |^2 \ud V (z,S)
	\end{gather*}
	whenever $0 < s \leq r$.
	
	Then there holds
	\begin{gather*}
		s^{-\vdim} \mu (s) + \tint{s}{t} u^{-\vdim-2}
		\ud_u \xi (u) = t^{-\vdim} \mu (t) + \tint{s}{t} u^{-\vdim-1}
		\nu (u) \ud \mathscr{L}^1 u
	\end{gather*}
	whenever $0 < s \leq t \leq r$.
\end{theorem}
\begin{proof}
	This readily follows from \ref{thm:monotonicity}.
\end{proof}
\begin{lemma} \label{lemma:monotonicity}
	Suppose $\vdim, \adim \in \nat$, $\vdim \leq \adim$, $a \in
	\rel^\adim$, $0 < r < \infty$, $V \in \Var_\vdim ( \oball{a}{r} )$ and
	$\| \delta V \|$ is a Radon measure.

	Then there holds
	\begin{gather*}
		s^{-\vdim} \measureball{\| V \|}{\cball{a}{s}} \leq t^{-\vdim}
		\measureball{\| V \|}{\cball{a}{t}} + \tint{s}{t} u^{-\vdim}
		\measureball{\| \delta V \|}{\cball{a}{u}} \ud \mathscr{L}^1 u
	\end{gather*}
	whenever $0 < s < t < r$.
\end{lemma}
\begin{proof}
	Let $I = \{ u \with 0 < u < r \}$ and define $f : I \to \rel$
	by $f (u) = u^{-\vdim} \measureball{\| V \|}{ \cball{a}{u} }$ for $u
	\in I$. Observing that \ref{thm:monotonicity} implies
	\begin{gather*}
		\tint{0}{r} \zeta' f \ud \mathscr{L}^1 \leq \tint{s}{t}
		u^{-\vdim} \measureball{\| \delta V \|}{ \cball{a}{u} } \ud
		\mathscr{L}^1 u
	\end{gather*}
	whenever $\zeta \in \mathscr{D}^0 ( I )$ with $\spt \zeta \subset \{ u
	\with s \leq u \leq t \}$ and $| \zeta (u) | \leq 1$ for $u \in I$,
	one may deduce the inequality in question provided $s$ and $t$ are
	$\mathscr{L}^1 \restrict I$ Lebesgue points of $f$ by letting $\zeta$
	approach the characteristic function of the interval $\{ u \with s
	\leq u \leq t \}$. Approximating $s$ and $t$ from above, the
	conclusion follows.
\end{proof}
\begin{lemma} \label{lemma:lower_density}
	Suppose $\adim \in \nat$ and $0 < \delta \leq 1$.

	Then there exists a positive, finite number $\Gamma$ with the
	following property.

	If $\vdim$, $\adim$, $U$, $V$, $p$, and $\psi$ are related to $\adim$
	as in \ref{miniremark:situation_general}, $a \in \spt \| V \|$, $0 < r <
	\infty$, $U = \oball{a}{r}$, $p = \vdim$, and
	$\measureball{\psi}{\oball{a}{r}} \leq \Gamma^{-1}$, then
	\begin{gather*}
		\measureball{\| V \|}{\oball{a}{r}} \geq (1-\delta)
		\unitmeasure{\vdim} r^\vdim.
	\end{gather*}
\end{lemma}
\begin{proof}
	This is a special case of \cite[2.6]{snulmenn.isoperimetric}.
\end{proof}
\begin{miniremark} \label{miniremark:situation}
	The following hypotheses will be used frequently.
	\begin{enumerate}
		\item \label{item:situation:basic} Suppose $\vdim, \adim \in
		\nat$, $1 < \vdim < \adim$, $U$ is an open subset of
		$\rel^\adim$, $a \in \rel^\adim$, $0 < r < \infty$,
		$\cball{a}{r} \subset U$, $V \in \Var_\vdim ( U )$, $\| \delta
		V \|$ is a Radon measure, $\eta ( V ; \cdot ) : U \to
		\mathbf{S}^{\adim-1}$ is a $\| \delta V \|$ measurable
		function with
		\begin{gather*}
			( \delta V ) ( g ) = \tint{}{} g (z) \bullet \eta ( V;
			z ) \ud \| \delta V \| z \quad \text{whenever $g \in
			\mathscr{D} ( U, \rel^\adim)$},
		\end{gather*}
		see Allard \cite[4.3]{MR0307015}, $I = \{ s : 0 < s \leq r
		\}$, and the functions $\mu : I \to \rel$, $\nu : I \to \rel$
		and $\xi : I \to \rel$ are given by
		\begin{align*}
			\mu (s) & = \measureball{\| V \|}{\cball{a}{s}}, \\
			\nu (s) & = \tint{\cball{a}{s}}{} (z-a) \bullet \eta (
			V ; z ) \ud \| \delta V \| z, \\
			\xi (s) & = \tint{\cball{a}{s} \times
			\grass{\adim}{\vdim}}{} | \perpproject{S} (z-a) |^2
			\ud V (z,S).
		\end{align*}
		\item \label{item:situation:l2} Suppose, additionally to
		\eqref{item:situation:basic},
		\begin{gather*}
			( \delta V ) ( g ) = - \tint{}{} \mathbf{h} ( V; z )
			\bullet g (z) \ud \| V \| z \quad \text{for $g \in
			\mathscr{D} ( U, \rel^\adim )$}, \\
			\mathbf{h} ( V; \cdot ) \in \Lp{2} ( \| V \| \restrict
			K, \rel^\adim ) \quad \text{whenever $K$ is a compact
			subset of $U$},
		\end{gather*}
		$\zeta = \mu - \vdim^{-1} \nu$, and $\phi : I \to \rel$ is
		defined by $\phi (s) = \tint{\cball{a}{s}}{} | \mathbf{h}
		(V;z) |^2 \ud \| V \| z$ for $s \in I$.
		\item \label{item:situation:perpendicular} Suppose,
		additionally to \eqref{item:situation:basic} and
		\eqref{item:situation:l2},
		\begin{gather*}
			\project{S} ( \mathbf{h} (V;z) ) = 0 \quad \text{for
			$V$ almost all $(z,S)$}.
		\end{gather*}
	\end{enumerate}
\end{miniremark}
\begin{remark} \label{remark:basic_integrability}
	A basic consequence is the following. If $0 < \delta < \infty$ then
	\begin{align*}
		& \tint{\cball{a}{r} \without \cball{a}{s}}{}
		|z-a|^{-\vdim+\delta} \ud \| V \| z = \tint{s}{r}
		t^{-\vdim+\delta} \ud_t \mu (t) \\
		& \quad \leq r^{-\vdim+\delta} \mu (r) + ( \vdim-\delta)
		\tint{s}{r} t^{-\vdim-1+\delta} \mu (t) \ud \mathscr{L}^1 t \\
		& \quad \leq r^{-\vdim+\delta} \mu (r) \big ( 1 + \sup \{0,
		\vdim-\delta\} \tint{0}{1} t^{-1+\delta} ( 1 + \vdim^{-1}
		\beta \log (1/t))^\vdim \ud \mathscr{L}^1 t \big ) < \infty.
	\end{align*}
\end{remark}

\section{An observation concerning the Lusin property}
\subsection{Notation}
The notation of \cite[\S 1,\,\S 2]{snulmenn.decay} is used. The reader
might want to recall in particular the following less commonly used symbols:
$\nat$ denoting the positive integers, $\oball{a}{r}$ and $\cball{a}{r}$
denoting respectively the open and closed ball with centre $a$ and radius $r$,
and $\bigodot^i ( V,W )$ and $\bigodot^i V$ denoting the vector space of all
$i$ linear symmetric functions (forms) mapping $V^i$ into $W$ and $\rel$
respectively, see \cite[2.2.6, 2.8.1, 1.10.1]{MR41:1976}.
\subsection{A model case}
\begin{lemma} \label{lemma:simple_interpolation}
	Suppose $k, \vdim, \adim \in \nat$, $\vdim < \adim$, $a \in
	\rel^\vdim$, $0 < r < \infty$, and $u : \oball{a}{r} \to \rel^\codim$
	is of class $k$.

	Then
	\begin{gather*}
		\tsum{i=0}{k} r^i \norm{D^iu}{\infty}{a,r} \leq \Gamma \big (
		r^k \norm{D^ku}{\infty}{a,r} + r^{-\vdim} \norm{u}{1}{a,r}
		\big )
	\end{gather*}
	where $\Gamma$ is a positive, finite number depending only on $k$ and
	$\adim$.
\end{lemma}
\begin{proof}
	See \cite[3.15]{snulmenn.decay}.
\end{proof}
\begin{lemma} \label{lemma:touching-appendix}
	Suppose $\vdim \in \nat$, $a \in \rel^\vdim$, $0 < r < \infty$, $u \in
	\Lp{1} ( \mathscr{L}^\vdim \restrict \oball{a}{r} )$, $u \geq 0$, $0
	\leq \gamma < \infty$ and $P_r : \rel^\vdim \to \rel$ for $0 < s \leq
	r$ are affine functions with $\lim_{s \to 0+} P_s (a) = 0$ and
	\begin{gather*}
		s^{-\vdim} \norm{u-P_s}{1}{a,s} \leq \gamma s \quad
		\text{whenever $0 < s \leq r$}.
	\end{gather*}

	Then
	\begin{gather*}
		| P_s (0) | / s + | D P_s (0) | \leq \Gamma \gamma \quad
		\text{for $0 < s \leq r$}
	\end{gather*}
	where $\Gamma$ is a positive, finite number depending only on $\vdim$.
\end{lemma}
\begin{proof}
	Whenever $0 < s/2 \leq t \leq s \leq r$
	\begin{gather*}
		| P_s (a) - P_t (a) | \leq \Delta_1 t^{-\vdim}
		\norm{P_s-P_t}{1}{a,t} \leq 2^{\vdim+1} \Delta_1 \gamma s
	\end{gather*}
	by \ref{lemma:simple_interpolation} where $\Delta_1 =
	\Gamma_{\ref{lemma:simple_interpolation}} ( 2, \vdim + 1 )$, hence,
	using induction,
	\begin{gather*}
		| P_s (a) | = \lim_{t \to 0+} | P_s (a) - P_t (a) | \leq
		2^{\vdim+2} \Delta_1 \gamma s \quad \text{for $0 < s \leq r$}.
	\end{gather*}
	Taking $0 < \Delta_2 < \infty$ such that $| \alpha | = \Delta_2
	\int_{\oball{0}{1}} \alpha^+ \ud \mathscr{L}^\vdim$ for $\alpha
	\in \Hom ( \rel^\vdim, \rel )$, one infers for $0 < r \leq s$ that
	\begin{gather*}
		\begin{aligned}
			s | D P_s (a) | & = \Delta_2 s^{-\vdim}
			\tint{\oball{a}{s}}{} \sup \{ 0, P_s (a) - P_s (x) \}
			\ud \mathscr{L}^\vdim x \\
			& \leq \Delta_2 s^{-\vdim} \tint{\oball{a}{s}}{} | P_s
			(a) + ( u-P_s ) (x) | \ud \mathscr{L}^\vdim x \\
			& \leq \Delta_2 \unitmeasure{\vdim} | P_s (a) | +
			\Delta_2 s^{-\vdim} \norm{u-P_s}{1}{a,s} \leq \Delta_3
			\gamma s
		\end{aligned}
	\end{gather*}
	where $\Delta_3 = \Delta_2 ( \unitmeasure{\vdim} 2^{\vdim+2}
	\Delta_1 + 1)$ and one may take $\Gamma=2^{\vdim+1} \Delta_1 +
	\Delta_3$.
\end{proof}
\begin{lemma} \label{lemma:oscillation_bound}
	Suppose $\vdim \in \nat$, $a \in \rel^\vdim$, $0 < r < \infty$, $u
	\in \Sob{}{2}{\vdim} ( \oball{a}{r} )$ with $\weakD^0 u = u$, $0 \leq
	\gamma < \infty$, and $Q : \rel^\vdim \to \rel$ is a polynomial
	function of degree at most $2$ such that
	\begin{gather*}
		u-Q \geq 0, \qquad u(a) = Q(a), \qquad \left < v^2, D^2 Q (a)
		\right > \geq - \gamma | v |^2 \quad \text{for $v \in
		\rel^\vdim$}.
	\end{gather*}

	Then $Du(a) = DQ (a)$ and for $P : \rel^\vdim \to \rel$ defined by
	$P(x) = Q(a) + \left < x-a, DQ(a) \right >$ for $x \in \rel^\vdim$
	\begin{gather*}
		r^{-1} \norm{u-P}{\infty}{a,r} + r^{-1} \norm{\weakD
		(u-P)}{\vdim}{a,r} \leq \Gamma \big ( \tint{\oball{a}{r}}{}
		(\gamma + | \weakD^2 u (x) |)^\vdim \ud \mathscr{L}^\vdim x
		\big)^{1/\vdim}
	\end{gather*}
	where $\Gamma$ is a positive, finite number depending only on $\vdim$.
\end{lemma}
\begin{proof}
	Possibly replacing $u(x)$ by $u(x)-P(x)-\gamma|x-a|^2/2$, one
	can assume $\gamma = 0$ and $Q = 0$. Note also that the first part of
	the conclusion follows from the second part as $u : \oball{a}{r} \to
	\rel$ is continuous.
	
	Twice applying Sobolev Poincar\'e's inequality, one obtains affine
	functions $P_s : \rel^\vdim \to \rel$ for $0 < s \leq r$ such that
	\begin{gather*}
		s^{-1} \norm{u-P_s}{\infty}{a,s} + s^{-1}
		\norm{\weakD(u-P_s)}{\vdim}{a,s} \leq \Delta_1 \norm{\weakD^2
		u}{\vdim}{a,s}
	\end{gather*}
	where $\Delta_1$ is a positive, finite number depending only on
	$\vdim$. Hence by \ref{lemma:touching-appendix}
	\begin{gather*}
		| P_r(0)|/r + |DP_r(0)| \leq \Delta_2 \norm{\weakD^2
		u}{\vdim}{a,r}
	\end{gather*}
	where $\Delta_2 = \Delta_1 \unitmeasure{\vdim}
	\Gamma_{\ref{lemma:touching-appendix}} ( \vdim )$ and one may take
	$\Gamma = \Delta_1 + (1+\unitmeasure{\vdim}^{1/\vdim}) \Delta_2$ in
	the remaining assertion.
\end{proof}
\begin{theorem}
	Suppose $\vdim \in \nat$, $U$ is an open subset of $\rel^\vdim$, $u
	\in \Sob{}{2}{\vdim} ( U )$ with $\weakD^0 u = u$, $A$ is the set of
	all $a \in U$ such that for some polynomial function $Q : \rel^\vdim
	\to \rel$ of degree at most $2$ there holds $u - Q \geq 0$ and $u (a)
	= Q(a)$.

	Then $A \subset \dmn Du$ and
	\begin{gather*}
		\mathscr{H}^\vdim ( Du \lIm E \rIm ) = 0 \quad \text{whenever
		$E\subset A$ with $\mathscr{L}^\vdim ( E ) = 0$}.
	\end{gather*}
\end{theorem}
\begin{proof}
	The first part of the conclusion follows from
	\ref{lemma:oscillation_bound}.

	Define $A_i$ for $i \in \nat$ to be set of all $a \in A$ such that for
	some polynomial function $Q : \rel^\vdim \to \rel$ of degree at most
	$2$
	\begin{gather*}
		u-Q \geq 0, \qquad u(a) = Q (a), \qquad \left < v^2, D^2 Q (a)
		\right > \geq - i | v |^2 \quad \text{for $v \in \rel^\vdim$}
	\end{gather*}
	and note $A = \bigcup \{ A_i \with i \in \nat \}$. Moreover, define
	for $a \in A$ the $1$ jet $P_a : \rel^\vdim \to \rel$ of $u$ at $a$ by
	$P_a (x) = u(a) + \left < x-a, Du (a) \right >$ for $x \in \rel^\vdim$
	and $f_i (x) = (i + | \weakD^2 u (x) |)^\vdim$ whenever $i \in \nat$ and
	$x \in \dmn \weakD^2u$.

	Next, it will be shown with $\Delta = 2^{\vdim+1}
	\unitmeasure{\vdim}^{-1/\vdim} \Gamma_{\ref{lemma:oscillation_bound}} (
	\vdim )$
	\begin{gather*}
		\diam Du \lIm A_i \cap \cball{a}{r} \rIm \leq \Delta
		\norm{f_i}{1}{a,3r}^{1/\vdim}
	\end{gather*}
	whenever $i \in \nat$, $a \in A_i$, $0 < r < \infty$ and 
	$\cball{a}{3r} \subset U$. For this purpose let
	$b,c \in A_i \cap \cball{a}{r}$ with $b \neq c$ and note with
	$s = | b-c |/2$ and $d = \frac{1}{2} (b+c)$
	\begin{gather*}
		\cball{d}{s} \subset \cball{b}{2s} \cap \cball{c}{2s}, \quad
		\cball{b}{2s} \cup \cball{c}{2s} \subset \cball{a}{3r}.
	\end{gather*}
	From \ref{lemma:oscillation_bound} one infers
	\begin{gather*}
		\begin{aligned}
			& s^{-\vdim} \sup \{ \norm{\weakD (u-P_b)}{1}{b,2s},
			\norm{\weakD (u-P_c)}{1}{c,2s} \} \\
			& \qquad \leq 2^{\vdim}
			\unitmeasure{\vdim}^{1-1/\vdim}
			\Gamma_{\ref{lemma:oscillation_bound}} ( \vdim )
			\norm{f_i}{1}{a,3r}^{1/\vdim},
		\end{aligned}
	\end{gather*}
	hence
	\begin{gather*}
		\begin{aligned}
			|Du(b)-Du(c)| & = | D(P_b-P_c)(d) | \\
			& = \unitmeasure{\vdim}^{-1} s^{-\vdim}
			\norm{D(P_b-P_c)}{1}{d,s} \leq \Delta
			\norm{f_i}{1}{a,3r}^{1/\vdim}.
		\end{aligned}
	\end{gather*}

	Now, suppose $E \subset A$ with $\mathscr{L}^\vdim ( E ) = 0$ and
	$\varepsilon > 0$. Assume $E \subset A_i$ for some $i \in \nat$ and
	choose an open subset $V$ of $\rel^\vdim$ such that $E \subset V
	\subset U$ and $\int_V f_i \ud \mathscr{L}^\vdim \leq
	\varepsilon$. Define $h : V \to \rel$ by
	\begin{gather*}
		h(x) = {\textstyle\frac{1}{30}} \inf \{ 1, \dist (x,
		\rel^\vdim \without V ) \} \quad \text{for $x \in V$}.
	\end{gather*}
	Using \cite[2.8.5]{MR41:1976}, one obtains a countable subset $S$ of
	$E$ such that
	\begin{gather*}
		E \subset {\textstyle\bigcup} \{ \cball{s}{5h(s)} \with s \in
		S \}, \quad \text{$\{ \cball{s}{h(s)} \with s \in S \}$ is
		disjointed}.
	\end{gather*}
	From \cite[3.1.12]{MR41:1976} one infers
	\begin{gather*}
		\card ( \classification{S}{s}{\cball{s}{15h(s)} \cap
		\cball{x}{15h(x)} \neq \emptyset} ) \leq (189)^\vdim \quad
		\text{for $x \in V$}.
	\end{gather*}
	This yields, together with the assertion of the preceding paragraph,
	that
	\begin{gather*}
		Du \lIm E \rIm \subset {\textstyle\bigcup} \{
		Du \lIm A_i \cap \cball{s}{5h(s)} \rIm \with s \in S \}, \\
		\tsum{s \in S}{} ( \diam Du \lIm A_i \cap \cball{s}{5h(s)} \rIm
		)^\vdim \leq \Delta^\vdim \tsum{s \in S}{}
		\norm{f_i}{1}{s,15h(s)} \leq (189\Delta)^\vdim \varepsilon
	\end{gather*}
	since $\cball{s}{15h(s)} \subset V$ for $s \in S$, and the conclusion
	follows.
\end{proof}
\begin{remark}
	The diameter estimate is termed a ``Rado-Reichelderfer'' condition in
	Mal\'y \cite[\S 3]{MR1669167} and Hencl \cite[\S 5]{MR1924813}. It is
	employed analogously to Mal\'y \cite[Theorem 3.4]{MR1669167} and
	Hencl \cite[Theorems 5.1 and 3.5]{MR1924813}.
\end{remark}

\paragraph{Acknowledgements} The author would like to thank Prof.~Dr.~Gerhard
Huisken for suggesting the investigation of the validity of Theorem
\ref{thm:lowerbound_hm} which was the intitial motivation for this paper.
Moreover, the author would like to thank Dr.~S{\l}awomir Kolasi{\'n}ski and
Dr.~Leobardo Rosales for several conversations concerning Part
\ref{part:weakly} of this paper.

\newpage

{\small \noindent Max Planck Institut for Gravitational Physics (Albert
Einstein Institute) \newline Am M{\"u}hlen\-berg 1, 14476 Potsdam, Germany
\newline \texttt{Ulrich.Menne@aei.mpg.de} \medskip \newline University of
Potsdam, Institute for Mathematics, \newline Am Neuen Palais 10, 14469
Potsdam, Germany \newline \texttt{Ulrich.Menne@math.uni-potsdam.de} }

\addcontentsline{toc}{section}{\numberline{}References}
\bibliographystyle{halpha}
\newcommand{\etalchar}[1]{$^{#1}$}
\newcommand{\noopsort}[1]{} \newcommand{\singleletter}[1]{#1} \def\cprime{$'$}
  \def\cprime{$'$}

\end{document}